\documentclass[10pt]{article}

\usepackage{authblk}

\usepackage[a4paper,top=3cm,bottom=2cm,left=2.3cm,right=2.3cm,marginparwidth=1.75cm]{geometry}

\usepackage{bbm}
\usepackage{amsmath}
\usepackage{amscd}
\usepackage{amsthm}
\usepackage{amssymb}
\usepackage{verbatim}
\usepackage{mathtools}
\usepackage{enumitem}
\usepackage[colorlinks]{hyperref}
\hypersetup{%
	colorlinks = true,
	linkcolor  = black
}
\usepackage{xcolor}
\usepackage{txfonts}
\usepackage{lscape}
\usepackage{mathabx}
\theoremstyle{plain}
\newtheorem{theorem}{Theorem}[section]
\newtheorem{lemma}[theorem]{Lemma}
\newtheorem{proposition}[theorem]{Proposition}
\newtheorem{corollary}[theorem]{Corollary}
\newtheorem{remark}[theorem]{Remark}

\theoremstyle{definition}
\newtheorem{definition}[theorem]{Definition}

\newtheorem{assumption}[theorem]{Assumption}

\newcommand{\N}{\mathbb{N}}
\newcommand{\Z}{\mathbb{Z}}

\newcommand{\R}{\mathbb{R}}
\newcommand{\C}{\mathbb{C}}
\newcommand{\A}{\mathbb{A}}

\newcommand{\simgrad}{\sym\nabla}

\newcommand{\eps}{{\varepsilon}}
\newcommand{\weak}{{\rightharpoonup}}

\DeclareMathOperator{\sym}{sym}

\newcommand{\vect}[1]{\boldsymbol #1}

\newcommand{\BBB}{\color{black}}
 
\newcommand{\PPP}{\color{blue}} 
\def\e{\varepsilon}

\let\oldsqrt\sqrt
\def\sqrt{\mathpalette\DHLhksqrt}
\def\DHLhksqrt#1#2{%
\setbox0=\hbox{$#1\oldsqrt{#2\,}$}\dimen0=\ht0
\advance\dimen0-0.2\ht0
\setbox2=\hbox{\vrule height\ht0 depth -\dimen0}%
{\box0\lower0.4pt\box2}}

\AtBeginDocument{
  \let\div\relax
  \DeclareMathOperator{\div}{div}
}

\usepackage{float}
\usepackage{caption}
\usepackage{subcaption}

\newcommand*{\nsection}[1]{
	\section*{#1}
	\addcontentsline{toc}{section}{#1}
}

\begin{document}

\title{\sc Effective behaviour of critical-contrast PDEs: micro-resonances, frequency conversion, and time dispersive properties. II.}

\def\correspondingauthor{\footnote{Corresponding author: k.cherednichenko@bath.ac.uk}}

\author[1]{Kirill Cherednichenko\correspondingauthor{}}
\author[1]{Alexander V. Kiselev}
\author[2]{Igor Vel\v{c}i\'{c}\,}
\author[2]{Josip Žubrinić}
\affil[1]{Department of Mathematical Sciences, University of Bath, Claverton Down, Bath,\qquad\qquad  BA2 7AY, United Kingdom}
\affil[2]{Faculty of Electrical Engineering and Computing, University of Zagreb, Unska 3,\qquad\qquad 10000 Zagreb, Croatia}
\maketitle

\begin{abstract}
\noindent We construct an order-sharp theory for a double-porosity model in the full linear elasticity setup. Crucially, we uncover time and frequency dispersive properties of highly oscillatory elastic composites.

 \vskip 0.5cm

\noindent {\bf Keywords:} Elasticity theory $\cdot$ Highly oscillatory media $\cdot$ Resolvent asymptotics $\cdot$ Time and frequency dispersion

\vskip 0.5cm

\noindent {\bf Mathematics Subject Classification (2020):} Primary 35Q99; Secondary 47F05, 47N50, 35B27, 47A10, 81U30

\end{abstract}
\tableofcontents

\section{Introduction}
Quantitative asymptotic methods in the analysis of parameter-dependent families of PDEs,
see e.g.  \cite{Zhikov_1989, BirmanSuslina, Gri04, Zhikov_Pastukhova, Kenig, CoopKamSmysh}, serve as a natural replacement of the classical ad hoc asymptotic approach, which is known to lead to errors, as pointed out in, e.g., \cite{CC_Maxwell_2015, CherCoop, CK, CherDO}. Their key feature is the pursuit of an estimate, in a uniform operator topology, on the difference between the ``exact" (usually inaccessible) solution 
and its asymptotic approximation.
 This problem formulation has brought forth the physically relevant possibility to account for degenerate problems (such as the ``double porosity", ``flipped double porosity" \cite{CherErKis}, and thin network \cite{CEK_networks} setups).
 The principal goal of the quantitative analysis is to develop a rigorous mathematical framework for metamaterials \cite{Veselago, Capolino}. It can be argued, see e.g. the discussion in \cite{CEKRS2022}, that generically metamaterial-like behaviour is accounted for by the ``corrector" terms in the asymptotic expansion \cite{CherErKis}. Indeed, if one assumed that the family admitted a ``limiting" operator in a strong enough topology, then the latter must inherit the positive-definiteness of the original formulation, which would not permit the ``negative" effects expected of a metamaterial. This calls for quantitatively tight asymptotic expansions 
 capturing the key features of the medium at hand.

  In connection with this goal, the operator theory has emerged as a source of powerful tools, a subset of which is based on the analysis of resolvents.
  Arguably, the norm-resolvent topology is indispensable if one seeks to control both the convergence of spectra and that of (generalised) eigenvectors. 
  Furthermore, the specific choice of topical models is governed by the established consensus that non-trivial, and in particular metamaterial-like, properties arise by equipping the medium with an infinite array of small resonators. In light of the recent advances in the operator-theoretic treatment of boundary value problems, spectral theory thus assumes a new prominent r\^{o}le.  
  Indeed, for problems involving resonators as well as heterogeneous thin structures, 
   (rods, plates, shells, and their combinations -- see 
   \cite{CV, BCVZ, CVZ, CEK_networks}), the relevant operator-theoretic setup is given by a parametrised family of ``transmission" or ``boundary value" problems for PDEs. 

   In \cite{CEK, CherErKis, CEKN} we proposed to utilise the link 
   (facilitated by the classical Kre\u\i n formula) between the resolvent and the Dirichlet-to-Neumann (DtN) operator on the interface between the medium and the resonators,
   to obtain sharp operator-norm convergence estimates. This has been done in a scalar version of the model commonly known as
    double porosity \cite{ADH, Allaire1992, Zhikov2000}.
    We point out that the idea to use DtN maps can be viewed as natural in this area, the first example of its application being traceable to \cite{Friedlander_2002}. However, prior to our work \cite{CEK, CherErKis, CEKN, CherKisSil, CV, CVZ} no attempts were made at employing this machinery to establish norm-resolvent convergence.    In the moderate contrast setting, a theory covering a wide class of problems has been known since the beginning of this century, due to seminal work \cite{BirmanSuslina, BirmanSuslina_corrector}. Up to the publication of the first part of the present work \cite{CherErKis}, nothing of the kind has been available for degenerate problems, of which the double-porosity setup is arguably the most well-studied, as the degenerating coefficients make the problem considerably more challenging. The results we obtain can be viewed as running in parallel with those of \cite{BirmanSuslina}, see Section \ref{main_results_sec}. Also, while we have treated the whole-space setting, bounded regions can be dealt with in a standard way, as in \cite{Suslina_Dirichlet}. 

Apart from addressing the specific problem of double porosity, we point out several generalisations.
First, although in view of clarity we only treat the prototypical model, a wide range of similar problems is amenable to the same approach, see e.g. \cite{Yi_Sheng}. Second, essentially the same technique with minor modifications, see \cite{CEK_networks}, is directly applicable to problems with a ``geometric'' contrast, e.g., elastic networks thinning to metric graphs \cite{KuchmentZeng, KuchmentZeng2004, Figotin_Kuchment, Post, Exner}. Third, via the analysis developed in \cite{CEN}, it transpires that our methodology leads to an explicit spectral resolution of identity for the operator describing the effective properties of the medium. Once this is done, scattering problems in degenerate highly inhomogeneous media come within grasp.

\section{Setup and main results}
\label{section2}

\subsection{Notation}

For a vector $\vect a\in \C^k,$ we denote by $a_j,$ $j=1, \dots, k$ its components.
Similarly, the entries of a matrix $\vect A\in \C^{k\times k}$ are referred to as $A_{ij},$ $i,j=1, \dots, k,$ and $\sym \vect A$ denotes the symmetric part of $\vect A$. 
 The vectors of the standard orthonormal basis in $\C^k$ are denoted by $\vect e_i$, $i=1,\dots,k$. Furthermore, for $\vect a,\vect b \in \C^k,$ we denote by $\vect a \otimes \vect b \in \C^{k \times k}$  the matrix with entries $a_ib_j,$ and
set $\vect{a} \odot \vect{b}:=\sym (\vect{a} \otimes \vect{b}).$
The Frobenius inner product of matrices 
$\vect A,$ $\vect B$ is denoted by ${\vect A}:\overline{\vect B}:={\rm Tr}(\vect B^*\vect A)$ where ${\vect B}^*$ stands for the adjoint of ${\vect B},$ and we set $|\vect A|:=(\vect A:\overline{\vect A})^{1/2}.$ 

For an operator ${\mathcal A}$ (or a sesquilinear form $a$) the domain of ${\mathcal A}$ (respectively $a$)  is denoted by ${\mathcal D}({\mathcal A})$ (respectively ${\mathcal D}(a)$).  We use the notation $\overline{\mathcal{A}}$ for the closure of a closable operator $\mathcal{A},$ and denote by $\sigma({\mathcal A})$ (respectively, $\rho(\mathcal{A})$) the spectrum (respectively, the resolvent set) of an operator $\mathcal A.$ For normed vector spaces $X,Y$, we denote by $\mathfrak{B}(X,Y)$ the space of bounded linear operators from $X$ to $Y$. Furthermore, when indicating a function space $X$ in the notation for a norm $\Vert\cdot\Vert_X,$ we omit the physical domain on which functions in $X$ are defined whenever it is clear from the context. For example, we often write $\Vert\cdot\Vert_{L^2},$ $\Vert\cdot\Vert_{H^1}$  instead of $\Vert\cdot\Vert_{L^2(\Omega;{\mathbb R}^k)},$ $\Vert\cdot\Vert_{H^1(\Omega;{\mathbb R}^k)},$ $k\in \N.$ 

For $A,B \subset \mathbb{C}^k,$ by ${\rm dist}(A,B)$ we denote the distance between the sets $A$ and $B$. For $f \in L^1(A)$, we set $\langle f\rangle:=\int_A f$ and $\mathbbm{1}_A$ denotes the indicator function of $A$. Finally, $\delta_{ij}$ denotes the Kronecker delta, and $C$ generically stands for a positive constant whose value is of no importance.

\subsection{Operator of linear elasticity}
\label{elast_op_sec}
Consider the ``reference cell" $Y := [0, 1)^3 \subset \R^3$ (which is without loss of generality for what is to follow), and let $Y_{\rm soft} \subset Y$ be a connected open set with $C^{1,1}$ boundary $\Gamma$ 
 such that the closure of $Y_{\rm soft}$ is a subset of the interior of $Y$ and $Y_{\rm stiff} = Y \setminus Y_{\rm soft}.$
For a fixed period $\varepsilon>0$ of material oscillations, we are interested in the behaviour of a composite elastic medium with components whose properties are in contrast to one another. We refer to the component materials as ``soft" and ``stiff" accordingly. With this goal in mind, we view $\R^3$ as being composed of two complementary subsets, the stiff part $\Omega_{\rm stiff}^\varepsilon$ (``matrix") and the soft complement $\Omega_{\rm soft}^\varepsilon$ (``inclusions"), see Fig.~\ref{figurehighcontrast}: 
\begin{equation*}
\Omega_{\rm stiff}^\varepsilon := \R^3 \setminus \Omega_{\rm soft}^\varepsilon,\qquad\Omega_{\rm soft}^\varepsilon := \bigcup_{z \in \Z^3} \bigl\{\varepsilon(Y_{\rm soft} + z)\bigr\}.
\end{equation*}
We are interested in the approximation properties, when $\varepsilon \to 0$, of the operator family $(\mathcal{A}_\varepsilon)_{\varepsilon > 0}$, where, for every $\varepsilon>0$, the operators $\mathcal{A}_\varepsilon$ are defined as self-adjoint unbounded operators on $L^2(\R^3; \C^3)$ corresponding to the differential expressions
$-\div\bigl(\A^\varepsilon(x/\varepsilon) \simgrad\bigr)$
with domains $\mathcal{D}(\mathcal{A}_\varepsilon) \subset  H^1(\R^3, \C^3).$ These operators are defined by the sesquilinear forms
\begin{equation*}
    a_\varepsilon (\vect u,\vect v) := 	\int_{\R^3} \A^\varepsilon(x/\varepsilon)\simgrad \vect u:\overline{\simgrad \vect v}, \quad \vect u,\vect v\in H^1(\R^3;\C^3),
\end{equation*}
where the tensor-valued function $\A^\varepsilon$ represents the spatially varying elastic moduli of the medium.  The properties of the stiff and soft components, modelled by tensor-valued functions $\A_{\rm stiff(soft)},$ are assumed to be in ``critical" contrast \cite{Zhikov2000} to each other, so the ratio between the stiff and soft moduli is of order $\varepsilon^{-2}:$   
\begin{equation*}
    \A^{\varepsilon}(y) =
\left\{\begin{array}{ll}
\A_{\rm stiff}(y), & y \in Y_{\rm stiff}, \\[0.25em]
\varepsilon^2\, \A_{\rm soft}(y), & y \in Y_{\rm soft}.
\end{array} \right.
\end{equation*}
The function $\A^\varepsilon$ is defined on the unit cell $Y$ and extended to $\R^3$ by periodicity. We make the following assumptions about $\A_{\rm stiff(soft)}.$
\begin{assumption}
\label{coffassumption}
\begin{itemize}
    \item 
    Uniform positive-definiteness and uniform boundedness on symmetric matrices: there exists $\nu>0$ such that
$\nu\vert\vect\xi\vert^2\leq \A_{\rm stiff(soft)}(y)\vect\xi:\vect\xi \leq\nu^{-1}\vert\vect\xi\vert^2 \quad \forall \vect\xi \in \R^{3\times 3},\,\vect\xi^\top=\vect\xi \quad \forall y \in Y. $
    \item 
    Material symmetries
    $[\A_{\rm stiff(soft)}]_{ijkl}=[\A_{\rm stiff(soft)}]_{jikl}=[\A_{\rm stiff(soft)}]_{klij}, \quad  i,j,k,l\in\{1,2,3\}.$
\item
Lipschitz continuity:
    $[\A_{\rm stiff(soft)}]_{ijkl} \in C^{0,1}(\overline{Y_{\rm stiff(soft)}}), \quad  i,j,k,l\in\{1,2,3\}.$
\end{itemize}
\end{assumption}
\begin{figure}[t]
	\centering
	\includegraphics[width = 0.8\linewidth]{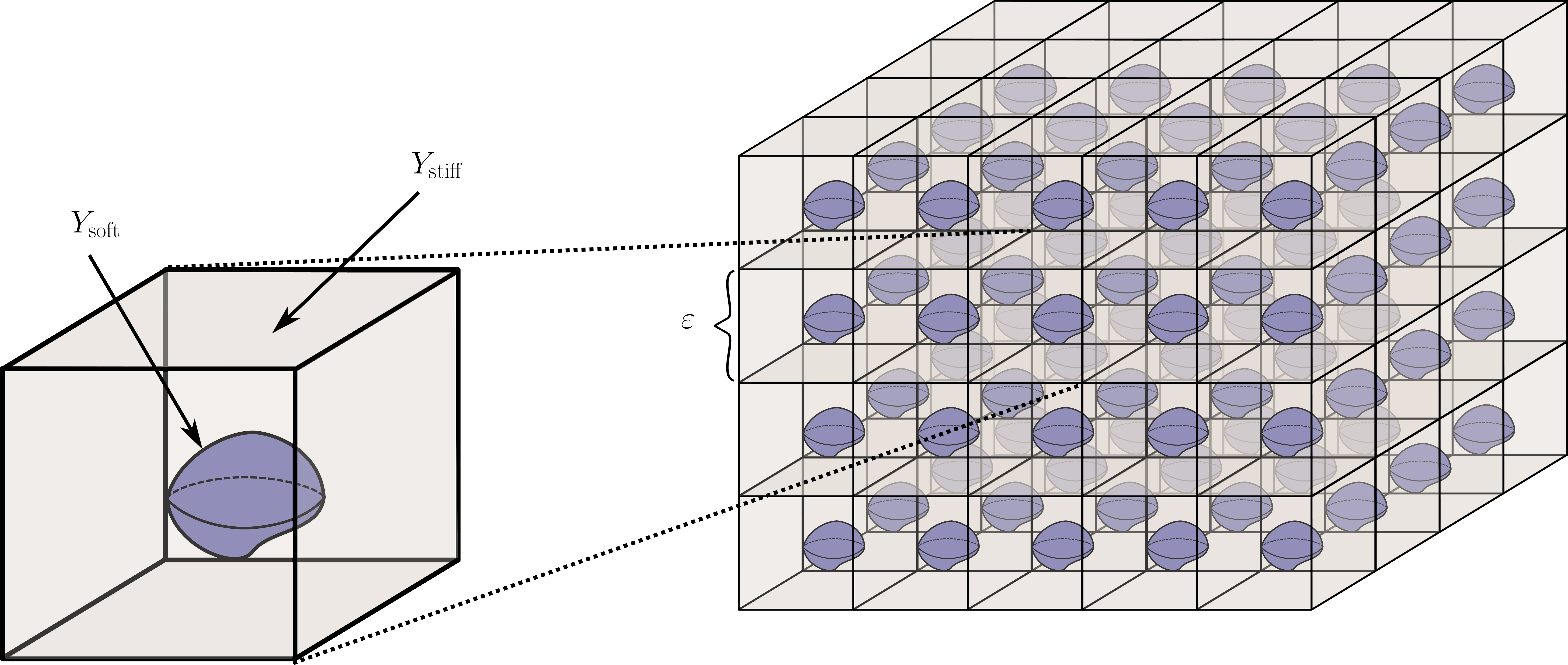}
	\caption{The illustration of a highly oscillating composite material consisting of stiff matrix with soft inclusions}
	\label{figurehighcontrast}
\end{figure}
Note that the setting of 2D (``plane-strain") elasticity, when $Y_{\rm soft}$ is a cylinder (and so $\Omega_{\rm soft}^\varepsilon$ is a union of unbounded ``fibres" parallel to one of the coordinate axes), does not satisfy the above geometric assumption about $Y_{\rm soft}$ but is still covered by our analysis with minor modifications (cf. Remark \ref{2Dremark} below), thanks to the said assumption being satisfied by the cross-section of $Y_{\rm soft}.$

We are interested in the properties of the operator $\mathcal{A}_\varepsilon$ for small $\varepsilon> 0$. Namely, we would like to obtain asymptotics of the resolvents $\left(\mathcal{A}_\varepsilon - zI \right)^{-1},$ as $\varepsilon\to0,$ in the $L^2 \to L^2$ operator norm. The spectral parameter $z \in \C$ is assumed uniformly separated from the spectrum of the operator $\mathcal{A}_\varepsilon:$ more precisely, we fix $\sigma>0$ and consider a compact 
    $K_\sigma \subset \left\{z \in \C:\ {\rm dist}( z, \R) \geq \sigma \right\}.$
We introduce the following space of functions supported only on the stiff (respectively, soft) component:
\begin{equation}
\label{L2stiffepsilon}
 L^{\rm stiff(soft)}_\varepsilon:=\Bigl\{\vect f \in L^2(\R^3; \C^3): \ \vect f = 0 \mbox{   on } \Omega_{\rm soft(stiff)}^{\varepsilon} \Bigr\}.
\end{equation}
Furthermore, we denote by $P^{\rm stiff}_\varepsilon$ the orthogonal projection from $L^2(\R^3;\C^3)$ onto $L^{\rm stiff}_\varepsilon$. To state the main result, we define the following operators, which are key components of the leading-order term of the resolvent asymptotics.

\begin{definition}[Macroscopic operator]
Consider the tensor $\A_{\rm macro} \in \R^{3\times 3\times 3\times 3}$ defined by
\begin{equation}
\label{macrodefinitionequation}
\begin{aligned}
     \A_{\rm macro}\,{\vect \xi} :{\vect \eta} = \int_{Y_{\rm stiff}} \A_{\rm stiff} \left( \simgrad \vect u_{\vect\xi} + \vect\xi \right)&:\left( \simgrad \vect u_{\vect\eta}+\vect\eta \right), \quad
     \vect\xi, \vect\eta \in \R^{3 \times 3},\ \vect\xi^\top=\vect\xi,\ \vect\eta^\top=\vect\eta. 
\end{aligned}
\end{equation}
    where $\vect u_{\vect \xi} \in H^1_{\#}(Y_{\rm stiff};\R^3)$ is the unique solution (guaranteed by the Lax-Milgram lemma) to the problem
\begin{equation*}
    \int_{Y_{\rm stiff}} \A_{\rm stiff} \left(\simgrad \vect u_{\xi} + \vect\xi \right):\simgrad \vect v= 0 \qquad \forall \vect v \in H^1_{\#}(Y_{\rm stiff};\R^3), \qquad \int_{Y_{\rm stiff}}\vect u_{\vect \xi}=0.
\end{equation*}
Henceforth  $H^1_{\#}(Y_{\rm stiff};\R^3)$ is the closure of $Y$-periodic smooth vector functions
on $Y_{\rm stiff}$ in the $H^1(Y_{\rm stiff}, \R^3)$ norm. 

We define the macroscopic operator (the operator of perforated domain) $\mathcal{A}_{\rm macro}$ as the self-adjoint unbounded operator on $L^2(\R^3;\C^3)$ corresponding to the differential expression
$-\div\left(\A_{\rm macro} \simgrad\right),$
with domain $\mathcal{D}(\mathcal{A}_{\rm macro}) \subset  H^1(\R^3, \C^3) $, defined by the sesquilinear form (cf. \cite{suslina10})
\begin{equation}
    a_{\rm macro} (\vect u,\vect v) := 	\int_{\R^3} \A_{\rm macro}\simgrad \vect u: \overline{\simgrad\vect v}, \quad \vect u,\vect v\in H^1(\R^3;\C^3).
    \label{macrooperatordef_form}
\end{equation}
\end{definition}
We will require the following lemma, whose proof is standard. 
\begin{lemma}
	\label{prop_lemma}
There tensor $\A_{\rm macro}$ is symmetric, in the sense that $[\A_{\rm macro}]_{ijkl} = [\A_{\rm macro}]_{jikl} =[\A_{\rm macro}]_{klij},$  $i,j,k,l\in\{1,2,3\},$ and positive definite: there exists a constant $\eta>0$ such that 
      $\A_{\rm macro}\vect \xi : \vect \xi \geq \eta |\vect \xi|^2$ for all $\vect \xi \in \R^{3 \times 3},$ $\vect\xi^\top=\vect\xi.$
\end{lemma}
\begin{proof} 
The proof is based on a standard extension argument, see e.g. \cite[Proposition 3.4]{BCVZ}.	
\end{proof} 	
In addition to the properties highlighted in Lemma \ref{prop_lemma}, the leading-order term in the resolvent asymptotics retains information on the microstructure, via the spectrum of the ``Bloch operator" $\mathcal{A}_{\rm Bloch}$ associated with the bilinear form
\begin{equation*}
    a_{\rm Bloch} (\vect u,\vect v) := 	\int_{Y_{\rm soft}} \A_{\rm soft}\simgrad \vect u:\overline{\simgrad \vect v}, \quad \vect u,\vect v\in H^1_0(Y_{\rm soft};\C^3), 
\end{equation*}
as a non-negative self-adjoint operator on $L^2(Y_{\rm soft};\C^3)$.
Furthermore, we define the matrix-valued ``Zhikov  function" $\mathcal{B}$ by
\begin{equation}
	\label{ante1} 
	\begin{split}
		\mathcal{B}(z)_{ij}:= z \delta_{ij}+ z^2\sum_{k = 1}^{\infty} \frac{\langle\varphi_k\rangle_i\langle\varphi_k\rangle_j}{\eta_k - z}, \quad i,j\in {1,2,3},
	\end{split}
\end{equation}
where $\eta_1<\dots\le\eta_k\le\dots\to\infty$ are the eigenvalues (indexed taking multiplicities into account) and $\varphi_k$, $k \in \N,$ are the corresponding (orthonormal) eigenfunctions of the associated Bloch operator $\mathcal{A}_{\rm Bloch}$ on $L^2(Y_{\rm soft};\C^3),$ and $z\neq\eta_k$ for all $k \in \mathbb{N}.$  
The 
function ${\mathcal B}$ 
has appeared in the context of ``qualitative" analysis of high contrast (see, e.g., \cite{zhikov2013, CC_Maxwell}). From this perspective, the results of the present paper can be viewed as demonstrating how it enters quantitative estimates,  with sharp error control, in the context of elasticity.

\subsection{Main results}
\label{main_results_sec}

We will now state the main results of the paper, which provide $O(\varepsilon)$ and $O(\varepsilon^2)$ approximations of the resolvent of the operator ${\mathcal A}_\varepsilon.$ When restricted to the stiff component, the $O(\varepsilon^2)$ approximation involves a pseudodifferential operator, which leads to a second-order differential operator at the cost of an $O(\varepsilon)$ correction. The operator estimates we prove (Theorems \ref{thmmain1}, \ref{thmamin2}) involve certain ``approximating" and ``effective" operators  $\mathcal{A}_\varepsilon^{\rm app}$ and $\mathcal{A}_\varepsilon^{\rm eff},$ which are described explicitly in Section \ref{Sectionopapp} and Section \ref{sectionopeff}, respectively. 
The proof of Theorem \ref{thmmain1} is given at the end of Section \ref{Sectionopapp}, the proof of claim (a) of Theorem \ref{thmamin2} is contained Section \ref{sectionopeff}, while its claim (b) is discussed in Section \ref{sectionopeffstiff}.
\begin{theorem} \label{thmmain1} 
	 There exists $C>0$ that depends only on $\sigma$ and ${\rm diam}(K_\sigma)$ such that for all $z \in K_\sigma$ one has
	\begin{equation}
 \label{estimate24}
	\bigl\lVert \left(\mathcal{A}_\varepsilon - z I \right)^{-1} - \Theta_\varepsilon^{\rm app} 
	\bigl(\mathcal{A}_\varepsilon^{\rm app} - z I \bigr)^{-1} \Theta^{\rm app}_\varepsilon \bigr\rVert_{L^2(\R^3;\C^3) \to L^2(\R^3;\C^3) } \leq C \varepsilon^2.
	\end{equation}
	  The operator $\Theta^{\rm app}_\varepsilon$ is an orthogonal projection (see Section \ref{Sectionopapp}), and  $\mathcal{A}_\varepsilon^{\rm app}$ is 
	  defined uniquely on $\Theta_\varepsilon^{\rm app}L^2(\R^3;\C^3)$ and then extended somehow (e.g., by the zero operator) to its orthogonal complement.
\end{theorem}

\begin{theorem} \label{thmamin2} 
    There exists $C>0$ that depends only on $\sigma$ and ${\rm diam}(K_\sigma)$ such that, for all $z \in K_\sigma:$
    \begin{enumerate}
    \item[(a)] The ``whole-space" homogenisation estimate 	 
    \begin{equation}
    \label{estimate25a}
    \bigl\lVert \left(\mathcal{A}_\varepsilon - z I \right)^{-1} - \Theta_\varepsilon^{\rm eff} 
    \bigl(\mathcal{A}_\varepsilon^{\rm eff} - z I \bigr)^{-1} \Theta^{\rm eff}_\varepsilon \bigr\rVert_{L^2(\R^3;\C^3) \to L^2(\R^3;\C^3) } \leq C \varepsilon
    \end{equation}
    	 holds. Here  $\Theta^{\rm eff}_\varepsilon$ is an  orthogonal projection, and $\mathcal{A}_\varepsilon^{\rm eff}$ is a self-adjoint operator initially defined uniquely on $\Theta_\varepsilon^{\rm eff}L^2(\R^3;\C^3)$ and then extended somehow (e.g., by the zero operator) to its orthogonal complement.
    \item[(b)] 
  The estimate on the ``stiff" component  
    \begin{equation*}
            \bigl\lVert P_\varepsilon^{\rm stiff}\left(\mathcal{A}_\varepsilon - z I \right)^{-1} P_\varepsilon^{\rm stiff} -  P_\varepsilon^{\rm stiff} \bigl(\mathcal{A}_{\rm macro} - \mathcal{B}(z) \bigr)^{-1}  P_\varepsilon^{\rm stiff} \bigr\rVert_{L^2(\R^3;\C^3) \to L^2(\R^3;\C^3) } \leq C \varepsilon
    \end{equation*}
holds.
    Here $\mathcal{A}_{\rm macro}$ is the differential operator of linear elasticity with constant coefficients defined by the form \eqref{macrooperatordef_form} and $P_\varepsilon^{\rm stiff}$ is the orthogonal projection onto the subspace of $L^2(\R^3;\C^3)$ consisting of functions vanishing on the soft component $\Omega^\varepsilon_{\rm soft}$ of the medium. 
    \end{enumerate} 
\end{theorem}

For additional discussions about the form of approximating operators in high-contrast homogenisation and the comparison of the above results to the ones of \cite{BirmanSuslina} in the moderate-contrast setting, see Remarks \ref{josipnak10}, \ref{nakk2} below. 

 One can also extend the claims of Theorems \ref{thmmain1} and \ref{thmamin2} beyond $K_{\sigma}$, provided the spectral parameter $z$ is confined to a bounded region away from the spectrum. Under this extension, Corollary \ref{ivan1}  reveals an explicit dependence of the constant $C$ on the distance of $z$ from the spectrum of the operator $\mathcal{A}_{\eps}$ and the modulus of $z,$ as follows. 

\begin{corollary} \label{ivan1} 
	The claim of Theorem \ref{thmmain1}  can be extended to all $z \notin \sigma(\mathcal{A}_{\eps}) \cup \sigma(\mathcal{A}^{\rm app}_{\eps})$ and the constant $C=C(z)$ depends on $z$ as follows:
	$$C(z)=C \left(1+(|z|+1)/{\rm dist}(z,\sigma(\mathcal{A}_{\eps}))\right) \left(1+(|z|+1)/{\rm dist}(z,\sigma(\mathcal{A}^{\rm app}_{\eps}))\right),
	$$
	where $C$ is independent of $z$. 
	A similar statement holds for the claim of Theorem \ref{thmamin2}~(a). 		
\end{corollary} 
\begin{proof}
	For the  unbounded operators $\mathcal{A}$ and $\mathcal{B}$ on the Banach space $\mathcal{X}$, the orthogonal projection $\mathcal P$ that commutes with the operator $\mathcal{B},$ and $z_1,z_2 \notin \sigma(\mathcal{A})\cup\sigma(\mathcal{B}),$ it is easy to check the identity 
	\begin{equation}
		\label{ante40} 
		\begin{aligned} 
		(\mathcal{A}-z_2I)^{-1}&-{\mathcal P}(\mathcal{B}-z_2I)^{-1}{\mathcal P}
		\\[0.3em] 
		&
		={\mathcal P}(\mathcal{B}-z_2I)^{-1}{\mathcal P}(\mathcal{B}-z_1I){\mathcal P}\left((\mathcal{A}-z_1I)^{-1}
        -{\mathcal P}(\mathcal{B}-z_1I)^{-1}{\mathcal P}\right)(\mathcal{A}-z_1I)(\mathcal{A}-z_2I)^{-1}
		\\[0.3em] &\hspace{35ex}
		+(I-{\mathcal P})(\mathcal{A}-z_1I)^{-1}(\mathcal{A}-z_1I)(\mathcal{A}-z_2I)^{-1}.
		\end{aligned}
	\end{equation} 
	Notice also that by the functional calculus for a self-adjoint $\mathcal{A}$ we have 
	\begin{equation*} 
		\bigl\|(\mathcal{A}-z_1I)(\mathcal{A}-z_2I)^{-1}\bigr\|_{\mathcal{X} \to \mathcal{X}}\leq 1+\frac{|z_2-z_1|}{{\rm dist}(z_2,\sigma\bigl(\mathcal{A})\bigr)}.
	\end{equation*} 
	We set $z_1={\rm i},$  $\mathcal{A}=\mathcal{A}_{\eps},$  $\mathcal{B}=\mathcal{A}_{\eps}^{\rm app},$ and ${\mathcal P}=\Theta^{\rm app}_\varepsilon$ in \eqref{ante40} and combine it with Theorem \ref{thmmain1}, where we set $z={\rm i},$ bearing in mind that, also as a consequence of Theorem \ref{thmmain1}, one has 
	\begin{equation*} 
		\bigl\lVert(I-\Theta^{\rm app}_\varepsilon) \left(\mathcal{A}_\varepsilon -{\rm i}I \right)^{-1}\bigr\rVert_{L^2(\R^3;\C^3) \to L^2(\R^3;\C^3) } \leq C \varepsilon^2.
	\end{equation*}
	The claim now follows immediately. 
\end{proof}
We next discuss implications of the main results for the asymptotic behaviour of the spectra of the operators ${\mathcal A}_\varepsilon.$ It is well known that these 
have a  band-gap structure (see \cite{Zhikov2000}).
 Theorem \ref{thmmain1} and Theorem \ref{thmamin2} enable us to estimate the gaps in the spectrum of $\mathcal{A}_{\varepsilon}$ on any compact interval by the gaps in the spectra of $\mathcal{A}_\varepsilon^{\rm app}$ and   $\mathcal{A}_\varepsilon^{\rm eff},$ respectively. 
\begin{corollary} \label{ivan0} 
	For every $M>0,$ one has
	\begin{eqnarray*} 
		{\rm dist} \left(\sigma(\mathcal{A}_{\eps})\cap [-M,M], \sigma(\mathcal{A}_\varepsilon^{\rm app})\cap[-M,M]\right) &\leq & C(M+1)^2\varepsilon^2,  \\[0.2em]  {\rm dist} \left(\sigma(\mathcal{A}_{\eps})\cap[-M,M],  \sigma(\mathcal{A}_\varepsilon^{\rm eff})\cap[-M,M]\right) &\leq& C (M+1)^2 \varepsilon.     
	\end{eqnarray*} 
	where $C>0$ is independent of $M.$
\end{corollary} 	
\begin{proof} 
	The proof is obtained by setting $z={\rm i}$ in Theorem \ref{thmmain1} and Theorem \ref{thmamin2}\,(a).
	It is well known that for self-adjoint bounded linear  operators $\mathcal{A},$  $\mathcal{B}$ on a Hilbert space $\mathcal{X},$ one has	
	${\rm dist} \left( \sigma(\mathcal{A}),\sigma(\mathcal{B})\right)\leq \|\mathcal{A}-\mathcal{B}\|_{\mathcal{X} \to \mathcal{X}}, $
	 see e.g. \cite{Kato}. 
	Furthermore, noting that
	\begin{equation*}
		\begin{aligned}
	 &\sigma \bigl(\bigl(\mathcal{A}_\varepsilon + I \bigr)^{-1}\bigr)=
		\bigl\{(\lambda + 1)^{-1}: \lambda \in \sigma(\mathcal{A}_\varepsilon) \bigr\},\qquad 
		\sigma \bigl(\Theta_\varepsilon^{\rm app}\bigl(\mathcal{A}_\varepsilon^{\rm app} + I\bigr)^{-1}\Theta_\varepsilon^{\rm app}\bigr)=\bigl\{(\lambda+ 1)^{-1}: \lambda \in \sigma(\mathcal{A}_\varepsilon^{\rm app}) \bigr\}, \\[0.3em]
		&\sigma\bigl(\Theta_\varepsilon^{\rm eff} \bigl(\mathcal{A}_\varepsilon^{\rm eff} + I \bigr)^{-1}\Theta_\varepsilon^{\rm eff}\bigr)= \bigl\{(\lambda + 1)^{-1}: \lambda \in \sigma(\mathcal{A}_\varepsilon^{\rm eff}) \bigr\},
		\end{aligned}	
	\end{equation*} 	
	the claim follows from the fact that for arbitrary $\lambda, \mu \in \mathbb{R}$ one has
	\begin{equation*} 
	|\lambda-\mu|= |\lambda +1||\mu + 1|\left|(\lambda+1)^{-1}- (\mu+1)^{-1}\right|\leq (|\lambda|+1)  (|\mu|+1) \left|(\lambda+1)^{-1}- (\mu+1)^{-1}\right|.  \qedhere\popQED
	\end{equation*}
\end{proof}

\begin{remark} 
	By combining Corollary \ref{ivan0} and \ref{ivan1} one can make the constant  $C=C(z)$ in Theorem \ref{thmmain1} and \ref{thmamin2} dependent only on $|z|$ and 	${\rm dist}(z,\sigma(\mathcal{A}_{\eps}))$, i.e. only on $|z|$ and ${\rm dist}(z,\sigma(\mathcal{A}^{\rm app/eff}_{\eps}))$. Also, an explicit numerical value for $C>0$ in Corollary \ref{ivan0} and \ref{ivan1} that corresponds to $z={\rm i}$ in Theorem \ref{thmmain1} and \ref{thmamin2} can be provided.  
\end{remark} 	
\BBB

\subsection{Gelfand transform}
\label{Gelf_sec}
The purpose of this chapter is to decompose the original differential operator into a family of differential operators with compact resolvents
that act on functions defined on the unit cell $Y$. This is carried out in a standard way by the Gelfand transform. 

 The Gelfand transform $\mathcal{G}$ is defined on $L^2(\R^3;\C^3)$ by the formula
\begin{equation*}
    (\mathcal{G} \vect u)(y,\chi):= \left(2\pi\right)^{-3/2} \sum_{n\in \Z^3}{\rm e}^{-{\rm i}\chi(y+n)}\vect u(y+n), \quad y \in  Y, \quad \chi\in Y',
\end{equation*}
where $Y':= [- \pi, \pi)^3$. (As noted above, without loss of generality we assume that $Y=[0,1)^3.$  Lattices with other periods can be treated by the same analysis with minor modifications --- the associated ``Brillouin zone" $Y'$ is then adjusted appropriately.) In the case of 
The Gelfand transform is a unitary operator: $$\mathcal{G} : L^2(\mathbb{R}^3;\C^3) \to L^2(Y';L^2(Y; \C^3)) = \int_{Y'}^\oplus L^2(Y; \C^3, \chi)d\chi,$$
in the sense that
    $\left\langle \vect u, \vect v \right\rangle_{L^2(\R^3;\C^3)} = \left\langle \mathcal{G} \vect u,\mathcal{G} \vect v \right\rangle_{L^2(Y';L^2(Y; \C^3))}$ for all $\vect u, \vect v \in L^2(\R^3;\C^3).$
A function can be reconstructed from its Gelfand transform as follows:
\begin{equation*}
   \vect  u(x)=(2\pi)^{-3/2} \int_{Y'} {\rm e}^{{\rm i}\chi\cdot x} (\mathcal{G} \vect u)(x,\chi)d\chi,\quad x\in\R^3.
\end{equation*}
For an overview of the properties of Gelfand transform in relation to homogenisation problems, we refer to \cite{BirmanSuslina}. 

In order to deal with the setting of highly oscillating material coefficients, we consider the following scaled version of Gelfand transform. For a fixed $\varepsilon>0$ and all $\vect u \in L^2(\R^3;\C^3)$, we set 
\begin{equation*}
    (\mathcal{G}_\varepsilon \vect u)(y,\chi):= \left(\frac{\varepsilon}{2\pi}\right)^{3/2} \sum_{n\in \Z^3}{\rm e}^{-{\rm i}\chi(y+n)}\vect u(\varepsilon(y+n)), \quad y \in  Y,\ \ \chi\in Y'.
\end{equation*}
Note that $\mathcal{G}_\varepsilon$ is a composition of $\mathcal{G}$ and the unitary scaling operator $\mathcal{S}_\varepsilon : L^2(\R^3;\C^3) \to L^2(\R^3;\C^3)$ defined by
\begin{equation*}
    \mathcal{S}_\varepsilon \vect u(x) := \varepsilon^{3/2} \vect u (\varepsilon x), \quad \vect u\in L^2(\R^3;\C^3).
\end{equation*}
It follows that $\mathcal{G}_\varepsilon$
is also unitary, i.e.
\begin{equation*}
    \left\langle \vect u, \vect v \right\rangle_{L^2(\R^3;\C^3)} = \left\langle \mathcal{G}_\varepsilon \vect u,\mathcal{G}_\varepsilon \vect v \right\rangle_{L^2(Y';L^2(Y; \C^3))} \qquad \forall \vect u, \vect v \in L^2(\R^3;\C^3).
\end{equation*}
The original function is recovered from its Gelfand transform by the formula
\begin{equation}
	\label{inversegelfand}
   \vect  u(x) = (2\pi \varepsilon)^{-3/2} \int_{Y'} {\rm e}^{{\rm i}\chi\cdot x/\varepsilon} (\mathcal{G}_\varepsilon \vect u)(x/\varepsilon,\chi)d\chi.
\end{equation}
Also, by noting that for the scaled Gelfand transform of a derivative of $\vect  u \in H^1(\R^3;\C^3)$ one has
\begin{equation*}
     \mathcal{G}_\varepsilon (\partial_{x_\alpha} \vect u )=\varepsilon^{-1} \bigl(\partial_{y_\alpha}\left(\mathcal{G}_\varepsilon \vect u\right) + {\rm i}\chi_\alpha \left( \mathcal{G}_\varepsilon \vect u\right)\bigr), \quad \alpha = 1, 2,3,
\end{equation*}
we infer that
\begin{equation}
\label{gelfandvsderivatives}
        \mathcal{G}_\varepsilon\left( \simgrad \vect u\right)(y,\chi) = \varepsilon^{-1}\bigl(\simgrad_y\left(\mathcal{G}_\varepsilon \vect u\right) + {\rm i}\sym\left(\left( \mathcal{G}_\varepsilon \vect u\right)\otimes\chi\right) \bigr) 
        =\varepsilon^{-1}\bigl(\simgrad_y\left(\mathcal{G}_\varepsilon \vect u\right)+{\rm i}X_\chi\left( \mathcal{G}_\varepsilon \vect u\right) \bigr),
\end{equation}
where the for each $\chi\in Y'$ the operator $X_{\chi}$ acting on $L^2( Y; \C^3)$ is defined by
\begin{equation*}
X_{\chi}\vect u  =  \sym \left(\vect u \odot\chi\right) = \begin{bmatrix}
\chi_1 u_1 & \frac{1}{2}(\chi_1 u_2 + \chi_2 u_1) & \frac{1}{2}(\chi_1 u_3 + \chi_3 u_1)  \\[0.7em]
\frac{1}{2}(\chi_1 u_2 + \chi_2 u_1) & \chi_2 u_2 & \frac{1}{2}(\chi_3 u_2 + \chi_2 u_3) \\[0.7em]
\frac{1}{2}(\chi_3 u_1 + \chi_1 u_3) & \frac{1}{2}(\chi_3 u_2 + \chi_2 u_3) & \chi_3 u_3
\end{bmatrix},\quad \vect u \in L^2( Y; \C^3).
\end{equation*}

\begin{remark}
	\label{2Dremark}
    Note that in the setting of 2D elasticity the operator $X_\chi$ takes the form \begin{equation*}
X_{\chi}\vect u  =  \begin{bmatrix}
\chi_1 u_1 & \frac{1}{2}(\chi_1 u_2 + \chi_2 u_1)   \\[0.7em]
\frac{1}{2}(\chi_1 u_2 + \chi_2 u_1) & \chi_2 u_2  
\end{bmatrix},\quad \vect u \in L^2( Y; \C^2).
\end{equation*}
\end{remark} 
It is straightforward to show the existence of $C_1, C_2>0$ such that that
\begin{equation}
C_1|\chi|||\vect u||_{L^2(Y;\C^3)} \leq ||X_{\chi}\vect u||_{L^2(Y;\C^{3 \times  3})} \leq C_2|\chi|||\vect u||_{L^2(Y;\C^3)}\qquad \forall\vect u\in L^2(Y;\C^3).
\label{Xchi_bounds}
\end{equation}
Denote by   $H_{\#}^1( Y;\C^3)$, $H_{\#}^2( Y;\C^3)$ the spaces of $Y$-periodic functions in $H^1(Y;\C^3)$, $H^2(Y;\C^3),$ respectively. We use similar notation when $Y$ is replaced by $Y_{\rm stiff}$. For Gelfand transform ${\mathcal G}_\varepsilon,$ one can show \cite{Kuchment_Floquet_book} that
\begin{equation*}
    a_\varepsilon(\vect u,\vect v) = \int_{Y'}\varepsilon^{-2}a_{\chi,\varepsilon}(\mathcal{G}_\varepsilon\vect u,\mathcal{G}_\varepsilon\vect v) d\chi,
\end{equation*}
where
\begin{equation}
\label{formaachiepsilon}
    a_{\chi,\varepsilon} (\vect u,\vect v) := 	\int_{ Y} \A^\varepsilon(\simgrad+{\rm i}X_\chi)\vect u: \overline{(\simgrad  +{\rm i}X_\chi )\vect v}, \qquad \vect u,\vect v\in H^1_{\#}(Y;\C^3).
\end{equation}
For each $\chi \in Y'$ and $\varepsilon>0,$ we introduce the self-adjoint operator
\begin{equation*}
    \mathcal{A}_{\chi,\varepsilon} := (\simgrad+{\rm i} X_\chi)^* \A^\varepsilon (\simgrad +{\rm i}X_\chi) : \mathcal{D}(\mathcal{A}_{\chi,\varepsilon}) \subset H^1_{\#}(Y;\C^3) \to L^2(Y;\C^3)
\end{equation*}
associated with the positive definite form $a_{\chi,\varepsilon}.$ Here we use the notation $(\cdot)^*$ for the formal adjoint of the operator. Applying the scaled Gelfand transform to the resolvent yields
\begin{equation}
\label{vonneumannformula}
   \left(\mathcal{A}_\varepsilon - zI \right)^{-1} = \mathcal{G}_\varepsilon^{-1} \left( \int_{Y'}^\oplus \left(\frac{1}{\varepsilon^2}\mathcal{A}_{\chi,\varepsilon} - zI \right)^{-1}  d\chi \right) \mathcal{G}_\varepsilon, \quad z \in \rho(\mathcal{A}_\varepsilon),
\end{equation}
which is an example of the classical von Neumann direct integral formula. Due to the compactness of the embedding $H^1_{\#}(Y ;\C^3) \hookrightarrow L^2(Y;\C^3)$, the resolvents $\left(\varepsilon^{-2}\mathcal{A}_{\chi,\varepsilon} - zI \right)^{-1}$ are compact. We interpret \eqref{vonneumannformula} as follows: by applying the Gelfand transform to the problem, we have decomposed the resolvent operator $\left(\mathcal{A}_\varepsilon - zI \right)^{-1}$ into a continuum family of resolvent operators $\left(\varepsilon^{-2}\mathcal{A}_{\chi,\varepsilon} -z  I \right)^{-1}$ indexed by $\chi\in Y'$. In contrast to the original resolvent operator, this family consists of compact operators, which have discrete spectra.

%
For each $\varepsilon>0,$ the $\chi$-fibre ($\chi \in Y'$) resolvent problem for $\mathcal{A}_\varepsilon$ consists in finding, for a fixed  
$z \in \rho(\mathcal{A}_{\chi,\varepsilon})$ and every $\vect f\in L^2(Y,\C^3),$ the solution $\vect u \in \mathcal{D}(\mathcal{A}_{\chi,\varepsilon})$ to the equation
    $\left(\varepsilon^{-2}\mathcal{A}_{\chi,\varepsilon}-zI\right)\vect u =\vect f.$

\begin{remark}[Transmission boundary value problem]
	\label{deinitiontransmission}
	The equation $\left(\varepsilon^{-2}\mathcal{A}_{\chi,\varepsilon}-zI\right)\vect u =\vect f$ can  be formally recast as follows: find $\vect u_{\rm stiff}\in H_\#^2(Y_{\rm stiff},\C^3)$,   $ \vect u_{\rm soft}\in H^2(Y_{\rm soft},\C^3)$ such that 
	\begin{equation}
		\label{transmissionboundaryproblem}
		\begin{aligned}
			\varepsilon^{-2}\left(\simgrad+{\rm i}X_\chi\right)^* \A_{\rm stiff} \left(\simgrad + {\rm i}X_\chi\right) \vect u_{\rm stiff}-z \vect u_{\rm stiff} = \vect f \quad &\mbox{on $Y_{\rm stiff}$, } \\[0.3em]
			\left(\simgrad+{\rm i}X_\chi\right)^* \A_{\rm soft} \left(\simgrad +{\rm i} X_\chi\right) \vect u_{\rm soft}- z \vect u_{\rm soft} = \vect f\quad &\mbox{on $Y_{\rm soft}$, } \\[0.4em]
			{\vect u}_{\rm stiff}={\vect u}_{\rm soft} \quad&{\rm on}\ \Gamma, \\[0.3em]
			\varepsilon^{-2}\A_{\rm stiff} \left(\simgrad + {\rm i}X_\chi\right) {\vect u}_{\rm stiff} \cdot {\vect n}_{\rm stiff} + \A_{\rm soft} \left(\simgrad +{\rm i} X_\chi\right) {\vect u}_{\rm soft} \cdot {\vect n}_{\rm soft} = 0\quad&{\rm on}\ \Gamma,
		\end{aligned}
	\end{equation}
	where ${\vect n}_{\rm stiff}$, ${\vect n}_{\rm soft}$ denote the 
	outward-pointing unit normals to $\Gamma$ from $Y_{\rm stiff},$ $Y_{\rm soft}.$ The question of making this reformulation rigorous is that of regularity of functions in the domain of ${\mathcal A}_{\chi,\varepsilon}.$ 
	
	The above regularity question can alternatively be framed in terms of the solution $\vect u\in H_\#^1(Y;\C^3)$ to the weak formulation of \eqref{transmissionboundaryproblem}:
	\begin{equation}
	a_{\chi, \varepsilon}(\vect u, \vect v)+\int\vect u\cdot\overline{\vect v}=\int_Y\vect f\cdot\overline{\vect v}\qquad \forall \vect v\in H_\#^1(Y;\C^3).
	\label{weak_form_above}
	\end{equation}
	It is then tempting to interpret the problem \eqref{transmissionboundaryproblem} via understanding its first two equations 
	in the sense of distributions, the penultimate equation in the sense of equality in $H^{1/2}(\Gamma;\C^3)$, and the last equation  in the sense of equality in $H^{-1/2}(\Gamma;\C^3)$. However, even this interpretation calls for some caution, in view of the moderate regularity assumptions made  in Section \ref{elast_op_sec} about the coefficient tensors $\A^{\rm stiff(soft)}$ and interface $\Gamma$. 
	
	We do not address the $H^2$ regularity in the present work, as it has no bearing on our results.  As the formulation \eqref{transmissionboundaryproblem} is often preferred in the engineering community, at present its link to the operators ${\mathcal A}_{\chi,\varepsilon}$ is an open question, to which some of the machinery of \cite{Mitrea} may be relevant. 
    
    We note however that, in the infinitely smooth smooth (i.e., $C^\infty$ coefficients and interface), the above regularity question was settled positively in \cite{Schechter}. Under the assumptions we made in the present paper (see Section \ref{elast_op_sec}), the results we present below can be shown to yield at least $H^{3/2}$ regularity. In other words, the weak formulation \eqref{weak_form_above} is shown to be equivalent to looking for $\vect u_{\rm stiff}\in H_\#^{3/2}(Y_{\rm stiff},\C^3)$,   $ \vect u_{\rm soft}\in H^{3/2}(Y_{\rm soft},\C^3)$ such that \eqref{transmissionboundaryproblem} holds.

	

	
\end{remark}

\section{Operator theoretic approach: Ryzhov triples}\label{section3} 

The purpose of this section is to introduce an abstract framework for the transmission problem \eqref{transmissionboundaryproblem}. In Section \ref{opthap1} we recall a general construction due to Ryzhov \cite{Ryzhov_spec}, while in Sections \ref{opapth2}, \ref{secnakkk1} we show how the problem \eqref{transmissionboundaryproblem} can be seen as part of this construction and prove key properties of the operators emerging in the process. 
  
We start by introducing some basic objects required to work with Ryzhov triples, an operator framework convenient for the analysis of boundary value problems for PDEs.

\subsection{Abstract notion of a Ryzhov triple} \label{opthap1} 

The concept of a Ryzhov triple was introduced  in \cite{MR2330831,Ryzhov_spec}.  The main results  of this section are Theorems \ref{theoremsolutionformularobin}, \ref{theoremkreinformula}, which provide an operator-theoretic formula for the solution to an abstract spectral boundary value problem. 
\begin{definition}
\label{ryzhovtriple}
Let $\mathcal{H}$ be a separable Hilbert space and $\mathcal{E}$ an auxiliary Hilbert space. Suppose that:
\begin{itemize}
    \item $\mathcal{A}_0$ is a self-adjoint operator on $\mathcal{H}$ with $0 \in \rho(\mathcal{A}_0)$,
    \item $\Pi:\mathcal{E}\to \mathcal{H}$ is a bounded operator such that
    $\mathcal{D}(\mathcal{A}_0)\cap \mathcal{R}(\Pi) = \{ 0\},$ $\ker(\Pi) = \{0\}.$
    \item $\Lambda$ is a self-adjoint operator on the domain $\mathcal{D}(\Lambda) \subset \mathcal{E}$.
\end{itemize}
We refer to the triple $\left(\mathcal{A}_0, \Pi, \Lambda \right)$ as a Ryzhov triple on $(\mathcal{H},\mathcal{E})$.
\end{definition}
Following \cite{Ryzhov_spec}, we introduce the following objects. 
\begin{definition}
\label{boundary operators}
Let $\mathcal{H}$ be a separable Hilbert space, $\mathcal{E}$ an auxiliary Hilbert space, and $\left(\mathcal{A}_0, \Pi, \Lambda \right)$ a Ryzhov triple on $(\mathcal{H},\mathcal{E})$. Define the operators ${\mathcal A},$ $\Gamma_0,$ $\Gamma_1$ as follows: 
\begin{equation}\label{defopp11} 
    \begin{aligned}
        \mathcal{D}(\mathcal{A}):=\mathcal{D}(\mathcal{A}_0) \dot + \Pi(\mathcal{E}), & \qquad \mathcal{A}:\mathcal{A}_0^{-1} \vect f + \Pi \vect g \to \vect f, \quad \vect f \in \mathcal{H}, \vect g \in \mathcal{E}, \\[0.3em]
        \mathcal{D}(\Gamma_0):=\mathcal{D}(\mathcal{A}_0) \dot + \Pi(\mathcal{E}), & \qquad \Gamma_0:\mathcal{A}_0^{-1} \vect f + \Pi \vect g \to \vect g, \quad \vect f \in \mathcal{H}, \vect g \in \mathcal{E}, \\[0.3em]
        \mathcal{D}(\Gamma_1):=\mathcal{D}(\mathcal{A}_0) \dot + \Pi(\mathcal{D}(\Lambda)), & \qquad \Gamma_1:\mathcal{A}_0^{-1} \vect f + \Pi \vect g \to \Pi^* \vect f + \Lambda \vect g, \quad \vect f \in \mathcal{H}, \vect g \in \mathcal{D}(\Lambda).
    \end{aligned}
\end{equation}
We say that $(\mathcal{A},\Gamma_0,\Gamma_1)$ is the boundary triple associated with the Ryzhov triple $\left(\mathcal{A}_0, \Pi, \Lambda \right)$.
\end{definition}
Notice that, by definition, we have
\begin{equation}
\label{dirichlettoneumanndecomp}
  \Lambda \vect g = \Gamma_1 \Pi \vect g \quad \forall \vect g\in \mathcal{D}(\Lambda), \qquad \Pi^* \vect f = \Gamma_1 \mathcal{A}_0^{-1} \vect f \quad  \forall \vect f  \in \mathcal{H}.
\end{equation}
In what follows, we will consider the case when ${\mathcal A}_0$ is an (unbounded) differential operator, and the operators $\Gamma_0$ and $\Gamma_1$ assume  the r\^{o}les of the trace of a function and of its co-normal derivative on the boundary, respectively. The next result is then well expected and can be found in \cite{Ryzhov_spec}.
\begin{theorem}[Green's formula]
	Let $(\mathcal{A}_0,\Pi, \Lambda)$ be a Ryzhov triple. Then for the associated boundary triple $(\mathcal{A}$, $\Gamma_0$, $\Gamma_1)$ the following identity holds:  
\begin{equation}
\label{Greenformula}
    \left\langle \mathcal{A} \vect u, \vect v \right\rangle_{\mathcal{H}} - \left\langle  \vect u, \mathcal{A} \vect v \right\rangle_{\mathcal{H}} = \left\langle \Gamma_1 \vect u, \Gamma_0 \vect v \right\rangle_{\mathcal{E}} - \left\langle \Gamma_0 \vect u, \Gamma_1 \vect v \right\rangle_{\mathcal{E}} \qquad \forall \vect u, \vect v \in \mathcal{D}(\mathcal{A}_0) \dot + \Pi\bigl(\mathcal{D}(\Lambda)\bigr).
\end{equation}
\end{theorem}
\begin{definition}
\label{solutionmfunction}
 Suppose $z \in \rho(\mathcal{A}_0)$. Define the operator $S(z)$ mapping $\vect g \in \mathcal{E}$ to the solution $\vect u \in \mathcal{D}(\mathcal{A})=\mathcal{D}(\Gamma_0)$ of the spectral boundary value problem
\begin{equation*}
         \mathcal{A}\vect u = z \vect u,
         \qquad
         \Gamma_0  \vect u = \vect g.
\end{equation*}
The operator-valued function $M(z)$ defined on $\mathcal{D}(\Lambda)$ by
    $M(z)\vect:=\Gamma_1 S(z),$
is called the Weyl $M$-function of the Ryzhov triple $(\mathcal{A}_0,\Pi,\Lambda)$.
\end{definition}
In \cite{Ryzhov_spec} the following formulae for the operators $S(z)$ and $M(z)$ were proven ($z\in\rho({\mathcal A}_0)$): 
\begin{equation}
\label{solutionrepresentation}
    S(z) =  \left(I - z \mathcal{A}_0^{-1} \right)^{-1}\Pi, \qquad     M(z)= \Gamma_1 \left(I - z \mathcal{A}_0^{-1} \right)^{-1}\Pi.
\end{equation}
Note also that 
\begin{equation}
\label{usefulidentitiesresolvent}
    \left(I - z \mathcal{A}_0^{-1} \right)^{-1} = I + z \mathcal{A}_0^{-1}\left(I - z \mathcal{A}_0^{-1} \right)^{-1} = I + z(\mathcal{A}_0 - zI)^{-1},
\end{equation}
and therefore
\begin{equation}
\label{soperatoranotherform}
    S(z) = \Pi + z(\mathcal{A}_0 - zI)^{-1}\Pi.
\end{equation}
The next proposition lists key properties of the $M$-function.
\begin{proposition} (Properties of the Weyl $M$-function)
\begin{itemize}
    \item     The following representation holds:
    \begin{equation}
    \label{representationmfunctionformula}
        M(z) = \Lambda + z \Pi^*\left(I - z\mathcal{A}_0^{-1} \right)^{-1} \Pi, \quad z \in \rho(\mathcal{A}_0).
    \end{equation}
    \item $M(z)$ is an operator-valued function with values in the set of closed operators on $\mathcal{E}$ with ($z$-independent) domain $\mathcal{D}(\Lambda)$ such that $M(z)- \Lambda$ is analytic.
    \item For $z, \xi \in \rho(\mathcal{A}_0)$ the operator $M(z)-M(\xi)$ is bounded, and 
       $M(z)-M(\xi) = (z-\xi)\left( S(\overline{z})\right)^* S(\xi).$
    \item For $\vect u \in \ker(\mathcal{A}-zI) \cap \left\{ \mathcal{D}(\mathcal{A}_0) \dot{+}\Pi \mathcal{D}(\Lambda) \right\}$, the following formula holds:
    \begin{equation}
    \label{dtnmap}
        M(z) \Gamma_0 \vect u = \Gamma_1 \vect u.
    \end{equation}
\end{itemize}
\end{proposition}
\begin{remark}
    In the case when the operators $\Gamma_0$ and $\Gamma_1$ 
    represent the trace of a function and the trace of its co-normal derivative, the formula \eqref{dtnmap} clearly reveals the $M$-function $M(z)$ to be the DtN map associated with the resolvent problem.
\end{remark}
\begin{remark}
Due to the fact that $\mathcal{D}(M(z))=\mathcal{D}(M(z)^*)=\mathcal{D}(\Lambda)$ independently on $z \in \C$, one can define 
the operators
\begin{equation*}
    \Re M(z) =2^{-1}\bigl(M(z) + M(z)^*\bigr), \quad \Im M(z) =(2{\rm i})^{-1}\bigl(M(z) - M(z)^*)
\end{equation*}
on $\mathcal{D}(\Lambda)$. Note that by \eqref{representationmfunctionformula} and the fact that $\Lambda$ is self-adjoint, one has $M(z)^* = M(\overline{z}),$ and therefore 
\begin{equation}
	\label{nakk301} 
      \Im M(z) =(2{\rm i })^{-1}\bigl(M(z) - M(z)^*\bigr) = \Im z\left( S(\overline{z})\right)^* S(\overline{z}), \quad \Im M(z) = \Im (M(z) - \Lambda) = \Im (M(z) -  M(0)).
\end{equation}
\end{remark}

\begin{remark}
It is clear from \eqref{representationmfunctionformula} and \eqref{usefulidentitiesresolvent} that
\begin{equation}
    \label{identityforasymptotics}
    M(z) = \Lambda + z\Pi^* \Pi + z^2\Pi^* (\mathcal{A}_0 - zI)^{-1} \Pi.
\end{equation}
This formula will prove to be one of the key elements in deriving the asymptotics, as $\varepsilon\to0,$ of the resolvents $\left(\varepsilon^{-2}\mathcal{A}_{\chi,\varepsilon} -zI\right)^{-1}$ of the operators $\mathcal{A}_{\chi,\varepsilon}$ introduced in Section \ref{Gelf_sec}. 
\end{remark}

For a given $\mathcal{A}_0,$ we define $\mathcal{A}_{00}$ to be the restriction of $\mathcal{A}_0$ to the set $\mathcal{D}(\mathcal{A}_{00}):=\ker(\Gamma_0)\cap\ker(\Gamma_1)$. 
\begin{remark}
It was shown in \cite{Ryzhov_spec} that  $\mathcal{D}(\mathcal{A}_{00})$ does not actually depend on the choice of the operator $\Gamma_1$ (or $\Lambda$) and can be characterised as the subspace of $\mathcal{D}(\mathcal{A}_0)$ consisting of those elements $\vect u$ for which $\mathcal{A}_0 \vect u $ is orthogonal to the range of $\Pi$.
\end{remark}
One can characterise a wide class of densely defined closed extensions of $\mathcal{A}_{00}$ contained in $\mathcal{A}$ to be the operators $\mathcal{A}_{\beta_0,\beta_1}$ associated with the spectral boundary value problem
    $\mathcal{A}\vect u - z \vect u = \vect f$
subject to an abstract Robin-type condition
\begin{equation}
\label{robintypecondition}
    \left(\beta_0 \Gamma_0 + \beta_1 \Gamma_1\right) \vect u = 0,
\end{equation}
by varying over the choice of the operators $\beta_0$, $\beta_1$ on $\mathcal{E}$. The rigorous definition of the extension operators $\mathcal{A}_{\beta_0,\beta_1}$ is postponed to Theorem \ref{theoremkreinformula}. Note that $\mathcal{A}_0$ is then the self-adjoint extension of $\mathcal{A}_{00}$ corresponding to the choice $\beta_0 = I$, $\beta_1 = 0$.
However, in order to clarify the meaning of  \eqref{robintypecondition}, it is necessary to make additional assumptions on the operators $\beta_0, \beta_1$, e.g., as follows.
\begin{assumption}
\label{assumptionbeta}
The operators $\beta_0, \beta_1$ are linear in $\mathcal{E}$ and such that
 $\mathcal{D}(\beta_0) \supset \mathcal{D}(\Lambda)$ and $\beta_1$ is bounded  on $\mathcal{E}.$  The operator $\beta_0 + \beta_1 \Lambda$, defined on $\mathcal{D}(\Lambda),$ is closable in $\mathcal{E}$. We denote its closure by $\mathfrak{B}$.
\end{assumption}
 Under the above assumption, the operator $\beta_0 + \beta_1 M(z)$ is also closable. The condition $\eqref{robintypecondition}$ is shown to be well posed on a certain Hilbert space associated with the closure of $\beta_0 + \beta_1 \Lambda$ in $\mathcal{E}$.
\begin{definition}
Consider a separable Hilbert space $\mathcal{H},$  an auxiliary Hilbert space $\mathcal{E},$ and suppose that $\left(\mathcal{A}_0, \Pi, \Lambda \right)$ is a Ryzhov triple on $(\mathcal{H},\mathcal{E})$. Suppose also that $\beta_0, \beta_1$ are linear operators on $\mathcal{E}$ satisfying Assumption \ref{assumptionbeta}. Consider the space \begin{equation*}
    \mathcal{H}_{\beta_0,\beta_1}:= \mathcal{D}(\mathcal{A}_0) \dot + \Pi\bigl(\mathcal{D}(\mathfrak{B})\bigr)  \subset \bigl\{\mathcal{A}_0^{-1} \vect{f}+\Pi \vect{g} :\vect{f} \in \mathcal{H}, \ \vect{g} \in \mathcal{E}\bigr\},
\end{equation*}
equipped with the norm
\begin{equation*}
    \bigl\lVert  \mathcal{A}_0^{-1}\vect f + \Pi \vect g  \bigr\rVert_{\mathcal{H}_{\beta_0,\beta_1}}:= \left(\left\lVert \vect f \right\rVert^2_{\mathcal{H}} + \left\lVert \vect g \right\rVert^2_{\mathcal{E}} + \left\lVert \mathfrak{B}\vect g \right\rVert_{\mathcal{E}}^2 \right)^{1/2},
\end{equation*}
\end{definition}
The following lemma is proved in \cite{Ryzhov_spec}.
\begin{lemma}
The space $(\mathcal{H}_{\beta_0,\beta_1}, \left\lVert \cdot \right\rVert_{\beta_0,\beta_1})$ is a Hilbert space. The operator $\beta_0\Gamma_0 + \beta_1 \Gamma_1 : \mathcal{H}_{\beta_0,\beta_1} \to \mathcal{E}$ is bounded.
\end{lemma}
We are now in a position to assign a meaning to 
the abstract spectral boundary value problem and establish its well-posedness.
\begin{theorem}
\label{theoremsolutionformularobin}
Suppose that $z \in \rho(\mathcal{A}_0)$ is such that the operator $\overline{\beta_0 + \beta_1 M(z)}$ is boundedly invertible in $\mathcal{E}$. Then, for given $\vect f \in \mathcal{H},$ $\vect g\in \mathcal{E},$ the unique solution $\vect u  \in \mathcal{H}_{\beta_0,\beta_1}$ to the  spectral boundary value problem
\begin{equation*}
         \mathcal{A}\vect u - z \vect u = \vect f,\qquad
         \left(\beta_0 \Gamma_0 + \beta_1 \Gamma_1\right) \vect u = \vect g,
\end{equation*}
is provided by the formula
\begin{equation}
\label{solutionabstractrobin}
    \vect u = (\mathcal{A}_0 -zI)^{-1}\vect f + \left(I-z \mathcal{A}_0^{-1}\right)^{-1}\Pi(\overline{\beta_0 + \beta_1 M(z)})^{-1}\bigl(\vect g - \beta_1 \Pi^*\left( I - z \mathcal{A}_0^{-1}\right)^{-1} \vect f \bigr).
\end{equation}
\end{theorem}
By setting $\vect g = 0,$ the formula \eqref{solutionabstractrobin} defines the resolvent of a closed, densely defined operator in $\mathcal{H}$ that is an extension of $\mathcal{A}_{00}$.
\begin{theorem}
\label{theoremkreinformula}
Suppose that $z \in \rho(\mathcal{A}_0)$ be such that the operator $\overline{\beta_0 + \beta_1 M(z)}$ defined on $\mathcal{D}(\mathfrak{B})$ is boundedly invertible in $\mathcal{E}$. Then the operator $\mathcal{R}_{\beta_0,\beta_1}(z)$ defined by
\begin{equation}
\label{kreinformula2}
 \begin{aligned}
   \mathcal{R}_{\beta_0,\beta_1}(z):=&(\mathcal{A}_0 -zI)^{-1} - \left(I-z \mathcal{A}_0^{-1}\right)^{-1}\Pi(\overline{\beta_0 + \beta_1 M(z)})^{-1}\beta_1\bigl(\Pi^*\left( I - z \mathcal{A}_0^{-1}\right)^{-1} \bigr)\\[0.3em]
   =&(\mathcal{A}_0 -zI)^{-1}-S(z)(\overline{\beta_0 + \beta_1 M(z)})^{-1}\beta_1S^*(\overline{z}).
\end{aligned}
\end{equation}
is the resolvent 
$\left( \mathcal{A}_{\beta_0, \beta_1} - zI\right)^{-1}$ of a closed densely defined operator $\mathcal{A}_{\beta_0,\beta_1}$ in $\mathcal{H}$
such that
\begin{equation*}
   \mathcal{A}_{00} \subset \mathcal{A}_{\beta_0,\beta_1} \subset \mathcal{A}, \quad  \mathcal{D}(\mathcal{A}_{\beta_0, \beta_1}) \subset \ker(\beta_0\Gamma_0 +\beta_1\Gamma_1).
\end{equation*}
\end{theorem}
\begin{remark}
    Notice that the pointwise nature of Theorem \ref{theoremsolutionformularobin} with respect to the parameter $z$ yields the same formula \eqref{solutionabstractrobin} under an even more general assumption that the operator $\beta_1$ depends on $z,$ see also \cite{Derkach}. 
\end{remark}

\begin{remark}
     Corollary 5.9 in \cite{Ryzhov_spec}  states that if $\Lambda$ is boundedly invertible then the map $M(z)$ is also boundedly invertible for all $z$ in the complement of the set $\sigma(\mathcal{A}_0) \cup \sigma(\mathcal{A}_{0,I})$. In the case that we will analyse below, the operator $\Lambda$ will be boundedly invertible for all non-zero values of quasimomentum $\chi.$
\end{remark}

\subsection{Operators associated with boundary value problems} \label{opapth2} 
In this section we introduce some operators required to reformulate the transmission value problem \eqref{transmissionboundaryproblem} in the context of the abstract theory of Ryzhov triples. First, we define the spaces
\begin{equation*}
    \mathcal{H}:=L^2(Y;\C^3),  \quad \mathcal{E}:=L^2(\Gamma;\C^3), \quad  \mathcal{H}^{\rm stiff}:= L^2(Y_{\rm stiff};\C^3), \quad \mathcal{H}^{\rm soft}:= L^2(Y_{\rm soft};\C^3).
\end{equation*}    
Note that we can identify $\mathcal{H}^{\rm stiff(soft)}$ with the following spaces:
\begin{equation*}
    \mathcal{H}^{\rm stiff} \equiv \bigl\{\vect u \in L^2(Y;\C^3), \quad \vect u = 0 \mbox{ on } Y_{\rm soft} \bigr\}, \quad \mathcal{H}^{\rm soft}\equiv\bigl\{\vect u \in L^2(Y;\C^3), \quad \vect u = 0 \mbox{ on } Y_{\rm stiff}\bigr\}.
\end{equation*} 
We also define the associated orthogonal projections
    $P_{\rm soft} :\mathcal{H} \mapsto \mathcal{H}^{\rm soft},$ $P_{\rm stiff} :\mathcal{H} \mapsto \mathcal{H}^{\rm stiff},$
so the following orthogonal decomposition holds:
\begin{equation}
\label{decompositionsoftstiff}
    \mathcal{H}= P_{\rm soft} \mathcal{H} \oplus P_{\rm stiff} \mathcal{H} = \mathcal{H}^{\rm soft} \oplus \mathcal{H}^{\rm stiff}.
\end{equation}
\subsubsection{Differential operators of linear elasticity}
Relative to the decomposition \eqref{decompositionsoftstiff}, we  define self-adjoint operators  $\mathcal{A}_{0, \chi}^{\rm stiff}$, $\mathcal{A}_{0, \chi}^{\rm soft}$ on the spaces $\mathcal{H}^{\rm stiff}$, $\mathcal{H}^{\rm soft}$, respectively, by the sesquilinear forms
\begin{equation}\label{nakk71} 
    \begin{aligned}
    	a_{0, \chi}^{\rm stiff} (\vect u,\vect v) &:= 	\int_{ Y_{\rm stiff}} \A_{\rm stiff}(\simgrad+{\rm i}X_\chi)\vect u: \overline{(\simgrad+{\rm i}X_\chi )\vect v}, \quad \vect u,\vect v\in \mathcal{D}(a_{0, \chi}^{\rm stiff}), \\[0.3em]
         a_{0, \chi}^{\rm soft} (\vect u,\vect v) &:= 	\int_{ Y_{\rm soft}} \A_{\rm soft}(\simgrad+{\rm i}X_\chi)\vect u: \overline{(\simgrad +{\rm i}X_\chi )\vect v}, \quad \vect u,\vect v\in \mathcal{D}(a_{0, \chi}^{\rm soft}), \\[0.3em]
                  \mathcal{D}(a_{0, \chi}^{\rm stiff})&:=\{ \vect u \in H_\#^1(Y_{\rm stiff};\C^3),\, \vect u|_{\Gamma} = 0\}, \qquad\mathcal{D}(a_{0, \chi}^{\rm soft}):= \{ \vect u \in H^1(Y_{\rm soft};\C^3),\, \vect u|_{\Gamma} = 0\}.
    \end{aligned}
\end{equation}
The following basic property is easy to prove. 
\begin{proposition}
The forms $a_{0,\chi}^{\rm stiff(soft)}$ are uniformly coercive, symmetric and closed.
\end{proposition}
\begin{proof}
Due to Assumption \ref{coffassumption}, there exist $\chi$-independent constants $C_1, C_2>0$ such that 
\begin{equation*}
	\begin{aligned}
    C_1 \left\lVert \left(\simgrad +{\rm i}X_\chi \right)\vect u\right\rVert_{L^2(Y_{\rm stiff(soft)};\C^{3 \times 3})}^2 \leq  a_{0, \chi}^{\rm (stiff)soft}(\vect u, \vect u)\leq C_2&\left\lVert \left(\simgrad+{\rm i}X_\chi \right)\vect u\right\rVert_{L^2(Y_{\rm stiff(soft)};\C^{3 \times 3})}^2\\[0.3em]
    &\forall\vect u \in \mathcal{D}(a_{0, \chi}^{\rm (stiff)soft}).
\end{aligned}
\end{equation*}

Furthermore, by Proposition \ref{josipapp1} (see the Appendix), there exists a $\chi$-independent $C>0$ such that
 \begin{equation*}
     a_{0, \chi}^{\rm (stiff)soft}(\vect u, \vect u)\geq C\lVert \vect u\rVert_{H^1(Y_{\rm stiff(soft)};\C^3)}^2\qquad\forall\vect u \in \mathcal{D}(a_{0, \chi}^{\rm (stiff)soft}),
 \end{equation*}
so both forms are uniformly coercive.
\end{proof}
Clearly, the operators $\mathcal{A}_{0, \chi}^{\rm stiff(soft)}$
correspond to the differential expressions
\begin{equation}
	\label{differentialoperators}
		\left(\simgrad+{\rm i}X_\chi\right)^* \A_{\rm stiff(soft)} \left(\simgrad +{\rm i}X_\chi\right)  \quad \mbox{ on $\mathcal{H}^{\rm stiff(soft)}$},
\end{equation}
%
subject to the zero boundary condition on $\Gamma$.

\subsubsection{Lift operators}
Here we introduce the lift operators required for the analysis of boundary value problems via the Kre\u\i n formula. First, we introduce classical lift operators $\widetilde{\Pi}_{\chi}^{\rm stiff(soft)}$ as the operators mapping $\vect g \in H^{1/2}(\Gamma;\C^3)$ to the weak solutions $\vect u \in H^{1}_\#(Y_{\rm stiff(soft)};\C^3)$ of the boundary value problems
\begin{equation}
\label{boundaryvalueproblemsoncomponents}
    \left\{ \begin{array}{ll}
         \left(\simgrad+{\rm i}X_\chi\right)^* \A_{\rm stiff(soft)}\left(\simgrad +{\rm i} X_\chi\right)\vect u = 0 \ \  \mbox{ on $Y_{\rm stiff(soft)}$},\\[0.4em]
         \vect u =\vect g  \ \  \mbox{ on $\Gamma,$ }\qquad 
         \mbox{ $\vect u$ is $Y$-periodic (in the case of $Y_{\rm stiff}$).}\end{array} \right.
\end{equation}
As a consequence of Proposition \ref{josipapp1}, the following statement holds. 
\begin{proposition}
\label{propositionpiestimate}
For every $\chi \in Y$ the problem \eqref{boundaryvalueproblemsoncomponents} has a unique weak solution. The associated operator $\widetilde{\Pi}_{\chi}^{\rm stiff(soft)}$ is bounded and satisfies the following bounds for all $\vect g\in H^{1/2}(\Gamma;\C^3):$ 
\begin{equation}
\label{piestimate2}
    \bigl\lVert \widetilde{\Pi}_\chi^{\rm stiff(soft)} \vect g \bigr\rVert_{H^1(Y_{\rm stiff(soft)};\C^3)} \leq C \left\lVert \vect g \right\rVert_{H^{1/2}(\Gamma;\C^3)},
\end{equation}
where $C>0$ is independent of $\chi$. 

\begin{proof} 
The proof uses the classical approach that involves rewriting the problem \eqref{boundaryvalueproblemsoncomponents} with zero boundary condition and non-zero right-hand side and using Proposition \ref{josipapp2} and Proposition \ref{josipapp1}.  	
\end{proof} 	
\end{proposition}
We next introduce specific realisations $\Pi_{\chi}^{\rm stiff(soft)}$ of the abstract lift operators $\Pi$ (see Section \ref{opthap1}), natural for the problem at hand, by their adjoints $\Xi_\chi^{\rm stiff(soft)}$  with respect to the pair of inner products $\left\langle \cdot, \cdot\right\rangle_{\mathcal{H}^{\rm stiff(soft)}}$ and $\left\langle \cdot, \cdot\right\rangle_{\mathcal{E}}$, namely
\begin{equation*}
	\Pi_\chi^{\rm stiff(soft)}:\mathcal{E} \to \mathcal{H}^{\rm stiff(soft)}, \qquad \bigl\langle \Pi_\chi^{\rm stiff(soft)} \vect g, \vect f \bigr\rangle_{\mathcal{H}^{\rm stiff(soft)}} = \bigl\langle  \vect g , \Xi_\chi^{\rm stiff(soft)} \vect f \bigr\rangle_{\mathcal{E}}.
\end{equation*}
We use the following definition:
\begin{equation}
\label{xidefinition}
    -\partial^{\rm stiff(soft)}_\nu  \bigl(\mathcal{A}_{0,\chi}^{\rm stiff(soft)}\bigr)^{-1}=:\Xi_\chi^{\rm stiff(soft)}: \mathcal{H}^{\rm stiff(soft)}\to{\mathcal E}.
\end{equation}
Here $\partial^{\rm stiff(soft)}_\nu$ \BBB is the trace of the co-normal derivative
   $\partial^{\rm stiff(soft)}_\nu \vect u := \bigl(\A_{\rm stiff(soft)} (\simgrad+{\rm i}X_{\chi}) \vect u\bigr)\cdot {\vect n}_{\rm stiff(soft)} |_{\Gamma}.$
 The adjoints of $\Xi_\chi^{\rm stiff(soft)}$ coincide with the closures of  $\widetilde{\Pi}_{\chi}^{\rm stiff(soft)}$ in the space  ${\mathcal E}=L^2(\Gamma;{\mathbb C}^3).$ Indeed, we have the following theorem.

\begin{theorem}\label{nakk200} 
 The operators $\Xi_\chi^{\rm stiff(soft)}$ defined by \eqref{xidefinition} are compact. 
 Their adjoints  $\Pi_\chi^{\rm stiff(soft)}$ are (compact) closures in $\mathcal{E}$ of the classical lift operators $\widetilde{\Pi}_{\chi}^{\rm stiff(soft)},$ and (cf. Definition \ref{ryzhovtriple})
 \begin{equation}
 	\label{nakk60_bis} 
     \ker\bigl(\Pi_\chi^{\rm stiff(soft)}\bigr) =  \left\{ 0\right\}, \qquad \mathcal{D}\bigl(\mathcal{A}_{0,\chi}^{\rm stiff(soft)}\bigr)\cap \mathcal{R}\bigl(\Pi_\chi^{\rm stiff(soft)}\bigr) = \{0\}.
 \end{equation}
\end{theorem}
\begin{proof}
Due to results on elliptic regularity, under Assumption \ref{coffassumption} the operators $\bigl(\mathcal{A}_{0,\chi}^{\rm stiff(soft)}\bigr)^{-1}$ are bounded from $\mathcal{H}^{\rm stiff(soft)}$ to $H^{2}( \Omega_{\rm stiff(soft)};\C^3)$ by Lemma \ref{appendixregularityofdirichlet}. Thus, using the trace theorem, we infer that $\Xi_\chi^{\rm stiff(soft)}$ is bounded as an operator to $H^{1/2}(\Gamma;\C^3).$
Due to the compactness of the embedding $H^{1/2}(\Gamma;\C^3)\hookrightarrow \mathcal{E}$, it follows that $\Xi_\chi^{\rm stiff(soft)}$ is compact.\footnote{By following the argument of \cite{Grubb}, it can be shown that not only $\Xi^{\rm stiff(soft)}$ are compact, but they also belong to certain weak Schatten -- von-Neumann classes.} Next, one can easily verify by integration by parts that for $\vect g \in H^{1/2}(\Gamma;\C^3)$, $\vect f \in \mathcal{H}^{\rm stiff(soft)}$ one has
\begin{equation*}
    \bigl\langle \widetilde{\Pi}_\chi^{\rm stiff(soft)} \vect g, \vect  f \bigr\rangle_{\mathcal{H}^{\rm stiff(soft)}} = \Bigl\langle \vect g , -\partial^{\rm stiff(soft)}_\nu\bigl(\mathcal{A}_{0,\chi}^{\rm stiff(soft)}\bigr)^{-1} \vect f \Bigr\rangle_{\mathcal{E}},
\end{equation*}
and hence
\begin{equation*}
    \Pi_\chi^{\rm stiff(soft)}\bigr|_{H^{1/2}(\Gamma;\C^3)} = \widetilde{ \Pi }_\chi^{\rm stiff(soft)}.
\end{equation*}

Proceeding to the proof of (\ref{nakk60_bis}), note that all $\vect g\in\mathcal{E}$ satisfy
\begin{equation*}
     \left(\simgrad+{\rm i}X_\chi\right)^* \A_{\rm stiff(soft)} \left(\simgrad + {\rm i}X_\chi\right)\Pi_\chi^{\rm stiff(soft)} \vect g = 0
\end{equation*}
in the sense of distributions.
Provided $\Pi_\chi^{\rm stiff(soft)} \vect g \in \mathcal{D}(\mathcal{A}_{0,\chi}^{\rm stiff(soft)})$, one then has, for all $\vect v \in \mathcal{C}_c^{\infty}(Y_{\rm stiff(soft)};\C^3),$
\begin{equation*}
	\begin{aligned}
	\int_{Y_{\rm stiff(soft)}}\mathcal{A}_{0,\chi}^{\rm stiff(soft)}\Pi_\chi^{\rm stiff(soft)}\vect g: \vect v
    &=\int_{Y_{\rm stiff(soft)}}\A_{\rm stiff(soft)}  \left(\simgrad+{\rm i} X_\chi\right)\Pi_\chi^{\rm stiff(soft)}\vect g: \left(\simgrad +{\rm i}X_\chi\right)\vect v\\[0.4em]
    &=\int_{Y_{\rm stiff(soft)}} \Pi_\chi^{\rm stiff(soft)} \vect g \cdot \left(\simgrad+{\rm i} X_\chi\right)^*\A_{\rm stiff(soft)}\left(\simgrad +{\rm i}X_\chi\right) \vect v
    =0,
    \end{aligned}
\end{equation*}
from which it follows immediately that $\Pi_\chi^{\rm stiff(soft)} \vect g = 0$.  This proves the second property in \eqref{nakk60_bis}. 

To prove the first property in \eqref{nakk60_bis}, choose an arbitrary $\vect g \in H^1(\Gamma;\C^3)$ and let $\vect u \in H^2(Y_{\rm stiff(soft)};\C^3)$ be such that \begin{equation*}
        \partial^{\rm stiff(soft)}_\nu \vect u\vert_\Gamma = -\vect g  
        \qquad 
          \vect u\vert_\Gamma = 0,\qquad 
          \mbox{$\vect u$ is $Y$-periodic. }
\end{equation*}
The existence of such $\vect u$ is guaranteed by the Lemma  \ref{traceextensionlemma}. Now, denoting
\begin{equation*}
    \vect f:= \left(\simgrad+{\rm i}X_\chi\right)^* \A_{\rm stiff(soft)}\left(\simgrad+{\rm i} X_\chi\right)\vect u \in \mathcal{H}^{\rm stiff(soft)},
\end{equation*}
it is clear that
    $\vect g = \Xi_\chi^{\rm stiff(soft)} \vect f,$
and, due to the density of $H^1(\Gamma;\C^3)$ in $\mathcal{E}$, we conclude that
\begin{equation*}
    \ker\bigl(\Pi_\chi^{\rm stiff(soft)}\bigr) =  \overline{\mathcal{R}\bigl(\Xi_\chi^{\rm stiff(soft)}\bigr)}^\perp = \left\{ 0\right\}.\qedhere\popQED
\end{equation*}
\end{proof}

Henceforth, we will be using the notation $\Pi_\chi^{\rm stiff(soft)}$ also for the operator $\widetilde{\Pi}_\chi^{\rm stiff(soft)},$ as it will always be clear from the context which one we are referring to. 

A more precise statement than that of Theorem \ref{nakk200} is available, although we do not use it in what follows -- as it is of an independent interest, we include it next.

\begin{theorem}
	The operators $\Xi_\chi^{\rm stiff(soft)}$ are bounded from ${\mathcal H}^{\rm stiff(soft)}$ to $H^{1/2}(\Gamma),$ and the their adjoints  $\Pi_\chi^{\rm stiff(soft)}$ are bounded from $L^2(\Gamma)$ to $H^{1/2}(Y_{\rm stiff(soft)}).$
\end{theorem}
\begin{proof}
	The claim pertaining to the operators $\Xi_\chi^{\rm stiff(soft)}$, which are the adjoints of $\Pi_\chi^{\rm stiff(soft)},$ is verified in the proof of Theorem \ref{nakk200} above.
	
	Passing over to the operators $\Pi_\chi^{\rm stiff(soft)}$, first notice that by Sobolev duality  $(\Xi_\chi^{\rm stiff(soft)})^*$ admit bounded extensions to operators acting from $H^{-1/2}(\Gamma)$ to ${\mathcal H}^{\rm stiff(soft)}$. The claim now follows by linear interpolation \cite{Triebel} between this fact and the boundedness of $\Pi_\chi^{\rm stiff(soft)}$ as operators from $H^{1/2}(\Gamma)$ to $H^1(Y_{\rm stiff(soft)})$, established in the estimate \eqref{piestimate2}.
\end{proof}

\subsubsection{Dirichlet-to-Neumann maps}
\label{DtN_sec} 
 
Here we introduce specific realisations  $\Lambda_\chi^{\rm stiff(soft)}$ of the operator $\Lambda$ (see Section \ref{opthap1}) and prove their self-adjointness. The main result of this section is Theorem \ref{dtnmapstheorem1}.

We begin by noting that, due to the elliptic regularity, under Assumption \ref{coffassumption} one can define the operators \cite{Agranovich}
\begin{equation*}
    \widetilde{\Lambda}_\chi^{\rm stiff(soft)}:H^{3/2}(\Gamma;\C^3) \to H^{1/2}(\Gamma;\C^3), \qquad \widetilde{\Lambda}_\chi^{\rm stiff(soft)}\vect g=-\partial_{\nu}^{\rm stiff(soft)} \vect{u},
\end{equation*}
where $\vect g\in H^{3/2}(\Gamma;\C^3)$ and $\vect u$ is the unique solution to \eqref{boundaryvalueproblemsoncomponents}. Notice that  as a consequence of Lemma \ref{appendixregularityofdirichlet} we have
\begin{equation} \label{nakk202} 
 \bigl\|\widetilde{\Lambda}_\chi^{\rm stiff(soft)} \vect{g}\|_{H^{1/2}(\Gamma;\C^3))} \leq C\bigr\| \vect{g} \|_{H^{3/2}(\Gamma;\C^3)}, 
\end{equation} 
where the constant $C$ is independent of $\chi.$ 
\begin{theorem}
\label{dtnmapstheorem1}
 The operators $\widetilde{\Lambda}_\chi^{\rm stiff(soft)}$ can be uniquely extended to unbounded non-positive self-adjoint operators $\Lambda_\chi^{\rm stiff(soft)}$ on $\mathcal{E}$ with domains $\mathcal{D}( \Lambda_\chi^{\rm stiff(soft)})=H^1(\Gamma;\C^3)$. Furthermore, $\Lambda_\chi^{\rm stiff(soft)}$ have compact resolvents.
\end{theorem}
\begin{proof}
The first step consists in obtaining a larger extension of the operator $\widetilde{\Lambda}_\chi^{\rm stiff(soft)}$ for which the desired operator $\Lambda_\chi^{\rm stiff(soft)}$ is simply a restriction onto $H^1(\Gamma;\C^3)$. The second step is to show that this restriction is, in fact, self-adjoint. For this, it suffices to show that the restriction is symmetric, and that the resolvent set of the restriction contains at least one real number, for example the unity. 

We begin by taking $\vect g \in H^{3/2}(\Gamma;\C^3),$ $\vect h \in H^{1/2}(\Gamma;\C^3)$ and considering the solutions $\vect u \in H^2(Y_{\rm stiff(soft)};\C^3),$ $\vect v \in H^1(Y_{\rm stiff(soft)};\C^3)$ to the problem \eqref{boundaryvalueproblemsoncomponents} such that
         $\vect u = \Pi_\chi^{\rm stiff(soft)}\vect g,$ $\vect v= \Pi_\chi^{\rm stiff(soft)}\vect h.$
Using Green's formula, we obtain
\begin{equation}
\label{bilinearformdtnmap}
\begin{aligned}
        \int_{\Gamma} \widetilde{\Lambda}_\chi^{\rm stiff(soft)} \vect g \cdot \vect h & = -\int_{Y_{\rm stiff(soft)}}\A_{\rm stiff(soft)}  \left(\simgrad+{\rm i} X_\chi\right)\vect u:  \left(\simgrad +{\rm i}X_\chi\right) \vect v \\
        & =-\int_{Y_{\rm stiff(soft)}}\A_{\rm stiff(soft)}  \left(\simgrad+{\rm i} X_\chi\right)\Pi_\chi^{\rm stiff(soft)}\vect g:  \left(\simgrad +{\rm i}X_\chi\right) \Pi_\chi^{\rm stiff(soft)}\vect h.
\end{aligned}
\end{equation}
It follows that $\widetilde{\Lambda}_\chi^{\rm stiff(soft)} \vect g$ defines an element of  $H^{-1/2}(\Gamma;\C^3)$. Due to \eqref{piestimate2}, the bound
\begin{equation}\label{eqnakk201} 
    \bigl\lVert \widetilde{\Lambda}_\chi^{\rm stiff(soft)} \vect g \bigr\rVert_{H^{-1/2}(\Gamma;\C^3)}\leq  C \left\lVert  \vect g \right\rVert_{H^{1/2}(\Gamma;\C^3)},
\end{equation}
holds with $C$ independent of $\chi\in Y'.$ In particular, we can define a unique bounded extension
\begin{equation*}
    \overline{\Lambda}_\chi^{\rm stiff(soft)}:H^{1/2}(\Gamma;\C^3) \to H^{-1/2}(\Gamma;\C^3),\qquad
    \overline{\Lambda}_\chi^{\rm stiff(soft)} |_{H^{3/2}(\Gamma;\C^3)} =  \widetilde{\Lambda}_\chi^{\rm stiff(soft)}.
\end{equation*}
Since $\left(H^1(\Gamma;\C^3),L^2(\Gamma;\C^3) \right)$ is an interpolation pair \cite{Triebel, KPS} with respect to the pairs $\left(H^{3/2}(\Gamma;\C^3),H^{1/2}(\Gamma;\C^3) \right)$ and  $\left(H^{1/2}(\Gamma;\C^3), H^{-1/2}(\Gamma;\C^3)\right)$, we conclude that $\overline{\Lambda}_\chi^{\rm stiff(soft)}$
is bounded as an operator from $H^1(\Gamma;\C^3)$ to $L^2(\Gamma;\C^3)$ and denote by $\Lambda_\chi^{\rm stiff(soft)}$ the restriction of $\overline{\Lambda}_\chi^{\rm stiff(soft)}$ to $H^1(\Gamma;\C^3)$. 

Proceeding to the second step, we prove that $\Lambda_\chi^{\rm stiff(soft)}$ is self-adjoint as an unbounded operator on $L^2(\Gamma;\C^3)$ with domain
    $\mathcal{D}(\Lambda_\chi^{\rm stiff(soft)}) = H^1(\Gamma;\C^3).$
Notice that $\Lambda_\chi^{\rm stiff(soft)}$ is non-positive and symmetric. Indeed,  as a consequence of \eqref{bilinearformdtnmap} holding for $\vect g \in H^{3/2}(\Gamma;\C^3),$
$\vect h \in H^{1/2}(\Gamma;\C^3),$ one has
\begin{equation*}
	\begin{aligned}
    \int_{\Gamma} \Lambda_\chi^{\rm stiff(soft)} \vect g \cdot \vect h  =-\int_{Y_{\rm stiff(soft)}}\A_{\rm stiff(soft)}  \left(\simgrad+{\rm i} X_\chi\right)\Pi_\chi^{\rm stiff(soft)}\vect g :&\left(\simgrad+{\rm i}X_\chi\right) \Pi_\chi^{\rm stiff(soft)}\vect h\\[0.3em] &\forall\vect g, \vect h \in H^1(\Gamma;\C^3).
    \end{aligned}
\end{equation*}
In order to prove self-adjointness of $\Lambda_\chi^{\rm stiff(soft)}$, it suffices to show that $\rho(\Lambda_\chi^{\rm stiff(soft)}) \cap \R \neq \emptyset$.   We claim that $(-\Lambda_\chi^{\rm stiff(soft)}+I)^{-1}:L^2(\Gamma;\C^3) \to H^1(\Gamma;\C^3)$ is bounded. To see this, first assume that $\vect h \in H^{1/2}(\Gamma;\C^3)$ and seek the solution $\vect g\in H^{3/2}(\Gamma;\C^3)$ to the problem
    $-\Lambda_\chi^{\rm stiff(soft)} \vect g + \vect g = \vect h.$ 
This is equivalent to seeking $\vect u$ 
such that
\begin{equation}
\label{robinproblem1}
        \left\{ \begin{array}{ll}
         \left(\simgrad+{\rm i}X_\chi\right)^* \A_{\rm stiff(soft)} \left(\simgrad + {\rm i}X_\chi\right) \vect u  = 0 \quad &\mbox{ on $Y_{\rm stiff(soft)}$},\\[0.35em]
         \partial_{\nu}^{\rm stiff(soft)} \vect u + \vect u = \vect h , \quad &\mbox{ on $\Gamma$} \end{array} \right.
\end{equation}
in the weak sense, and then setting $\vect{g}:=\vect{u}|_{\Gamma}.$ Invoking Lemma \ref{appendixregularitiyofrobin} for the existence and uniqueness of such a solution, we thus obtain an operator $(-\Lambda_\chi^{\rm stiff(soft)}+I)^{-1}: H^{1/2}(\Gamma;\C^3) \to H^{3/2}(\Gamma;\C^3).$

Second, consider the form $a_{\chi}^{\rm stiff(soft)}$ defined by the same expression as  $a_{0, \chi}^{\rm stiff (soft)},$ see \eqref{nakk71}, but on the domain $\mathcal{D}(a_\chi^{\rm stiff(soft)}):= H_\#^1(Y_{\rm stiff(soft)};\C^3).$
Applying Korn's inequality (see Proposition \ref{josipapp1}) to the weak form of \eqref{robinproblem1}, namely
\begin{equation*}
	a_\chi^{\rm stiff(soft)}( \vect u, \vect v) + \int_{\Gamma} \vect u \cdot \vect v = \int_{\Gamma} \vect h \cdot \vect v \qquad \forall \vect v \in H^1(Y_{\rm stiff(soft)};\C^3),
\end{equation*}
  where we set $\vect v=\vect u,$ we obtain the apriori estimate
    $\left\lVert \vect u \right\rVert_{H^1(Y_{\rm stiff(soft)};\C^3)} \leq C  \left\lVert \vect h\right\rVert_{H^{-1/2}(\Gamma;\C^3)}.$
Therefore, the resolvent $(-\Lambda_\chi^{\rm stiff(soft)} + I )^{-1}$ can be extended to a bounded operator
    from $H^{-1/2}(\Gamma;\C^3)$ to $H^{1/2}(\Gamma;\C^3).$
Using an interpolation argument once again, we conclude that
    $(-\Lambda_\chi^{\rm stiff(soft)}+ I)^{-1}$
    is bounded from $L^2(\Gamma;\C^3)$ to  $H^1(\Gamma;\C^3)$ and is therefore compact 
as an operator on $L^2(\Gamma;\C^3).$  Thus, unity is in the resolvent set of $\Lambda_\chi^{\rm stiff(soft)}$.
 \end{proof}
\begin{remark}
	\label{igor3} 
Notice that as a consequence of the linear interpolation theory (see, e.g., \cite{Triebel, KPS}) as well as the bounds \eqref{nakk202}, \eqref{eqnakk201}, we have  
\begin{equation} 
	\label{nakk203} 
\bigl\|\Lambda_\chi^{\rm stiff(soft)} \vect{g}\bigr\|_{L^2(\Gamma;\C^3)} \leq C\|\vect{g}\|_{H^1(\Gamma;\C^3)} \qquad \forall \vect{g} \in H^1(\Gamma;\C^3),
\end{equation} 
where the constant $C$ does not  depend on $\chi$. 
\end{remark} 

Applying the Ryzhov triple framework \eqref{defopp11} to the lift and DtN operators provided by Theorems \ref{nakk200}, \ref{dtnmapstheorem1}, we define the operators $\mathcal{A}_{\chi}^{\rm stiff(soft)}, \Gamma_{1,\chi}^{\rm stiff(soft)}, \Gamma_{0,\chi}^{\rm stiff(soft)}$. In particular, one has 
\begin{align}
	\label{identity1}
	\Gamma_{0,\chi}^{\rm stiff(soft)} \vect u &= \vect u|_{\Gamma} \qquad  \forall \vect u \in  H^2_{\#}(Y_{\rm stiff(soft)};\mathbb{C}^3) + \Pi_\chi^{\rm stiff(soft)}(H^{1/2}(\Gamma;\C^3)), \\[0.3em]
	\label{identity2}
	\Gamma_{1,\chi}^{\rm stiff(soft)} \vect u &=-\partial_{\nu}^{\rm stiff(soft)} \vect u\big\vert_\Gamma \qquad \forall \vect u \in  H^2_{\#}(Y_{\rm stiff(soft)};\mathbb{C}^3), \\[0.3em]
	\label{identity3}
	\mathcal{A}_{\chi}^{\rm stiff(soft)} \vect u & = \left(\simgrad+{\rm i}X_\chi\right)^* \A_{\rm stiff(soft)} \left(\simgrad + {\rm i}X_\chi\right) \vect u \qquad \forall \vect u \in  H^2_{\#}(Y_{\rm stiff(soft)};\mathbb{C}^3).
\end{align}
(For clarity, we note that the domains of the operators $\Gamma_{1,\chi}^{\rm stiff(soft)}$ and $\mathcal{A}_{\chi}^{\rm stiff(soft)}$ are wider than the indicated spaces.)

An alternative approach to defining the operators $-\Lambda_\chi^{\rm stiff(soft)}$ (and thus $\Lambda_\chi^{\rm stiff(soft)}$) can be developed under weaker assumptions on the regularity of the domain and the operator coefficients by considering  sesquilinear forms $\lambda_\chi^{\rm stiff(soft)}:H^{1/2}(\Gamma;\C^3)\times H^{1/2}(\Gamma;\C^3) \to \R,$ as follows:
\begin{equation}
    \lambda_\chi^{\rm stiff(soft)}( \vect  g, \vect h):=a_{\chi}^{\rm stiff(soft)}\bigl(\Pi_\chi^{\rm stiff(soft)} \vect g, \Pi_\chi^{\rm stiff(soft)} \vect h\bigr),\qquad \vect g, \vect h \in H^{1/2}(\Gamma;\C^3).
    \label{lambda_chi_form_def}
\end{equation}
The next proposition provides a coercivity estimate for the form $\lambda^{\rm stiff}_\chi,$ which implies an estimate for the lowest eigenvalue of the operator $\Lambda^{\rm stiff}_\chi,$ see Section \ref{Steklov_eigenvalue_sec}.
\begin{proposition}
\label{generaldtn}
Let the boundary $\Gamma$ be of class $C^{0,1},$ and suppose the entries of the elasticity tensor $\A_{\rm stiff(soft)}$ are in  $L^\infty(Y_{\rm stiff(soft)}).$ The forms $\lambda_\chi^{\rm stiff(soft)}$ are symmetric, non-negative, closed, and densely defined in $\mathcal{E}.$ Furthermore, there exists a $\chi$-independent constant $C>0$ such that
\begin{equation*}
    |\chi|^2\lVert \vect g\rVert^2_{H^{1/2}(\Gamma;\C^3)}\leq C \lambda_\chi^{\rm stiff}(\vect g, \vect g)\qquad  \forall\vect g\in H^{1/2}(\Gamma;\C^3).
\end{equation*}
\end{proposition}

\begin{proof}
For $\vect g \in H^{1/2}(\Gamma;\C^3)$, due to the pointwise coercivity of $\A_{\rm stiff},$ one has
\begin{equation*}
        \left\lVert \left(\simgrad+{\rm i}X_\chi \right)\Pi_{\chi}^{\rm stiff} \vect g\right\rVert_{L^2(Y_{\rm stiff};\C^3)}^2 \leq C \lambda_\chi^{\rm stiff}(\vect g, \vect g).
\end{equation*}
Furthermore, due to the estimate \eqref{estimate2} and the trace theorem, we obtain 
\begin{equation*}
    |\chi|^2\lVert \vect g\rVert^2_{H^{1/2}(\Gamma;\C^3)} \leq |\chi|^2\left\lVert \Pi_{\chi}^{\rm stiff} \vect g\right\rVert_{H^1(Y_{\rm stiff};\C^3)}^2 \leq C \lambda_\chi^{\rm stiff}(\vect g, \vect g).\qedhere\popQED
\end{equation*}
\end{proof}
A representation theorem by Kato \cite{Kato} yields the existence of (non-positive) self-adjoint
operators $\Lambda_\chi^{\rm stiff(soft)}:\mathcal{E}\supseteq\mathcal{D}(\Lambda_\chi^{\rm stiff(soft)}) \to \mathcal{E}$
such that
\begin{equation*}
  \prescript{}{H^{-1/2}(\Gamma;\C^3)}{\bigl\langle}  \Lambda_\chi^{\rm stiff(soft)} \vect g, \vect v \bigr\rangle_{ H^{1/2}(\Gamma;\C^3)} =-\lambda_\chi^{\rm stiff(soft)}( \vect g, \vect  v)\qquad \forall \vect g,  \vect v  \in H^{1/2}(\Gamma;\C^3). 
\end{equation*}
 Here, $\prescript{}{H^{-1/2}}{\bigl\langle}  \cdot, \cdot \bigr\rangle_{ H^{1/2}}$ denotes the duality pairing between the spaces $H^{-1/2}$ and $H^{1/2}$. 
This approach to defining DtN maps allows us to relax the regularity assumptions (on the boundary and the operator coefficients). However, as a result, we lose the information on the domains of these maps. In particular, we no longer have the equality
    $\mathcal{D}(\Lambda_\chi^{\rm stiff}) =  \mathcal{D}(\Lambda_\chi^{\rm soft}) = H^1(\Gamma;\C^3),$
as in the Theorem \ref{dtnmapstheorem1}. We will not pursue this approach, but we will use Proposition \ref{generaldtn}   for estimating the smallest eigenvalues of the operator  $\Lambda_\chi^{\rm stiff(soft)}$.  The eigenvalues and eigenfunctions of DtN maps are usually referred to as Steklov eigenvalues and eigenfunctions, respectively.

\subsection{Transmission problem: reformulation in terms of Ryzhov triple} \label{secnakkk1}

The decoupled operator  $\mathcal{A}_{0, \chi, \varepsilon}$ is defined by 
\begin{equation*}
\mathcal{A}_{0,\chi,\varepsilon} := \mathcal{A}_{0, \chi}^{\rm soft} \oplus \varepsilon^{-2}\mathcal{A}_{0, \chi}^{\rm stiff}, \qquad \mathcal{D}\bigl( \mathcal{A}_{0,\chi,\varepsilon}\bigr) = \mathcal{D}\bigl(\mathcal{A}_{0, \chi}^{\rm soft} \bigr) \oplus \mathcal{D}\bigl(\mathcal{A}_{0, \chi}^{\rm stiff}\bigr),
\end{equation*}
relative to the decomposition \eqref{decompositionsoftstiff}.
Equivalently, its form is given by
\begin{equation*}
\begin{aligned}
a_{0, \chi, \varepsilon} (\vect u,\vect v) := 	\int_{ Y} \A^\varepsilon(\simgrad +{\rm i}X_\chi)\vect u: \overline{(\simgrad+{\rm i}X_\chi )\vect v},\qquad\quad
\vect u, \vect  v\in\mathcal{D}(a_{0, \chi, \varepsilon}) = \bigl\{ \vect u \in H_\#^1(Y;\C^3),\, \vect u |_\Gamma = 0\bigr\}.
\end{aligned}
\end{equation*}
We also define the coupled lift operator by
\begin{equation*}
\Pi_\chi : \mathcal{E} \to \mathcal{H}^{\rm soft} \oplus \mathcal{H}^{\rm stiff}, \quad  \Pi_\chi \vect g:= \Pi_\chi^{\rm soft} \vect g \oplus \Pi_\chi^{\rm stiff} \vect g,\qquad \vect g\in \mathcal{E}.
\end{equation*}
From Theorem \ref{nakk200} we have 
$\ker (\Pi_\chi) = \{0\},$ $\mathcal{D}(\mathcal{A}_{0,\chi,\varepsilon})\cap \mathcal{R}(\Pi_\chi) = \{0\}.$
In order to describe the condition of continuity of normal derivatives of \eqref{transmissionboundaryproblem}, we introduce the operator
$\Lambda_{\chi,\varepsilon}:=\varepsilon^{-2}\Lambda_\chi^{\rm stiff} + \Lambda_\chi^{\rm soft}.$

Both $\Lambda_\chi^{\rm stiff}$ and $\Lambda_\chi^{\rm soft}$  are self-adjoint on the common domain $H^1(\Gamma;\C^3)$ and non-positive. Therefore, the operator $\Lambda_{\chi,\varepsilon}$ is self-adjoint on $\mathcal{D}(\Lambda_{\chi,\varepsilon}):=H^1(\Gamma;\C^3)$. This fact very probably belongs to the domain of folklore, still we include the proof of it as kindly shared with the authors by Dr\,V.\,Sloushch, to whom we express our deep gratitude.

\begin{theorem}\label{thm:Voffka}
	
	Assuming there are two self-adjoint operators ${\mathcal{A}}\geq 0$ and ${\mathcal{B}}\geq 0$ with the same domain $\mathcal{D}(\mathcal{A})=\mathcal{D}(\mathcal{B})$ in the Hilbert space $\mathcal{H},$ the sum ${\mathcal{A}}+{\mathcal{B}}$ is self-adjoint on the same domain.
	
\end{theorem}

\begin{proof}
	
	Without loss of generality, assume that ${\mathcal{A}}$ and ${\mathcal{B}}$  are positive definite. By the closed graph theorem, ${\mathcal{A}}{\mathcal{B}}^{-1}$ and ${\mathcal{B}}{\mathcal{A}}^{-1}$ are bounded in $\mathcal{H}$. Pick a $\kappa$ such that $\|\kappa {\mathcal{B}}{\mathcal{A}}^{-1}\|\leq 1$. One clearly has $\mathcal{D}(\mathcal{A}) \subseteq \mathcal{D}(\kappa {\mathcal{B}})$ and $\|\kappa {\mathcal{B}} u\|\leq \|{\mathcal{A}}\vect u\|$ for $\vect u\in\mathcal{D}(\mathcal{A})$.

	By the Heinz inequality (see, e.g., \cite[Chapter 10, Section 4.2, Theorem. 3]{MR1192782}),
	$$
	\mathcal{D}(\mathcal{A})^{1/2}\subset \mathcal{D}(\kappa {\mathcal{B}})^{1/2},\qquad \|\kappa^{1/2}{\mathcal{B}}^{1/2}u\|\leq \|{\mathcal{A}}^{1/2}\vect u\|, \quad \vect u\in \mathcal{D}(\mathcal{A})^{1/2},
	$$
	hence $\mathcal{D}(\mathcal{A})^{1/2}\subset \mathcal{D}(\mathcal{B})^{1/2}$ and ${\mathcal{B}}^{1/2}{\mathcal{A}}^{-1/2}$ is bounded. Swapping the r\^{o}les of ${\mathcal{A}}$ and ${\mathcal{B}}$, $\mathcal{D}(\mathcal{B})^{1/2}\subset \mathcal{D}(\mathcal{A})^{1/2}$ and ${\mathcal{A}}^{1/2}{\mathcal{B}}^{-1/2}$ is bounded.

	It follows that ${\mathcal{B}}^{-1/2}{\mathcal{A}}{\mathcal{B}}^{-1/2}$ is bounded and non-negative. Since obviously
	$$
	({\mathcal{A}}+{\mathcal{B}})\vect u={\mathcal{B}}^{1/2}({\mathcal{B}}^{-1/2}{\mathcal{A}}{\mathcal{B}}^{-1/2}+I){\mathcal{B}}^{1/2}\vect u\qquad\forall\vect u\in \mathcal{D}({\mathcal{A}}),
	$$
	and ${\mathcal{B}}^{-1/2}{\mathcal{A}}{\mathcal{B}}^{-1/2}+I$ is boundedly invertible, $({\mathcal{A}}+{\mathcal{B}})^{-1}$ is bounded and defined on the whole $\mathcal{H}$, and therefore closed. Therefore, ${\mathcal{A}}+{\mathcal{B}}$ is self-adjoint on $\mathcal{D}(\mathcal{A})$, as required.	
\end{proof}

Thus we can define the ``coupled" Rhyzov triple
$\left(\mathcal{A}_{0,\chi,\varepsilon},  \Pi_\chi,  \Lambda_{\chi,\varepsilon} \right)$
associated with the transmission problem \eqref{transmissionboundaryproblem}. Note that, while the operators $\Lambda_{\chi}^{\rm stiff}$ and $\Lambda_{\chi}^{\rm soft}$ are used to calculate the normal derivative on the boundary $\Gamma$, the map $ \Lambda_{\chi,\varepsilon}$ instead yields the jump on $\Gamma$ from the conormal derivative on the soft component to the scaled conormal derivative on the stiff component. 
We can also introduce the transmission operators $\mathcal{A}_{\chi,\varepsilon},  \Gamma_{0,\chi},  \Gamma_{1,\chi,\varepsilon}$ associated with the above triple.
Clearly, we have
$\Gamma_{1,\chi,\varepsilon} =\varepsilon^{-2}\Gamma_{1,\chi}^{\rm stiff}P_{\rm stiff} + \Gamma_{1,\chi}^{\rm soft}P_{\rm soft}.$
Based on the above boundary triples, we use Definition \ref{solutionmfunction} to define the following pairs of solution operators and $M$-functions: 
\begin{equation*}
\bigl(S_{\chi}^{\rm stiff}(z),M_{\chi}^{\rm stiff}(z)\bigr), \ \  z \in \rho\bigl(\mathcal{A}_{0,\chi}^{\rm stiff}\bigr),    \quad \bigl(S_{\chi}^{\rm soft}(z),M_{\chi}^{\rm soft}(z)\bigr), \ \  z \in \rho\bigl(\mathcal{A}_{0,\chi}^{\rm soft}\bigr),\quad \bigl(S_{\chi, \varepsilon}(z),M_{\chi, \varepsilon}(z)\bigr), \ \  z \in \rho\bigl(\mathcal{A}_{0,\chi}\bigr).
\end{equation*}
Obviously, one has 
\begin{equation}
S_{\chi,\varepsilon}(z)= S_{\chi}^{\rm soft}(z) + S_{\chi}^{\rm stiff}(\varepsilon^2 z),
\qquad
M_{\chi,\varepsilon}(z)= M_{\chi}^{\rm soft}(z) +\varepsilon^{-2}M_{\chi}^{\rm stiff}(\varepsilon^2 z), \qquad z\in \rho(\mathcal{A}_{\chi,\varepsilon,0}).
\label{MMM} 
\end{equation}

In the context of introduced boundary triples, the transmission problem \eqref{transmissionboundaryproblem} can be formulated as finding $\vect u \in \mathcal{D}(\mathcal{A}_{\chi,\varepsilon})$ such that
$\mathcal{A}_{\chi,\varepsilon}\vect u - z \vect u = \vect f,$
$\Gamma_{1,\chi,\varepsilon}\vect u = 0.$
The corresponding solution operator is given by the ``Kre\u\i n formula"
\begin{equation}
\label{transmissionkrein}
\mathcal{R}_{\chi,\varepsilon}(z)=\left( \mathcal{A}_{0,\chi,\varepsilon} -z I  \right)^{-1} - S_{\chi, \varepsilon}(z) M_{\chi, \varepsilon}(z)^{-1} S_{\chi, \varepsilon}(\overline{z})^*,
\end{equation}
and we know it to be the resolvent of a closed extension $\left(\mathcal{A}_{\chi,\varepsilon}\right)_{0,I}$ of $\mathcal{A}_{0,\chi,\varepsilon}$.
\begin{remark}
	\label{tracescoincide}
	Notice that the continuity condition ${\vect u}_{\rm stiff}(y)= {\vect u}_{\rm soft}(y)$, $ y \in \Gamma$, is built into the domain of $\mathcal{A}_{\chi,\varepsilon}$. Namely, let $\vect u \in \mathcal{D}(\mathcal{A}_{\chi,\varepsilon}) = \mathcal{D}(\mathcal{A}_{0,\chi,\varepsilon})\dot{+} \Pi_\chi(\mathcal{E})$. Then
	\begin{equation*}
	\begin{aligned}
	\vect u = \mathcal{A}_{0,\chi,\varepsilon}^{-1} \vect f + \Pi_\chi \vect g&=\left(\bigl( \varepsilon^{-2}\mathcal{A}_{0,\chi,\varepsilon}^{\rm stiff}\bigr)^{-1}  P_{\rm stiff}\vect f + \Pi_\chi^{\rm stiff} \vect g \right)\oplus\left(\bigl( \mathcal{A}_{0,\chi,\varepsilon}^{\rm soft}\bigr)^{-1}  P_{\rm soft}\vect f + \Pi_\chi^{\rm soft} \vect g \right) 
	=: \vect u_{\rm stiff} \oplus \vect u_{\rm soft},
	\end{aligned}
	\end{equation*}
	and
	$\Gamma_{0,\chi}^{\rm stiff}(\vect u_{\rm stiff}) = \vect g = \Gamma_{0,\chi}^{\rm soft}(\vect u_{\rm soft}).$
	Thus, the domain $\mathcal{D}(\mathcal{A}_{\chi,\varepsilon})$ contains pairs of functions whose traces coincide on $\Gamma$, and the operator $\Gamma_{0,\chi}$ maps such pairs to this common trace.
\end{remark}

\begin{remark}
	In order to see that the abstract operator $\left(\mathcal{A}_{\chi,\varepsilon}\right)_{0,I} $ corresponds to the operator defined by the sesquilinear form $a_{\chi,\varepsilon}$  given by \eqref{formaachiepsilon} (cf.  Remark \ref{deinitiontransmission}), one can use the following observations: 
	\begin{itemize}
		\item Both operators are self-adjoint (to see that the first one is self adjoint, see \cite{Ryzhov_spec} Corollary 5.9 ). 
		\item The resolvent problem for the operator defined with the form $a_{\chi,\varepsilon}$ admits the following weak formulation: for $\vect f \in L^2(Y;\C^3)$, find $\vect u \in H^1_\#(Y;\C^3)$ such that
		$$\int_{ Y} \A^\varepsilon(\simgrad +{\rm i}X_\chi)\vect u: \overline{(\simgrad+{\rm i}X_\chi )\vect v} -z \int_{Y} \vect u \overline{\vect v} = \int_{Y} \vect f \overline{\vect v}, \qquad \forall \vect v \in H^2_\#(Y;\C^3). $$ Actually, the test functions can be taken to be in $H^1_\#(Y;\C^3)$.
		\item In order to conclude that the resolvent problem for the operator $\left(\mathcal{A}_{\chi,\varepsilon}\right)_{0,I} $ is also given by the above formula, one considers the problem of finding $\vect u \in \mathcal{D} (\left(\mathcal{A}_{\chi,\varepsilon}\right)_{0,I})$ such that
		$$ \bigl\langle  \left( (\mathcal{A}_{\chi,\varepsilon}\right)_{0,I} - z I)\vect u , \vect v\bigr\rangle_{L^2(Y;\C^3)} = \langle \vect f, \vect v\rangle_{L^2(Y;\C^3)} \qquad \forall \vect v \in H^2_\#(Y;\C^3).$$ One then uses Green's formula \eqref{Greenformula}, integration by parts, as well as the identities 
		\eqref{identity1}, \eqref{identity2}, and \eqref{identity3}.  
	\end{itemize}
\end{remark}

\section{Transmission problem: Ryzhov triple asymptotics}\label{section4} 
The goal of this section is to provide operator asymptotics with respect to the quasimomentum $\chi \in Y'$ for the operators of the Ryzhov triple associated with the stiff component that were introduced in the previous section. As we will see, the approximation on the stiff component pays a key r\^{o}le in the overall approximation, cf. Remark \ref{nakk2} below. 

We also show that the eigenspace of the DtN map $\Lambda_\chi^{\rm stiff}$ corresponding, for small $\chi$, to Steklov  eigenvalues of order $|\chi|^2$ is finite-dimensional. This fact is one of the key ingredients for providing resolvent asymptotics for the transmission problem \eqref{transmissionboundaryproblem}.

As we are about to see, the vector nature of the problem does not allow one to infer an asymptotic expansion for Steklov eigenvalues. However, it turns out to be sufficient for our purposes to obtain the asymptotics of the resolvents of DtN maps, and consequently, the asymptotics of  their spectral projections.

We conclude by providing simple approximations for the boundary operators on the soft component (Section \ref{soft_sec}).

\subsection{Lift operators: asymptotic properties}

We are interested in the asymptotics for the lift operator $\Pi_\chi^{\rm stiff}:H^{1/2}(\Gamma;\C^3) \to H_\#^1(Y_{\rm stiff};\C^3)$ defined by \eqref{boundaryvalueproblemsoncomponents}. (Recall that for each $\chi$ the operator $\Pi_\chi^{\rm stiff}$ is nothing but the closure of $\widetilde{\Pi}_\chi^{\rm stiff},$ which is defined by \eqref{boundaryvalueproblemsoncomponents}.) Naturally, the leading-order term is the operator $\Pi_0^{\rm stiff}:H^{1/2}(\Gamma;\C^3) \to H_\#^1(Y_{\rm stiff};\C^3),$ which does not depend on $\chi$. Note that for $\vect g \in H^{1/2}(\Gamma;\C^3)$ the function $\Pi_0^{\rm stiff}\vect g \in H_\#^1(Y_{\rm stiff};\C^3)$ satisfies the identity
\begin{equation}
    \label{pi0definition}
    \int_{Y_{\rm stiff}} \A_{\rm stiff} \left[ \simgrad \Pi_0^{\rm stiff} \vect g\right] : \overline{\simgrad \vect v}=0 \qquad \forall \vect v \in H_\#^1(Y_{\rm stiff};\C^3),\ \vect v\vert_\Gamma=0.
\end{equation}
The operator $\Pi_0^{\rm stiff}$ satisfies the estimate provided by the following lemma.
\begin{lemma}
There exist constants $C_1$, $C_2>0$ such that for all $\vect g \in H^{1/2}(\Gamma;\mathbb{C}^3)$ one has
\begin{equation*}
    C_1 \left\lVert \simgrad \Pi_0^{\rm stiff} \vect g \right\rVert_{L^2(Y;\C^{3 \times 3})} \leq \left\lVert \vect g - \frac{1}{|\Gamma|} \int_\Gamma \vect g \right\rVert_{H^{1/2}(\Gamma;\C^3)} \leq C_2 \left\lVert \simgrad \Pi_0^{\rm stiff} \vect g \right\rVert_{L^2(Y;\C^{3 \times 3})}.
\end{equation*}
\end{lemma}
\begin{proof}
The right-hand inequality is proved in Proposition \ref{nakk112} (see the Appendix). To prove the left-hand one, we observe that
\begin{equation*}
    \begin{aligned}
        \left\lVert \simgrad \Pi_0^{\rm stiff} \vect g  \right\rVert_{L^2(Y_{\rm stiff};\C^{3 \times 3})}  & =   \left\lVert \simgrad \Pi_0^{\rm stiff} \left(\vect g - \frac{1}{|\Gamma|} \int_\Gamma \vect g \right) \right\rVert_{L^2(Y_{\rm stiff};\C^{3 \times 3})} \leq \left\lVert \Pi_0^{\rm stiff} \left( \vect g - \frac{1}{|\Gamma|} \int_\Gamma \vect g \right) \right\rVert_{H^{1}(Y_{\rm stiff};\C^3)} \\[0.4em]
        & \stackrel{\text{Estimate \eqref{piestimate2}}}{\leq} C \left\lVert \vect g - \frac{1}{|\Gamma|} \int_\Gamma \vect g \right\rVert_{H^{1/2}(\Gamma;\C^3)},
    \end{aligned}
\end{equation*}
which concludes the proof.
\end{proof}
\begin{remark}
\label{pi0const}
    Notice here that the operator $\Pi_0^{\rm stiff}$ lifts constant functions on the boundary $\vect g \equiv C \in \C^3 \hookrightarrow \mathcal{E}$ to constant functions  defined by the same constant:
        $\Pi_0^{\rm stiff} \vect g = \vect G,$ $\vect G \equiv C \in \C^3 \hookrightarrow \mathcal{H}^{\rm stiff}.$
\end{remark}

The following theorem is crucial for understanding the asymptotics of the DtN map. As its proof follows a standard asymptotic argument \cite{Selden, Friedlander_2002}, we provide it in the Appendix.
\begin{theorem}
\label{Piasymptheorem}
 For each $n\in{\mathbb N},$ the operator $\Pi_\chi^{\rm stiff}$ admits the asymptotic expansion
 \begin{equation*}
    \Pi_\chi^{\rm stiff} = \Pi_0^{\rm stiff} + \widetilde{\Pi}_{\chi,1} + \widetilde{\Pi}_{\chi,2} + \dots +  \widetilde{\Pi}_{\chi,n} + \widetilde{\Pi}_{\chi,n}^{\rm error}, 
\end{equation*}
where the operators
 $\widetilde{\Pi}_{\chi,k}, \widetilde{\Pi}_{\chi,n}^{\rm error} :H^{1/2}(\Gamma;\C^3) \to  H_\#^1(Y_{\rm stiff};\C^3)$, $ k = 1, \dots, n$ are bounded
and satisfy
\begin{equation}
    \bigl\lVert \widetilde{\Pi}_{\chi,k} \bigr\rVert_{H^{1/2}(\Gamma;\C^3)\to H^1(Y_{\rm stiff};\C^3) } \leq C |\chi|^k, \quad k = 1,\dots, n, \quad \bigl\lVert \widetilde{\Pi}_{\chi,n}^{\rm error} \bigr\rVert_{H^{1/2}(\Gamma;\C^3)\to H^1(Y_{\rm stiff};\C^3) } \leq C |\chi|^{n+1},
    \label{Pi_error}
\end{equation}
and the constant $C>0$ does not depend on $\chi \in Y'$. Furthermore, 
\begin{equation*}
    \widetilde{\Pi}_{\chi,k} = \Pi_k : \chi^{\otimes k}, \qquad \left[\Pi_k\right]_{i_1,i_2,...,i_k} \in\mathfrak{B}\bigl( H^{1/2}(\Gamma;\C^3), H_\#^1(Y_{\rm stiff};\C^3)\bigr),
\end{equation*}
and the operator-valued tensors $\Pi_k$ are symmetric, i.e.
    $\left[\Pi_k\right]_{\sigma(i_1,i_2,...,i_k)} = \left[\Pi_k\right]_{i_1,i_2,...,i_k}$
    for all  $\sigma \in S_n,$
where $S_n$ is the permutation group of order $n.$
\end{theorem}
\begin{remark}
Here we make  an observation that proves to be crucial in the understanding of the homogenisation properties of the effective operator (see Lemma \ref{josipnak1}). By following the proof of Theorem \ref{Piasymptheorem} and taking into account the equation \eqref{definitionofpichi1}, which actually serves as the definition of the operator $\widetilde{\Pi}_{\chi,1}$, one concludes that for $\vect g \equiv C \in \C^3 \hookrightarrow \mathcal{E}$ one has
\begin{equation}
    \label{pichi1formula}
    \int_{Y_{\rm stiff}} \A_{\rm stiff} \bigl(\simgrad \widetilde{\Pi}_{\chi,1} \vect g +{\rm i}X_\chi \Pi_0^{\rm stiff} \vect g \bigr) :\simgrad \vect v \,dy =0 \qquad \forall \vect v\in H_\#^1(Y_{\rm stiff};\C^3),\ \vect v\vert_\Gamma=0.
\end{equation}
\end{remark}
The following lemma yields estimates on the stiff component that are useful in the spectral analysis to follow.  We identify the spaces of constant functions on $Y_{\rm stiff}$ and $\Gamma$ with the space $\C^3$.
\begin{proposition}
\label{lemmascalar1}
There exists a constant $C>0$ such that for all $\chi \in Y'\setminus \{0\}:$
\begin{itemize}
\item For every $ \vect g  \in H^{1/2}(\Gamma;\C^3),$ one has
	\begin{equation}
	\label{estim1}
	\left\lVert \vect  g\right\rVert_{H^{1/2}(\Gamma;\C^3)} \leq C|\chi|^{-1} \left\lVert\left(\simgrad + {\rm i}X_\chi  \right) \Pi_\chi^{\rm stiff}  \vect g\right\rVert_{L^2(Y_{\rm stiff};\C^{3 \times 3})};
	\end{equation}
\item For every $ \vect g \in H^{1/2}(\Gamma;\C^3), $ $\vect g \perp \C^3$ (in $L^2(\Gamma;\C^3)$ inner product), one has
	\begin{equation}
	\label{estim2}
	\left\lVert \vect g \right\rVert_{H^{1/2}(\Gamma;\C^3)} \leq  C\left\lVert\left(\simgrad+{\rm i}X_\chi  \right) \Pi_\chi^{\rm stiff} \vect  g \right\rVert_{L^2(Y_{\rm stiff};\C^{3 \times 3})}.
	\end{equation}
\end{itemize}
\end{proposition}
\begin{proof}
The estimate \eqref{estim1} is a straightforward  consequence of
Proposition \ref{Korn_est} 
and the trace theorem.
In order to prove \eqref{estim2}, it suffices to show that for $\vect g \in H^{1/2}(\Gamma;\C^3)$, $\vect g \perp \C^3$ one has
	\begin{equation}
	\label{part1}
	\left\lVert \vect g \right\rVert_{L^2(\Gamma;\C^3)} \leq  C\left\lVert\left(\simgrad + {\rm i}X_\chi  \right) \Pi_\chi^{\rm stiff} \vect  g \right\rVert_{L^2(Y_{\rm stiff};\C^{3 \times 3})}.
	\end{equation}
Indeed, suppose \eqref{part1} holds. Next, employing 
Proposition \ref{josipapp1} and the trace theorem,
we obtain
\begin{equation*}
\begin{aligned}
    \lVert \vect g \rVert_{H^{1/2}(\Gamma;\C^3)} \leq C\bigl\lVert \Pi^{\rm stiff}_\chi \vect g\bigr\rVert_{H^1(Y_{\rm stiff};\C^3)} & \leq C\left\lVert\left(\simgrad +{\rm i}X_\chi  \right) \Pi_\chi^{\rm stiff} \vect  g \right\rVert_{L^2(Y_{\rm stiff};\C^{3 \times 3})} + C\left\lVert \vect g \right\rVert_{L^2(\Gamma;\C^3)} \\[0.25em]
    & \leq C\left\lVert\left(\simgrad +{\rm i}X_\chi  \right) \Pi_\chi^{\rm stiff} \vect  g \right\rVert_{L^2(Y_{\rm stiff};\C^{3 \times 3})}.
    \end{aligned}
\end{equation*}
Finally, \eqref{part1} is 
is obtained by plugging $\vect u = \Pi_\chi^{\rm stiff}{\vect g}$ into the second estimate of Proposition \ref{nakk112}. 
\end{proof}

\subsection{Smallest Steklov eigenvalues}

\label{Steklov_eigenvalue_sec}

The operator $\Lambda_\chi^{\rm stiff}$ on $L^2(\Gamma;\C^3)$ has discrete spectrum, due to the compactness of its resolvent established above. 
The eigenvalues $\nu^\chi$ of $\Lambda_\chi^{\rm stiff}$  are equivalently characterised as solutions to either of the following two problems:
\begin{equation*}
    \Lambda_\chi^{\rm stiff} \vect g = \nu^\chi \vect g, 
    \qquad \vect g \in \mathcal{D}(\Lambda_\chi^{\rm stiff})\setminus\{0\},\qquad\qquad 
 \left\{ \begin{array}{ll}
         \mathcal{A}_{\chi}^{\rm stiff} \vect u = 0, \\[0.4em]
         \Gamma_{1,\chi}^{\rm stiff}  \vect u = \nu^\chi \Gamma_{0,\chi}^{\rm stiff}  \vect u, 
         \qquad \vect u \in \Pi_\chi^{\rm stiff}\mathcal{D}(\Lambda_\chi^{\rm stiff})\setminus\{0\}. \end{array} \right.
\end{equation*}
Next, we define the Rayleigh quotient associated with $\Lambda_\chi^{\rm stiff}$, namely
\begin{equation*}
    \mathcal{R}_\chi(\vect g):= \frac{\lambda_\chi^{\rm stiff}(\vect g,  \vect g)}{\left\lVert \vect g \right\rVert_{L^2(\Gamma;\C^3)}^2}, \quad \vect g\in H^{1/2}(\Gamma;\C^3),
\end{equation*}
where $\lambda_\chi^{\rm stiff}$ is defined by \eqref{lambda_chi_form_def}. The sequence $(\nu_n^\chi)_{n \in \N}$ 
is characterised by the min-max principle
\begin{equation}
	\label{nakk100} 
    -\nu_n^\chi = \min_{\substack{\mathcal{G} \subset H^{1/2}(\Gamma;\C^3)\\ \dim \mathcal{G}=n}} \max_{\vect g\in \mathcal{ G}} \mathcal{R}_\chi(\vect g), \qquad n\in \N.
\end{equation}
The following lemma quantifies the behaviour of the smallest eigenvalues.
\begin{lemma}
\label{Rayleighestimscalar}
There exist constants $C_1 > C_2 > 0$ such that
\begin{itemize}
\item $\mathcal{R}_\chi(\vect g) \geq C_2|\chi|^2\qquad \forall \vect g \in H^{1/2}(\Gamma;\C^3),$
\item $\mathcal{R}_\chi(\vect g) \leq C_1|\chi|^2\qquad \forall \vect g \in \C^3,$
\item $\mathcal{R}_\chi(\vect g) \geq C_2\qquad \forall \vect g \in (\C^3)^\perp.$
\end{itemize}
\begin{proof} 
The proof of the first and third points is a direct consequence of Proposition \ref{lemmascalar1}. The second point is verified by a direct computation. 	
\end{proof} 	

\end{lemma}
The asymptotics of eigenvalues with respect to $|\chi|$ is given in the following lemma.
\begin{lemma}
\label{lemmasteklovorder}
The three smallest eigenvalues of $\Lambda_\chi^{\rm stiff}$ are of order $\mathcal{O}(|\chi|^2),$ and the remaining eigenvalues uniformly separated from zero. Namely, there exist constants $c_1, c_2>0$ that do not depend on $\chi$ such that
\begin{equation*}
        c_1 |\chi|^2 \leq -\nu_n^\chi \leq c_2 |\chi|^2,  \quad n = 1,2,3, \qquad 
        c_1 \leq -\nu_n^\chi,  \quad n\ge4.
\end{equation*}
\end{lemma}
\begin{proof} 
The proof is a direct consequence of \eqref{nakk100} and Lemma \ref{Rayleighestimscalar}. 
\end{proof}	

In what follows, we refer to $\nu_n^\chi,$ $n=1,2,3,$
as $\mathcal{O}(|\chi|^2)$ eigenvalues and to $\nu_n^\chi,$ $n\ge4,$ as $\mathcal{O}(1)$ eigenvalues. Consider the decomposition
\begin{equation}
\label{steklovdecomposition}
    \mathcal{E} := \widehat{\mathcal{E}}_\chi  \oplus \widecheck{\mathcal{E}}_\chi =  \widehat{P}_\chi \mathcal{E} \oplus \widecheck{P}_\chi \mathcal{E},
\end{equation}
where $\widehat{P}_\chi$ is the orthogonal projection onto the
 three-dimensional space $\widehat{\mathcal{E}}_\chi < \mathcal{E}$ spanned by the eigenfunctions associated with order $\mathcal{O}(|\chi|^2)$ eigenvalues of $\Lambda^{\rm stiff}_\chi,$ and $\widecheck{P}_\chi$ is the orthogonal projection onto $\widecheck{\mathcal{E}}_\chi < \mathcal{E}$, the infinite-dimensional space spanned by the eigenfunctions associated with order $\mathcal{O}(1)$ eigenvalues of $\Lambda_\chi^{\rm stiff}$, so that
    $\widehat{P}_\chi = I - \widecheck{P}_\chi.$
Since the decomposition \eqref{steklovdecomposition} is spectral for $\Lambda_\chi^{\rm stiff}$, we have
\begin{equation*}
    \Lambda_\chi^{\rm stiff} = \begin{bmatrix}
    \widehat{\Lambda}_\chi^{\rm stiff} & 0 \\[0.3em]
    0 & \widecheck{\Lambda}_\chi^{\rm stiff} 
\end{bmatrix},
\end{equation*}
where
\begin{equation}
\label{steklovtruncation1}
    \widehat{\Lambda}_\chi^{\rm stiff}:= \widehat{P}_\chi\Lambda_\chi^{\rm stiff}|_{\widehat{\mathcal{E}}_\chi}, \quad \widecheck{\Lambda}_\chi^{\rm stiff}:= \widecheck{P}_\chi\Lambda_\chi^{\rm stiff}|_{\widecheck{\mathcal{E}}_\chi}.
\end{equation}
Both $\widehat{\Lambda}_\chi^{\rm stiff}$ and $\widecheck{\Lambda}_\chi^{\rm stiff}$ are self-adjoint operators on $\widehat{\mathcal{E}}_\chi$ and $\widecheck{\mathcal{E}}_\chi,$ respectively, since $\widehat{P}$ is a spectral projection for $\Lambda_\chi^{\rm stiff}$.   The operator $\widehat{\Lambda}_\chi^{\rm stiff}$ is clearly bounded since it is finite-rank. Note also that the domain of the second operator is a subset of $H^1(\Gamma;\mathbb{C}^3)$, see Section \ref{DtN_sec}.
Furthermore, due to Lemma \ref{lemmasteklovorder} we have the uniform bound
    $\lVert \widehat{\Lambda}_\chi^{\rm stiff}\rVert_{\mathcal{E} \to \mathcal{E}} \leq C|\chi|^2.$
On the other hand, the same lemma implies that the operator $\widecheck{\Lambda}_\chi^{\rm stiff}$, while unbounded, is uniformly bounded from below, where the estimate does not depend on $|\chi|$, namely
    $\lVert \bigl( \widecheck{\Lambda}_\chi^{\rm stiff} \bigr)^{-1}\rVert_{\mathcal{E} \to \mathcal{E}} \leq C.$
Moreover, $(\widecheck{\Lambda}_\chi^{\rm stiff})^{-1}$ is compact and 
 \begin{equation} 
 	\label{nakk204} 
 \bigl\lVert \bigl(\widecheck{\Lambda}_\chi^{\rm stiff} \bigr)^{-1} \vect{g} \bigr\rVert_{H^1(\Gamma;\C^3)} \leq C\|\vect g\|_{L^2(\Gamma;\C^3)} \qquad \forall \vect{g} \in L^2(\Gamma;\C^3),
\end{equation} 
where the constant $C$ is independent of $\chi.$

\subsection{Asymptotics of $\bigl(|\chi|^{-2}\Lambda_\chi^{\rm stiff} - I\bigr)^{-1}$} 
\label{AsymDtN}

While the results of this section are not  necessary for the proof of Theorem \ref{thmmain1}, they are essential in the proof of Theorem \ref{thmamin2}, see Sections \ref{sectionopeff}, \ref{sectionopeffstiff}.
The main tool for proving the latter theorem is finding an approximating homogenised operator by developing an asymptotics of the DtN map using its variational definition (see \eqref{homdefinitionshort}). Note that for that we cannot follow the approach of \cite{CherErKis}, since for PDE systems one cannot expand eigenfunctions or eigenvalues with respect to the  quasimomentum $\chi,$ see, e.g., \cite[Example 5.12]{Kato}. Instead, we construct an asymptotics for the resolvent of the DtN map; in Sections \ref{sectionopeff}, \ref{sectionopeffstiff} this suffices to prove Theorem \ref{thmamin2}. Note also that the variational definition of the approximating operator \eqref{homdefinitionshort} implies its characterisation in terms of $\mathbb{A}_{\rm macro},$ see Lemma \ref{josipnak1}.

We calculate the estimates on the distance between the resolvents of $-|\chi|^{-2}\Lambda_\chi^{\rm stiff}$ and  $-|\chi|^{-2}\Lambda_\chi^{\rm hom}$, where the latter plays the r\^{o}le of the effective DtN map and is introduced below. In Corollary \ref{epsilonnormestimates} we use this to obtain the approximation error with respect to the resolvents of $\varepsilon^2$-scaled operators. A similar approach (in a different context) was used in \cite{CV,CVZ}. 

Recall that the lift operator $\Pi_\chi^{\rm stiff}:H^{1/2}(\Gamma;\C^3) \to H_\#^1(Y_{\rm stiff};\C^3)$ admits the decomposition
\begin{equation*}
    \Pi_\chi^{\rm stiff} = \Pi_0^{\rm stiff} + \widetilde{\Pi}_{\chi,1} + \widetilde{\Pi}_{\chi,1}^{\rm error}
\end{equation*}
in the sense of Theorem \ref{Piasymptheorem}, where the dependence of the operator  $\widetilde{\Pi}_{\chi,1}$ 
on the parameter $\chi$ is linear. The error term $\widetilde{\Pi}_{\chi,1}^{\rm error}$ satisfies the bound (see \eqref{Pi_error})
    $\lVert \widetilde{\Pi}_{\chi,1}^{\rm error}\rVert_{H^{1/2}(\Gamma;\C^3) \to H_\#^1(Y_{\rm stiff};\C^3)} \leq C|\chi|^2.$
For each $\chi\in Y',$ consider the expression $\Psi_\chi:={\rm i}X_\chi  \Pi_0^{\rm stiff}  + \simgrad  \widetilde{\Pi}_{\chi,1}.$
We define the homogenised operator to be a constant matrix $\Lambda^{\rm hom}_\chi \in \C^{3 \times 3}$ such that
\begin{equation}
	\label{homdefinitionshort}
	\left\langle-\Lambda^{\rm hom}_\chi \vect c,\vect d \right\rangle_{\C^3} := \frac{1}{|\Gamma|}\int_{Y_{\rm stiff}} \A_{\rm stiff}\left[  \simgrad  \Pi_0^{\rm stiff} \vect c_{\rm corr}  +  \Psi_\chi \vect c\right] : \overline{ \Psi_\chi\vect d}\qquad \forall \vect c,\vect d \in \C^3,
\end{equation}
where $\vect c_{\rm corr} \in H_\#^1(Y_{\rm stiff};\C^3)$ is the unique solution with $\int_{Y_{\rm stiff}}{\vect c}_{\rm corr}=0$ to the cell problem
\begin{equation}
\label{focorr2}
    \int_{Y_{\rm stiff}} \A_{\rm stiff}\left[  \simgrad  \Pi_0^{\rm stiff} \vect c_{\rm corr}+\Psi_\chi\vect c
    \right]: \overline{ \simgrad   \Pi_0^{\rm stiff}\vect v} = 0 \qquad \forall\vect v \in H^{1/2}(\Gamma;\C^3).
\end{equation}
In next lemma we provide important properties of the matrix  $\Lambda_\chi^{\rm hom}$.    

\begin{lemma}\label{josipnak1} 
    The matrix $\Lambda_\chi^{\rm hom} $ is quadratic in $\chi,$ in particular 
    $\Lambda_\chi^{\rm hom}=-|\Gamma|^{-1}\BBB\left({\rm i}X_\chi \right)^*\A_{\rm macro}{\rm i}X_\chi,$
where $\A_{\rm macro}$ 
is the constant symmetric tensor defined by \eqref{macrodefinitionequation}. As a consequence, $|\chi|^{-2}\Lambda_\chi^{\rm hom}$ is positive and bounded uniformly in $\chi.$ 
\end{lemma}
\begin{proof}
    Note first that for $\vect c \in \C^3$ the solution $\vect c_{\rm corr}$ to \eqref{focorr2} satisfies
\begin{equation}
\label{expression1}
    \int_{Y_{\rm stiff}} \A_{\rm stiff}\left[\simgrad  \Pi_0^{\rm stiff} \vect c_{\rm corr}+\Psi_\chi\vect c
    \right]:\overline{ \simgrad   \vect v} = 0 \qquad \forall\vect v \in H_\#^1(Y_{\rm stiff};\C^3).
\end{equation}
This is seen by noting that for an arbitrary $\vect v \in  H_\#^1(Y_{\rm stiff};\C^3)$ one has the decomposition
\begin{equation*} 
    \vect v = \Pi_0^{\rm stiff} \vect h + \vect w,  \qquad \vect h \in H^{1/2}(\Gamma;\C^3), \qquad \vect w \in H_\#^1(Y_{\rm stiff};\C^3), \quad \vect w\vert_\Gamma= 0.
\end{equation*}
The identity \eqref{expression1} when  $\vect v=\Pi_0^{\rm stiff} \vect h$, $\vect h \in H^{1/2}(\Gamma;\C^3)$ is stated in \eqref{focorr2}, while for  $\vect v=\vect w \in H_\#^1(Y_{\rm stiff};\C^3),$ $\vect w\vert_\Gamma=0$ is covered by \eqref{pi0definition} and \eqref{pichi1formula}. 

For arbitrary $\vect c^1, \vect c^2 \in \C^3,$ define $\vect c^j_{\rm corr},$ $j=1,2,$ as in (\ref{focorr2}) and introduce the notation $\vect u_{\rm corr}^{j}:= \Pi_0^{\rm stiff} \vect c_{\rm corr}^j + \widetilde{\Pi}_{\chi,1} \vect c^j$, $j=1,2.$ Invoking \eqref{expression1}, we obtain
\begin{equation*}
    \int_{Y_{\rm stiff}} \A_{\rm stiff}\bigl[  \simgrad  \vect u_{\rm corr}^{j}  +{\rm i}X_\chi  \Pi_0^{\rm stiff} \vect c^j \bigr] : \overline{ \simgrad   \vect v} = 0 \quad \forall\vect v \in H_\#^1(Y_{\rm stiff};\C^3), \quad j=1,2,
\end{equation*}
as well as
\begin{equation*}
\begin{aligned}
        \left\langle-\Lambda^{\rm hom}_\chi \vect c^1,\vect c^2 \right\rangle_{\C^3} & := \frac{1}{|\Gamma|}\int_{Y_{\rm stiff}} \A_{\rm stiff}\left[  \simgrad  \vect u_{\rm corr}^1  +{\rm i}X_\chi  \Pi_0^{\rm stiff} \vect c^1 \right] : \overline{ \left[\simgrad  \vect u_{\rm corr}^2  +{\rm i}X_\chi  \Pi_0^{\rm stiff} \vect c^2 \right]} \\[0.3em]
                & = \frac{1}{|\Gamma|}\int_{Y_{\rm stiff}} \A_{\rm stiff}\left[  \simgrad  \vect u_{\rm corr}^1  +{\rm i}X_\chi   \vect c^1 \right] : \overline{ \left[\simgrad  \vect u_{\rm corr}^2  + {\rm i}X_\chi  \vect c^2 \right]} 
        = \frac{1}{|\Gamma|}\A_{\rm macro}{\rm i}X_\chi \vect c^1 : \overline{{\rm i}X_\chi \vect c^2},
\end{aligned}
\end{equation*} 
where we have employed the definition of the macroscopic tensor \eqref{macrodefinitionequation}. Using Lemma \ref{prop_lemma} completes the proof.
\end{proof}

Next, we state a theorem on the norm-resolvent estimates of DtN maps.
\begin{theorem}
\label{dirichletotneumannresolventasymptotics}
 There exists a constant $C>0$, which does not depend on $\chi$, such that the following norm-resolvent estimate holds:
 \begin{equation*}
    \left\lVert \left(|\chi|^{-2}\Lambda_\chi^{\rm stiff} - I \right)^{-1} - \left(|\chi|^{-2}\Lambda_\chi^{\rm hom}-I \right)^{-1}S  \right\rVert_{L^2(\Gamma;\C^3) \to H^{1/2}(\Gamma;\C^3)} \leq C |\chi|.
\end{equation*}
Here $S: \vect{h}\mapsto |\Gamma|^{-1}\int_\Gamma\vect{h}$ is the averaging operator on $\Gamma,$ 
and $\Lambda_\chi^{\rm hom}$ is the effective operator defined by \eqref{homdefinitionshort}.
\end{theorem}
\begin{proof}
The proof is by construction, following a version of the standard asymptotic expansion approach. We start with the weak formulation of the resolvent problem for the operator  $|\chi|^{-2}\Lambda_\chi^{\rm stiff}$. For $\vect h \in L^2(\Gamma;\C^3),$ our goal is to expand the solution $\vect g$ of the problem 
\begin{equation}
\label{chiproblem}
    \frac{1}{|\chi|^2} \int_{Y_{\rm stiff}} \A_{\rm stiff} \left( \simgrad + {\rm i}X_\chi \right)\Pi_\chi^{\rm stiff} \vect g  : \overline{\left( \simgrad + {\rm i}X_\chi \right) \Pi_\chi^{\rm stiff} \vect v} + \int_\Gamma \vect g \cdot \overline{\vect v} = \int_\Gamma \vect h \cdot \overline{\vect v} \qquad \forall \vect v\in H^{1/2}(\Gamma;\C^3)
\end{equation}
in the form
   $\vect g = \vect g_0 + \vect g_1 + \vect g_2 + \vect g_{\rm err},$
where  the terms satisfy the bounds
$$\vect g_0 = \mathcal{O}(1), \quad \vect g_1 = \mathcal{O}(|\chi|), \quad \vect g_2 = \mathcal{O}(|\chi|^2), \quad \vect g_{\rm err} = \mathcal{O}(|\chi|^3),$$  with respect to the $H^{1/2}(\Gamma;\C^3)$ norm, and are calculated from a sequence of boundary value problems obtained from \eqref{chiproblem}, the first few of which are introduced below, see \eqref{correctoru2}. By equating the terms in \eqref{chiproblem} which are of order $\mathcal{O}(1)$, we obtain the following identity for the leading-order term $\vect g_0:$
\begin{equation*}
    \int_{Y_{\rm stiff}} \A_{\rm stiff}  \simgrad  \Pi_0^{\rm stiff} \vect g_0  : \overline{ \simgrad   \Pi_0^{\rm stiff}\vect v} = 0 \qquad \forall\vect v \in H^{1/2}(\Gamma;\C^3),
\end{equation*}
hence $\vect g_0\in\C^3.$ 
Furthermore, by combining the terms of order $\mathcal{O}(|\chi|)$, we obtain the identity
\begin{equation}
\label{corr2}
\begin{split}
    \int_{Y_{\rm stiff}} \A_{\rm stiff}  \simgrad  \Pi_0^{\rm stiff} \vect g_1  : \overline{ \simgrad   \Pi_0^{\rm stiff}\vect v} = - \int_{Y_{\rm stiff}} \A_{\rm stiff}  \Psi_\chi \vect g_0  : \overline{ \simgrad   \Pi_0^{\rm stiff}\vect v}\qquad  \forall\vect v \in H^{1/2}(\Gamma;\C^3),
\end{split}
\end{equation}
which has a unique solution satisfying $\int_\Gamma \vect g_1 = 0$. Next, comparing the terms of order $\mathcal{O}(|\chi|^2)$, we define $\vect g_2$ as the solution to the identity
\begin{equation}
\label{correctoru2}
\begin{aligned}
     \int_{Y_{\rm stiff}} &\A_{\rm stiff} \simgrad  \Pi_0^{\rm stiff} \vect g_2 : \overline{ \simgrad   \Pi_0^{\rm stiff}\vect v}=\\[0em]
    &-\int_{Y_{\rm stiff}} \A_{\rm stiff}  \bigl(\Psi_\chi  \vect g_1+{\rm i}X_\chi  \widetilde{\Pi}_{\chi,1} \vect g_0+ \sym \nabla\widetilde{\Pi}_{\chi,1}^{\rm error}\vect{g}_0):\overline{ \simgrad  \Pi_0^{\rm stiff}\vect v}\\[0.2em]
     &-\int_{Y_{\rm stiff}} \A_{\rm stiff} \bigl(\simgrad  \Pi_0^{\rm stiff} \vect g_1+\Psi_\chi  \vect g_0\bigr):\overline{  \Psi_\chi  \vect v}  
    -|\chi|^2\int_\Gamma \vect g_0 \cdot\overline{\vect v} + |\chi|^2\int_\Gamma \vect h \cdot \overline{\vect v}\qquad \forall\vect v \in H^{1/2}(\Gamma;\C^3).
\end{aligned}
\end{equation}
The existence and uniqueness of solution 
to (\ref{correctoru2}) can be established under two additional constraints. We require $\vect g_2$ to have zero mean, $\int_\Gamma \vect g_2 = 0$, and the right-hand side of (\ref{correctoru2}) to vanish when tested against constants.\footnote{This is in fact the Fredholm alternative.} The second part of this requirement is satisfied by choosing an appropriate $\vect g_0 \in \C^3.$ Namely, 
setting $\vect v  = \vect v_0 \in \C^3$ in \eqref{correctoru2} yields
\begin{equation*}
     0 =
     - \int_{Y_{\rm stiff}} \A_{\rm stiff} (\simgrad  \Pi_0^{\rm stiff} \vect g_1+\Psi_\chi  \vect g_0 ) : \overline{  \Psi_\chi \vect v_0}
    - |\chi|^2\int_\Gamma \vect g_0 \cdot \overline{\vect v_0} + |\chi|^2\int_\Gamma \vect h \cdot \overline{\vect v_0} \qquad \forall\vect v_0 \in \C^3.
\end{equation*}
By virtue of \eqref{homdefinitionshort} and \eqref{corr2}, it follows that
\begin{equation*}
    -\frac{1}{|\chi|^2}\left\langle \Lambda_\chi^{\rm hom} \vect g_0,\vect v_0 \right\rangle_{\C^3} + \frac{1}{|\Gamma|}\int_\Gamma \vect g_0 \cdot \overline{\vect v_0} = \frac{1}{|\Gamma|}\int_\Gamma \vect h \cdot \overline{\vect v_0}\qquad \forall\vect v_0 \in \C^3,
\end{equation*} 
or, equivalently,
\begin{equation}
\label{leadingorderterm}
    -\frac{1}{|\chi|^2}\Lambda_\chi^{\rm hom} \vect g_0 + \vect g_0 = \frac{1}{|\Gamma|}\int_\Gamma \vect h.
\end{equation}
Thus, by defining the leading-order term as the solution to the above resolvent problem, we have ensured the solvability of \eqref{correctoru2}.

Next, we prove the estimates for the correctors and the final error estimate.
It is clear from \eqref{leadingorderterm} and the coercivity of $\Lambda_\chi^{\rm hom}$ that
    $\left\lVert \vect g_0 \right\rVert_{H^{1/2}(\Gamma;\C^3)} \leq C \left\lVert \vect h \right\rVert_{L^2(\Gamma;\C^3)}.$
Furthermore, it follows from \eqref{corr2} that
\begin{equation*}
    \left\lVert \simgrad \Pi_0^{\rm stiff} \vect g_1 \right\rVert_{L^2(Y; {\mathbb C}^{3\times3})} \leq \left\lVert {\rm i}X_\chi \Pi_0^{\rm stiff} \vect g_0 \right\rVert_{L^2(Y; {\mathbb C}^{3\times3})} + \bigl\lVert \simgrad \widetilde{\Pi}_{\chi,1} \vect g_0 \bigr\rVert_{L^2(Y; {\mathbb C}^{3\times3})}.
\end{equation*} By virtue of Theorem \ref{Piasymptheorem} and Proposition \ref{nakk112}, we obtain 
\begin{equation}
	\label{nakk113} 
    \left\lVert \vect g_1 \right\rVert_{H^{1/2}(\Gamma;\C^3)} \leq C|\chi| \left\lVert \vect h \right\rVert_{L^2(\Gamma;\C^3)}.
\end{equation}
In a similar manner, it can be seen from \eqref{correctoru2} that
\begin{equation}
	\label{nakk114} 
    \left\lVert \vect g_2 \right\rVert_{H^{1/2}(\Gamma;\C^3)} \leq C|\chi|^2 \left\lVert \vect h \right\rVert_{L^2(\Gamma;\C^3)}.
\end{equation}
Next, we formulate the equation for the error term $\vect g_{\rm err}:=\vect  g - \vect g_0 - \vect g_1 -\vect g_2$, where
$\vect g$ is the solution to the full problem \eqref{chiproblem}. Using the above identities for $\vect g_0,$ $\vect g_1,$ $\vect g_2,$ it is straighforward to infer that  
\begin{equation}
    \label{firsterrorequation}
        \frac{1}{|\chi|^2} \int_{Y_{\rm stiff}} \A_{\rm stiff}\left( \simgrad +{\rm i}X_\chi \right)\Pi_\chi^{\rm stiff} \vect g_{\rm err}  \cdot \overline{\left( \simgrad +{\rm i}X_\chi \right) \Pi_\chi^{\rm stiff}\vect v}  = \frac{1}{|\chi|^2}\mathcal{R}_{\rm err}(\vect v)\quad \forall\vect v\in H^{1/2}(\Gamma;\C^3),
\end{equation}
where the functional $\mathcal{R}_{\rm err}$ satisfies the bound
    $\bigl|\mathcal{R}_{\rm err}(\vect v)\bigr| \leq C|\chi|^3 \left\lVert\vect v \right\rVert_{H^{1/2}(\Gamma;\C^3)} \left\lVert \vect h \right\rVert_{L^2(\Gamma;\C^3)}.$
Thus, by testing \eqref{firsterrorequation} with $\vect g_{\rm err}$ and invoking \eqref{estim1}, we obtain
\begin{equation*}
	\begin{aligned}
    \left\lVert \vect g_{\rm err} \right\rVert_{H^{1/2}(\Gamma;\C^3)}^2&\leq C|\chi|^{-2}\left\lVert \left(\simgrad+{\rm i}X_\chi \right) \Pi_\chi^{\rm stiff} \vect g_{\rm err} \right\rVert_{L^2(Y; {\mathbb C}^{3\times3})}^2
    \leq C|\chi|^{-2}\bigl|\mathcal{R}_{\rm err}(\vect g_{\rm err})\bigr| \leq C|\chi| \lVert \vect g_{\rm err} \rVert_{H^{1/2}(\Gamma;\C^3)}\left\lVert \vect h \right\rVert_{L^2(\Gamma;\C^3)}.
    \end{aligned}
\end{equation*}
Therefore, one has  
    $\lVert \vect g - \vect g_0 - \vect g_1 - \vect g_2\rVert_{H^{1/2}(\Gamma;\C^3)} \leq C |\chi|\left\lVert \vect h \right\rVert_{L^2(\Gamma;\C^3)}.$
It now follows from \eqref{nakk113} and \eqref{nakk114} that
    $\lVert \vect g - \vect g_0  \rVert_{H^{1/2}(\Gamma;\C^3)} \leq C |\chi|\left\lVert \vect h \right\rVert_{L^2(\Gamma;\C^3)},$
as required.
\end{proof}

\begin{remark}
    The averaging operator $S: \mathcal{E} \to \mathcal{E}$ coincides with the projection operator $\widehat{P}_0= \widehat{P}_\chi |_{\chi = 0}$ .
\end{remark}
\begin{remark} \label{nakk120} 
    The norm-resolvent estimate provided in Theorem \ref{dirichletotneumannresolventasymptotics} also yields 
 $$ \left\lVert \left(|\chi|^{-2}\Lambda_\chi^{\rm stiff} -zI \right)^{-1} - \left(|\chi|^{-2}\Lambda_\chi^{\rm hom} -zI \right)^{-1}S  \right\rVert_{L^2(\Gamma;\C^3) \to H^{1/2}(\Gamma;\C^3)} \leq C(z) |\chi|, $$
    where the constant $C(z)$ depends on the distance of $z$ to the spectrum of $|\chi|^{-2}\Lambda_\chi^{\rm stiff(hom)}.$ (It is bounded uniformly in $z$ if $z$ belongs to a set for which both $|z|$ and $\{{\rm dist}\,(z, \sigma(|\chi|^{-2}\Lambda_\chi^{\rm stiff(hom)}))\}^{-1}$ are bounded.) This can be seen by using resolvent identities or by revisiting the proof of Theorem \ref{dirichletotneumannresolventasymptotics}. 
\end{remark}

Next, we discuss the resolvent asymptotics with respect to $\varepsilon$ of the DtN maps on the stiff component.

\begin{corollary}
\label{epsilonnormestimates}
There exists a constant $C>0$, independent of $\chi \in Y'$, $\varepsilon>0$, such that the following norm-resolvent estimate holds:
  \begin{eqnarray*}
\left\Vert \left(\varepsilon^{-2}\Lambda_\chi^{\rm stiff} - I\right)^{-1} -\left( \varepsilon^{-2}\Lambda_\chi^{\rm hom} - I\right)^{-1} S\right\Vert_{L^2(\Gamma;\C^3) \to H^{1/2}(\Gamma;\C^3)}
 \leq C \varepsilon,
\end{eqnarray*}
where the homogenised operator $\Lambda_\chi^{\rm hom}$ is defined by the formula \eqref{homdefinitionshort}.
\end{corollary}
Before proceeding to the proof, we provide some auxiliary results.
There are several important points to make that allow us to rewrite norm-resolvent estimates in terms of the parameter $\varepsilon$.
Both operators $|\chi|^{-2}\Lambda_\chi^{\rm stiff}$ and $|\chi|^{-2}\Lambda_\chi^{\rm hom}$ (where the latter is in fact a multiplication by a constant matrix depending on $\chi$) have exactly $3$ eigenvalues of order $\mathcal{O}(1)$, and the set of these eigenvalues can be enclosed with a fixed contour $\gamma \subset \C$ uniformly in $|\chi|$ (for small enough $|\chi|$). These properties are summarised in the following lemma. 
\begin{lemma}
\label{lemmacontour}
    There exist $\zeta,\eta>0$ and a contour $\gamma \subset \{z\in{\mathbb C}: \Re(z)<-\eta\}$ such that, for all $0<|\chi| \leq \zeta:$
    \begin{itemize}
    \item All eigenvalues of the operators $|\chi|^{-2}\Lambda_\chi^{\rm stiff}$ and $|\chi|^{-2}\Lambda_\chi^{\rm hom}$ that admit an $\mathcal{O}(1)$  estimate in $|\chi|$ are enclosed by $\gamma;$
    \item No other eigenvalue of $|\chi|^{-2}\Lambda_\chi^{\rm stiff}$ and $|\chi|^{-2}\Lambda_\chi^{\rm hom}$ is enclosed by $\gamma;$
\item There exists $\rho_0>0$ such that for all eigenvalues $\nu^\chi$ of
$|\chi|^{-2}\Lambda_\chi^{\rm stiff}$ and $|\chi|^{-2}\Lambda_\chi^{\rm hom}$ one has $\inf_{z\in \gamma}|z-\nu^\chi|\geq \rho_0.$
\end{itemize}
\end{lemma}
\begin{proof}
    Note that for small enough $|\chi|$ the spectrum of order $\mathcal{O}(1)$ is uniformly separated from the remaining spectrum -- this is guaranteed by the estimates in Lemma \ref{lemmasteklovorder}. The same estimates also show that this part of the spectrum lies in a fixed $\chi$-independent interval that does not contain zero.
\end{proof}

For every fixed $\varepsilon>0$, $\chi \neq 0$ we consider the function
   $g_{\varepsilon,\chi}(z) := \left(\varepsilon^{-2}|\chi|^2z - 1 \right)^{-1},$
   $\Re(z) < 0,$
 which satisfies the following lemma.
 \begin{lemma}
 For every fixed $\eta > 0$, the function   $g_{\varepsilon,\chi}$ is bounded in the half-plane $\{z \in \C, \Re(z) < -\eta\}$: 
 \begin{equation*}
     \bigl|g_{\varepsilon,\chi}(z)\bigr| \leq C(\eta) \left(\max\left\{\varepsilon^{-2}|\chi|^2, 1\right\}\right)^{-1}.
 \end{equation*}
 \end{lemma}
 \begin{proof} The required bound follows from the estimate 
 \begin{equation*}
 \bigl|g_{\varepsilon,\chi}(z)\bigr|^{-1} =  \bigl|\varepsilon^{-2}|\chi|^2z-1\bigr| \geq \varepsilon^{-2}|\chi|^2\eta+ 1 \geq
C(\eta)\max\left\{\varepsilon^{-2}|\chi|^2, 1\right\}. 
 \qedhere\popQED
 \end{equation*}
 \end{proof}

\begin{proof}[Proof of Corollary \ref{epsilonnormestimates}]
    Applying the integral Cauchy formula, we obtain
\begin{equation*}
    \begin{aligned}
        \widehat{P}_\chi\left(\frac{1}{\varepsilon^{2}}\Lambda_\chi^{\rm stiff}- I\right)^{-1}\widehat{P}_\chi&= \frac{1}{2\pi{\rm i}} \oint_\gamma g_{\chi,\varepsilon}(z)\left(zI-\frac{1}{|\chi|^2}\Lambda_\chi^{\rm stiff}\right)^{-1}  dz, \\
        \left(\frac{1}{\varepsilon^{2}}\Lambda^{\rm hom}_\chi-I\right)^{-1}S&= \frac{1}{2\pi{\rm i}} \oint_\gamma g_{\chi,\varepsilon}(z)\left(zI-\frac{1}{|\chi|^2}\Lambda_\chi^{\rm hom}\right)^{-1}Sdz,
    \end{aligned}
\end{equation*}
where $\widehat{P}_\chi$ is the operator of projection onto the $3$-dimensional span of the eigenfunctions of $\Lambda_\chi^{\rm stiff}$ associated with eigenvalues of order $|\chi|^2.$ Furthermore, since $\widehat{P}_\chi$ is a spectral projection for $\Lambda_\chi^{\rm stiff}$, using the estimates from Lemma \ref{lemmasteklovorder} yields
\begin{equation}\label{nakk121} 
    \begin{aligned}
        &\left\Vert
        \left( \varepsilon^{-2}\Lambda_\chi^{\rm stiff}-I\right)^{-1}
        -\widehat{P}_\chi\left( \varepsilon^{-2}\Lambda_\chi^{\rm stiff}-I\right)^{-1}\widehat{P}_\chi
        \right\Vert_{L^2(\Gamma;\C^3) \to H^{1/2}(\Gamma;\C^3)}  \\[0.4em]
        &\hspace{17ex}=\left\Vert
        (I-\widehat{P}_\chi)\left( \varepsilon^{-2}\Lambda_\chi^{\rm stiff}-I\right)^{-1}
        (I-\widehat{P}_\chi)
        \right\Vert_{L^2(\Gamma;\C^3) \to H^{1/2}(\Gamma;\C^3)} \leq C\varepsilon^{2},
    \end{aligned}
\end{equation}
On the other hand,
\begin{equation}
	\begin{aligned}
&\left\Vert \widehat{P}_\chi\left(\varepsilon^{-2}\Lambda_\chi^{\rm stiff}- I\right)^{-1}\widehat{P}_\chi -\left(\varepsilon^{-2}\Lambda_\chi^{\rm hom}-I\right)^{-1}S \right\Vert_{L^2(\Gamma;\C^3) \to H^{1/2}(\Gamma;\C^3)}
\\[0.4em]
&\leq \frac{1}{2\pi} \oint_{\gamma} 
\bigl|g_{\chi,\varepsilon}(z)\bigr|\left\Vert \left(  z I - \frac{1}{|\chi|^{2}}\Lambda_\chi^{\rm stiff} \right)^{-1} -\left(  zI - \frac{1}{|\chi|^{2}}\Lambda_\chi^{\rm hom} \right)^{-1} S\right\Vert_{L^2(\Gamma;\C^3) \to H^{1/2}(\Gamma;\C^3)} dz
\\[0.4em]
&\leq C|\chi| \left(\max\left\{\varepsilon^{-2}|\chi|^{2}, 1\right\}\right)^{-1} \leq C \varepsilon,
\end{aligned}
\label{nakk122}
\end{equation}
where we have used Remark \ref{nakk120}. 
The proof is concluded by combining \eqref{nakk121} and \eqref{nakk122}. 
\end{proof}
 Theorem \ref{dirichletotneumannresolventasymptotics} has another direct consequence, namely the asymptotics of the spectral projections $\widehat{P}_\chi$ and the truncated lift operators $ \Pi_\chi^{\rm stiff} \widehat{P}_\chi$ with respect to the quasimomentum $\chi$.
\begin{corollary}
\label{pipasymptotics_stiff}
The operators $\widehat{P}_\chi$ and $\Pi_\chi^{\rm stiff} \widehat{P}_\chi$ satisfy the asymptotics
\begin{align}
\label{pasimp}
       &\left\lVert \widehat{P}_\chi - \widehat{P}_0  \right\rVert_{L^2(\Gamma;\C^3) \to H^1(\Gamma;\C^3)} \leq C|\chi|,\\[0.4em]
\label{piasimp}
     &\left\lVert\Pi_\chi^{\rm stiff} \widehat{P}_\chi - \Pi_0^{\rm stiff} \widehat{P}_0 \right\rVert_{L^2(\Gamma;\C^3) \to H^1(Y_{\rm stiff};\C^3)} \leq C|\chi|,\\[0.4em]
\label{pistarasymp}
	&\left\lVert\widehat{P}_\chi \left(\Pi_\chi^{\rm stiff}\right)^* - \widehat{P}_0 \left(\Pi_0^{\rm stiff}\right)^* \right\rVert_{L^2(Y_{\rm stiff};\C^3) \to L^2(\Gamma;\C^3)} \leq C|\chi|.
\end{align}
where the constant $C>0$ does not depend on $\chi \in Y'$.
\end{corollary}
\begin{proof}
By choosing a contour $\gamma$ as above, and applying the Cauchy formula to the constant function $g(z)=1$, we obtain the asymptotics of projections $\widehat{P}_\chi$ defined by \eqref{steklovdecomposition}:
\begin{equation*}
\begin{aligned}
    \widehat{P}_\chi&=\frac{1}{2\pi i} \oint_\gamma \left(zI-\frac{1}{|\chi|^2}\Lambda_\chi^{\rm stiff}\right)^{-1}  dz & = \frac{1}{2\pi i} \oint_\gamma \left(zI-\frac{1}{|\chi|^2}\Lambda_\chi^{\rm hom}\right)^{-1}Sdz +  \oint_\gamma \widetilde{R}_{\chi}^{\rm corr}(z) dz 
    =\widehat{P}_0 +  \oint_\gamma \widetilde{R}_{\chi}^{\rm corr}(z) dz,
\end{aligned}
\end{equation*}
\begin{equation*}
	\widetilde{R}_{\chi}^{\rm corr}(z):= \left(|\chi|^{-2}\Lambda_\chi^{\rm stiff} - zI \right)^{-1} - \left(|\chi|^{-2}\Lambda_\chi^{\rm hom} - zI \right)^{-1}S, \qquad\ \  
	\biggl\lVert\oint_\gamma \widetilde{R}_{\chi}^{\rm corr}(z) dz\biggr\rVert_{L^2(\Gamma;\C^3) \to H^{1/2}(\Gamma;\C^3)} \leq C|\chi|,
\end{equation*}
as a consequence of Remark \ref{nakk120}.
This proves the estimate 
\begin{equation}
	\label{igor1}
	\left\lVert \widehat{P}_\chi - \widehat{P}_0  \right\rVert_{L^2(\Gamma;\C^3) \to H^{1/2}(\Gamma;\C^3)} \leq C|\chi|.
\end{equation}
To upgrade it to an $H^1$ estimate, we write
\begin{equation}
	\label{igor2}
	\widehat{P}_\chi - \widehat{P}_0 = \widehat{P}_\chi(\widehat{P}_\chi-\widehat{P}_0) + (\widehat{P}_\chi-\widehat{P}_0)\widehat{P}_0.
\end{equation}
Now note that $\widehat{P}_\chi:L^2(\Gamma;\C^3) \to H^1(\Gamma;\C^3)$ is bounded, owing to the facts that  $\mathcal{D}\left(\Lambda_\chi^{\rm stiff}\right)= H^1(\Gamma;\C^3)$ and that $\widehat{P}_\chi$ is the projection onto the eigenspace corresponding to ${\mathcal O}(|\chi|^2)$ eigenvalues. The first term in \eqref{igor2} admits an ${\mathcal O}(|\chi|)$ estimate in the $L^2(\Gamma;\C^3) \to H^1(\Gamma;\C^3)$ norm, due to \eqref{igor1} and the boundedness of $\widehat{P}_\chi$. The second term in \eqref{igor2} admits the same bound with respect to the same norm, by virtue of \eqref{igor1} and the fact that $\widehat{P}_0$ is finite-rank.
 Together with the assertion of Theorem \ref{Piasymptheorem} this yields
\begin{equation*}
       \Pi_\chi^{\rm stiff} \widehat{P}_\chi = \bigl( \Pi_0^{\rm stiff} + \widetilde{\Pi}_{\chi,1}^{\rm error}\bigr)\bigl(\widehat{P}_0 +  \widetilde{P}_{\chi}^{\rm error}\bigr) = \Pi_0^{\rm stiff} \widehat{P}_0 + \mathcal{O}\bigl(|\chi|\bigr), 
\end{equation*}
where $O(|\chi|)$ is understood in the sense of the $L^2(\Gamma;\C^3) \to H^1(Y_{\rm stiff};\C^3)$ norm. Similarly,
\begin{equation*}
   \widehat{P}_\chi \left(\Pi_\chi^{\rm stiff}\right)^*  =  \widehat{P}_0 \left(\Pi_0^{\rm stiff}\right)^* + \mathcal{O}(|\chi|), 
\end{equation*}
where $O(|\chi|)$ is understood in the sense of the $L^2(\Gamma;\C^3) \to L^2(Y_{\rm stiff};\C^3)$ norm.
\end{proof}

\subsection{Soft component asymptotics}
\label{soft_sec}
Next, we state some simpler asymptotic properties of the boundary operators for the soft component. These results will be used in Section \ref{sectionopeff} for proving Theorem \ref{thmamin2}~(a). 
\begin{lemma}
For the boundary operators on the soft component, one has 
\begin{equation}
\label{softidentities}
    \Pi_\chi^{\rm soft} = {\rm e}^{-{\rm i}\chi y} \Pi_0^{\rm soft} {\rm e}^{{\rm i}\chi y}, \qquad \Lambda_\chi^{\rm soft} = {\rm e}^{-{\rm i}\chi y} \Lambda_0^{\rm soft} {\rm e}^{{\rm i}\chi y}, \qquad \Gamma_{\chi,1}^{\rm soft} = {\rm e}^{-{\rm i}\chi y} \Gamma_{0,1}^{\rm soft} {\rm e}^{{\rm i}\chi y}.
\end{equation}
Moreover, the eigenvalues of the operator $\mathcal{A}_{0,\chi}^{\rm soft}$ are independent of $\chi,$ and 
\begin{equation}
\label{softidentites2}
    {\rm e}^{{\rm i}\chi y} \bigl(\mathcal{A}_{0,\chi}^{\rm soft} - zI\bigr)^{-1} {\rm e}^{-{\rm i}\chi y} = \bigl(\mathcal{A}_{0,0}^{\rm soft} - zI\bigr)^{-1},
\end{equation}
where the operator $\mathcal{A}_{0,0}^{\rm soft}$ is defined by \eqref{differentialoperators} with $\chi = 0$ (and so coincides with ${\mathcal A}_{\rm Bloch}$ introduced in Section \ref{elast_op_sec}.) 
Furthermore, let $\{\eta_k\}$ be a sequence of eigenvalues of the operator $\mathcal{A}_{0,0}^{\rm soft}$ and $\{\varphi_k\}$ the associated sequence of eigenfunctions. Then the corresponding sequence of eigenfunctions of $\mathcal{A}_{0,\chi}^{\rm soft}$ is given by $\{\varphi_k^\chi\}= \{{\rm e}^{-{\rm i}\chi y} \varphi_k\}$.
\end{lemma}
\begin{proof}
 By the definition of $\Pi_0^{\rm soft},$ one has
\begin{equation*}
    \Pi_0^{\rm soft}\vect g = \vect u \quad \iff \left\{ \begin{array}{ll}
        - \div\left(\A_{\rm soft}\simgrad \vect u \right) = 0  & \mbox{ on } Y_{\rm soft},  \\[0.25em]
        \vect u = \vect g & \mbox{ on } \Gamma.
    \end{array} \right.
\end{equation*}
Writing $\vect u = {\rm e}^{{\rm i}\chi y} \vect w \in H^1(Y_{\rm soft};\C^3)$, we obtain
\begin{equation*}
\begin{aligned}
    \Pi_0^{\rm soft}\vect g = \vect u \quad & \iff \left\{ \begin{array}{ll}
        - \div\left(\A_{\rm soft}\simgrad \left({\rm e}^{{\rm i}\chi y} \vect w \right)  \right) = 0  & \mbox{ on } Y_{\rm soft},  \\[0.25em]
        {\rm e}^{{\rm i}\chi y} \vect w = \vect g &\mbox{ on } \Gamma
    \end{array} \right.  
    \\[0.6em]
    & 
    \iff \left\{\begin{array}{ll}
        - \div \left({\rm e}^{{\rm i}\chi y}\A_{\rm soft}\left( \simgrad +{\rm i}X_\chi  \right) \vect w\right) = 0   &\mbox{ on } Y_{\rm soft},  \\[0.25em]
        \vect w = {\rm e}^{-{\rm i}\chi y} \vect g &\mbox{ on } \Gamma
    \end{array} \right. \\[0.6em]
    & \iff \vect w = \Pi_\chi^{\rm soft} {\rm e}^{-{\rm i}\chi y} \vect g \quad \iff \quad \vect u = {\rm e}^{{\rm i}\chi y}\Pi_\chi^{\rm soft}{\rm e}^{-{\rm i}\chi y}\vect g.
\end{aligned}
\end{equation*}
Similarly, one has
\begin{equation*}
\begin{aligned}
    \Lambda_0^{\rm soft}\vect g = \vect h \quad & \iff \left\{ \begin{array}{lll}
        - \div\left(\A_{\rm soft}\simgrad \vect  u  \right) = 0   &\mbox{ on } Y_{\rm soft},\\[0.25em]
        \vect u = \vect g &\mbox{ on } \Gamma, \\[0.25em]
        \A_{\rm soft}\simgrad \vect u \cdot {\vect n}_{\rm soft} = \vect h  &\mbox{ on } \Gamma
    \end{array} \right.
   \iff \left\{ \begin{array}{lll}
        - \div\left(\A_{\rm soft} \simgrad \left({\rm e}^{{\rm i}\chi y} \vect w \right)  \right) = 0  &\mbox{ on } Y_{\rm soft},  \\[0.25em]
        {\rm e}^{{\rm i}\chi y} \vect w = \vect g & \mbox{ on } \Gamma, \\[0.25em]
        \A_{\rm soft}\simgrad \left({\rm e}^{{\rm i}\chi y} \vect w \right) \cdot {\vect n}_{\rm soft} = \vect h & \mbox{ on } \Gamma
    \end{array} \right. 
\\[0.7em]
    & 
    \iff \left\{\begin{array}{lll}
        - \div \left({\rm e}^{{\rm i}\chi y}\A_{\rm soft}\left( \simgrad + {\rm i}X_\chi  \right) \vect w\right) = 0  & \mbox{ on } Y_{\rm soft},  \\[0.25em]
        \vect w = {\rm e}^{-{\rm i}\chi y} \vect g & \mbox{ on } \Gamma, \\[0.25em]
        \A_{\rm soft}\left( \simgrad+{\rm i}X_\chi  \right) \vect w  \cdot {\vect n}_{\rm soft} = {\rm e}^{-{\rm i}\chi y}\vect h & \mbox{ on } \Gamma
    \end{array} \right. \\[0.6em]
    & \iff  {\rm e}^{-{\rm i}\chi y}\vect h = \Lambda_\chi^{\rm soft} {\rm e}^{-{\rm i}\chi y}\vect g \quad \iff \quad \vect h ={\rm e}^{{\rm i}\chi y} \Lambda_\chi^{\rm soft}{\rm e}^{-{\rm i}\chi y}\vect g.
\end{aligned}
\end{equation*}
\end{proof}
\begin{corollary}
\label{msoftasymptotics}
For the $M$-function $M_\chi^{\rm soft}(z)$ of the soft component, one has
\begin{equation*}
    M_\chi^{\rm soft}(z) = {\rm e}^{- i\chi y}M_0^{\rm soft}(z) {\rm e}^{i \chi y}, \qquad M_0^{\rm soft}(z) = \Lambda_0^{\rm soft} + z \left( \Pi_0^{\rm soft}\right)^*\Pi_0^{\rm soft} + z^2 \left( \Pi_0^{\rm soft}\right)^* \left( \mathcal{A}_{0,0}^{\rm soft} - zI\right)^{-1}\Pi_0^{\rm soft}.
\end{equation*}
Furthermore, there exists a $\chi$-independent constant $C=C(z)>0$ such that
\begin{equation*}
    \left\lVert\widehat{P}_\chi M_\chi^{\rm soft}(z)\widehat{P}_\chi -\widehat{P}_0 M_0^{\rm soft}(z)\widehat{P}_0\right\rVert_{\mathcal{E} \to \mathcal{E}} \leq C |\chi|.
\end{equation*}
\end{corollary}
\begin{proof}
    Recalling the identity \eqref{identityforasymptotics}, we have the  representation formula
    \begin{equation*}
        M_\chi^{\rm soft}(z) = \Lambda_\chi^{\rm soft} + z \left( \Pi_\chi^{\rm soft}\right)^*\Pi_\chi^{\rm soft} + z^2 \left( \Pi_\chi^{\rm soft}\right)^* \bigl(\mathcal{A}_{0,\chi}^{\rm soft}-zI\bigr)^{-1}\Pi_\chi^{\rm soft}.
    \end{equation*} 
Employing the identities \eqref{softidentities}, \eqref{softidentites2}, Remark \ref{igor3}, and the estimate \eqref{pasimp} yields the claim.
\end{proof}
\begin{remark}
	\label{nakk1000}
    Notice that, due to the fact that 
    $\Lambda_0^{\rm soft} |_{\widehat{\mathcal{E}}_0} = 0,$ one also has
    \begin{equation}
    \label{msoftzeroformula}
       M_0^{\rm soft}(z) |_{\widehat{\mathcal{E}}_0} =  z  \left( \Pi_0^{\rm soft}\right)^*\Pi_0^{\rm soft}|_{\widehat{\mathcal{E}}_0} + z^2  \left( \Pi_0^{\rm soft}\right)^* \left( \mathcal{A}_{0,0}^{\rm soft} - zI\right)^{-1}\Pi_0^{\rm soft}|_{\widehat{\mathcal{E}}_0}
    \end{equation}
\end{remark}

Combining the above lemma with  \eqref{pasimp} yields the following corollary, which we use in the proof of Theorem \ref{nakk1}.
\begin{corollary}
\label{pipasymptotics_soft}
The operator $\Pi_\chi^{\rm soft} \widehat{P}_\chi$ satisfies the estimates
\begin{align}
\label{pisoftasimp}
     &\bigl\lVert\Pi_\chi^{\rm soft} \widehat{P}_\chi - \Pi_0^{\rm soft} \widehat{P}_0\bigr\rVert_{L^2(\Gamma;\C^3) \to L^2(Y_{\rm soft};\C^3)} \leq C|\chi|,
     \\[0.4em]
	&
	\bigl\lVert\widehat{P}_\chi \left(\Pi_\chi^{\rm soft}\right)^* - \widehat{P}_0 \left(\Pi_0^{\rm soft}\right)^* \bigr\rVert_{L^2(Y_{\rm soft};\C^3) \to L^2(\Gamma;\C^3)} \leq C|\chi|.\nonumber
\end{align}
where the constant $C>0$ does not depend on $\chi \in Y'$.
\end{corollary}

\section{Transmission problem: $O(\varepsilon^2)$ resolvent asymptotics} \label{section5} 

 In this section we aim at proving Theorem \ref{thmmain1}. The starting point is the Kre\u\i n formula \eqref{transmissionkrein}. The approximation is carried out in two steps. The first step (see Section \ref{steklovtruncation}, Theorem \ref{firstapproximationtheorem}) is to use the Schur-Frobenius inversion formula by restricting traces to the space $\widehat{\mathcal{E}}_{\chi}$ and by imposing the equality of projections onto the same space of the traces of co-normal derivatives. The second step (Section \ref{refinement}, Theorem \ref{thmrefinement}) is to approximate the $M$-function on the stiff component. 
Recalling \eqref{usefulidentitiesresolvent}, we write
\begin{equation}
\label{asymptoticsforsformula}
     S_{\chi}^{\rm stiff}(\varepsilon^2 z) = \Bigl(I + \varepsilon^2 z\bigl( \mathcal{A}_{0,\chi}^{\rm stiff} - \varepsilon^2 zI\bigr)^{-1}\Bigr)\Pi^{\rm stiff}_\chi = \Pi^{\rm stiff}_\chi +\mathcal{O}(\varepsilon^2), 
\end{equation}
where $\mathcal{O}(\varepsilon^2)$ is understood in the sense of the $\mathcal{E}\mapsto \mathcal{H}^{\rm stiff}$ norm. The formula  \eqref{identityforasymptotics} yields
\begin{equation}
\label{Mfunctionasymptotics}
\begin{aligned}
        M_{\chi}^{\rm stiff}(\varepsilon^2 z)
        & = \Lambda_{\chi}^{\rm stiff} + \varepsilon^2 z \left(\Pi_\chi^{\rm stiff}\right)^*\Bigl(I-\varepsilon^2 z \bigr(\mathcal{A}_{0,\chi}^{\rm stiff}\bigr)^{-1}\Bigr)^{-1} \Pi_\chi^{\rm stiff}\\[0.3em]
        &= \Lambda_{\chi}^{\rm stiff} + \varepsilon^2 z \bigl(\Pi_\chi^{\rm stiff}\bigr)^*\Pi_\chi^{\rm stiff} + \varepsilon^4 z^2 \Bigl(\left(\Pi_\chi^{\rm stiff}\right)^* \bigl(\mathcal{A}_{0,\chi}^{\rm stiff} - \varepsilon^2 zI\bigr)^{-1} \Pi_\chi^{\rm stiff}\Bigr).
\end{aligned}
\end{equation}

\subsection{Steklov truncation}\label{steklovtruncation} 
 Similarly to \eqref{steklovtruncation1}, we define the following truncated DtN maps:
\begin{equation*}
        \widehat{\Lambda}_\chi^{\rm soft}:= \widehat{P}_\chi\Lambda_\chi^{\rm soft} |_{\widehat{\mathcal{E}}_\chi}, 
        \quad \widecheck{\Lambda}_\chi^{\rm soft}:= \widecheck{P}_\chi\Lambda_\chi^{\rm soft} |_{\widecheck{\mathcal{E}}_\chi},\qquad 
        \widehat{\Lambda}_{\chi, \varepsilon}:= \widehat{P}_\chi\Lambda_{\chi,\varepsilon}|_{\widehat{\mathcal{E}}_\chi}, 
        \quad \widecheck{\Lambda}_{\chi,\varepsilon}:= \widecheck{P}_\chi\Lambda_{\chi,\varepsilon}|_{\widecheck{\mathcal{E}}_\chi},
\end{equation*}
so one obviously has
\begin{equation}
    \widehat{\Lambda}_{\chi, \varepsilon}=\varepsilon^{-2}  \widehat{\Lambda}_\chi^{\rm stiff} + \widehat{\Lambda}_\chi^{\rm soft}, \qquad \widecheck{\Lambda}_{\chi, \varepsilon}=\varepsilon^{-2}  \widecheck{\Lambda}_\chi^{\rm stiff} + \widecheck{\Lambda}_\chi^{\rm soft}.
    \label{operator_sums}
\end{equation}
The first operator sum in (\ref{operator_sums}) is self-adjoint due to the fact that its terms are finite-rank self-adjoint (Hermitian) operators. The second sum in (\ref{operator_sums}) also defines a self-adjoint operator (noting that $\widecheck{\Lambda}_\chi^{\rm soft}$ is also self-adjoint, which can be checked by considering the associated sesquilinear form), by an argument similar to that of Theorem \ref{thm:Voffka}, where the operator domains are given by 
\begin{equation*}
    \mathcal{D}(\widecheck{\Lambda}_{\chi, \varepsilon}) = \mathcal{D}(\widecheck{\Lambda}_\chi^{\rm stiff}) = \mathcal{D}(\widecheck{\Lambda}_\chi^{\rm soft}) = \widecheck{P}_\chi \bigl(\mathcal{D}(\Lambda_{\chi,\varepsilon})\bigr).
\end{equation*}
Additionally, we denote the truncated lift operators by
\begin{equation*}
        \widehat{\Pi}_\chi^{\rm stiff(soft)}:=\Pi_\chi^{\rm stiff(soft)}|_{\widehat{\mathcal{E}}_\chi},  \quad \widecheck{\Pi}_\chi^{\rm stiff(soft)}:=\Pi_\chi^{\rm stiff(soft)}|_{\widecheck{\mathcal{E}}_\chi}, \quad \widehat{\Pi}_\chi:=  \widehat{\Pi}_\chi^{\rm stiff}\oplus  \widehat{\Pi}_\chi^{\rm soft},\quad \widecheck{\Pi}_\chi:=  \widecheck{\Pi}_\chi^{\rm stiff}\oplus  \widecheck{\Pi}_\chi^{\rm soft}.
\end{equation*}
Thus, we have defined the following ``Steklov-truncated'' Ryzhov triples
\begin{equation*}
        \bigl(\mathcal{A}_{0,\chi}^{\rm stiff (soft)}, \widehat{\Pi}_\chi^{\rm stiff(soft)},\widehat{\Lambda}_\chi^{\rm stiff (soft)} \bigr) \mbox{ on } \bigl(\mathcal{H}^{\rm stiff(soft)}, \widehat{\mathcal{E}}_\chi \bigr),\quad  
        \bigl(\mathcal{A}_{0,\chi}^{\rm stiff (soft)}, \widecheck{\Pi}_\chi^{\rm stiff(soft)},\widecheck{\Lambda}_\chi^{\rm stiff (soft)} \bigr) \mbox{ on } \bigl(\mathcal{H}^{\rm stiff(soft)}, \widecheck{\mathcal{E}}_\chi\bigr),
\end{equation*}
as well as the coupled Steklov-truncated triples
        $(\mathcal{A}_{0,\chi,\varepsilon}, \widehat{\Pi}_\chi,\widehat{\Lambda}_{\chi,\varepsilon}),$ $(\mathcal{A}_{0,\chi,\varepsilon}, \widecheck{\Pi}_\chi,\widecheck{\Lambda}_{\chi,\varepsilon})$ on $(\mathcal{H},  \widehat{\mathcal{E}}_\chi),$  $(\mathcal{H},  \widecheck{\mathcal{E}}_\chi),$ respectively.
The triple properties as stated in the Definition \ref{ryzhovtriple} are easily checked. In particular, one has
\begin{equation*}
    \begin{aligned}
        \mathcal{D}(\mathcal{A}_{0,\chi,\varepsilon})\cap \mathcal{R}(\widehat{\Pi}_\chi)&= \mathcal{D}(\mathcal{A}_{0,\chi,\varepsilon})\cap \Pi_\chi(\widehat{P}_\chi(\mathcal{E})) \subset  \mathcal{D}(\mathcal{A}_{0,\chi,\varepsilon})\cap \Pi_\chi(\mathcal{E}) = \{0\}, \\[0.3em]
        \ker(\widehat{\Pi}_\chi)&=\ker(\Pi_\chi \widehat{P}_\chi) \subset \ker(\Pi_\chi) = \{0\}.
    \end{aligned}
\end{equation*}
One can define the boundary triples and $M$-functions associated with the Ryzhov triples introduced above, denoted in the same fashion. Notice that, for example, the domain of the operator $\widehat{\mathcal{A}}_{\chi}^{\rm soft}$ coincides with $\mathcal{D}(\mathcal{A}_{0,\chi}^{\rm soft}) \dot{+} \widehat{\Pi}_\chi(\widehat{\mathcal{E}}_\chi)$, so the trace operator $\widehat{\Gamma}_{0,\chi}$ takes values in $\widehat{\mathcal{E}}_\chi$.
By recalling that representation formula \eqref{representationmfunctionformula} decomposes the $M$-function into the sum of a self-adjoint operator (DtN map) and a bounded operator, we know that its domain actually coincides with the domain of the associated ``$\Lambda$-operator", and thus one has
\begin{equation*}
\begin{aligned}
        \widehat{M}_\chi(z)^{\rm stiff(soft)} &= \widehat{\Lambda}_\chi^{\rm stiff(soft)} + z \bigl(\widehat{\Pi}_\chi^{\rm stiff(soft)}\bigr)^*\Bigl(I - z\bigl(\mathcal{A}_{0,\chi}^{\rm stiff(soft)}\bigr)^{-1} \Bigr)^{-1} \widehat{\Pi}^{\rm stiff(soft)}_\chi \\[0.25em]
        &=\widehat{P}_\chi \Bigl(\Lambda_\chi^{\rm stiff(soft)} + z \bigl(\Pi_\chi^{\rm stiff(soft)}\bigr)^*\Bigl(I - z\bigl(\mathcal{A}_{0,\chi}^{\rm stiff(soft)}\bigr)^{-1} \Bigr)^{-1} \Pi_\chi^{\rm stiff(soft)} \Bigr)\Big|_{\widehat{\mathcal{E}}_\chi} 
         =\widehat{P}_\chi M_\chi(z)^{\rm stiff(soft)}\big|_{\widehat{\mathcal{E}}_\chi},
\end{aligned}
\end{equation*}
and similar claims hold for  $\widecheck{M}_\chi(z)^{\rm stiff(soft)}$, $\widecheck{M}_{\chi,\varepsilon}(z),$ and $\widehat{M}_{\chi,\varepsilon}(z)$.

We introduce the notation 
$\widehat{\mathcal{H}}_\chi^{\rm stiff(soft)}:= \Pi_\chi^{\rm stiff(soft)} \widehat{\mathcal{E}}_\chi,$
together with the notation $P_{\widehat{\mathcal{H}}_\chi^{\rm stiff(soft)}}$ for the respective orthogonal projections (with respect to the $\mathcal{H}$ inner product). We also define $\Theta_\chi:\mathcal{H}\to \mathcal{H}$ as the orthogonal projection
\begin{equation}
	\label{nakk1011}
	\Theta_\chi\left( {\vect u}_{\rm soft} \oplus {\vect u}_{\rm stiff}\right) = {\vect u}_{\rm soft} \oplus  P_{\widehat{\mathcal{H}}_\chi^{\rm stiff}} \vect u_{\rm stiff}.
\end{equation}
Before stating our approximation result, we need one helpful lemma, whose proof is found in the Appendix. It establishes the equivalence of the $H^1$ and $L^2$ norms on $\widehat{\mathcal{H}}_\chi^{\rm stiff(soft)}$ uniformly in the quasimomentum $\chi.$ The proof of this lemma can be found in the Appendix. 
\begin{lemma}
	\label{lemmanormequivalence_f}
	There exists a $\chi$-independent constant $C>0$ such that
	\begin{equation*}
	\lVert \vect f \rVert_{L^2(Y_{\rm stiff(soft)};\C^3)} \geq C \lVert  \vect f\rVert_{H^1(Y_{\rm stiff(soft)};\C^3)} \qquad \forall \vect f \in \widehat{\mathcal{H}}_\chi^{\rm stiff(soft)}.
	\end{equation*}
\end{lemma}

The next theorem provides the basis for Theorem \ref{thmmain1} .

\begin{theorem}
\label{firstapproximationtheorem}
 There exists $C>0$
 such that for the resolvent of the transmission problem \eqref{transmissionboundaryproblem} one has 
 \begin{equation*}
     \Bigl\lVert\bigl((\mathcal{A}_{\chi,\varepsilon})_{0,I} -zI \bigr)^{-1} - \bigl( \left(\mathcal{A}_{\chi,\varepsilon}\right)_{\widecheck{P}_\chi,\widehat{P}_\chi} -zI \bigr)^{-1} \Bigr\rVert_{\mathcal{E} \to \mathcal{E}} \leq C \varepsilon^2\qquad  \forall\chi \in Y',\ z \in K_\sigma.
 \end{equation*}
\end{theorem}
\begin{proof}
Notice that, as a consequence of \eqref{identityforasymptotics} and the second equality in \eqref{MMM}, we can write
\begin{equation}
\label{formulamboundedb}
    M_{\chi, \varepsilon}(z) = \Lambda_{\chi,\varepsilon} + \mathcal{B}_{\chi,\varepsilon}(z)=\varepsilon^{-2}\Lambda_{\chi}^{\rm stiff} +\Lambda_{\chi}^{\rm soft}+ \mathcal{B}_{\chi,\varepsilon}(z),
\end{equation}
where the operator $\mathcal{B}_{\chi,\varepsilon}$ is bounded uniformly in $\chi,\eps$.  So, the question of boundedness of a certain truncation of $M_{\chi,\varepsilon} $ actually comes down to the boundedness of associated truncation of DtN map. But, since $\widehat{\mathcal{E}}_\chi \subset \mathcal{D}(\Lambda_{\chi,\varepsilon}) = \mathcal{D}(\Lambda_{\chi}^{\rm stiff}) = \mathcal{D}(\Lambda_{\chi}^{\rm soft})$ and $\widehat{\mathcal{E}}_\chi$ is finite-dimensional, one has
\begin{equation*}
        \left\lVert \widehat{P}_\chi \Lambda_{\chi,\varepsilon}\widecheck{P}_\chi\right\rVert_{\mathcal{E} \to \mathcal{E}} \leq \left\lVert \Lambda^{\rm soft}_{\chi,\varepsilon}\big|_{\widehat{\mathcal{E}}_\chi} \right\rVert_{\mathcal{E} \to \mathcal{E}}, \qquad \left\lVert\widecheck{P}_\chi \Lambda_{\chi,\varepsilon} \widehat{P}_\chi \right\rVert_{\mathcal{E} \to \mathcal{E}} \leq \left\lVert \Lambda^{\rm soft}_{\chi,\varepsilon}\big|_{\widehat{\mathcal{E}}_\chi} \right\rVert_{\mathcal{E} \to \mathcal{E}},
\end{equation*}
and thus the operators $ \widehat{P}_\chi M_{\chi, \varepsilon}(z)\widecheck{P}_\chi$, $\widecheck{P}_\chi M_{\chi, \varepsilon}(z) \widehat{P}_\chi $,  are bounded as well,  uniformly in $\chi,\eps$.  We next show that the operator $\widecheck{P}_\chi M_{\chi, \varepsilon}(z) \widecheck{P}_\chi $ is boundedly invertible with a bound depending on $\varepsilon$. This is the point where we stress the importance of Steklov truncations and the bound \eqref{nakk204}.

To prove the mentioned boundedness, our first observation is that the operator $\widecheck{\Lambda}_\chi^{\rm soft}(\widecheck{\Lambda}_\chi^{\rm stiff})^{-1}$ is bounded independently of $\chi$, as a consequence of \eqref{nakk203} and \eqref{nakk204}. Using the formula \eqref{formulamboundedb}, we have
\begin{equation}
\label{formula11}
    \widecheck{M}_{\chi,\varepsilon}(z) = \Bigl( I + \varepsilon^2 \widecheck{\Lambda}_\chi^{\rm soft}\bigl(\widecheck{\Lambda}_\chi^{\rm stiff} \bigr)^{-1} + \varepsilon^2 \widecheck{P}_{\chi}\mathcal{B}_{\chi,\varepsilon}(z)\bigl(\widecheck{\Lambda}_\chi^{\rm stiff} \bigr)^{-1} \Bigr)\varepsilon^{-2}\widecheck{\Lambda}_\chi^{\rm stiff}.
\end{equation}
Since the operators $\widecheck{\Lambda}_\chi^{\rm soft}(\widecheck{\Lambda}_\chi^{\rm stiff})^{-1}$ and $\mathcal{B}_{\chi,\varepsilon}(\widecheck{\Lambda}_\chi^{\rm stiff})^{-1}$ are bounded uniformly in $\chi$ and $\eps$, we can choose $\varepsilon$ small enough so that  \eqref{formula11} is invertible and 
\begin{equation}
   \Bigl\lVert \Bigl( I + \varepsilon^2 \widecheck{\Lambda}_\chi^{\rm soft}\bigl(\widecheck{\Lambda}_\chi^{\rm stiff} \bigr)^{-1} + \varepsilon^2 \mathcal{B}_{\chi,\varepsilon}(z)\bigl(\widecheck{\Lambda}_\chi^{\rm stiff} \bigr)^{-1} \Bigr)^{-1} \Bigr\rVert_{\mathcal{E} \to \mathcal{E}} \leq \frac{1}{2}.
   \label{combin2}
\end{equation}
Combining \eqref{formula11} and \eqref{combin2} yields
    $\lVert\widecheck{M}_{\chi,\varepsilon}(z)^{-1}\rVert_{\mathcal{E} \to \mathcal{E}} \leq C \varepsilon^2.$
We write $M_{\chi,\varepsilon}$ as a block operator matrix relative to the decomposition \eqref{steklovdecomposition}:
\begin{equation*}
    M_{\chi,\varepsilon} = \begin{bmatrix} \mathbb{A} & \mathbb{B} \\ \mathbb{E} & \mathbb{F} \end{bmatrix},
\end{equation*}
where
\begin{equation*}
        \mathbb{A}:= \widehat{M}_{\chi,\eps}(z),\quad \mathbb{B}:= \widehat{P}_\chi M_{\chi,\eps}(z) \widecheck{P}_\chi, \quad \mathbb{E}:= \widecheck{P}_\chi M_{\chi,\eps}(z) \widehat{P}_\chi, \quad
        \mathbb{F}:=\widecheck{M}_{\chi,\eps}(z).
\end{equation*}
We have shown that  $\mathbb{A}$, $\mathbb{B}$, $\mathbb{E}$ are bounded (where the bound of $\mathbb{A}$ depends on $\eps$), and $\mathbb{F}$ is boundedly invertible: 
    $\lVert \mathbb{F}^{-1} \rVert_{\mathcal{E} \to \mathcal{E}}  \leq C\varepsilon^2,$
where $C$ does not depend on $\eps$. Our next objective is to show that $\mathbb{A}$ is boundedly invertible with a $\chi$-independent bound. To this end, notice that  \eqref{Mfunctionasymptotics} implies
\begin{equation} 
	\label{nakk300} 
\eps^{-2}\widehat{M}_{\chi}^{\rm stiff} (\eps^2 z)=\eps^{-2} \widehat{\Lambda}_{\chi}^{\rm stiff} +z\bigl(\widehat{\Pi}_\chi^{\rm stiff}\bigr)^* \widehat{\Pi}_\chi^{\rm stiff}+O(\eps^2),
\end{equation} 
where $O(\eps^2)$ is understood in the sense of the $L^2 \to L^2$ operator norm, uniformly in $\chi.$ Using \eqref{nakk301}, \eqref{nakk300}, Lemma \ref{lemmanormequivalence_f}, the trace inequality, and the fact that $z \in K_{\sigma},$ we infer the existence of a $\chi$-independent constant $C>0$ such that, for all $\vect f \in L^2(\Gamma;\C^3)$ and $\eps$ small enough, one has  
\begin{align*} 
\Bigl\vert\bigl\langle \Im \widehat{M}_{\chi,\eps} (z)\vect f ,\vect f&\bigr\rangle_{\mathcal{E}}\Bigr\vert = \Bigl\vert\bigl\langle\eps^{-2}\Im \widehat{M}^{\rm stiff}_{\chi} (\eps^2 z) \vect f ,\vect f\bigr\rangle_{\mathcal{E}}\Bigr\vert+ \Bigl\vert\bigl\langle \Im \widehat{M}^{\rm soft}_{\chi} (z) \vect f ,\vect f\bigr\rangle_{\mathcal{E}}\Bigr\vert
\geq \Bigl\vert\bigl\langle\eps^{-2}\Im \widehat{M}^{\rm stiff}_{\chi} (\eps^2 z) \vect f ,\vect f \bigr\rangle_{\mathcal{E}}\Bigr\vert\\[0.4em] 
&=\vert\Im z\vert\|\widehat{\Pi}^{\rm stiff}_{\chi}\vect f \|_{L^2(Y_{\rm stiff};\C^3)}+O(\eps^2)\|\vect f\|_{L^2(Y_{\rm stiff};\C^3)}
\geq C\|\widehat{\Pi}^{\rm stiff}_{\chi} \vect f\|_{H^1(Y_{\rm stiff};\C^3)}+O(\eps^2)\|\vect f\|_{L^2(Y_{\rm stiff};\C^3)}\\[0.4em]
& \geq C \|\widehat{P}_{\chi}\vect f\|_{L^2(Y_{\rm stiff};\C^3)}.
\end{align*} 
By virtue of Corollary \ref{boundfrombelowappendix}, it now follows that
\begin{equation} \label{nakk400} 
\left\lVert \mathbb{A}^{-1} \right\rVert_{\mathcal{E} \to \mathcal{E}}  \leq C,
\end{equation}
where $C>0$ does not depend on $\chi.$ 
Using the Schur-Frobenius inversion formula, see \cite{tretter2008}, we have
\begin{equation*}
    M_{\chi,\varepsilon}^{-1}(z) = \overline{\begin{bmatrix} \mathbb{A} & \mathbb{B} \\ \mathbb{E} & \mathbb{F} \end{bmatrix}}^{-1} = \begin{bmatrix} \mathbb{A}^{-1} & 0 \\ 0 & 0 \end{bmatrix} + \begin{bmatrix} \overline{\mathbb{A}^{-1}\mathbb{B}}\overline{\mathbb{S}}^{-1}\mathbb{E}\mathbb{A}^{-1} & -\overline{\mathbb{A}^{-1}\mathbb{B}}\overline{\mathbb{S}}^{-1} \\ \overline{\mathbb{S}}^{-1}\mathbb{E}\mathbb{A}^{-1} & \overline{\mathbb{S}}^{-1} \end{bmatrix},
\end{equation*} 
where  $\mathbb{S}:=\mathbb{F}-\mathbb{E}\mathbb{A}^{-1}\mathbb{B}$. Furthermore, since
    $\lVert \mathbb{S}^{-1}\rVert_{\mathcal{E} \to \mathcal{E}} \leq\lVert(I - \mathbb{F}^{-1}\mathbb{E}\mathbb{A}^{-1}\mathbb{B})^{-1}\mathbb{F}^{-1} \rVert_{\mathcal{E} \to \mathcal{E}} \leq C \varepsilon^2,$
we write
\begin{equation*}
    M_{\chi,\varepsilon}^{-1}(z) =\begin{bmatrix} \bigl( \widehat{P}_\chi M_{\chi,\varepsilon}(z)\widehat{P}_\chi \bigr)^{-1} & 0 \\ 0 & 0 \end{bmatrix} + \mathcal{O}(\varepsilon^2)  = \widehat{P}_\chi\bigl( \widehat{P}_\chi M_{\chi,\varepsilon}(z)\widehat{P}_\chi \bigr)^{-1}\widehat{P}_\chi   + \mathcal{O}(\varepsilon^2).
\end{equation*}
On the other hand, the Schur-Frobenius formula implies
\begin{equation*}
        \bigl(\widecheck{P}_\chi+\widehat{P}_\chi M_{\varepsilon,\chi}(z)\bigr)^{-1} = \begin{bmatrix}  \widehat{P}_\chi M_{\chi,\varepsilon}(z)\widehat{P}_\chi  & \widehat{P}_\chi M_{\chi,\varepsilon}(z)\widecheck{P}_\chi\\ 0 & I \end{bmatrix}^{-1} 
        = \begin{bmatrix} \bigl( \widehat{P}_\chi M_{\chi,\varepsilon}(z)\widehat{P}_\chi \bigr)^{-1} & -\bigl( \widehat{P}_\chi M_{\chi,\varepsilon}(z)\widehat{P}_\chi\bigr)^{-1}\widehat{P}_\chi M_{\chi,\varepsilon}(z)\widecheck{P}_\chi \\ 0 & I \end{bmatrix}.
\end{equation*}
Now, clearly
     $M_{\chi,\varepsilon}^{-1}(z) = \bigl(\widecheck{P}_\chi+\widehat{P}_\chi M_{\varepsilon,\chi}(z)\bigr)^{-1}\widehat{P}_\chi + \mathcal{O}(\varepsilon^2).$
Thus, by putting $\beta_0 = \widecheck{P}_\chi $, $\beta_1 = \widehat{P}_\chi$ and recalling the Theorem \ref{theoremkreinformula} we  recognise that the operator-valued function
\begin{equation*}
    z \mapsto \left( \mathcal{A}_{0,\chi,\varepsilon} -z I  \right)^{-1} - S_{\chi, \varepsilon}(z) \bigl(\widecheck{P}_\chi+\widehat{P}_\chi M_{\varepsilon,\chi}(z)  \bigr)^{-1}\widehat{P}_\chi S_{\chi, \varepsilon}(\overline{z})^*
\end{equation*}
is the resolvent of the closed extension of $\mathcal{A}_{0,\chi,\varepsilon}$ associated, in the sense of Theorem \ref{theoremkreinformula}, with the boundary condition
    $(\widecheck{P}_\chi \Gamma_{0,\chi} + \widehat{P}_\chi \Gamma_{1,\chi}) \vect u = 0,$
and the result follows.
\end{proof}
\begin{remark}
The resolvent $(\left(\mathcal{A}_{\chi,\varepsilon}\right)_{\widecheck{P}_\chi,\widehat{P}_\chi} -zI)^{-1}$ is the solution operator of the boundary value problem ($\vect f\in\mathcal{H}$)
\begin{equation*}
         \mathcal{A}_{\chi,\varepsilon}\vect u - z \vect u = \vect f,\qquad 
          \bigl(\widecheck{P}_\chi \Gamma_{0,\chi} + \widehat{P}_\chi \Gamma_{1,\chi}\bigr) \vect u = 0.
\end{equation*}
The solution $\vect u \in \mathcal{H}$ satisfies the following constraints:
\begin{itemize}
    \item The $3$-dimensional projections of the traces on $\Gamma$ of conormal derivatives of $\vect u$ (from inside $Y_{\rm soft}$, $Y_{\rm stiff}$) onto the space $\widehat{\mathcal{E}}_\chi$ coincide.
    \item The traces of $\vect u$ from both  $Y_{\rm soft}$ and $Y_{\rm stiff}$  belong to the $3$-dimensional space $\widehat{\mathcal{E}}_\chi$ (they clearly coincide, as per Remark \ref{tracescoincide}). \BBB
\end{itemize}
Thus, by approximating the resolvent $((\mathcal{A}_{\chi,\varepsilon})_{0,I} -zI)^{-1}$ associated with the transmission problem \eqref{transmissionboundaryproblem} by the resolvent  $((\mathcal{A}_{\chi,\varepsilon})_{\widecheck{P}_\chi,\widehat{P}_\chi} -zI)^{-1}$, one relaxes the condition on the continuity of co-normal derivatives and tightens the constraint on the traces, which leads to an error of order $\varepsilon^2.$

Intuitively, the homogenisation procedure should replace the solution on the stiff component with a $3$-dimensional constant vector. However, at this point only the trace of the solution is finite-dimensional.
\end{remark}

\subsection{Approximation refinement: truncation of $\widehat{M}_\chi^{\rm stiff}$} \label{refinement} 
We introduce the following notation for the truncated solution operator $\widehat{S}_{\chi,\varepsilon}(z):=S_{\chi,\varepsilon}(z)|_{\widehat{\mathcal{E}}_\chi}$.
Using  \eqref{soperatoranotherform}, we have
\begin{equation}
\label{soperatorhatformula}
\begin{aligned}
        \widehat{S}_{\chi,\varepsilon}(z):=S_{\chi,\varepsilon}(z)\big|_{\widehat{\mathcal{E}}_\chi} &= \bigl(\Pi_\chi + z \bigl(\mathcal{A}_{0,\chi,\varepsilon} -z I  \bigr)^{-1} \Pi_\chi \bigr)\big|_{\widehat{\mathcal{E}}_\chi} 
        =\widehat{\Pi}_\chi + z\bigl( \mathcal{A}_{0,\chi,\varepsilon} -z I\bigr)^{-1} \widehat{\Pi}_\chi,
\end{aligned}
\end{equation}
which is precisely the solution operator associated with the triple $(\mathcal{A}_{0,\chi,\varepsilon},\widehat{\Pi}_\chi,\widehat{\Lambda}_{\chi,\varepsilon})$ in the sense of Definition \ref{solutionmfunction}. Similar representation formulae are obtained for the operators $\widehat{S}^{\rm soft(stiff)}_{\chi}(z),$ which are defined in an obvious way.

\begin{remark}
	Notice that 
\begin{equation}
\label{Sadjointformulae}
	\bigl(\widehat{S}_\chi^{\rm stiff(soft)}(z)\bigr)^* = \widehat{P}_\chi \bigl(S_\chi^{\rm stiff(soft)} \bigr)^* , \qquad \bigl(\widehat{\Pi}_\chi^{\rm stiff(soft)} \bigr)^* = \widehat{P}_\chi\bigl(\Pi_\chi^{\rm stiff(soft)} \bigr)^*.
\end{equation}
 Also, one has
	\begin{equation}
		\label{pi_projections}			 P_{\widehat{\mathcal{H}}_\chi^{\rm stiff(soft)}} \widehat{\Pi}_\chi^{\rm stiff(soft)} = \widehat{\Pi}_\chi^{\rm stiff(soft)}, \quad  \bigl(\widehat{\Pi}_\chi^{\rm stiff(soft)} \bigr)^* P_{\widehat{\mathcal{H}}_\chi^{\rm stiff(soft)}} =   \bigl(\widehat{\Pi}_\chi^{\rm stiff(soft)}\bigr)^*,
	\end{equation}
which follows directly from the definition of the operators involved.
\end{remark}

The operator $\widehat{\Gamma}_{0,\chi}$ is the left inverse of the operator $\widehat{\Pi}_\chi$ in the sense of Definition \ref{boundary operators}, so by virtue of \eqref{soperatorhatformula} it is the left inverse of  $\widehat{S}_{\chi,\varepsilon}(z)$ as well. A similar claim applies to the operators $\widehat{S}_{\chi}^{\rm stiff(soft)}(z)$. In particular, we have
\begin{equation*}
    \widehat{\Gamma}_{0,\chi}\widehat{S}_{\chi,\varepsilon}(z) = \widehat{\Gamma}_{0,\chi}^{\rm stiff(soft)}\widehat{S}_{\chi}^{\rm stiff(soft)}(z) = I|_{\widehat{\mathcal{E}}_\chi}.
\end{equation*}
In Theorem \ref{firstapproximationtheorem} we have obtained an approximation of the original resolvent in terms of the resolvent of another operator, where the relative simplification is not immediately evident.  However, by doing simple additional approximations the result becomes much more transparent. We carry out these additional approximations by analyzing the block components of the resolvent (see \eqref{Sadjointformulae}) separately.
\begin{equation*}
\begin{aligned}
     \bigl( \left(\mathcal{A}_{\chi,\varepsilon}\right)_{\widecheck{P}_\chi,\widehat{P}_\chi} -zI \bigr)^{-1} &= \left( \mathcal{A}_{0,\chi,\varepsilon} -z I  \right)^{-1} - S_{\chi, \varepsilon}(z)\widehat{P}_\chi \bigl(\widehat{P}_\chi M_{\varepsilon,\chi}(z) \widehat{P}_\chi \bigr)^{-1}\widehat{P}_\chi S_{\chi, \varepsilon}(\overline{z})^* \\[0em]
     &= \left( \mathcal{A}_{0,\chi,\varepsilon} -z I  \right)^{-1} - \widehat{S}_{\chi, \varepsilon}(z) {\widehat{M}_{\varepsilon,\chi}(z)}^{-1} \widehat{S}_{\chi, \varepsilon}(\overline{z})^*,
\end{aligned}
\end{equation*}
relative to the decomposition \eqref{decompositionsoftstiff}.

It follows from \eqref{nakk300} and \eqref{nakk400} that
\begin{equation}
\label{Masymptoticformula}
    {\widehat{M}_{\varepsilon,\chi}(z)}^{-1} = \bigl(\varepsilon^{-2}\widehat{M}_{\chi}^{\rm stiff}(\varepsilon^2 z) + \widehat{M}_{\chi}^{\rm soft}(z) \bigr)^{-1}= \bigl(\varepsilon^{-2} \widehat{\Lambda}_\chi^{\rm stiff} + z \bigl(\widehat{\Pi}_\chi^{\rm stiff}\bigr)^* \widehat{\Pi}_{\chi}^{\rm stiff} + \widehat{M}_{\chi}^{\rm soft}(z) \bigr)^{-1}+ \mathcal{O}(\varepsilon^2),
\end{equation}
which we use in the proof of Theorem \ref{thmrefinement} below. We introduce the following operator-valued function featuring prominently in our homogenisation results.
\begin{definition}
	We refer to the operator-valued function 
	$\widehat{Q}^{\rm app}_{\chi,\varepsilon}(z): \widehat{\mathcal{E}}_\chi \to \widehat{\mathcal{E}}_\chi$ given by  
\begin{equation}
\label{qoperator}
    \widehat{Q}^{\rm app}_{\chi,\varepsilon}(z):= \varepsilon^{-2}\widehat{\Lambda}_\chi^{\rm stiff} + z \bigl(\widehat{\Pi}_\chi^{\rm stiff}\bigr)^* \widehat{\Pi}_{\chi}^{\rm stiff} + \widehat{M}_{\chi}^{\rm soft}(z)
\end{equation}
as the \emph{transmission function}.
\end{definition}
We next prove a result on  resolvent asymptotics that simplifies the solution on the stiff component.
\begin{theorem}\label{thmrefinement} 
There exists $C>0$, which depends only on $\sigma$ and ${\rm diam}(K_\sigma),$ such that for the resolvent of the transmission problem \eqref{transmissionboundaryproblem}  one has
 \begin{equation*}
     \bigl\lVert \bigl( \left(\mathcal{A}_{\chi,\varepsilon}\right)_{0,I}-zI\bigr)^{-1} - \mathcal{R}_{\chi,\varepsilon}^{\rm app}(z) \bigr\rVert_{\mathcal{H} \to \mathcal{H}} \leq C \varepsilon^2\qquad \forall\chi \in Y',
 \end{equation*}
where the operator-valued function $\mathcal{R}_{\chi,\varepsilon}^{\rm app}(z)$ is defined by 
 \begin{equation}
 \label{resolventblockmatrix}
     \mathcal{R}_{\chi,\varepsilon}^{\rm app}(z):=
     \begin{bmatrix}\bigl( \mathcal{A}_{0,\chi}^{\rm soft} -z I  \bigr)^{-1} - \widehat{S}^{\rm soft}_{\chi}(z){\widehat{Q}^{\rm app}_{\chi,\varepsilon}(z)}^{-1} \widehat{S}_{\chi}^{\rm soft}(\overline{z})^* & -\widehat{S}^{\rm soft}_{\chi}(z){\widehat{Q}^{\rm app}_{\chi,\varepsilon}(z)}^{-1} \bigl(\widehat{\Pi}_\chi^{\rm stiff}\bigr)^*\\[0.5em]
     -\widehat{\Pi}_\chi^{\rm stiff}{\widehat{Q}^{\rm app}_{\chi,\varepsilon}(z)}^{-1} \widehat{S}_{\chi}^{\rm soft}(\overline{z})^* & - \widehat{\Pi}_\chi^{\rm stiff}{\widehat{Q}^{\rm app}_{\chi,\varepsilon}(z)}^{-1} \bigl(\widehat{\Pi}_\chi^{\rm stiff}\bigr)^*\end{bmatrix},
 \end{equation}
and the block-operator matrix is understood relative to the decomposition $\mathcal{H}^{\rm soft} \oplus \mathcal{H}^{\rm stiff}$.
\end{theorem}
\begin{proof}
The proof consists in applying the formula \eqref{Masymptoticformula} to the individual blocks of the resolvent $( \left(\mathcal{A}_{\chi,\varepsilon}\right)_{\widecheck{P}_\chi,\widehat{P}_\chi} -zI)^{-1}$. We have 
\begin{equation*}
\begin{aligned}
    P_{\rm soft}\bigl( (\mathcal{A}_{\chi,\varepsilon})_{\widecheck{P}_\chi,\widehat{P}_\chi} -zI \bigr)^{-1}P_{\rm soft} = & \bigl( \mathcal{A}_{0,\chi}^{\rm soft} -z I  \bigr)^{-1} - \widehat{S}^{\rm soft}_{\chi}(z) {\widehat{M}_{\varepsilon,\chi}(z)}^{-1} \widehat{S}_{\chi}^{\rm soft}(\overline{z})^* \\[0.3em]
     = & \bigl( \mathcal{A}_{0,\chi}^{\rm soft} -z I  \bigr)^{-1} - \widehat{S}^{\rm soft}_{\chi}(z){\widehat{Q}^{\rm app}_{\chi,\varepsilon}(z)}^{-1} \widehat{S}_{\chi}^{\rm soft}(\overline{z})^* + \mathcal{O}(\varepsilon^2).
\end{aligned}
\end{equation*}
Next, we use the asymptotic formula \eqref{asymptoticsforsformula} for $S_\chi^{\rm stiff}$:
\begin{equation*}
\begin{aligned}
    P_{\rm stiff}\bigl( (\mathcal{A}_{\chi,\varepsilon})_{\widecheck{P}_\chi,\widehat{P}_\chi} -zI \bigr)^{-1}P_{\rm soft} = &  - \widehat{S}^{\rm stiff}_{\chi}( \varepsilon^2 z) {\widehat{M}_{\varepsilon,\chi}(z)}^{-1} \widehat{S}_{\chi}^{\rm soft}(\overline{z})^* 
    = -\widehat{\Pi}_\chi^{\rm stiff}{\widehat{Q}^{\rm app}_{\chi,\varepsilon}(z) }^{-1} \widehat{S}_{\chi}^{\rm soft}(\overline{z})^* + \mathcal{O}(\varepsilon^2),\\[0.35em]
    P_{\rm soft}\bigl( (\mathcal{A}_{\chi,\varepsilon})_{\widecheck{P}_\chi,\widehat{P}_\chi} -zI \bigr)^{-1}P_{\rm stiff} = &  - \widehat{S}^{\rm soft}_{\chi}( z) {\widehat{M}_{\varepsilon,\chi}(z)}^{-1} \widehat{S}_{\chi}^{\rm stiff}(\varepsilon^2 \overline{z})^* 
    =-\widehat{S}^{\rm soft}_{\chi}( z){\widehat{Q}^{\rm app}_{\chi,\varepsilon}(z)}^{-1} \bigl(\widehat{\Pi}_\chi^{\rm stiff}\bigr)^* + \mathcal{O}(\varepsilon^2).
\end{aligned}
\end{equation*}
For calculating the remaining block, we use the fact that
\begin{equation*}
    P_{\rm stiff}\left( \mathcal{A}_{0,\chi,\varepsilon} -z I  \right)^{-1}P_{\rm stiff} = \bigl( \varepsilon^{-2}\mathcal{A}_{0,\chi}^{\rm stiff}-zI\bigr)^{-1} = \mathcal{O}(\varepsilon^2)
\end{equation*}
in the $\mathcal{H} \to \mathcal{H}$ operator norm. Finally, we have
\begin{equation*}
    \begin{aligned}
    P_{\rm stiff}\bigl( (\mathcal{A}_{\chi,\varepsilon})_{\widecheck{P}_\chi,\widehat{P}_\chi} -zI \bigr)^{-1}P_{\rm stiff} =  \bigl( \varepsilon^{-2}\mathcal{A}_{0,\chi}^{\rm stiff} -z I  \bigr)^{-1} - \widehat{S}^{\rm stiff}_{\chi}(z) {\widehat{M}_{\varepsilon,\chi}(z)}^{-1} \widehat{S}_{\chi}^{\rm stiff}(\overline{z})^* \\[0.3em]
     = &  
     -\widehat{\Pi}_\chi^{\rm stiff} {\widehat{Q}^{\rm app}_{\chi,\varepsilon}(z)}^{-1} \bigl(\widehat{\Pi}_\chi^{\rm stiff}\bigr)^* + \mathcal{O}(\varepsilon^2),
\end{aligned}
\end{equation*}
which, combined with Theorem \ref{firstapproximationtheorem}, completes the proof. 
\end{proof}
\begin{remark}
Note that one can rewrite \eqref{resolventblockmatrix} as follows:
\begin{equation}
 \label{resolventcondensedform}
     \mathcal{R}_{\chi,\varepsilon}^{\rm app}(z):= \bigl(\mathcal{A}_{0,\chi}^{\rm soft} -z I  \bigr)^{-1} P_{\mathcal{H}^{\rm soft}} - \begin{bmatrix}
     \widehat{S}^{\rm soft}_{\chi}(z) & \widehat{\Pi}_\chi^{\rm stiff}
     \end{bmatrix}{\widehat{Q}^{\rm app}_{\chi,\varepsilon}(z)}^{-1} \begin{bmatrix}
     \widehat{S}^{\rm soft}_{\chi}(z) & \widehat{\Pi}_\chi^{\rm stiff}
     \end{bmatrix}^*.
 \end{equation}
\end{remark}

\subsection{Fiberwise approximating operator} \label{Sectionopapp} 
It remains to provide an explicit description of the selfadjoint operator whose resolvent is given by \eqref{resolventblockmatrix}.
To this end, we consider the Hilbert space $\mathcal{H}^{\rm soft} \oplus \widehat{\mathcal{H}}_\chi^{\rm stiff}$ and define the 
operator $\mathcal{A}_{\chi,\varepsilon}^{\rm app}$ as follows:\PPP \footnote{Since the operators $\widehat{\Pi}_\chi^{\rm stiff}$ are acting from 3-dimensional space to a 3-dimensional space they themselves and their inverses can be represented by the appropriate matrices.} \BBB
\begin{equation}
\begin{aligned}
    \mathcal{D}\bigl(\mathcal{A}_{\chi,\varepsilon}^{\rm app}\bigr)&:=\bigl\{ (\vect u,\widehat{\vect u})^\top\in \mathcal{H}^{\rm soft} \oplus \widehat{\mathcal{H}}_\chi^{\rm stiff}, \quad \vect u \in \mathcal{D}(\widehat{\mathcal{A}}_\chi^{\rm soft}), \quad  \widehat{\vect u} = \widehat{\Pi}_\chi^{\rm stiff}\widehat{\Gamma}_{0,\chi}^{\rm soft} \vect u  \bigr\} ,\\[0.3em]
    \mathcal{A}_{\chi,\varepsilon}^{\rm app} \begin{bmatrix}
    \vect u \\ \widehat{\vect u}
    \end{bmatrix}&:= \begin{bmatrix}
    \widehat{\mathcal{A}}_\chi^{\rm soft} & 0 \\[0.3em]
     - \bigl( \bigl(\widehat{\Pi}_\chi^{\rm stiff} \bigr)^* \bigr)^{-1} \widehat{\Gamma}_{1,\chi}^{\rm soft} &     -\varepsilon^{-2}\bigl( \bigl(\widehat{\Pi}_\chi^{\rm stiff} \bigr)^* \bigr)^{-1} \widehat{\Gamma}_{1,\chi}^{\rm stiff}
    \end{bmatrix}\begin{bmatrix}
    \vect u \\ \widehat{\vect u}
    \end{bmatrix}. 
\end{aligned}
\label{homoperator}
\end{equation}
The following theorem links $\mathcal{R}_{\chi,\varepsilon}^{\rm app}(z),$ see
\eqref{resolventblockmatrix}, to the resolvent of $\mathcal{A}_{\chi,\varepsilon}^{\rm app}.$
\begin{theorem}
\label{theoremappoperator}
 For every $\chi \in Y'$, the operator $\mathcal{A}_{\chi,\varepsilon}^{\rm app}$ is self-adjoint and its resolvent for all $z\in \rho(\mathcal{A}_{\chi,\varepsilon}^{\rm app})$ is given by the formula \eqref{resolventblockmatrix}, relative to the decomposition $\mathcal{H}^{\rm soft} \oplus \widehat{\mathcal{H}}_\chi^{\rm stiff}.$
\end{theorem}
\begin{proof}
First we show that the operator $\mathcal{A}_{\chi,\varepsilon}^{\rm app}$ is symmetric. For $(\vect u, \widehat{\vect u})^\top, (\vect v, \widehat{\vect v})^\top\in \mathcal{D}(\mathcal{A}_{\chi,\varepsilon}^{\rm app})$ one has
\begin{equation*}
\begin{aligned}
        \biggl\langle \mathcal{A}_{\chi,\varepsilon}^{\rm app} \begin{bmatrix}
    \vect u \\ \widehat{\vect u}
    \end{bmatrix},\begin{bmatrix}
    \vect v \\ \widehat{\vect v}
    \end{bmatrix} \biggr\rangle 
& = \bigl\langle \widehat{\mathcal{A}}_\chi^{\rm soft}\vect u , \vect v\bigr\rangle_{\mathcal{H}^{\rm soft}}
-\Bigl\langle \bigl( \bigl(\widehat{\Pi}_\chi^{\rm stiff} \bigr)^* \bigr)^{-1} \widehat{\Gamma}_{1,\chi}^{\rm soft}\vect u, \widehat{\vect v} \Bigr\rangle_{\widehat{\mathcal{H}}_\chi^{\rm stiff}} 
-\varepsilon^{-2}
    \Bigl\langle
      \bigl(\bigl(\widehat{\Pi}_\chi^{\rm stiff} \bigr)^* \bigr)^{-1} \widehat{\Gamma}_{1,\chi}^{\rm stiff} \widehat{\vect u},\widehat{\vect v}\Bigr\rangle_{\widehat{\mathcal{H}}_\chi^{\rm stiff}} \\[0.3em]      
      & = \bigl\langle \widehat{\mathcal{A}}_\chi^{\rm soft}\vect u , \vect v \bigr\rangle_{\mathcal{H}^{\rm soft}}  - \Bigl\langle \bigl( \bigl(\widehat{\Pi}_\chi^{\rm stiff} \bigr)^* \bigr)^{-1} \widehat{\Gamma}_{1,\chi}^{\rm soft}\vect u, \widehat{\Pi}_\chi^{\rm stiff}\widehat{\Gamma}_{0,\chi}^{\rm soft} \vect v \Bigr\rangle_{\widehat{\mathcal{H}}_\chi^{\rm stiff}} \\[0.3em] 
      &\hspace{20ex}-\varepsilon^{-2}
    \Bigl\langle
      \bigl( \bigl(\widehat{\Pi}_\chi^{\rm stiff} \bigr)^* \bigr)^{-1} \widehat{\Gamma}_{1,\chi}^{\rm stiff}\widehat{\Pi}_\chi^{\rm stiff}\widehat{\Gamma}_{0,\chi}^{\rm soft} \vect u,\widehat{\Pi}_\chi^{\rm stiff}\widehat{\Gamma}_{0,\chi}^{\rm soft} \vect v\Bigr\rangle_{\widehat{\mathcal{H}}_\chi^{\rm stiff}} \\[0.3em]
      & = \bigl\langle \widehat{\mathcal{A}}_\chi^{\rm soft}\vect u , \vect v \bigr\rangle_{\mathcal{H}^{\rm soft}}  - \bigl\langle  \widehat{\Gamma}_{1,\chi}^{\rm soft}\vect u, \widehat{\Gamma}_{0,\chi}^{\rm soft} \vect v \bigr\rangle_{\widehat{\mathcal{E}}_\chi}-\varepsilon^{-2}
    \bigl\langle   \widehat{\Lambda}_\chi^{\rm stiff}\widehat{\Gamma}_{0,\chi}^{\rm soft} \vect u,\widehat{\Gamma}_{0,\chi}^{\rm soft} \vect v\bigr\rangle_{\widehat{\mathcal{E}}_\chi}.
\end{aligned}
\end{equation*}
By Green's formula (see \eqref{Greenformula})  and the self-adjointness of $\widehat{\Lambda}_\chi^{\rm stiff},$ one has
\begin{equation*}
    \left\langle \mathcal{A}_{\chi,\varepsilon}^{\rm app} \begin{bmatrix}
    \vect u \\ \widehat{\vect u}
    \end{bmatrix},\begin{bmatrix}
    \vect v \\ \widehat{\vect v}
    \end{bmatrix} \right\rangle = \left\langle  \begin{bmatrix}
    \vect u \\ \widehat{\vect u}
    \end{bmatrix},\mathcal{A}_{\chi,\varepsilon}^{\rm app}\begin{bmatrix}
    \vect v \\ \widehat{\vect v}
    \end{bmatrix} \right\rangle.
\end{equation*}
Next, we fix $\vect f \in \mathcal{H}^{\rm soft}$, $\,\;\widehat{\!\!\vect f} \in \widehat{\mathcal{H}}_\chi^{\rm stiff}$. For every $z \in \rho(\mathcal{A}_{\chi,\varepsilon}^{\rm app})$ we consider the problem
\begin{equation*}
    \left(\mathcal{A}_{\chi,\varepsilon}^{\rm app} - zI \right)\begin{bmatrix}
    \vect u\\ \widehat{\vect u}
    \end{bmatrix} = \begin{bmatrix}
    \vect f\\ \,\;\widehat{\!\!\vect f}
    \end{bmatrix}, \quad \begin{bmatrix}
    \vect u\\ \widehat{\vect u}
    \end{bmatrix} \in \mathcal{D}\left( \mathcal{A}_{\chi,\varepsilon}^{\rm app}\right).
\end{equation*}
Component-wise, we have
\begin{equation}
\label{blockresolvent2}
        \left\{ \begin{array}{ll}
         \widehat{\mathcal{A}}_{\chi}^{\rm soft}\vect u - z \vect u = \vect f,\\[0.4em]
         - \bigl( \bigl(\widehat{\Pi}_\chi^{\rm stiff} \bigr)^* \bigr)^{-1} \widehat{\Gamma}_{1,\chi}^{\rm soft}  \vect u  -\varepsilon^{-2}\bigl( \bigl(\widehat{\Pi}_\chi^{\rm stiff}\bigr)^* \bigr)^{-1} \widehat{\Gamma}_{1,\chi}^{\rm stiff}\widehat{\vect u} - z \widehat{\vect u} = \,\;\widehat{\!\!\vect f},\end{array} \right.\quad \begin{bmatrix}
    \vect u\\ \widehat{\vect u}
    \end{bmatrix} \in \mathcal{D}\left( \mathcal{A}_{\chi,\varepsilon}^{\rm app}\right),
\end{equation}
which is equivalent to
\begin{equation}
\label{blockresolvent3}
        \left\{ \begin{array}{ll}
         \widehat{\mathcal{A}}_{\chi}^{\rm soft}\vect u - z \vect u = \vect f,\\[0.4em]
         -  \widehat{\Gamma}_{1,\chi}^{\rm soft}  \vect u -\varepsilon^{-2}   \widehat{\Gamma}_{1,\chi}^{\rm stiff}\widehat{\vect u} - z \bigl( \widehat{\Pi}_\chi^{\rm stiff} \bigr)^*\widehat{\vect u} = \bigl( \widehat{\Pi}_\chi^{\rm stiff} \bigr)^*\,\;\widehat{\!\!\vect f},\end{array} \right.\quad \begin{bmatrix}
    \vect u\\ \widehat{\vect u}
    \end{bmatrix} \in \mathcal{D}\left( \mathcal{A}_{\chi,\varepsilon}^{\rm app}\right).
\end{equation}
Due to the fact that $\widehat{\vect u} = \widehat{\Pi}_\chi^{\rm stiff}\widehat{\Gamma}_{0,\chi}^{\rm soft} \vect u$, the problem \eqref{blockresolvent3} is equivalent to finding a vector $\vect u \in \mathcal{D}(\widehat{\mathcal{A}}_{\chi}^{\rm soft})$ such that
\begin{equation*}
        \left\{ \begin{array}{ll}
         \widehat{\mathcal{A}}_{\chi}^{\rm soft}\vect u - z \vect u = \vect f,\\[0.3em]
         -  \widehat{\Gamma}_{1,\chi}^{\rm soft}  \vect u -\varepsilon^{-2} \widehat{\Gamma}_{1,\chi}^{\rm stiff}\widehat{\Pi}_\chi^{\rm stiff}\widehat{\Gamma}_{0,\chi}^{\rm soft} \vect u - z \bigl(\widehat{\Pi}_\chi^{\rm stiff} \bigr)^*\widehat{\Pi}_\chi^{\rm stiff}\widehat{\Gamma}_{0,\chi}^{\rm soft} \vect u = \bigl(\widehat{\Pi}_\chi^{\rm stiff} \bigr)^*\,\;\widehat{\!\!\vect f}\end{array} \right.
\end{equation*}
and then setting $\widehat{\vect u} =\widehat{\Pi}_\chi^{\rm stiff}\widehat{\Gamma}_{0,\chi}^{\rm soft} \vect u.$ By recalling \eqref{dirichlettoneumanndecomp}, one has $\widehat{\Gamma}_{1,\chi}^{\rm stiff}\widehat{\Pi}_\chi^{\rm stiff} = \widehat{\Lambda}_{\chi}^{\rm stiff},$ and therefore the above is equivalent to finding $\vect u \in \mathcal{D}( \widehat{\mathcal{A}}_{\chi}^{\rm soft})$ such that  
\begin{equation*}
         \widehat{\mathcal{A}}_{\chi}^{\rm soft}\vect u - z \vect u = \vect f,\qquad\quad
          \widehat{\Gamma}_{1,\chi}^{\rm soft}  \vect u  +\bigl(\varepsilon^{-2}\widehat{\Lambda}_{\chi}^{\rm stiff} + z\bigl( \widehat{\Pi}_\chi^{\rm stiff} \bigr)^*\widehat{\Pi}_\chi^{\rm stiff}\bigr)\widehat{\Gamma}_{0,\chi}^{\rm soft}  \vect u = -\bigl( \widehat{\Pi}_\chi^{\rm stiff} \bigr)^*\,\;\widehat{\!\!\vect f}.
\end{equation*}

We next define the operators $\beta_{0,\chi,\varepsilon}(z):=\varepsilon^{-2}\widehat{\Lambda}_{\chi}^{\rm stiff}+z (\widehat{\Pi}_\chi^{\rm stiff})^*\widehat{\Pi}_\chi^{\rm stiff}$, $\beta_1 = I,$ so the  transmission function \eqref{qoperator} can be written as 
\begin{equation}
    \widehat{Q}^{\rm app}_{\chi,\varepsilon}(z)= \beta_{0,\chi,\varepsilon}(z) + \beta_1\widehat{M}_{\chi}^{\rm soft}(z).
    \label{Qeps_ref}
\end{equation}
The operator $\widehat{Q}^{\rm app}_{\chi,\varepsilon}(z)$ is boundedly invertible (as can be seen by considering its imaginary part and using Corollary \ref{boundfrombelowappendix}) and satisfies the assumptions of Theorem \ref{theoremsolutionformularobin}. The solution $\vect u$ is then given by  \eqref{solutionabstractrobin} with $\vect g=-(\widehat{\Pi}_\chi^{\rm stiff})^*\,\;\widehat{\!\!\vect f}$:
\begin{equation*}
\begin{aligned}
    \vect u &= (\mathcal{A}_{0,\chi}^{\rm soft} -zI)^{-1}\vect f - \widehat{S}_{\chi}^{\rm soft}(z)\widehat{Q}^{\rm app}_{\chi,\varepsilon}(z)^{-1}\bigl( (\widehat{S}_{\chi}^{\rm soft}(\overline{z}))^* \vect f +\bigl( \widehat{\Pi}_\chi^{\rm stiff}\bigr)^*\,\;\widehat{\!\!\vect f} \bigr) \\[0.3em]
    & = \begin{bmatrix}
    (\mathcal{A}_{0,\chi}^{\rm soft} -zI)^{-1} - \widehat{S}_{\chi}^{\rm soft}(z) \widehat{Q}^{\rm app}_{\chi,\varepsilon}(z)^{-1} (\widehat{S}_{\chi}^{\rm soft}(\overline{z}))^*  & -\widehat{S}_{\chi}^{\rm soft}(z) \widehat{Q}^{\rm app}_{\chi,\varepsilon}(z)^{-1}\bigl(\widehat{\Pi}_\chi^{\rm stiff}\bigr)^*
    \end{bmatrix}\begin{bmatrix}
    \vect f \\\,\;\widehat{\!\!\vect f}
    \end{bmatrix} .
\end{aligned}
\end{equation*}
Now, one has
\begin{equation*}
\begin{aligned}
    \widehat{\vect u} &= \widehat{\Pi}_\chi^{\rm stiff}\widehat{\Gamma}_{0,\chi}^{\rm soft}\vect u =  -\widehat{\Pi}_\chi^{\rm stiff}\widehat{\Gamma}_{0,\chi}^{\rm soft} \,\widehat{S}_{\chi}^{\rm soft}(z) \widehat{Q}^{\rm app}_{\chi,\varepsilon}(z)^{-1}\bigl((\widehat{S}_{\chi}^{\rm soft}(\overline{z}))^* \vect f +\bigl(\widehat{\Pi}_\chi^{\rm stiff} \bigr)^*\,\;\widehat{\!\!\vect f}\bigr) 
    \\[0.3em]
    & = \begin{bmatrix}
   -\widehat{\Pi}_\chi^{\rm stiff} \widehat{Q}^{\rm app}_{\chi,\varepsilon}(z)^{-1} (\widehat{S}_{\chi}^{\rm soft}(\overline{z}))^*  & -\widehat{\Pi}_\chi^{\rm stiff} \widehat{Q}^{\rm app}_{\chi,\varepsilon}(z)^{-1}\bigl(\widehat{\Pi}_\chi^{\rm stiff} \bigr)^*
    \end{bmatrix}\begin{bmatrix}
    \vect f \\\,\;\widehat{\!\!\vect f}
    \end{bmatrix} .
\end{aligned}
\end{equation*}
\end{proof}
Another insight into the operator 
$\mathcal{A}_{\chi,\varepsilon}^{\rm app}$ is obtained by considering its sesquilinear form.
\begin{lemma}
The sesquilinear form $a_{\chi,\varepsilon}^{\rm app}$ on $\mathcal{H}\times \mathcal{H}$ associated with $\mathcal{A}_{\chi,\varepsilon}^{\rm app}$
is given by 
\begin{equation*}
	\begin{aligned}
    \mathcal{D}\left(a_{\chi,\varepsilon}^{\rm app}\right)&:=\bigl\{ (\vect u,\widehat{\vect u})^\top\in \mathcal{H}^{\rm soft} \oplus \widehat{\mathcal{H}}_\chi^{\rm stiff}, \quad \vect u \in \mathcal{D}(a_{0,\chi}^{\rm soft})\dot +  \widehat{\Pi}_\chi^{\rm soft} \widehat{\mathcal{E}}_\chi, \quad  \widehat{\Gamma}_{0,\chi}^{\rm stiff}\widehat{\vect u} = \widehat{\Gamma}_{0,\chi}^{\rm soft} \vect u \bigr\},\\[0.4em]
	 a_{\chi,\varepsilon}^{\rm app}\left( \begin{bmatrix}
    \vect u \\ \widehat{\vect u}
    \end{bmatrix} ,\begin{bmatrix}
    \vect v \\ \widehat{\vect v}
    \end{bmatrix}\right)&:=\int_{Y_{\rm soft}} \A_{\rm soft} \left(\simgrad +{\rm i}X_\chi \right)\vect u : \overline{ \left(\simgrad +{\rm i}X_\chi \right)\vect v}\\[0.4em]
&+\frac{1}{\varepsilon^2}\int_{Y_{\rm stiff}} \A_{\rm stiff} \left(\simgrad +{\rm i}X_\chi \right)\widehat{\vect u} : \overline{\left(\simgrad +{\rm i}X_\chi \right)\widehat{\vect v}}\qquad \forall \begin{bmatrix}
    \vect u \\ \widehat{\vect u}
\end{bmatrix} ,\begin{bmatrix}
\vect v \\ \widehat{\vect v}
\end{bmatrix} \in \mathcal{D}( a_{\chi,\varepsilon}^{\rm app}).
\end{aligned}
\end{equation*}
\end{lemma}
The proof is obtained by a direct computation.
\begin{remark}
	\label{nakk1020}
For $(\vect u, \widehat{\vect u})^\top, (\vect v, \widehat{\vect v})^\top\in \mathcal{D}( a_{\chi,\varepsilon}^{\rm app})$, one has  
    \begin{equation*}
    \vect u = \mathring{\vect u} + \widehat{\Pi}_\chi^{\rm soft} \vect g_{\vect u}, \quad \vect v = \mathring{\vect v} + \widehat{\Pi}_\chi^{\rm soft} \vect g_{\vect v}, \quad \widehat{\vect u} = \widehat{\Pi}_\chi^{\rm stiff}\vect g_{\vect u}, \quad \widehat{\vect v} = \widehat{\Pi}_\chi^{\rm stiff} \vect g_{\vect v},
    \end{equation*}
where $ \mathring{ \vect u}, \mathring{ \vect v} \in \mathcal{D}(a_{0,\chi}^{\rm soft})$ and $  \vect g_{\vect u},\vect g_{\vect v} \in \widehat{\mathcal{E}}_\chi$.
    With this notation at hand, the form $a_{\chi,\varepsilon}^{\rm app}$ can be written as
    \begin{equation*}
    	\begin{aligned}
			    a_{\chi,\varepsilon}^{\rm app}\left( \begin{bmatrix}
    \vect u \\ \widehat{\vect u}
    \end{bmatrix} ,\begin{bmatrix}
    \vect v \\ \widehat{\vect v}
    \end{bmatrix}\right)&=\int_{Y_{\rm soft}} \A_{\rm soft} \left(\simgrad +{\rm i}X_\chi \right)\mathring{\vect u} : \overline{ \left(\simgrad +{\rm i}X_\chi \right)\mathring{\vect v}}\\[0.4em]
&\hspace{5ex}+ \int_{Y_{\rm stiff}} \A_{\rm stiff} \left(\simgrad + {\rm i}X_\chi \right)\widehat{\Pi}_\chi^{\rm soft} \vect g_{\vect u} : \overline{\left(\simgrad +{\rm i}X_\chi \right)\widehat{\Pi}_\chi^{\rm soft} \vect g_{\vect v}} \\[0.4em] 
&\hspace{10ex}+\frac{1}{\varepsilon^2}\int_{Y_{\rm stiff}} \A_{\rm stiff} \left(\simgrad +{\rm i}X_\chi \right)\widehat{\Pi}_\chi^{\rm stiff} \vect g_{\vect u} : \overline{\left(\simgrad +{\rm i}X_\chi \right)\widehat{\Pi}_\chi^{\rm stiff} \vect g_{\vect v}} \\[0.4em]
   &=a_{0,\chi}^{\rm soft}(\mathring{\vect u},\mathring{\vect v}) + \lambda_\chi^{\rm soft}(\vect g_{\vect u},\vect g_{\vect v}) + \frac{1}{\varepsilon^2} \lambda_\chi^{\rm stiff}(\vect g_{\vect u},\vect g_{\vect v}) \qquad \forall \begin{bmatrix}
    \vect u \\ \widehat{\vect u}
    \end{bmatrix} ,\begin{bmatrix}
    \vect v \\ \widehat{\vect v}
    \end{bmatrix} \in \mathcal{D}\bigl(a_{\chi,\varepsilon}^{\rm app}\bigr).
    	\end{aligned}
    \end{equation*}
\end{remark}

The following theorem contains the main result of this section. 

\begin{theorem}\label{thmpremain1} 
 There exists $C>0$, which depends only on $\sigma$ and ${\rm diam}(K_\sigma),$ such that for the resolvent of the transmission problem \eqref{transmissionboundaryproblem} one has
 \begin{equation*}
     \bigl\lVert \bigl((\mathcal{A}_{\chi,\varepsilon})_{0,I} -zI \bigr)^{-1} - \Theta_\chi\left(\mathcal{A}_{\chi,\varepsilon}^{\rm app} - zI \right)^{-1}\Theta_\chi \bigr\rVert_{\mathcal{H} \to \mathcal{H}} \leq C \varepsilon^2\qquad \forall\chi \in Y',
 \end{equation*}
 where the operator $\mathcal{A}_{\chi,\varepsilon}^{\rm app}$ is defined by 
  \eqref{homoperator}, 
 and
   $\Theta_\chi : \mathcal{H} = \mathcal{H}^{\rm soft} \oplus \mathcal{H}^{\rm stiff} \to   \mathcal{H}^{\rm soft} \oplus \widehat{\mathcal{H}}_\chi^{\rm stiff}$
 is an orthogonal projection defined by
     $\Theta_\chi\left( {\vect u}_{\rm soft} \oplus {\vect u}_{\rm stiff}\right) := {\vect u}_{\rm soft} \oplus  P_{\widehat{\mathcal{H}}_\chi^{\rm stiff}} \vect u_{\rm stiff},$
  with respect to the $L^2(Y;\C^3)$ inner product.
\end{theorem}
\begin{proof}
The proof consists in combining \eqref{pi_projections} and \eqref{resolventcondensedform} with Theorem \ref{theoremappoperator}, to infer that
 	$\mathcal{R}_{\chi,\varepsilon}^{\rm app}(z) = \Theta_\chi\left(\mathcal{A}_{\chi,\varepsilon}^{\rm app} - zI \right)^{-1}\Theta_\chi.$ The orthogonality of $\Theta_\chi$ is obvious.
\end{proof}
\begin{proof}[Proof of Theorem \ref{thmmain1}]  
The proof is a direct consequence of Theorem \ref{thmpremain1}, using the fact that the (scaled) Gelfand transform is an isometry and defining
$\Theta^{\rm app}_\eps:=\mathcal{G}_{\eps}^{-1} \Theta_{\chi} \mathcal{G}_{\eps},$ 
$\mathcal{A}^{\rm app}_{\eps}:=\mathcal{G}_{\eps}^{-1} \mathcal{A}_{\chi,\eps}^{\rm app} \mathcal{G}_{\eps}.$ 
\end{proof} 

\subsection{General outlook on the approach}
An alternative way to rewrite \eqref{resolventblockmatrix} is as follows:
 \begin{equation*}
     \mathcal{R}_{\chi,\varepsilon}^{\rm app}(z):=
     \begin{bmatrix}\mathcal{R}_{\chi,\varepsilon}^{\rm app, soft}(z) &  \Bigl(\mathcal{R}_{\chi,\varepsilon}^{\rm app, soft}(z) - \bigl(\mathcal{A}_{0,\chi}^{\rm soft} - zI \bigr)^{-1}\Bigr)\bigl( \widehat{\Pi}_\chi^{\rm stiff}\widehat{\Gamma}_{0,\chi}^{\rm soft}\bigr)^*\\[0.5em]
     \widehat{\Pi}_\chi^{\rm stiff}\widehat{\Gamma}_{0,\chi}^{\rm soft}\Bigl(\mathcal{R}_{\chi,\varepsilon}^{\rm app, soft}(z) - \bigl(\mathcal{A}_{0,\chi}^{\rm soft} - zI \bigr)^{-1}\Bigr) &\widehat{\Pi}_\chi^{\rm stiff}\widehat{\Gamma}_{0,\chi}^{\rm soft}\Bigl(\mathcal{R}_{\chi,\varepsilon}^{\rm app, soft}(z) - \bigl(\mathcal{A}_{0,\chi}^{\rm soft} - zI \bigr)^{-1}\Bigr)\bigl( \widehat{\Pi}_\chi^{\rm stiff}\widehat{\Gamma}_{0,\chi}^{\rm soft}\bigr)^*\end{bmatrix},
 \end{equation*}
relative to the decomposition $\mathcal{H}^{\rm soft} \oplus \widehat{\mathcal{H}}_\chi^{\rm stiff}$, where
\begin{equation*}
\begin{aligned}
        \mathcal{R}_{\chi,\varepsilon}^{\rm app, soft}(z) &:= \bigl( \mathcal{A}_{0,\chi}^{\rm soft} -z I \bigr)^{-1} - \widehat{S}^{\rm soft}_{\chi}(z){\widehat{Q}^{\rm app}_{\chi,\varepsilon}(z)}^{-1} \widehat{S}_{\chi}^{\rm soft}(\overline{z})^* \\[0.3em]
        &= \bigl(\mathcal{A}_{0,\chi}^{\rm soft} -z I  \bigr)^{-1} - \widehat{S}^{\rm soft}_{\chi}(z)\bigl(\widehat{M}_{\chi}^{\rm soft}(z)+\varepsilon^{-2} \widehat{\Lambda}_\chi^{\rm stiff} + z \bigl(\widehat{\Pi}_\chi^{\rm stiff}\bigr)^* \widehat{\Pi}_{\chi}^{\rm stiff} \bigr)^{-1}\widehat{S}_{\chi}^{\rm soft}(\overline{z})^*.
\end{aligned}
\end{equation*}
The operator-valued function $\mathcal{R}_{\chi,\varepsilon}^{\rm app, soft}(z)$ is the solution operator (in the sense of Theorem \ref{theoremkreinformula}) to the problem of finding $\vect u \in \mathcal{D}\left( \mathcal{A}_{\chi}^{\rm soft}\right)$ such that
\begin{equation*}
         \widehat{\mathcal{A}}_{\chi}^{\rm soft}\vect u - z \vect u = \vect f,\qquad\quad
          \beta_{0,\chi,\varepsilon}(z)\widehat{\Gamma}_{0,\chi}^{\rm soft}  \vect u +
          \beta_1\widehat{\Gamma}_{1,\chi}^{\rm soft}  \vect u  = 0,
\end{equation*}
where transmission operators are given by
  $\beta_{0,\chi,\varepsilon}(z)=\varepsilon^{-2}\widehat{\Lambda}_{\chi}^{\rm stiff} + z(\widehat{\Pi}_\chi^{\rm stiff})^*\widehat{\Pi}_\chi^{\rm stiff},$ $\beta_1 = I.$ 
Problems like this, namely those where the boundary conditions depend on the spectral parameter, are often referred to as impedance boundary value problems. One should note that the ``impedance" here is linear in the spectral parameter. 

On the other hand, one can consider the ``sandwiched resolvent''  
\begin{equation*}
    \mathcal{R}_{\chi,\varepsilon}^{\rm soft}(z):= P_{\rm soft}\bigl( (\mathcal{A}_{\chi,\varepsilon})_{0,I} -zI \bigr)^{-1} P_{\rm soft},
\end{equation*}
and use \eqref{transmissionkrein} to infer that
\begin{equation*}
    \begin{aligned}
        \mathcal{R}_{\chi,\varepsilon}^{\rm soft}(z) &= P_{\rm soft}\bigl(( \mathcal{A}_{0,\chi,\varepsilon} -z I)^{-1} - S_{\chi, \varepsilon}(z) M_{\chi, \varepsilon}(z)^{-1} S_{\chi, \varepsilon}(\overline{z})^* \bigr) P_{\rm soft} \\[0.3em]
        &= \bigl( \mathcal{A}_{0,\chi}^{\rm soft} -z I \bigr)^{-1} - S^{\rm soft}_{\chi}(z)\bigl(M_{\chi}^{\rm soft}(z) + \varepsilon^{-2}M_{\chi}^{\rm stiff}(\varepsilon^2 z) \bigr)^{-1}S_{\chi}^{\rm soft}(\overline{z})^*.
    \end{aligned}
\end{equation*}
Comparing this to \eqref{kreinformula2} yields the following proposition.
\begin{proposition}
The generalised resolvent $\mathcal{R}_{\chi,\varepsilon}^{\rm soft}(z)$ is the solution operator of the ``impedance'' boundary value problem on $\mathcal{H}^{\rm soft}$ that consists in finding $\vect u \in \mathcal{D}\left( \mathcal{A}_{\chi}^{\rm soft}\right)$ such that
\begin{equation}
\label{bvp2}
         \mathcal{A}_{\chi}^{\rm soft}\vect u - z \vect u = \vect f,\qquad\quad
          \widetilde{\beta}_{0,\chi,\varepsilon}(z)\Gamma_{0,\chi}^{\rm soft}  \vect u +
          \widetilde{\beta}_1 \Gamma_{1,\chi}^{\rm soft}  \vect u  = 0,
\end{equation}
where the transmission operators are given by
    $\widetilde{\beta}_{0,\chi,\varepsilon}(z):=\varepsilon^{-2}M_\chi^{\rm stiff}(\varepsilon^2 z),$ $\widetilde{\beta}_1 = I.$
\end{proposition}
The ``impedance'' of the boundary value problem \eqref{bvp2} is highly nonlinear, due to the structure of the $M$-function $M_\chi^{\rm stiff}(z)$.
On the abstract level, both solution operators $\mathcal{R}^{\rm app, soft}_{\chi,\e}$ and $\mathcal{R}^{\rm soft}_{\chi,\e}$ are \emph{generalised resolvents} \cite{Neumark1940,Neumark1943,Strauss,Strauss_ext,Strauss_survey}. A generalised resolvent can be equivalently characterised as either an operator of the form $P(\mathcal{A}-zI)^{-1}|_{P}$ for a self-adjoint $\mathcal{A}$ in a Hilbert space $\mathcal{H}$ and an orthogonal projection $P$, or a solution operator of an abstract spectral boundary value problem
\begin{equation}\label{eq:gen_res}
\mathcal{A} \vect u = z \vect u,\quad \Gamma_1 \vect u = \mathcal{B}(z) \Gamma_0 \vect u,
\end{equation}
where $\mathcal{A}$ is a densely defined linear operator on $P \mathcal{H}$ and $(\mathcal{E},\Gamma_0,\Gamma_1)$ is an abstract boundary triple of $\mathcal{A}$, while $-\mathcal{B}(z)$ is an analytic in the upper half-plane operator-valued function with positive imaginary part (i.e., an operator $R$-function) on $\mathcal{E},$ extended into the region $\Im z<0$ by the identity $\mathcal{B} (z)=\mathcal{B}^* (\bar z)$.

The system $\eqref{eq:gen_res}$ can be thus re-cast in the form of the operator equation
$
\mathcal{A}_z u = zu,
$
where $\mathcal{A}_z$ is a closed densely defined linear operator on $P \mathcal{H}$ with domain
$
\mathcal D (\mathcal{A}_z)=\{\vect u\in \mathcal D(\mathcal{A})\subset P \mathcal{H}: \Gamma_1 \vect u = \mathcal{B}(z) \Gamma_0 \vect u\}.
$
The operator $\mathcal{A}_z$ is shown to be maximal dissipative for $z\in\mathbb{C}_-$ and maximal antidissipative for $z\in\mathbb{C}_+$. 

From the point of view of generalised resolvents, one can therefore view the homogenisation procedure we have performed above as obtaining the main order term in the asymptotic expansion of the generalised resolvent $\mathcal{R}^{\rm soft}_{\chi,\e}$ for every fixed $\chi\in Y'$ as $\e\to 0$.

Moreover, we point out that in order to determine the main order term of $((\mathcal A_{\chi,\e})_{0,I}-zI)^{-1}$ as $\e\to 0$, or in other words to recover the operator describing the homogenised medium, it is in fact necessary and sufficient to construct an asymptotic expansion of the generalised resolvent described above. This follows from the fact that under a natural and non-restrictive ``minimality'' condition the operator $\mathcal{A}$ giving rise to the generalised resolvent $P(\mathcal{A}-z I)^{-1}|_{P \mathcal{H}}$ is in fact uniquely determined based on the latter up to a unitary gauge transform $\Phi$ such that $\Phi|_{P \mathcal{H}}=I_{P\mathcal{H}}$.

This can be viewed in the homogenisation problem at hand as taking the ``down, right, up'' detour in the commutative diagram
\begin{equation}\label{diagram}
\begin{CD}
\bigl((\mathcal A_{\chi,\e})_{0,I}-zI\bigr)^{-1}@>\text{(1)}>> (\mathcal A^{\text{app}}_{\chi,\e}-zI)^{-1}\\
@V\text{(2)}VV @|\\
\mathcal{R}^{\rm soft}_{\chi,\e}@>\text{(3)}>>\mathcal{A}^{\text{app}}_{\chi,\e}=
(\mathcal{A}^{\text{app}}_{\chi,\e})^*\text{ in } \mathfrak H\supset \mathcal H^{\rm soft}:\\
@. \mathcal{R}^{\rm app, soft}_{\chi,\e}=
P|_{\mathcal H^{\rm soft}} (\mathcal{A}^{\text{app}}_{\chi,\e}-zI)^{-1}|_{\mathcal H^{\rm soft}},
\end{CD}
\end{equation}
where the double solid line represents the unitary gauge.

As far as the asymptotic analysis of the generalised resolvent $\mathcal{R}^{\rm soft}_{\chi,\e}$ is concerned, the required analysis is essentially reduced to the derivation of the asymptotics of the operator $\widetilde{\beta}_{0,\chi,\e}(z)$ which governs its impedance boundary conditions. This, due to \eqref{Qeps_ref}, in turn reduces to a well-understood problem of perturbation theory for the DtN map pertaining to the stiff component of the medium and thus presents no complications.

Having said that, we point out that however appealing this argument appears, it meets two significant difficulties. Firstly, at present we don't have an explicit way to construct the operator $\mathcal{A}^{\text{app}}_{\chi,\e}$ in \eqref{diagram} for arbitrary impedance boundary conditions parameterised by a generic $(-R)$-operator function $\mathcal{B}(z)$ in \eqref{eq:gen_res}. In the problem at hand, this presents no challenge as the main order term \eqref{blockresolvent2} of the boundary operator is in fact linear in $z$. Generalised resolvents of this form have already appeared in problems of dimension reduction, most notably in the works concerned with the convergence of PDEs defined on ``thin'' networks to ODEs on limiting metric graphs, see, e.g., \cite{Post,KuchmentZeng,KuchmentZeng2004,Exner}, and in particular our recent paper \cite{CEK_networks} where an approach akin to the one utilised in the present work is extended to the context of thin networks. In the area of linear elasticity in particular this analysis is thought to be applicable to the analysis of pentamodes \cite{pentamodes}, which will be further discussed elsewhere.

Secondly and crucially, once the asymptotics of the family of generalised resolvents is obtained in some desired strong topology, the same type of convergence for the family of resolvents $((\mathcal A_{\chi,\e})_{0,I}-zI)^{-1}$ cannot be inferred from the general operator theory. In fact, one can argue that norm-resolvent convergence of $\mathcal{R}^{\rm soft}_{\chi,\e}$ only yields \emph{strong} convergence of $((\mathcal A_{\chi,\e})_{0,I}-zI)^{-1}$. It is here that the specifics of the problem at hand must play a crucial r\^{o}le in the analysis, leading to a result of the type formulated in Theorem \ref{thmrefinement} above.

Despite the deficiencies of the general operator-theoretic outlook based on generalised resolvents explained above, we point out that this way of considering the dimension reduction problem at hand is very natural in that it presents one with a physically motivated understanding of the problem.

As the argument of \cite{Figotin_Schenker_2005,Figotin_Schenker_2007b}, see also references therein, demonstrates, generalised resolvents appear naturally in physical setups where one forcefully removes certain degrees of freedom from consideration in an otherwise conservative setting in view of simplifying the latter. Conversely, the procedure of reconstructing the self-adjoint generator of conservative dynamics must be viewed as adding those ``hidden'', or concealed, degrees of freedom back in a proper way. In doing so, one frequently faces a situation (and in particular, in the setup of linear elasticity discussed in the present paper) where the resulting model is drastically simplified owing to only a certain limited number of concealed degrees of freedom appearing in it in a handily transparent way. The procedure of the diagram \eqref{diagram} can be therefore seen as a non-trivial generalisation of the seminal idea of Lax and Phillips \cite{Lax}, with a dissipative generator expressing the scattering properties of the system being replaced by a more general one, corresponding to an $R$-function which non-trivially depends on the spectral parameter $z$.

At the same time, as explained in \cite{CEKRS2022} (see also references therein), the concept of \emph{dilating} (in the sense of \eqref{eq:gen_res}) a generalised resolvent to a resolvent of a self-adjoint generator gives rise to the understanding of homogenisation limits in the setup of double porosity models as essentially operators on soft component of the media with singular surface potentials possessing internal structure, see also \cite{CK} where a similar argument applied to high-contrast ODEs has led to a Kronig-Penney-type model.  In the problem considered in the present paper, the mentioned singular surface can be shown to be the periodic lattice of the original composite.

Moreover, the argument of \cite{CEN} can be immediately invoked for the homogenised family \eqref{homoperator} to obtain its functional model in an explicit form in certain explicitly constructed Hilbert space of complex-analytic functions, giving rise to a Clark-Alexandrov measure serving as the spectral measure of the family. This latter program will be pursued elsewhere, together with the study of effective scattering problems of the high-contrast composite which can be considered naturally on this basis.

\section{Transmission problem: $O(\varepsilon)$ resolvent asymptotics } \label{sectionopeff} 
The goal of this section is to further approximate the resolvent related to the transmission problem and prove Theorem \ref{thmamin2}, (a).
In doing so, we will worsen the order in $\varepsilon$ of the estimate but will obtain more familiar objects in the asymptotics. 

Our aim is to provide a further approximation to the operator 
		$\bigl[\widehat{S}^{\rm soft}_{\chi}(z)\ 
		\widehat{\Pi}_\chi^{\rm stiff}\bigr]
	{\widehat{Q}^{\rm app}_{\chi,\varepsilon}(z)}^{-1} 
		\bigl[\widehat{S}^{\rm soft}_{\chi}(z)\ \  
		\widehat{\Pi}_\chi^{\rm stiff}\bigr]^*$
entering the resolvent \eqref{resolventcondensedform} so the associated error is not worse than $O(\varepsilon)$. 
For this, the following estimate on the inverse of the transmission function is crucial.
\begin{lemma}
	\label{lemmaestimateforq}
	There exists $C>0$ which does not depend on $\varepsilon>0$, $z \in K_\sigma$, $\chi \in Y'$, such that
	\begin{equation}
		\label{estimateq}
		\bigl\lVert{\widehat{Q}^{\rm app}_{\chi,\varepsilon}(z)}^{-1} \widehat{P}_\chi \bigr\rVert_{\mathcal{E}\to \mathcal{E}} \leq C \min \left\{\vert\chi\vert^{-2}\varepsilon^2,1 \right\}.
	\end{equation}
\end{lemma}
\begin{proof}
	First, note that
	\begin{align*}
		\Im\bigl(\widehat{Q}^{\rm app}_{\chi,\varepsilon}(z) \bigr)
		= \Im z \bigl( \widehat{S}^{\rm soft}_{\chi}(\overline{z})\bigr)^* \widehat{S}^{\rm soft}_{\chi}(\overline{z})+\Im z\bigl(\widehat{\Pi}_\chi^{\rm stiff} \bigr)^*  \widehat{\Pi}_\chi^{\rm stiff},\quad 
		\Re\bigl(\widehat{Q}^{\rm app}_{\chi,\varepsilon}(z) \bigr) 
		=\varepsilon^{-2}\widehat{\Lambda}_{\chi}^{\rm stiff}  + \Re\widehat{M}_{\chi}^{\rm soft}(z) + \Re z \bigl(\widehat{\Pi}_\chi^{\rm stiff} \bigr)^*  \widehat{\Pi}_\chi^{\rm stiff}.
	\end{align*}
	For $\vect u \in \widehat{\mathcal{E}}_\chi$  using Lemma \ref{lemmanormequivalence_f} and the trace inequality, we write
	\begin{equation*}
		\begin{aligned}
			\bigl\lvert\bigl\langle  \Im\bigl( \widehat{Q}^{\rm app}_{\chi,\varepsilon}(z)\bigr)\vect u , \vect u\bigr\rangle_{\mathcal{E}} \bigr\rvert & = |\Im z|\Bigl( \bigl\langle \widehat{\Pi}_\chi^{\rm stiff}\vect u, \widehat{\Pi}_\chi^{\rm stiff} \vect u\bigr\rangle_{\mathcal{H}^{\rm stiff}} + \bigr\langle \widehat{S}^{\rm soft}_{\chi}(\overline{z}) \vect u, \widehat{S}^{\rm soft}_{\chi}(\overline{z}) \vect u\bigr\rangle_{\mathcal{H}^{\rm soft}} \Bigr) \\[0.3em]
			& \geq |\Im z| \bigl\lVert  \widehat{\Pi}_\chi^{\rm stiff} \vect u\bigr\rVert_{\mathcal{H}^{\rm stiff}}^2 \geq C |\Im z| \bigl\lVert  \widehat{\Pi}_\chi^{\rm stiff} \vect u\bigr\rVert^2_{H^1(Y_{\rm stiff}; \C^3)} 
			\geq C |\Im z|\left\lVert   \vect u\right\rVert_{\mathcal{E}}^2,
		\end{aligned}
	\end{equation*}
	where $C>0$ depends only on $K_\sigma$. Thus, due to Corollary \ref{boundfrombelowappendix}, one has 
	\begin{equation}\label{nakk1001} 
		\bigl\lVert {\widehat{Q}^{\rm app}_{\chi,\varepsilon}(z)}^{-1} \widehat{P}_\chi\bigr\rVert_{\mathcal{E} \to\mathcal{E}} \leq C,
	\end{equation}
	where $C>0$ is independent of $\chi$ and $z$. Furthermore, by Corollary \ref{msoftasymptotics} and Remark \ref{nakk1000}, we infer the existence of $\widetilde{C}>0$, which depends on $|z|$ and $\sigma,$ such that
	\begin{equation*}
		\left\lVert \widehat{M}_{\chi}^{\rm soft}(z) \right\rVert_{L^2(\Gamma;\C^3) \to L^2(\Gamma;\C^3)} \leq \widetilde{C}.
	\end{equation*}
	Using Lemma \ref{lemmasteklovorder}  and the fact that $\widehat{\Pi}_\chi^{\rm stiff}$ is uniformly bounded, we infer the existence of constants $D, C_2$,  independent of $\chi$, $\varepsilon$, such that for  $|\chi| \geq D \varepsilon$, one has
	\begin{equation*}
			\bigl\lvert \bigl\langle \Re\bigl( \widehat{Q}^{\rm app}_{\chi,\varepsilon}(z) \bigr) \vect u , \vect u \bigr\rangle_{\mathcal{E}} \bigr\rvert  \geq    C_2 \varepsilon^{-2}|\chi|^2\left \lVert  \vect u \right\rVert_{\mathcal{E}}^2 \qquad \forall \vect u \in \widehat{\mathcal{E}}_\chi.
	\end{equation*}
	For such $|\chi|$, by applying Corollary \ref{boundfrombelowappendix}, we obtain
	\begin{equation*}
		\bigl\lVert {\widehat{Q}^{\rm app}_{\chi,\varepsilon}(z)}^{-1} \widehat{P}_\chi\bigr\rVert_{L^2(Y_{\rm stiff};\C^3)\to L^2(Y_{\rm stiff};\C^3)} \leq C_2|\chi|^{-2}\varepsilon^2,
	\end{equation*}
	which, combined with \eqref{nakk1001},  concludes the proof.
\end{proof}

Next we introduce the version of the transmission function that will appear in the final homogenisation result. The above two lemmata allow us to replace the subscript $\chi$ by zero everywhere except  $\widehat{\Lambda}^{\rm hom}_\chi,$ which leads to a more transparent result involving a differential, rather than a pseudodifferential, form of the operator asymptotics.
\begin{definition}
	We refer to the operator valued function $ \widehat{Q}_{\varepsilon,\chi}^{\rm eff}(z): \widehat{\mathcal{E}}_0 \to \widehat{\mathcal{E}}_0$ given by 
	\begin{equation}
		\label{qoperatoreff}
		\widehat{Q}_{\varepsilon,\chi}^{\rm eff}(z):= \varepsilon^{-2} \widehat{\Lambda}_\chi^{\rm hom}+ z \bigl(\widehat{\Pi}_0^{\rm stiff}\bigr)^* \widehat{\Pi}_0^{\rm stiff} + \widehat{M}_{0}^{\rm soft}(z)
	\end{equation}
	as the \emph{effective transmission function}. We introduce the following associated operator-valued function on $\mathcal{H}:$
	\begin{equation*}
		\mathcal{R}_{\chi,\varepsilon}^{\rm eff}(z):= \bigl( \mathcal{A}_{0,\chi}^{\rm soft} -z I  \bigr)^{-1} P_{\mathcal{H}^{\rm soft}} - \begin{bmatrix}
			\widehat{S}^{\rm soft, eff}_{\chi}(z) & \widehat{\Pi}_0^{\rm stiff}
		\end{bmatrix} {\widehat{Q}^{\rm eff}_{\chi,\varepsilon}(z)}^{-1} \begin{bmatrix}
			\widehat{S}^{\rm soft, eff}_{\chi}(z) & \widehat{\Pi}_0^{\rm stiff}
		\end{bmatrix}^*,
	\end{equation*}
	where the \emph{effective solution operator} $\widehat{S}^{\rm soft, eff}_{\chi}(z):\widehat{\mathcal{E}}_0 \to \mathcal{H}_{\rm soft}$ is defined by 
	\begin{equation*}
		\widehat{S}^{\rm soft, eff}_{\chi}(z) := \widehat{\Pi}_0^{\rm soft} + z\bigl( \mathcal{A}_{0,\chi}^{\rm soft} -z I  \bigr)^{-1}\widehat{\Pi}_0^{\rm soft}.
	\end{equation*}
\end{definition}
\begin{remark}
\label{remark_estimates_igorpetak}
	Due to the estimate $\eqref{pisoftasimp}$ and the boundedness of the resolvent $(\mathcal{A}_{0,\chi}^{\rm soft} -z I)^{-1},$ we have 
	\begin{equation*}
		\left\lVert \widehat{S}^{\rm soft}_{\chi}(z) \widehat{P}_\chi - \widehat{S}^{\rm soft, eff}_{\chi}(z) \widehat{P}_0 \right\rVert_{L^2(\Gamma;\C^3) \to L^2(Y_{\rm stiff};\C^3)} \leq C|\chi|,
	\end{equation*}
	where the constant $C>0$ is independent of $\chi$ and $z$. Using the estimate \eqref{pistarasymp} and the identity \eqref{pi_projections} yields
	\begin{equation*}
		\left\lVert \bigl(\widehat{\Pi}_\chi^{\rm stiff}\bigr)^*  -  \bigl(\widehat{\Pi}_0^{\rm stiff}\bigr)^*   \right\rVert_{L^2(Y_{\rm stiff};\C^3) \to L^2(\Gamma;\C^3)} \leq C|\chi|.
	\end{equation*}
\end{remark}
For the inverse of the effective transmission function we have an estimate similar to \eqref{estimateq}. 
\begin{lemma}\label{nakk1003} 
	There exists $C>0$ which does not depend on $\varepsilon>0$, $z \in K_\sigma$, $\chi \in Y'$, such that 
	\begin{equation*}
		\bigl\lVert{\widehat{Q}_{\varepsilon,\chi}^{\rm eff}(z)}^{-1} \widehat{P}_0 \bigr\rVert_{\mathcal{E}\to \mathcal{E}} \leq C \min \left\{|\chi|^{-2}\varepsilon^2,1 \right\}.
	\end{equation*}
\end{lemma}
\begin{proof}
	The proof follows the steps of the proof of Lemma \ref{lemmaestimateforq} while utilising Lemma \ref{josipnak1}. 
\end{proof}

The following lemma provides an estimate on the distance between the two transmission functions.
\begin{lemma}\label{nakk1004} 
	There exists a constant $C>0,$ independent of $\varepsilon>0$, $z \in K_\sigma$, $\chi \in Y',$ such that 
	\begin{equation*}
		\left\lVert\widehat{Q}^{\rm app}_{\chi,\varepsilon}(z)\widehat{P}_\chi - \widehat{Q}_{\varepsilon,\chi}^{\rm eff}(z) \widehat{P}_0 \right\rVert_{\mathcal{E}\to \mathcal{E}} \leq C \max \left\{\varepsilon^{-2}|\chi|^3, |\chi|\right\}.
	\end{equation*}
\end{lemma}
\begin{proof}
	The case of $\chi = 0$ is trivial. It is clear that for all $\chi \in Y' \setminus\{0\}$ we have
	\begin{equation*}
		\frac{1}{|\chi|^2}\widehat{\Lambda}_\chi^{\rm stiff}\widehat{P}_\chi -\frac{1}{|\chi|^2}\Lambda_\chi^{\rm hom}\widehat{P}_0 = \frac{1}{2\pi{\rm i}} \oint_\gamma  z\left(\left(zI-\frac{1}{|\chi|^2}\Lambda_\chi^{\rm stiff}\right)^{-1}- \left(zI-\frac{1}{|\chi|^2}\Lambda_\chi^{\rm hom}\right)^{-1} \right) dz,
	\end{equation*}
	where $\gamma$ is the contour provided by Lemma \ref{lemmacontour}.
	Therefore, by applying the Theorem \ref{dirichletotneumannresolventasymptotics} (cf. Remark \ref{nakk120}), we obtain
	\begin{equation*}
		\bigl\lVert\varepsilon^{-2} \widehat{\Lambda}_\chi^{\rm stiff}\widehat{P}_\chi -\varepsilon^{-2} \Lambda_\chi^{\rm hom} \widehat{P}_0\bigr\rVert_{\mathcal{E} - \mathcal{E}} \leq C \varepsilon^{-2}|\chi|^3.
	\end{equation*}
	The claim now follows from \eqref{qoperator} and \eqref{qoperatoreff}, by invoking Corollary \ref{msoftasymptotics} and Corollary  \ref{pipasymptotics_stiff}.
\end{proof}

The following lemma is crucial for  obtaining $\varepsilon$-order asymptotics of the resolvent $((\mathcal{A}_{\chi,\varepsilon})_{0,I} -zI)^{-1}$ and relating it to an object that incorporates the effective transmission function.
\begin{lemma}
	\label{lemmaqasymptotics}
	There exists a constant $C>0$ which does not depend on $\varepsilon>0$, $z \in K_\sigma$, $\chi \in Y'$, such that
	\begin{equation*}
		\bigl\lVert {\widehat{Q}^{\rm app}_{\chi,\varepsilon}(z)}^{-1}\widehat{P}_\chi - {\widehat{Q}^{\rm eff}_{\chi,\varepsilon}(z)}^{-1} \widehat{P}_0 \bigr\rVert_{\mathcal{E}\to \mathcal{E}} \leq C \varepsilon.
	\end{equation*}
\end{lemma}
\begin{proof}
	By a direct calculation, we see that
		\begin{equation*}
		{\widehat{Q}^{\rm app}_{\chi,\varepsilon}\BBB(z)}^{-1}\widehat{P}_\chi - {\widehat{Q}^{\rm eff}_{\chi,\varepsilon}(z)}^{-1} \widehat{P}_0 = \rm I + II + III,
	\end{equation*}
	where
	\begin{equation*}
		\begin{aligned}
			& \text{I} :=  \bigl(\widehat{Q}^{\rm app}_{\chi,\varepsilon}(z)\bigr)^{-1}\widehat{P}_\chi\bigl(\widehat{Q}_{\varepsilon,\chi}^{\rm eff}(z) \widehat{P}_0-\widehat{Q}^{\rm app}_{\chi,\varepsilon}(z)\widehat{P}_\chi\bigr){\widehat{Q}^{\rm eff}_{\chi,\varepsilon}(z)}^{-1} \widehat{P}_0, \\[0.3em]
			& \text{II} := \bigl( \widehat{Q}^{\rm app}_{\chi,\varepsilon}(z)\bigr)^{-1}\widehat{P}_\chi \bigl(\widehat{P}_\chi - \widehat{P}_0 \bigr),\qquad 
			\text{III} := \bigl( \widehat{P}_0-\widehat{P}_\chi  \bigr){\widehat{Q}^{\rm eff}_{\chi,\varepsilon}(z)}^{-1} \widehat{P}_0. 
		\end{aligned}
	\end{equation*}
	 Next, using  Lemma \ref{lemmaestimateforq}, Lemma \ref{nakk1003} and Lemma \ref{nakk1004}, we obtain
	\begin{equation*}
			\left\lVert \text{I} \right\rVert_{\mathcal{E} \to \mathcal{E}}
			\leq C \min \left\{ \frac{\varepsilon^2}{|\chi|^2},1 \right\}\max\left\{|\chi|, \frac{|\chi|^3}{\varepsilon^2} \right\}\min \left\{ \frac{\varepsilon^2}{|\chi|^2},1 \right\}
			\leq \left\{\begin{array}{ll}
				1 \cdot \varepsilon \cdot 1 \leq \varepsilon  \textrm{ if } |\chi| \leq \varepsilon,\\[0.5em]
				\dfrac{\varepsilon^2}{|\chi|^2}\cdot \dfrac{|\chi|^3}{\varepsilon^2} \cdot  \dfrac{\varepsilon^2}{|\chi|^2} = \dfrac{\varepsilon^2}{|\chi|}    \leq \varepsilon & \textrm{ if }  |\chi| \geq \varepsilon.
			\end{array}\right.
	\end{equation*}
	Furthermore, by employing Corollary \ref{pipasymptotics_stiff} and Lemma \ref{lemmaestimateforq}, one easily estimates
	\begin{equation}
		\label{nakk1005} 
			\left\lVert \text{II} \right\rVert_{\mathcal{E} \to \mathcal{E}}
			\leq C |\chi |\min \left\{ \frac{\varepsilon^2}{|\chi|^2},1 \right\} 
			\leq \left\{\begin{array}{ll}
				|\chi| \cdot 1  \leq \varepsilon & \textrm{ if }  |\chi| \leq \varepsilon, \\[0.4em]				
				|\chi| \cdot \dfrac{\varepsilon^2}{|\chi|^2} = \dfrac{\varepsilon^2}{|\chi|}    \leq \varepsilon & \textrm{ if }   |\chi| \geq \varepsilon.
			\end{array}\right.
	\end{equation}
	Similarly, using again Corollary \ref{pipasymptotics_stiff} and Lemma \ref{nakk1003} we estimate
	\begin{equation*}
			\left\lVert \text{III} \right\rVert_{\mathcal{E} \to \mathcal{E}}
			\leq C |\chi| \min \left\{ \vert\chi\vert^{-2}\varepsilon^2,1 \right\}  
			\leq C \varepsilon,
	\end{equation*}
	which concludes the proof. 
\end{proof}

Finally, we can summarise these results as the following theorem. 
\begin{theorem} \label{nakk1} 
	There exists a constant $C>0$ which does not depend on $\varepsilon>0$, $z \in K_\sigma$, $\chi \in Y'$, such that 
	\begin{equation*}
		\left\lVert \mathcal{R}_{\chi,\varepsilon}^{\rm app}(z) - \mathcal{R}_{\chi,\varepsilon}^{\rm eff}(z) \right\rVert_{\mathcal{H} \to \mathcal{H}} \leq C \varepsilon,
	\end{equation*}
	where
	\begin{equation}
		\label{effresolventblockmatrix}
		\mathcal{R}_{\chi,\varepsilon}^{\rm eff}(z):=
		\begin{bmatrix}\bigl( \mathcal{A}_{0,\chi}^{\rm soft} -z I  \bigr)^{-1} - \widehat{S}^{\rm soft, eff}_{\chi}(z)
			{\widehat{Q}_{\varepsilon,\chi}^{\rm eff}(z)}^{-1} \widehat{S}_{\chi}^{\rm soft, eff}(\overline{z})^* & -\widehat{S}^{\rm soft, eff}_{\chi}( z)
			{\widehat{Q}_{\varepsilon,\chi}^{\rm eff}(z)}^{-1} \bigl(\widehat{\Pi}_0^{\rm stiff}\bigr)^*\\[0.4em]
			-\widehat{\Pi}_0^{\rm stiff}
			{\widehat{Q}_{\varepsilon,\chi}^{\rm eff}(z)}^{-1} \widehat{S}_{\chi}^{\rm soft, eff}(\overline{z})^* & - \widehat{\Pi}_0^{\rm stiff} 
			{\widehat{Q}_{\varepsilon,\chi}^{\rm eff}(z)}^{-1} \bigl(\widehat{\Pi}_0^{\rm stiff}\bigr)^*\end{bmatrix},
	\end{equation}
	and the block decomposition is relative to the decomposition $\mathcal{H}^{\rm soft} \oplus \mathcal{H}^{\rm stiff}$.
\end{theorem}
\begin{proof}
	The proof of this fact consists of estimating the four blocks of the matrix 
	\begin{equation*}
		\begin{bmatrix}
			\widehat{S}^{\rm soft}_{\chi}(z) & \widehat{\Pi}_\chi^{\rm stiff}
		\end{bmatrix} {\widehat{Q}^{\rm app}_{\chi,\varepsilon}(z)}^{-1} \begin{bmatrix}
			\widehat{S}^{\rm soft}_{\chi}(z) & \widehat{\Pi}_\chi^{\rm stiff}
		\end{bmatrix}^* - \begin{bmatrix}
			\widehat{S}^{\rm soft, eff}_{\chi}(z) & \widehat{\Pi}_0^{\rm stiff}
		\end{bmatrix}{\widehat{Q}^{\rm eff}_{\chi,\varepsilon}(z)}^{-1} \begin{bmatrix}
			\widehat{S}^{\rm soft,eff}_{\chi}(z) & \widehat{\Pi}_0^{\rm stiff}
		\end{bmatrix}^* 
	\end{equation*}
	for which all the arguments go analogously. Thus, we prove the estimate only for one of the blocks. Combining the triangle inequality, Corollary \ref{pipasymptotics_soft},  identity \eqref{Sadjointformulae}, Lemma \ref{lemmaestimateforq}, and Remark \ref{remark_estimates_igorpetak},  we obtain, similarly to \eqref{nakk1005}: 
	\begin{equation*}
			\widehat{\Pi}_\chi^{\rm stiff}  {\widehat{Q}^{\rm app}_{\chi,\varepsilon}(z)}^{-1} \bigl( \widehat{\Pi}_\chi^{\rm stiff}\bigr)^* - \widehat{\Pi}_0^{\rm stiff}  {\widehat{Q}^{\rm eff}_{\chi,\varepsilon}(z)}^{-1}\bigl( \widehat{\Pi}_0^{\rm stiff}\bigr)^* 
			= \widehat{\Pi}_0^{\rm stiff}  \Bigl(\bigl(\widehat{Q}^{\rm app}_{\chi,\varepsilon}(z) \bigr)^{-1} \widehat{P}_\chi-{\widehat{Q}^{\rm eff}_{\chi,\varepsilon}(z)}^{-1} \widehat{P}_0 \Bigr) \bigl( \widehat{\Pi}_0^{\rm stiff}\bigr)^* + \mathcal{O}(\varepsilon),
	\end{equation*}
	where $O(\varepsilon)$ is of order $\varepsilon$ with respect to the  $L^2 \to L^2$ norm.	Using Lemma \ref{lemmaqasymptotics} concludes the estimate for this block. 
\end{proof}
\begin{definition} We define the effective operator $\mathcal{A}_{\chi,\varepsilon}^{\rm eff}$ as follows:
	\begin{equation}
	\begin{aligned}
		\mathcal{D}\left(\mathcal{A}_{\chi,\varepsilon}^{\rm eff}\right)&:=\bigl\{ (\vect u,\widehat{\vect u})^\top \in \mathcal{H}^{\rm soft} \oplus \widehat{\mathcal{H}}_0^{\rm stiff}, \quad \vect u \in \mathcal{D}(\mathcal{A}_{0,\chi}^{\rm soft}) \dot +  \widehat{\Pi}_0^{\rm soft} \widehat{\mathcal{E}}_0, \quad  \widehat{\Gamma}_{0,0}^{\rm stiff}\widehat{\vect u} = \widehat{\Gamma}_{0,0}^{\rm soft} \vect u  \bigr\},
		\\[0.4em]
		\mathcal{A}_{\chi,\varepsilon}^{\rm eff} \begin{bmatrix}
			\vect u \\ \widehat{\vect u}
		\end{bmatrix}&:= \begin{bmatrix}
			\mathring{\mathcal{A}}_\chi^{\rm soft} & 0 \\[0.2em]
			- \bigl( \bigl(\widehat{\Pi}_0^{\rm stiff} \bigr)^* \bigr)^{-1} \mathring{\Gamma}_{1,0}^{\rm soft} & -\varepsilon^{-2}\bigl( \bigl(\widehat{\Pi}_0^{\rm stiff} \bigr)^* \bigr)^{-1} \Lambda_{\chi}^{\rm hom}\bigl(\widehat{\Pi}_0^{\rm stiff} \bigr)^{-1}
		\end{bmatrix}\begin{bmatrix}
			\vect u \\ \widehat{\vect u}
		\end{bmatrix},
	\end{aligned}
\label{effoperator}
\end{equation}
	where 
	\begin{equation*}
		\begin{aligned}
			\mathcal{D}(\mathring{\mathcal{A}}_\chi^{\rm soft})&:=\mathcal{D}(\mathcal{A}_{0,\chi}^{\rm soft}) \dot + \widehat{\Pi}_0^{\rm soft}(\widehat{\mathcal{E}}_0), \quad \mathring{\mathcal{A}}_\chi^{\rm soft}:\bigl(\mathcal{A}_{0,\chi}^{\rm soft}\bigr)^{-1} \vect f + \widehat{\Pi}_0^{\rm soft} \vect g \to \vect f, \quad \vect f \in \mathcal{H}, \vect g \in \widehat{\mathcal{E}}_0, \\[0.3em]
			\mathcal{D}(\mathring{\Gamma}_{1,0}^{\rm soft})&:=\mathcal{D}(\mathcal{A}_{0,\chi}^{\rm soft})  \dot +\widehat{\Pi}_0^{\rm soft}(\widehat{\mathcal{E}}_0),  \quad \mathring{\Gamma}_{1,0}^{\rm soft}:\bigl(\mathcal{A}_{0,\chi}^{\rm soft}\bigr)^{-1} \vect f + \widehat{\Pi}_0^{\rm soft} \vect g \to \bigl( \widehat{\Pi}_0^{\rm soft}\bigr)^* \vect f + \widehat{\Lambda}_0^{\rm soft} \vect g, \quad \vect f \in \mathcal{H}, \vect g \in \widehat{\mathcal{E}}_0.
		\end{aligned}
	\end{equation*}
\end{definition}

\begin{remark} \label{ante31} 
	It is straightforward to check that $(\widehat{\Pi}_0^{\rm stiff(soft)})^{*}=|Y_{\rm stiff(soft)}||\Gamma|^{-1}\widehat{\Gamma}_{0,\chi}^{\rm stiff(soft)}P_{\mathcal{\widehat{H}}^{\rm stiff(soft)}_{0}}$. 
 	\end{remark} 

 Recall that $\widehat{\Lambda}_0^{\rm soft} \vect g=0$ for every $\vect g\in \widehat{\mathcal{E}}_0,$ as $\widehat{\mathcal{E}}_0$ consists of constant functions. Similarly to Theorem \ref{theoremappoperator}, one can establish the following statement, whose proof we omit. 
\begin{theorem}
	For every $\chi \in Y'$, the operator $\mathcal{A}_{\chi,\varepsilon}^{\rm eff}$ is self-adjoint and its resolvent is given, for all $z \in \rho(\mathcal{A}_{\chi,\varepsilon}^{\rm eff}),$ by the formula \eqref{effresolventblockmatrix} relative to the decomposition $\mathcal{H}^{\rm soft} \oplus \widehat{\mathcal{H}}_0^{\rm stiff}.$
\end{theorem}

\begin{remark}
	\label{latter}
	The sesquilinear form $a_{\chi,\varepsilon}^{\rm eff}$ on $\mathcal{H}\times \mathcal{H}$ associated with the operator \eqref{effoperator} is given by
	\begin{equation*}
	\begin{aligned}
		&\mathcal{D}\left(a_{\chi,\varepsilon}^{\rm eff}\right):=\bigl\{ (\vect u,\widehat{\vect u})^\top \in \mathcal{H}^{\rm soft} \oplus \widehat{\mathcal{H}}_0^{\rm stiff}, \quad \vect u \in \mathcal{D}(a_{0,\chi}^{\rm soft})\dot +  \widehat{\Pi}_0^{\rm soft} \widehat{\mathcal{E}}_0, \quad  \widehat{\Gamma}_{0,0}^{\rm stiff}\widehat{\vect u} = \widehat{\Gamma}_{0,0}^{\rm soft} \vect u  \bigr\},\\[0.3em]
		&\begin{aligned}
		a_{\chi,\varepsilon}^{\rm eff}\left( \begin{bmatrix}
			\vect u \\ \widehat{\vect u}
		\end{bmatrix} ,\begin{bmatrix}
			\vect v \\ \widehat{\vect v}
		\end{bmatrix}\right)&:= \int_{Y_{\rm soft}} \A_{\rm soft} \left(\simgrad +{\rm i}X_\chi \right)\vect u : \overline{ \left(\simgrad + {\rm i}X_\chi \right)\vect v} + \frac{1}{\varepsilon^2}\Lambda_{\chi}^{\rm hom}\widehat{\Gamma}_{0,0}^{\rm stiff}\widehat{\vect u} \cdot \widehat{\Gamma}_{0,0}^{\rm stiff}\widehat{\vect v},
		\quad \begin{bmatrix}
		\vect u \\ \widehat{\vect u}
	\end{bmatrix} ,\begin{bmatrix}
		\vect v \\ \widehat{\vect v}
	\end{bmatrix} \in \mathcal{D}\bigl(a_{\chi,\varepsilon}^{\rm eff}\bigr).
	\end{aligned}
\end{aligned}
\end{equation*}
	Recalling Lemma \ref{josipnak1}, one can see that a similar form was obtained in \cite{CherCoop} as an 
	$O(\varepsilon)$-approximation in the case of a scalar equation by using a different technique.
\end{remark}

By a slight abuse of notation, Remark \ref{latter}  allows us to identify the operator ${\mathcal A}^{\rm eff}_{\chi.\varepsilon}$ with an operator acting 
in a subspace of ${\mathcal H}^{\rm soft}\oplus{\mathcal H}^{\rm stiff}\equiv{\mathcal H}.$ We then extend it by zero to the whole $\mathcal H,$ while still keeping the same notation for the extension, hoping that it does not lead to any confusion. 

The following theorem provides 
norm-resolvent asymptotics of order $\varepsilon$ in the form of a ``sandwiched" resolvent of the effective operator $\mathcal{A}_{\chi,\varepsilon}^{\rm eff}$. It is a direct consequence of Theorem \ref{nakk1}, cf. the proof of Theorem \ref{thmpremain1}. 
\begin{theorem} \label{thmpremain2} 
	There exists $C>0$, independent of $z \in K_{\sigma}$ and $\varepsilon$, such that for the resolvent of the transmission problem \eqref{transmissionboundaryproblem} one has
	\begin{equation}
		\label{asymptoticformularesolventepseffective}
		\bigl\lVert \bigl((\mathcal{A}_{\chi,\varepsilon})_{0,I} -zI \bigr)^{-1} - \Theta_0\left(\mathcal{A}_{\chi,\varepsilon}^{\rm eff} - zI \right)^{-1}\Theta_0 \bigr\rVert_{\mathcal{H} \to \mathcal{H}} \leq C \varepsilon\qquad \forall\chi \in Y',
	\end{equation}
	where the operator $\mathcal{A}_{\chi,\varepsilon}^{\rm eff} $ is defined by 
	\eqref{effoperator}, and
	$\Theta_0$ is the orthogonal projection \eqref{nakk1011}.
\end{theorem}
\begin{proof}[Proof of Theorem \ref{thmamin2} (a)] 
	This is a direct consequence of Theorem \ref{thmpremain2}, based on the fact that the (scaled) Gelfand transform is an isometry, by setting
	$\Theta^{\rm eff}:=\mathcal{G}_{\eps}^{-1} \Theta_0 \mathcal{G}_{\eps},$ 
	$\mathcal{A}^{\rm eff}_{\eps}:=\mathcal{G}_{\eps}^{-1} \mathcal{A}_{\chi,\eps}^{\rm eff} \mathcal{G}_{\eps}.$
\end{proof}


\begin{remark} \label{comparison} 
	In \cite{CherErKis}, a version of Theorem \ref{thmpremain2} is proved by expanding the least eigenvalue and the corresponding eigenfunction of the operator $\Lambda_{\chi}^{\rm stiff}$ with respect to the quasimomentum $\chi$.  As explained in the introduction to  Section \ref{AsymDtN}, this is not possible  when dealing with systems.  Thus we expand  the resolvent of an appropriately scaled operator $\Lambda_{\chi}^{\rm stiff},$ which as we have shown, suffices to prove Theorem \ref{thmpremain2}. We also improve the error estimate $O(\varepsilon^{2/3})$ obtained in \cite{CherErKis} to $O(\varepsilon).$
\end{remark}  
\begin{remark} \label{josipnak10} 
	For every $\varepsilon>0$ we define the space $\mathcal{S}_{\varepsilon}^{\rm stiff} \subset H^1(\mathbb{R}^3;\mathbb{R}^3)$ as the space of functions whose scaled Gelfand transform is constant for every $\chi \in Y'$. It is easy to see that this space consists of functions $\vect u$ such that $\mathcal{F}(\vect{u})(\xi)=0$ when $|\xi|_\infty>(2\varepsilon)^{-1}$, where $|\xi|_\infty=\max\{|\xi_1|,|\xi_2|,|\xi_3|\}$ (cf. \eqref{inversegelfand}). Here  $\mathcal{F}(\cdot)$ stands for the Fourier transform and $\xi \in \C^3$ is the Fourier variable.  We also introduce the space 
	 $\mathcal{S}_{\varepsilon}^{\rm soft}:=H^1(\mathbb{R}^3;\mathbb{R}^3) \cap L_{\varepsilon}^{\rm soft}$.
	Define a bilinear form $a^{\rm eff}_{\eps}$ by
	\begin{equation*}
	\begin{aligned}
		&\mathcal{D}\left(a^{\rm eff}_{\eps}\right):=\left\{ \widehat{\vect u}+\vect{u}: (\widehat{\vect u},\vect{u}) \in \mathcal{S}_{\varepsilon}^{\rm stiff} \times \mathcal{S}_{\varepsilon}^{\rm soft}  \right\},\\[0.4em]
		&\begin{aligned}
		a^{\rm eff}_{\eps}\left(
		\vect u+ \widehat{\vect u},
		\vect v + \widehat{\vect v}
		\right):= \varepsilon^2\int_{\Omega^{\eps}_{\rm soft}} \A_{\rm soft} \simgrad  (\vect u+\widehat{\vect u}): \simgrad (\vect v+\widehat{\vect v})+ \int_{\Omega}\mathbb{A}_{\rm macro}&\simgrad  \widehat{\vect u} : \simgrad \widehat{\vect v,}\\[0.1em] 
		&(\vect u+\widehat{\vect u}),  (\vect v+\widehat{\vect v}) \in \mathcal{D}\left(a^{\rm eff}_{\eps}\right).
		\end{aligned}
	\end{aligned}
\end{equation*}
	It is easily seen that the scaled Gelfand transform of the form $a^{\rm eff}_{\eps}$ equals $a^{\rm eff}_{\chi,\eps}$. By an appropriate modification of the definition of $\mathcal{S}_{\varepsilon}^{\rm stiff}$ we can also treat the form from Remark \ref{nakk1020} that defines the operator $\mathcal{A}_{\chi,\varepsilon}^{\rm app}$. 
\end{remark}

\section{Stiff component analysis} \label{sectionopeffstiff} 
In this section we study implications  of the estimates of the previous section. Our goal here is to prove Theorem \ref{thmamin2}\,(b). 
We are interested in the properties of the effective operator
\eqref{effoperator} when restricted to the stiff component.
A representation formula for this operator will be obtained that will bring to focus some known features of high-contrast homogenisation. To this end, we define the following  operators that unitarily identify the spaces $\widehat{\mathcal{E}}_0$ and $\widehat{\mathcal{H}}_0^{\rm stiff}$ (spanned by constant functions) with $\C^3$:
\begin{equation*}
   \begin{aligned}
        &\iota_{\Gamma}: \widehat{\mathcal{E}}_0 \to \C^3, \quad \iota_{\Gamma} \vect c = |\Gamma|^{1/2} \vect c, \quad \vect c \in  \widehat{\mathcal{E}}_0,\qquad\qquad 
        \iota_{\rm stiff}: \widehat{\mathcal{H}}_0^{\rm stiff} \to \C^3, \quad \iota_{\rm stiff} \vect c = |Y_{\rm stiff}|^{1/2} \vect c, \quad \vect c \in  \widehat{\mathcal{H}}_0^{\rm stiff}.
    \end{aligned}
\end{equation*}
Notice that (cf. Remark \ref{ante31}) 
\begin{equation*}
    \widehat{\Pi}_0^{\rm stiff} = |Y_{\rm stiff}|^{1/2}|\Gamma|^{-1/2}\iota_{\rm stiff}^* \iota_{\Gamma}.
\end{equation*}
With these operators at hand, we obtain the following representation formula (recall Lemma \ref{josipnak1} and \eqref{msoftzeroformula}):
\begin{equation*}
    \begin{aligned}
         P_{\widehat{\mathcal{H}}_0^{\rm stiff}}\left(\mathcal{A}_{\chi,\varepsilon}^{\rm eff} - zI \right)^{-1}|_{\widehat{\mathcal{H}}_0^{\rm stiff}} &= - \widehat{\Pi}_0^{\rm stiff} 
			{\widehat{Q}_{\varepsilon,\chi}^{\rm eff}(z)}^{-1} \bigl(\widehat{\Pi}_0^{\rm stiff}\bigr)^* = - \widehat{\Pi}_0^{\rm stiff} \left(\varepsilon^{-2} \widehat{\Lambda}_\chi^{\rm hom}+ z \bigl(\widehat{\Pi}_0^{\rm stiff}\bigr)^* \widehat{\Pi}_0^{\rm stiff} + \widehat{M}_{0}^{\rm soft}(z) \right)^{-1} \bigl(\widehat{\Pi}_0^{\rm stiff}\bigr)^* \\[0.4em]
   & = - |Y_{\rm stiff}||\Gamma|^{-1}\iota_{\rm stiff}^* \iota_{\Gamma} \left(\varepsilon^{-2} \widehat{\Lambda}_\chi^{\rm hom}+ z |Y_{\rm stiff}||\Gamma|^{-1} I_{\widehat{\mathcal{E}}_0} + \widehat{M}_{0}^{\rm soft}(z) \right)^{-1} \iota_{\Gamma}^* \iota_{\rm stiff} \\[0.4em]
   &= - |Y_{\rm stiff}|\iota_{\rm stiff}^*  \left(\varepsilon^{-2} |\Gamma|\iota_{\Gamma}\widehat{\Lambda}_\chi^{\rm hom}\iota_{\Gamma}^*+ z |Y_{\rm stiff}| I_{\C^3} + |\Gamma|\iota_{\Gamma}\widehat{M}_{0}^{\rm soft}(z)\iota_{\Gamma}^* \right)^{-1}  \iota_{\rm stiff} \\[0.4em]
   &= |Y_{\rm stiff}|\iota_{\rm stiff}^*  \left(\varepsilon^{-2} \left({\rm i}X_\chi \right)^*\A_{\rm macro}{\rm i}X_\chi - \mathcal{B}(z) \right)^{-1}  \iota_{\rm stiff}, 
    \end{aligned}
\end{equation*}
where the matrix-valued function $\mathcal{B}(z)$ is defined by 
\begin{equation*}
	\mathcal{B} (z):= z |Y_{\rm stiff}| I_{\C^3} + |\Gamma|\iota_{\Gamma}\widehat{M}_{0}^{\rm soft}(z)\iota_{\Gamma}^*.
\end{equation*}
In what follows, we show (see \eqref{silly_label}, \eqref{Keff_final}) that, owing to \eqref{msoftzeroformula}, a natural matrix representation of $\mathcal{B}$ is given by \eqref{ante1}.  
\begin{remark}
	Regarding the estimates from above and below for the operator $\mathcal{B}(z)$, note that for $\vect c \in \mathbb{C}^3$ one has
	\begin{equation}
		\label{marin1} 
			\bigl|\left\langle \Im(\mathcal{B}(z)) \vect c, \vect c\right\rangle_{\mathbb{C}^3}\bigr|=\bigl|\Im(z)\bigr|\Bigl\lvert\left|Y_{\rm stiff}|\langle  \vect c, \vect c\right\rangle_{\mathbb{C}^3} + |\Gamma|\bigl\langle \widehat{S}_0^{\rm soft}(\overline{z}) \iota_{\Gamma}^* \vect c, \widehat{S}_0^{\rm soft}(\overline{z}) \iota_{\Gamma}^*\vect c\bigr\rangle_{\mathcal{H}^{\rm soft}}\Bigr\rvert\geq C |\vect c |^2,
	\end{equation}
	and thus, by Corollary, \ref{boundfrombelowappendix} one has 
		$\lVert\mathcal{B}(z)^{-1}\rVert_{\C^3 \to \C^3} \leq C_1,$
	where $C_1>0$ depends on the set $K_\sigma$. Also, clearly
	\begin{equation}
		\label{estimateonB2}
		\bigl\lVert \mathcal{B}(z)\bigr\rVert_{\C^3 \to \C^3} \leq C_2,
	\end{equation}
	where $C_2>0$ depends only on the $\max_{z \in K_\sigma} |z|$.

The above is an explicit proof of the property for a Herglotz matrix function to be bounded together with its inverse away from the real line. 
\end{remark}

\begin{remark}
	We will keep the same notation $\mathcal{B}(z)$ for the operator of (pointwise) multiplication by the said matrix.
\end{remark}

In order to pass to the real domain, what remains is to apply the inverse Gelfand transform. Before doing so, we introduce a smoothing operator: 
$\Xi_\varepsilon :L^2(\R^3;\C^3) \to L^2(\R^3;\C^3)$ defined by\footnote{The operator $\mathcal{G}_\varepsilon^{-1}$ is here applied to a function depending only on $\chi$ i.e. for fixed $\chi$ constant in $y$.}
\begin{equation*}
	{\Xi}_\varepsilon \vect u := \mathcal{G}_\varepsilon^{-1}\int_{Y}  (\mathcal{G}_\varepsilon \vect u)(y,\cdot)dy.    
\end{equation*} 
Next, we note that the projection operator $P_{\rm stiff}$ is simply a multiplication with an indicator function associated with $Y_{\rm stiff}$, namely
	$P_{\rm stiff} \vect u = \mathbbm{1}_{Y_{\rm stiff}}(y) \vect u, \quad \vect u \in \mathcal{H}.$
Similarly, for the operator $P_\varepsilon^{\rm stiff}$ i.e. the orthogonal projector from $L^2(\R^3;\C^3) $ onto $ L_\varepsilon^{\rm stiff}$, which is defined by \eqref{L2stiffepsilon}, we have
\begin{equation*}
	P_\varepsilon^{\rm stiff} \vect u =  \mathbbm{1}_{\Omega_{\rm stiff}^\varepsilon}(x) \vect u, \quad \vect u \in L^2(\R^3;\C^3).
\end{equation*}
Also, for $\vect u \in \mathcal{H}_{\rm stiff}$, we have
\begin{equation*}
	P_{\widehat{\mathcal{H}}_0^{\rm stiff}} \vect u ={\mathbbm 1}_{Y_{\rm stiff}}(y)\frac{1}{|Y_{\rm stiff}|} \int_{Y_{\rm stiff}} \vect u(y) dy.
\end{equation*}
Note that \footnote{Again, the inverse Gelfand transform is applied to a function depending only on $\chi$ i.e. for fixed $\chi$ constant in $y$.}
\begin{equation}
	\label{xiformula}
	\begin{aligned}
		\Xi_\varepsilon P_{\varepsilon}^{\rm stiff} \vect u &=  \Xi_\varepsilon \left(\mathbbm{1}_{\Omega_\varepsilon^{\rm stiff}} \vect u \right) = \mathcal{G}_\varepsilon^{-1} \left(\int_{Y}  (\mathbbm{1}_{Y_{\rm stiff}}\mathcal{G}_\varepsilon \vect u)(y,\cdot) dy \right) \\[0.3em]
		& = |Y_{\rm stiff}|\mathcal{G}_\varepsilon^{-1} \left(\frac{1}{|Y_{\rm stiff}|}
		\int_{Y_{\rm stiff}}(\mathcal{G}_\varepsilon \vect u)(y,\cdot) dy \right) = |Y_{\rm stiff}|^{1/2}\mathcal{G}_\varepsilon^{-1} \bigl(\iota_{\rm stiff}P_{\widehat{\mathcal{H}}_0^{\rm stiff}} \mathcal{G}_\varepsilon \vect u \bigr).
	\end{aligned}
\end{equation}
We have the following lemma. 
\begin{lemma} 
	The following formula holds:
	\begin{equation}
		\label{gelfandpullbackformula}
		\mathcal{G}_{\eps }^{-1}\left(\int_{Y'}^{\oplus}\left(\frac{1}{\varepsilon^2}\left({\rm i}X_\chi \right)^*\A_{\rm macro}{\rm i}X_\chi  - \mathcal{B}(z) \right)^{-1} 
	 \iota_{\rm stiff}P_{\widehat{\mathcal{H}}_0^{\rm stiff}}\,d \chi\right)\mathcal{G}_\varepsilon =\frac{1}{\sqrt{|Y_{\rm stiff}|}} \bigl(\mathcal{A}_{\rm macro} - \mathcal{B}(z) \bigr)^{-1}\Xi_\varepsilon P_{\eps}^{\rm stiff},
	\end{equation} 
	where the operator $\mathcal{A}_{\rm macro}$ is defined by the form \eqref{macrooperatordef_form}.
\end{lemma}
\begin{proof}
	To see this, we consider the operator $\mathcal{A}_{\chi,\rm macro}$ on $\mathcal{H}$ with the sesquilinear form 
	\begin{equation*}
		a_{\chi,\rm macro}(\vect u,\vect v) = \int_Y \mathbb{A}_{\rm macro} \left(\simgrad + {\rm i}X_\chi \right)\vect u : \overline{\left(\simgrad + {\rm i}X_\chi \right)\vect v}, \quad \vect u, \vect v \in H^1_\#(Y;\C^3). 
	\end{equation*}
	By invoking the properties of the Gelfand transform \eqref{gelfandvsderivatives}, it is clear that 
	\begin{equation*}
		\mathcal{G}_\varepsilon^{-1} \left( \int_{Y'}^{\oplus} \frac{1}{\varepsilon^2}\mathcal{A}_{\chi,\rm macro} \, d\chi \right) \mathcal{G}_\varepsilon = \mathcal{A}_{\rm macro}.
	\end{equation*}
	By virtue of \eqref{xiformula}, it remains to show that 
	\begin{equation}
		\label{resolventequivalentformula}
		\bigl(\varepsilon^{-2}\mathcal{A}_{\chi,\rm macro} -\mathcal{B}(z) \bigr)^{-1} \iota_{\rm stiff}P_{\widehat{\mathcal{H}}_0^{\rm stiff}}= \bigl(\varepsilon^{-2}\left({\rm i}X_\chi \right)^*\A_{\rm macro}{\rm i}X_\chi  - \mathcal{B}(z) \bigr)^{-1} 
		 \iota_{\rm stiff}P_{\widehat{\mathcal{H}}_0^{\rm stiff}}.
	\end{equation}
	First conclusion is that for $z \in K_\sigma$, $\vect u  \in \mathcal{H}$, we have 
	\begin{equation*}
		\Im \bigl(\varepsilon^{-2}\left({\rm i}X_\chi \right)^*\A_{\rm macro}{\rm i}X_\chi  - \mathcal{B}(z) \bigr)  =  \Im \bigl( \varepsilon^{-2}\mathcal{A}_{\chi,\rm macro} -\mathcal{B}(z) \bigr) = -\Im\bigl(\mathcal{B}(z)\bigr),
	\end{equation*}
	so the operators are invertible by taking into account Corollary \ref{boundfrombelowappendix} and the estimates \eqref{marin1}.
	In order to show \eqref{resolventequivalentformula}, we take $\vect f \in \mathcal{H}$ and consider the unique solution $\vect u \in \C^3$ 
	to the resolvent problem
	\begin{equation*}
		\varepsilon^{-2}\left({\rm i}X_\chi \right)^*\A_{\rm macro}{\rm i} X_\chi \vect u  - \mathcal{B}(z) \vect u =  \iota_{\rm stiff}P_{\widehat{\mathcal{H}}_0^{\rm stiff}}\vect f.
	\end{equation*}
	Multiplying the above equation with arbitrary $\vect v \in \C^3$, and integrating over $Y$ one obtains 
	\begin{equation*}
		\frac{1}{\varepsilon^2}\int_Y \A_{\rm macro}{\rm i}X_\chi \vect u : \overline{{\rm i} X_\chi \vect v} - \int_Y \mathcal{B}(z) \vect u \cdot \overline{\vect v} = \int_Y \iota_{\rm stiff}P_{\widehat{\mathcal{H}}_0^{\rm stiff}} \vect f \cdot \overline{\vect v}.
	\end{equation*} 
	Furthermore, it is checked that, as an element of $H^1_\#(Y;\C^3),$ the constant function $\vect u$ solves the problem
	\begin{equation*}
		\frac{1}{\varepsilon^2}\int_Y \A_{\rm macro}\left( \simgrad + {\rm i} X_\chi \right) \vect u : \overline{ \left(\simgrad + {\rm i}X_\chi \right) \vect v} - \int_Y \mathcal{B}(z) \vect u \cdot \overline{\vect v} = \int_Y \iota_{\rm stiff}P_{\widehat{\mathcal{H}}_0^{\rm stiff}}\vect f \cdot \overline{\vect v} \qquad \forall \vect v \in H^1_\#(Y;\C^3),
	\end{equation*} 
	which is unique. The formula \eqref{gelfandpullbackformula} now follows from \eqref{resolventequivalentformula}.
\end{proof}
The following lemma allows us to drop the smoothing operator $\Xi_\varepsilon$ from the resolvent asymptotics while not making the error of higher order then $\varepsilon^2$.
\begin{lemma}
	\label{dropxi}
	Let $z \in K_\sigma$. There exists a constant $C>0$ such that
	\begin{eqnarray} \label{prva} 
		\bigl\lVert \bigl(\mathcal{A}_{\rm macro} - \mathcal{B}(z) \bigr)^{-1} \left(I- \Xi_{\varepsilon} \right) \bigr\rVert_{L^2(\R^3;\C^3) \to L^2(\R^3;\C^3) } &\leq&  C \varepsilon^2,
	\end{eqnarray}
	where $\mathcal{A}_{\rm macro}$ is a differential operator of linear elasticity with constant coefficients defined by the form \eqref{macrooperatordef_form}.    
\end{lemma}
\begin{proof}
	We start with the identity 
	\begin{equation*}
		\mathcal{F}({\Xi}_{\eps}\vect f)(\xi)=\mathbbm{1}_{[-1/(2\eps), 1/(2\eps)]^3}(\xi)\mathcal{F}(\vect f)(\xi)\qquad \forall\vect f \in L^2(\R^3;\C^3),	\end{equation*}
	where $\mathcal{F}$ denotes, as before, the Fourier transform, and $\xi\in\C^3$ is the Fourier variable (see, e.g. \cite[Section 2.5.3]{CVZ}).
	The estimate \eqref{prva} follows from the fact that for 
	$\vect f \in L^2(\R^3;\C^3)$ one has
	\begin{equation*}
			\mathcal{F}\bigl(\bigl(\mathcal{A}_{\rm macro} - \mathcal{B}(z) \bigr)^{-1}  \vect f \bigr)(\xi)  = \left(\left({\rm i}X_\xi \right)^*\A_{\rm macro}{\rm i} X_\xi  - \mathcal{B}(z)\right)^{-1}\mathcal{F}(\vect f)(\xi).
	\end{equation*} 
	Namely, introducing  ${\vect u}(\xi):= \left(\left({\rm i}X_\xi \right)^*\A_{\rm macro}{\rm i}X_\xi  - \mathcal{B}(z)\right)^{-1}\mathcal{F}(\vect f)(\xi),$ one has
	\begin{equation*}
		\bigl|{\vect u}(\xi)\bigr| \leq(1 - \varepsilon^2 C_2)^{-1}\varepsilon^2\bigl|\mathcal{F}(\vect f)(\xi)\bigr|\leq C \varepsilon^2\bigl|\mathcal{F}(\vect f)(\xi)\bigr|,\qquad |\xi|>(2\varepsilon)^{-1},\quad  \varepsilon\in(0,C_2^{-1/2}),
	\end{equation*}
	where $C_2$ is given by \eqref{estimateonB2}. 
\end{proof}
Finally, we proceed to the proof of Theorem \ref{thmamin2} (b). 
\begin{proof}[Proof of Theorem \ref{thmamin2} (b)]
 The asymptotic estimate \eqref{asymptoticformularesolventepseffective} immediately yields 
	\begin{equation}
		\label{ante10}
		\begin{aligned} 
		\left\lVert P_{\rm stiff}\bigl((\mathcal{A}_{\chi,\varepsilon})_{0,I} -zI \bigr)^{-1}P_{\rm stiff} -|Y_{\rm stiff}|^{1/2}P_{\rm stiff} \bigl(\varepsilon^{-2}\left({\rm i}X_\chi \right)^*\A_{\rm macro}{\rm i}X_\chi  - \mathcal{B}(z) \bigr)^{-1}\iota_{\rm stiff}P_{\widehat{\mathcal{H}}_0^{\rm stiff}} \right\rVert_{\mathcal{H} \to \mathcal{H}} \leq C \varepsilon\\[0.3em]
		\hspace{50ex}\forall \chi \in Y'.
		\end{aligned}
	\end{equation}
	Invoking \eqref{gelfandpullbackformula}, we obtain  
	\begin{equation*}
		\begin{aligned}
			\mathcal{G}_\varepsilon^{-1}\int_{Y'}^{\oplus}&\left(
		P_{\rm stiff}\bigl((\mathcal{A}_{\chi,\varepsilon})_{0,I} -zI \bigr)^{-1}P_{\rm stiff} -|Y_{\rm stiff}|^{1/2}P_{\rm stiff} \left(\frac{1}{\varepsilon^2}({\rm i}X_\chi)^*\A_{\rm macro}{\rm i}X_\chi  - \mathcal{B}(z) \right)^{-1}\iota_{\rm stiff}P_{\widehat{\mathcal{H}}_0^{\rm stiff}}
		\right)\,d\chi\,\mathcal{G}_\varepsilon \\[0.3em]			
			&\hspace{27ex}=P_\varepsilon^{\rm stiff}\left(\mathcal{A}_\varepsilon - z I \right)^{-1} P_\varepsilon^{\rm stiff}-P_\varepsilon^{\rm stiff} \bigl(\mathcal{A}_{\rm macro} - \mathcal{B}(z)\bigr)^{-1} \Xi_\varepsilon  P_\varepsilon^{\rm stiff}.
		\end{aligned}
	\end{equation*}
	Combining this with \eqref{ante10} and using the fact that Gelfand transform is a unitary operator, we obtain 
	\begin{equation*}
		\bigl\lVert P_\varepsilon^{\rm stiff}\left(\mathcal{A}_\varepsilon - z I \right)^{-1} P_\varepsilon^{\rm stiff} - P_\varepsilon^{\rm stiff} \bigl(\mathcal{A}_{\rm macro} - \mathcal{B}(z) \bigr)^{-1}  \Xi_\varepsilon P_\varepsilon^{\rm stiff}\bigr\rVert_{L^2(\R^3;\C^3) \to L^2(\R^3;\C^3) } \leq C \varepsilon,
	\end{equation*}
	The last step is to drop the smoothing operator for which we use Lemma \ref{dropxi}.
\end{proof}
\begin{remark}
	The operator $\mathcal{A}_{\rm macro} - \mathcal{B}(z),$ which plays the r\^{o}le of the leading-order term in the resolvent asymptotics of Theorem \ref{thmamin2} (b), is clearly a second-order differential operator with constant coefficients. 
\end{remark}

\begin{remark} \label{nakk2}  
	The operator $\mathcal{A}^{\rm eff}_{\eps}$ has a more concise form than the operator  $\mathcal{A}^{\rm app}_{\eps}$ and admits no further simplification. 
	 On the other hand, in what applies to the operator $\mathcal{A}^{\rm app}_{\eps},$ one can still obtain a simpler approximation, by going further in the expansion of DtN map in Section \ref{AsymDtN}. The error bound thus obtained can be seen as  $O(\eps^2),$ i.e., the same as for the $\mathcal{A}^{\rm app}_{\eps}.$ It can be further seen the thus obtained operator will non-local in the spatial variable, i.e. non-differential. For brevity, we refrain from discussing this in detail. 
\end{remark}

\subsection{Dispersion relation}

The assertion of Theorem \ref{thmpremain2}, as well as a more precise (at the cost of being more involved) statement of Theorem \ref{thmpremain1}, pertains to the asymptotic behaviour of the resolvent in the whole space. In applications however, and in particular in applications to periodic problems, it is often desirable to relate the spectrum of the problem to the so-called wave vector, or equivalently quasimomentum $\chi,$ of the problem at hand. Indeed, the dispersion relation, which expresses the mentioned relationship, becomes of a paramount importance when the question of which monochromatic waves are supported by the medium at hand, as well as when the group velocity of wave packets spreading in the medium is brought to the forefront of investigation. The latter question is to arise naturally in the context of metamaterials, to which the model considered in the present paper is thought to be intimately related (the precise mathematical formulation of this relationship will however be discussed elsewhere; see also \cite{CEKRS2022} and references therein for a related discussion), as in these both the phase and group velocity need to negate, \cite{Veselago}.

\begin{definition}
	We refer to the operator valued function defined by 
	\begin{equation}
		\label{ante30}
		\begin{aligned} 
		\mathcal{K}^{\rm app} _{\chi,\varepsilon}(z):= -\bigr(\bigl(\widehat{\Pi}_\chi^{\rm stiff} \bigr)^*\bigr)^{-1} \widehat{Q}^{\rm app}_{\chi,\varepsilon}(z) &\bigl(\widehat{\Pi}_\chi^{\rm stiff} \bigr)^{-1}+ z I
		=-\bigl( \bigl( \widehat{\Pi}_\chi^{\rm stiff} \bigr)^*\bigr)^{-1}\bigl(\varepsilon^{-2}\widehat{\Lambda}_{\chi}^{\rm stiff}  + \widehat{M}_{\chi}^{\rm soft}(z)\bigr)\bigl(\widehat{\Pi}_\chi^{\rm stiff}\bigr)^{-1}
		\end{aligned}
	\end{equation}
	as the \emph{dispersion function} associated with the operator $\mathcal{A}_{\chi,\varepsilon}^{\rm app}$. Similarly, we define the \emph{effective dispersion function} associated with the operator $\mathcal{A}_{\chi,\varepsilon}^{\rm eff}$:
	\begin{equation}
		\label{ante32}
		\begin{aligned} 
		\mathcal{K}_{\chi,\varepsilon}^{\rm eff}(z):= -\bigl( \bigl( \widehat{\Pi}_0^{\rm stiff} \bigr)^*\bigr)^{-1}\widehat{Q}_{\chi,\varepsilon}^{\rm app}(z)&\bigl( \widehat{\Pi}_0^{\rm stiff} \bigr)^{-1} + z I
		=-\bigl( \bigl( \widehat{\Pi}_0^{\rm stiff} \bigr)^*\bigr)^{-1}\bigl(\varepsilon^{-2}\Lambda_{\chi}^{\rm hom}  + \widehat{M}_{0}^{\rm soft}(z)\bigr)\bigl( \widehat{\Pi}_0^{\rm stiff} \bigr)^{-1}.
		\end{aligned}
	\end{equation}
\end{definition}

\begin{remark}
	Notice that  
	\begin{equation*}
		\bigl( \mathcal{K}^{\rm app} _{\chi,\varepsilon}(z) - zI \bigr)^{-1} = - \widehat{\Pi}_\chi^{\rm stiff} \bigl(\widehat{Q}^{\rm app}_{\chi,\varepsilon}(z)\bigr)^{-1} \bigl( \widehat{\Pi}_\chi^{\rm stiff} \bigr)^{*}, \qquad \bigl( \mathcal{K}_{\chi,\varepsilon}^{\rm eff}(z) - zI \bigr)^{-1} = - \widehat{\Pi}_0^{\rm stiff} \bigl( \widehat{Q}_{\chi,\varepsilon}^{\rm eff}(z)\bigr)^{-1} \bigl( \widehat{\Pi}_0^{\rm stiff}\bigr)^{*}.
	\end{equation*}
	By comparing this with resolvent formulas: \eqref{resolventblockmatrix} and \eqref{effresolventblockmatrix}, one can see that these operators are, in fact, resolvents of the appropriate operators sandwiched with projections onto the stiff component:
	\begin{equation*}
		P_{\widehat{\mathcal{H}}_\chi^{\rm stiff}}\bigl(\mathcal{A}_{\chi,\varepsilon}^{\rm app} - zI \bigr)^{-1}|_{\widehat{\mathcal{H}}_\chi^{\rm stiff}} =  \bigl(\mathcal{K}^{\rm app} _{\chi,\varepsilon}(z) -z I\bigr)^{-1}, \qquad P_{\widehat{\mathcal{H}}_0^{\rm stiff}}\bigl(\mathcal{A}_{\chi,\varepsilon}^{\rm eff} - zI \bigr)^{-1}|_{\widehat{\mathcal{H}}_0^{\rm stiff}} =  \bigl(\mathcal{K}^{\rm eff}_{\chi,\varepsilon}(z) -z I\bigr)^{-1}.
	\end{equation*}
\end{remark}

Denote by $\{\psi_j^\chi\}_{j=1}^3 \subset \widehat{\mathcal{E}}_\chi$
an orthonormal basis of eigenfunctions of $\widehat{\Lambda}_\chi^{\rm stiff}$ and by $\{\nu_j^{\chi}\}_{j=1}^3$
the associated set of eigenvalues.
Note that  the  $\{ \widehat{\Pi}_\chi^{\rm stiff} \psi_j^\chi\}_{j=1}^3$
is a basis in $\widehat{\mathcal{H}}_\chi^{\rm stiff}.$ The function $\widehat{\vect u} \in \widehat{\mathcal{H}}_\chi^{\rm stiff}$ is represented in this basis with a vector $\vect \alpha = (\alpha_1,\alpha_2,\alpha_3) \in \C^3$ as 
\begin{equation}
	\label{uhatinbasis}
	\widehat{\vect u} = \sum_{i = 1}^3\alpha_i \widehat{\Pi}_\chi^{\rm stiff}\psi_i^\chi.
\end{equation}
Furthermore, we denote by $\{\Psi_j^\chi\}_{j=1}^3\subset \widehat{\mathcal{H}}_\chi^{\rm stiff}$
the contravariant dual for the basis $\{ \widehat{\Pi}_\chi^{\rm stiff} \psi_j^\chi\}_{j=1}^3$,
namely, the set of functions such that
\begin{equation*}
	\bigl\langle  \widehat{\Pi}_\chi^{\rm stiff} \psi_i^\chi, \Psi_j^\chi\bigr\rangle_{\mathcal{H}^{\rm stiff}} = \delta_{ij}, \quad i,j = 1,2,3.
\end{equation*}
One can easily check that
	$\Psi_j^\chi=((\widehat{\Pi}_\chi^{\rm stiff})^*)^{-1} \psi_j^\chi,$ $j=1,2,3.$
Thus we find the following expressions for the coefficients 
in \eqref{uhatinbasis}:
\begin{equation*}
	\alpha_j = \langle \widehat{\Gamma}_{0,\chi} \widehat{\vect u},\psi_j^\chi\rangle_{\widehat{\mathcal{E}}_\chi}, \quad j=1,2,3.
\end{equation*}
Next, we calculate the matrix $\mathbb{K}^{\rm app }_{\chi,\varepsilon}(z)$ of the operator $\mathcal{K}^{\rm app}_{\chi,\varepsilon}(z)$ in the basis $\{\widehat{\Pi}_\chi^{\rm stiff} \psi_i^\chi\}_{i=1}^3:$
\begin{equation*}
		\mathbb{K}^{\rm app}_{\chi,\varepsilon}(z)_{ij}= \Bigl\langle \mathcal{K}^{\rm app} _{\chi,\varepsilon}(z) \widehat{\Pi}_\chi^{\rm stiff}  \psi_j^{\chi}, \bigl( \bigl(\widehat{\Pi}_\chi^{\rm stiff}\bigr)^*\bigr)^{-1} \psi_i^\chi \Bigr\rangle_{\widehat{\mathcal{H}}_\chi^{\rm stiff}}.
\end{equation*}
By invoking the formulas \eqref{usefulidentitiesresolvent}, \eqref{solutionrepresentation} and \eqref{ante30}, the dispersion function can be expressed as
\begin{equation*}
	\mathcal{K}^{\rm app}_{\chi,\varepsilon}(z):= -\bigl( \bigl( \widehat{\Pi}_\chi^{\rm stiff} \bigr)^*\bigr)^{-1}\bigl(\varepsilon^{-2}\widehat{\Lambda}_{\chi}^{\rm stiff}  + \widehat{\Lambda}_\chi^{\rm soft} + z \bigl( \widehat{\Pi}_\chi^{\rm soft}\bigr)^*\widehat{\Pi}_\chi^{\rm soft} + z^2 \bigl(\widehat{\Pi}_\chi^{\rm soft}\bigr)^* \bigl(\mathcal{A}_{0,\chi}^{\rm soft} - zI\bigr)^{-1}\widehat{\Pi}_\chi^{\rm soft}  \bigr)\bigl(\widehat{\Pi}_\chi^{\rm stiff} \bigr)^{-1}.
\end{equation*}
Denote now by $(\eta_k)_{k \in \mathbb{N}} \subset \mathbb{R}^+$, $(\varphi_k^\chi)_{k \in \mathbb{N}} \subset \mathcal{H}^{\rm soft}$ the eigenvalues and associated eigenfunctions of the operator $\mathcal{A}_{0,\chi}^{\rm soft}$. The resolvent of $\mathcal{A}_{0,\chi}^{\rm soft}$ admits the expansion
\begin{equation*}
	\bigl(\mathcal{A}_{0,\chi}^{\rm soft} - zI \bigr)^{-1} = \sum_{k = 1}^{\infty} (\eta_k - z)^{-1}\mathcal{P}_k^\chi, \qquad \mathcal{P}_k^\chi:= \left\langle \cdot, \varphi_k^\chi \right\rangle_{\mathcal{H}^{\rm soft}} \varphi_k^\chi.
\end{equation*}
Upon a straightforward computation, we obtain
\begin{equation}
	\label{dispersionfunctionmatrixform}
\begin{aligned}
	\mathbb{K}^{\rm app} _{\chi,\varepsilon}(z)_{ij} &= -\varepsilon^{-2}\nu^{j}_{\chi}\bigl\langle \psi_j^{\chi},\mathbb{H}_\chi\psi_i^\chi\bigr\rangle_{\widehat{\mathcal{E}}^{\rm stiff}_{\chi}}-\bigl\langle\widehat{\Lambda}_\chi^{\rm soft }\psi_j^{\chi}, \mathbb{H}_\chi\psi_i^\chi\bigr\rangle_{\widehat{\mathcal{E}}^{\rm stiff}_{\chi}}\\[0.3em] 
	&-z\bigl\langle \widehat{\Pi}_\chi^{\rm soft}\psi_j^{\chi}, \widehat{\Pi}_\chi^{\rm soft}\mathbb{H}_\chi\psi_i^{\chi} \bigr\rangle_{\mathcal{H}^{\rm soft}}  - \sum_{k = 1}^{\infty} \frac{z^2}{\eta_k - z} \bigl\langle \widehat{\Pi}_\chi^{\rm soft}\psi_j^{\chi}, \varphi_k^\chi \bigr\rangle_{\mathcal{H}^{\rm soft}}\bigl\langle \varphi_k^\chi,\widehat{\Pi}_\chi^{\rm soft}\mathbb{H}_\chi\psi_i^{\chi}  \bigr\rangle_{\mathcal{H}^{\rm soft}}.
	\end{aligned}
\end{equation}
where  $\mathbb{H}_\chi:=\bigl(( \widehat{\Pi}_\chi^{\rm stiff})^* \widehat{\Pi}_\chi^{\rm stiff}\bigr)^{-1}:\widehat{\mathcal{E}}_\chi \to \widehat{\mathcal{E}}_\chi$ .

Similarly, by using \eqref{ante32}, one establishes the matrix representation $\mathbb{K}_{\chi,\varepsilon}^{\rm{eff}}(z)$ of the effective dispersion function $\mathcal{K}_{\chi,\varepsilon}^{\rm eff}(z)$ in the canonical basis $\{e_j\}_{j=1}^3$ 
of $\mathbb{C}^3:$
\begin{equation*}
	\mathbb{K}_{\chi,\varepsilon}^{\rm{eff}}(z)_{ij} := \left\langle \mathcal{K}_{\chi,\varepsilon}^{\rm eff}(z) e_j, e_i \right\rangle_{\widehat{\mathcal{H}}_0^{\rm stiff}}=-\varepsilon^{-2}\bigl\langle \Lambda_{\chi}^{\rm hom} e_j, e_i\bigr\rangle_{\widehat{\mathcal{E}}_0} - \bigl\langle \widehat{M}_{0}^{\rm soft}(z) e_j, e_i \bigr\rangle_{\widehat{\mathcal{E}}_0} . 
\end{equation*}
Recalling \eqref{msoftzeroformula}, we obtain
\begin{equation}
	\begin{aligned}
		&\bigl\langle \widehat{M}_{0}^{\rm soft}(z) e_j, e_i \bigr\rangle_{\widehat{\mathcal{E}}_0^{\rm stiff}} = z  \left\langle \Pi_0^{\rm soft} e_j,\Pi_0^{\rm soft}e_i\right\rangle_{\widehat{\mathcal{H}}_0^{\rm soft}} + z^2 \left\langle \left( \mathcal{A}_{0,0}^{\rm soft} - zI\right)^{-1}\Pi_0^{\rm soft} e_j,\Pi_0^{\rm soft}e_i\right\rangle_{\widehat{\mathcal{H}}_0^{\rm soft}} \\[0.1em]
		& 
		= z  \left\langle \Pi_0^{\rm soft} e_j,\Pi_0^{\rm soft}e_i\right\rangle_{\widehat{\mathcal{H}}_0^{\rm soft}} +    z^2\sum_{k=1}^\infty \frac{\left\langle \Pi_0^{\rm soft} e_j, \varphi_k\right\rangle_{\widehat{\mathcal{H}}_0^{\rm soft}}\left\langle \varphi_k,\Pi_0^{\rm soft}e_i\right\rangle_{\widehat{\mathcal{H}}_0^{\rm soft}}}{\eta_k - z} 
		= z |Y_{\rm soft}|\delta_{ij} + \sum_{k=1}^\infty \frac{z^2}{\eta_k - z}\langle\varphi_k\rangle_j\langle {\varphi_k}\rangle_i. 
	\end{aligned}
\label{silly_label}
\end{equation}
where $(\eta_k,\varphi_k)_{k\in \N}$ are the eigenpairs of the operator $\mathcal{A}_{0,0}^{\rm soft}$.
Finally, we have 
\begin{equation}
	\label{Keff_final}
	\left\langle \left(\mathcal{K}_{\chi,\varepsilon}^{\rm eff}(z) - zI \right) e_j, e_i \right\rangle_{\widehat{\mathcal{H}}_0^{\rm stiff}} = - \frac{1}{\varepsilon^2}\left\langle \Lambda_{\chi}^{\rm hom} e_j, e_i \right\rangle_{\widehat{\mathcal{E}}_0} -  z \delta_{ij} -\sum_{k=1}^\infty \frac{z^2}{\eta_k - z}\langle\varphi_k\rangle_j\langle {\varphi_k}\rangle_i.
\end{equation}
The calculations above prove the following theorem.
\begin{theorem}
The dispersion relation  for the operator $\mathcal{A}_{\chi,\varepsilon}^{\rm app}$ is given by
	\begin{equation*}
		\det \left( \mathbb{K}^{\rm app}_{\chi,\varepsilon}(z) - z |Y_{\rm stiff}| \mathbb{I}\right) = 0,
	\end{equation*}
	It links the parameters $z$ and $\chi$, where the matrix $\mathbb{K}_{\chi,\varepsilon}(z) \in \C^{3 \times 3}$ is given by  \eqref{dispersionfunctionmatrixform}. The \emph{effective} dispersion relation (see also \cite{Smyshlyaev_Mech_Mater, CC_Maxwell}) is given by 
	\begin{equation}
		\det\bigl( \A_{\chi,\varepsilon}^{\rm eff} - \mathcal{B}(z)\bigr) = 0,
		\label{drelation}
	\end{equation}
	where (cf. Lemma \ref{josipnak1})
	\begin{equation*}
		\A_{\chi,\varepsilon}^{\rm eff}:=\varepsilon^{-2}X_\chi^*\A_{\rm macro}X_\chi, \qquad \mathcal{B}(z)_{ij}= z\delta_{ij} +  \sum_{k=1}^\infty \frac{z^2}{\eta_k - z}\langle\varphi_k\rangle_i\langle {\varphi_k}\rangle_j.  
	\end{equation*}
\end{theorem}

Denoting $\theta:=|\chi|^{-1}\chi,$ one has
$	\A_{\chi,\varepsilon}^{\rm eff} = |\chi|^2\A_{{\theta},\varepsilon}^{\rm eff} = \varepsilon^{-2}|\chi|^2X_{\theta}^*\A_{\rm macro}X_{\theta}.$
The matrix $X_{\theta}^*\A_{\rm macro}X_{\theta}$ is positive definite with the lower bound uniform with respect to ${\theta},$ with $\eta, C_1>0$:
\begin{equation*}
	\left\langle X_{\theta}^*\A_{\rm macro}X_{\theta}\, \vect\zeta, \vect\zeta \right\rangle = \left\langle\A_{\rm macro}X_{\theta}\, \vect\zeta, X_{\theta}\, \vect\zeta \right\rangle \geq \eta \left\lvert X_{\theta}\, \vect\zeta\right\rvert^2 \geq \eta C_1\vert {\theta}\vert^2  \vert\vect\zeta\vert^2 =\eta C_1 \vert\vect\zeta\vert^2\qquad \forall \vect\zeta\in{\mathbb R}^2,
\end{equation*}
where we use the coercivity estimate of Lemma \ref{prop_lemma} and the lower bound in \eqref{Xchi_bounds}. Denote by $\A_{{\theta}}^{1/2}$ the positive definite (symmetric) square root of $X_{{\theta}}^*\A_{\rm macro}X_{{\theta}}$. Clearly, the dispersion relation \eqref{drelation} is equivalent to 
\begin{equation}
	\label{drelation2}
	\det\bigl( |\chi|^2 I - \varepsilon^2\A_{{\theta}}^{-1/2}\mathcal{B}(z)\A_{{\theta}}^{-1/2}\bigr) = 0.
\end{equation}
Note that for fixed  $z > 0$ and ${\theta},$ the matrices $\mathcal{B}(z)$ and $\A_{{\theta}}^{-1/2}\mathcal{B}(z)\A_{{\theta}}^{-1/2}$ have the same number of 
non-negative 
eigenvalues. The eigenvalues $\beta(z)$ of $\mathcal{B}(z)$ are shown to be strictly increasing in $z$ in every interval of analyticity of $\mathcal{B}$ on the real line. This follows from the Herglotz property of $\mathcal{B}(z),$ since the latter implies the positive-definiteness of the derivative $\mathcal{B}'(z)$ on the real line. It follows that the same holds for the eigenvalues of $\A_{{\theta}}^{-1/2}\mathcal{B}(z)\A_{{\theta}}^{-1/2},$ and therefore all the branches of the multivalued function $z \mapsto |\chi(z)|$ defined implicitly by \eqref{drelation2} are strictly increasing. This, in turn, implies that the gradient of $z$ is parallel to $\chi$: 
$$
\nabla_{\chi}z=z'(|\chi|)\frac{\chi}{|\chi|}=
\biggl(\frac{d|\chi(z)|}{dz}\biggr)^{-1}
\frac{\chi}{|\chi|}.
$$
We have thus proved the following statement. 
	\begin{theorem}
		For every direction ${\theta}$, the number of solutions $|\chi|$ to \eqref{drelation2} is equal to the number of non-negative eigenvalues of $\mathcal{B}(z).$
		 Furthermore, 
		the corresponding group velocities \cite{Brillouin, Milonni} are positive \cite{Capolino} on every interval of analyticity of ${\mathcal B}$.
	\end{theorem}

\subsection{Numerical example}

In this section we present a numerical example in which we calculate the solutions to the dispersion relation (\ref{drelation}).  
The problem solved here is in 2D, so as to allow visualising its solutions. We consider the unit cell $\overline{Y}= [0,1]^2 =  \overline{Y_{\rm stiff}} \cup \overline{Y_{\rm soft}}$ with an elliptical soft inclusion $Y_{\rm soft}$ with axes $0.04$ and $0.045$ in the stiff matrix $Y_{\rm stiff}$. The elastic material considered is isotropic, with the  material coefficients constant on each of the two components $Y_{\rm soft},$ $Y_{\rm stiff}.$ The contrast between the components is assumed to be $\varepsilon^{-2},$ where $\varepsilon>0$ is small. By fixing the Lam\'e constants $\lambda= 1$, $\mu = 0.1$, we define the tensor of material coefficients
\begin{equation*}
	\A^{\varepsilon}_{ijkl}(y)=
	\left\{\begin{array}{ll} \lambda \, \delta_{ij} \delta_{kl} + 2 \mu \left(\delta_{ik} \delta_{jl} + \delta_{il} \delta_{kj} \right)=:\A^{\rm stiff}_{ijkl}, & y \in Y_{\rm stiff}, \\[0.35em]
		\varepsilon^2\left( \lambda \, \delta_{ij} \delta_{kl} + 2 \mu \left(\delta_{ik} \delta_{jl} + \delta_{il} \delta_{kj} \right)\right)=:\A^{\rm soft}_{ijkl}, & y \in Y_{\rm soft}.
	\end{array} \right.
\end{equation*}
The macroscopic tensor $\mathbb{A}_{\rm macro},$ which constitutes the fiberwise effective operator $\A_{\chi,\varepsilon}^{\rm eff},$ is represented by its action on the elements of an orthonormal basis $\{E_i\}_{i=1}^3$ of $\R^{2 \times 2}_{\rm sym}$: 
\begin{equation*}
	\mathbb{A}_{\rm macro} E_j : E_i := \int_{Y_{\rm stiff}} \A^{\rm stiff} \left(\simgrad \vect u_{E_i} + E_i \right): E_j,
	\quad i,j = 1,2,3.
\end{equation*}
where for $E\in\R^{2 \times 2}_{\rm sym}$ the displacement $\vect u_{E} \in H^1_\#(Y_{\rm stiff};\R^2)$, $\int_{Y_{\rm stiff}} \vect u_{E} = 0$  is calculated by solving the cell problem for the perforated domain $Y_{\rm stiff}$:
\begin{equation*}
	\int_{Y_{\rm stiff}} \A^{\rm stiff} \left(\simgrad \vect u_{E} + E \right): \simgrad \vect v =0\quad \forall \vect v \in H^1_\#(Y_{\rm stiff};\R^2), \int_{Y_{\rm stiff}} \vect v = 0.
\end{equation*}
This results in an explicit construction of the matrix $\A_{\chi,\varepsilon}^{\rm eff}.$ The function $\mathcal{B}(z)$ is approximated by a finite sum  
\begin{equation}
\mathcal{B}^n(z)_{i,j} = z\delta_{i,j} +  \sum_{k=1}^n \frac{z^2}{\eta_k - z}\langle \varphi_k\rangle_j\langle\varphi_k\rangle_i,
\label{Btruncation}
\end{equation}
where $\eta_k,\, k=1,\dots,n,$ are the $n$ smallest eigenvalues of $\mathcal{A}_{0,0}^{\rm soft}$, see Fig.~\ref{figureinclusions}.The graphs of the real eigenvalues $\beta_1(z)$, $\beta_2(z)$ (ordered by the relation $\beta_1(z)\le\beta_2(z)$) of the symmetric matrix-valued function $\mathcal{B}^n(z), z>0,$ are shown in
Fig.~\ref{figurebetafunction}. The dispersion surfaces for the effective problem, which are determined by the relation \eqref{drelation} are shown in Fig.\,\ref{figuredispersion}. 

\begin{figure}[t]
	\centering
	\begin{subfigure}{0.49\linewidth}
		\includegraphics[width=0.46\linewidth]{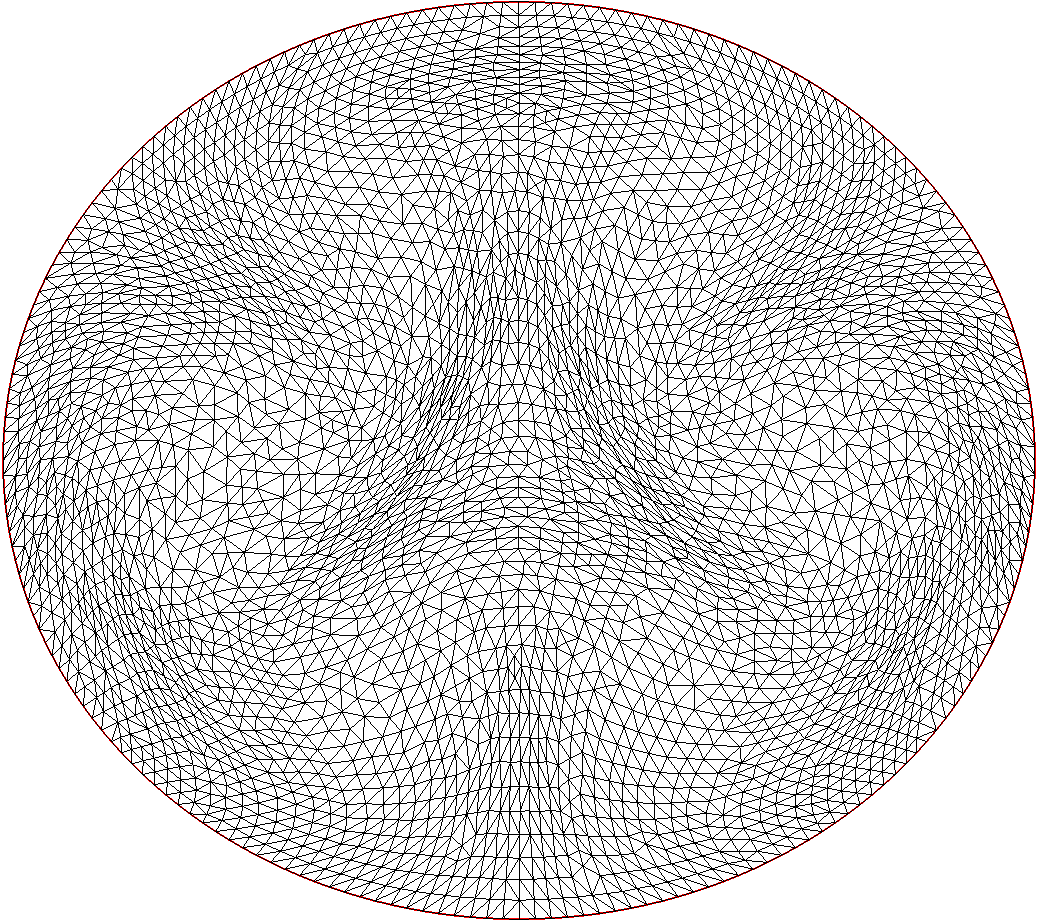}
		\includegraphics[width=0.46\linewidth]{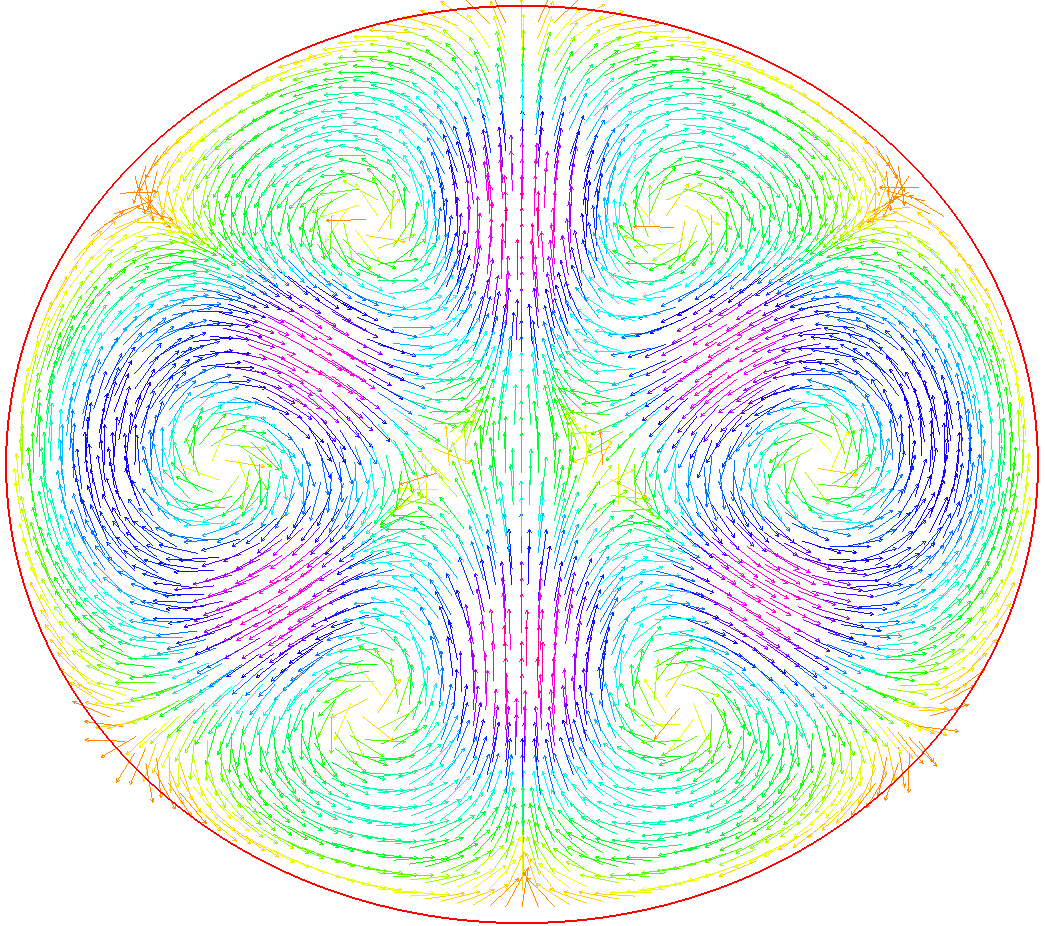}
		\caption*{$\eta = 41.4271$}
	\end{subfigure}
	\begin{subfigure}{0.49\linewidth}
		\includegraphics[width=0.46\linewidth]{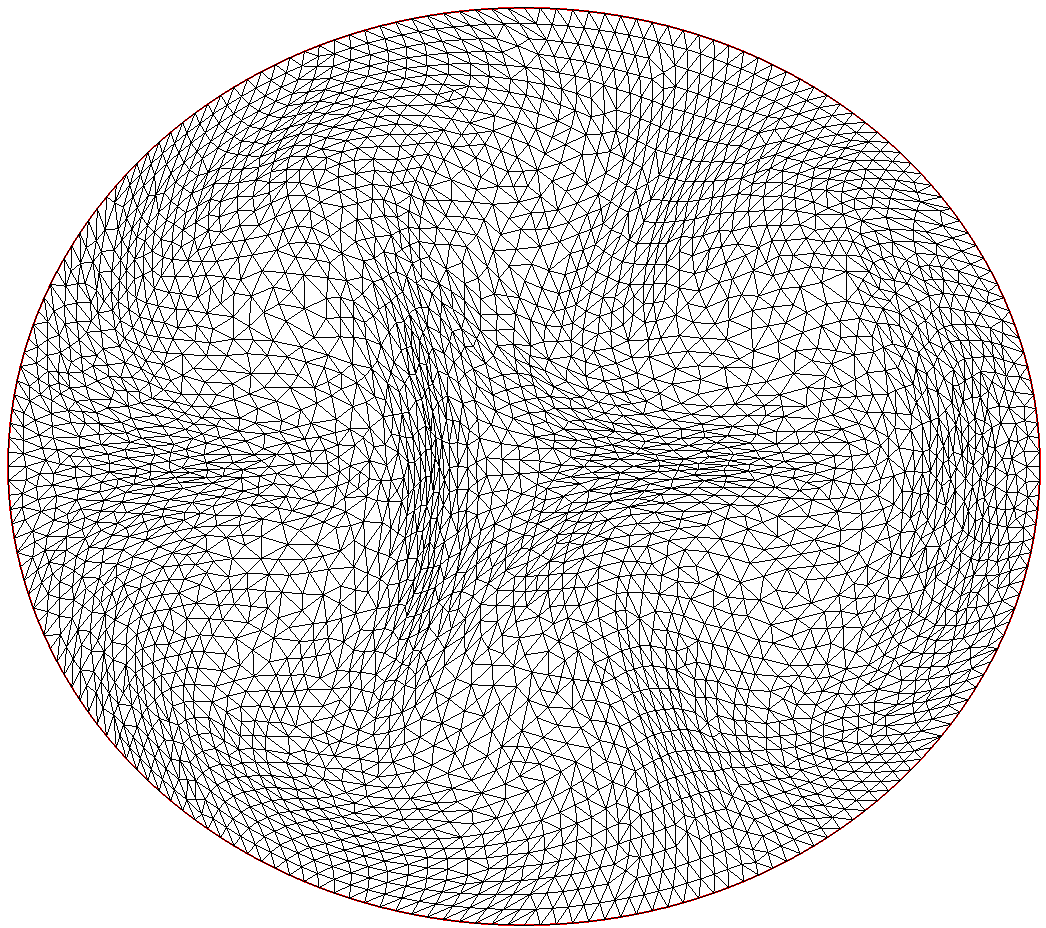}
		\includegraphics[width=0.46\linewidth]{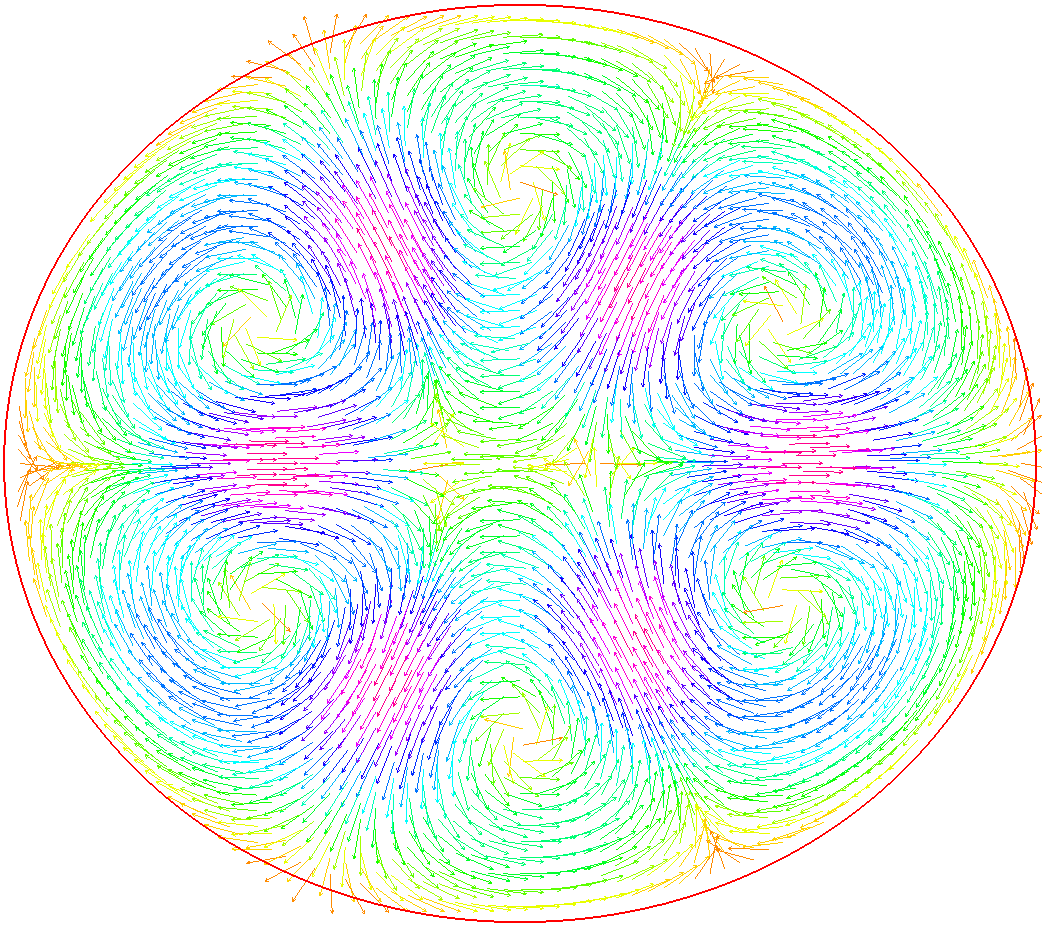}
		\caption*{$\eta = 41.555$}
	\end{subfigure}
	\vfill
	\begin{subfigure}{0.49\linewidth}
		\includegraphics[width=0.46\linewidth]{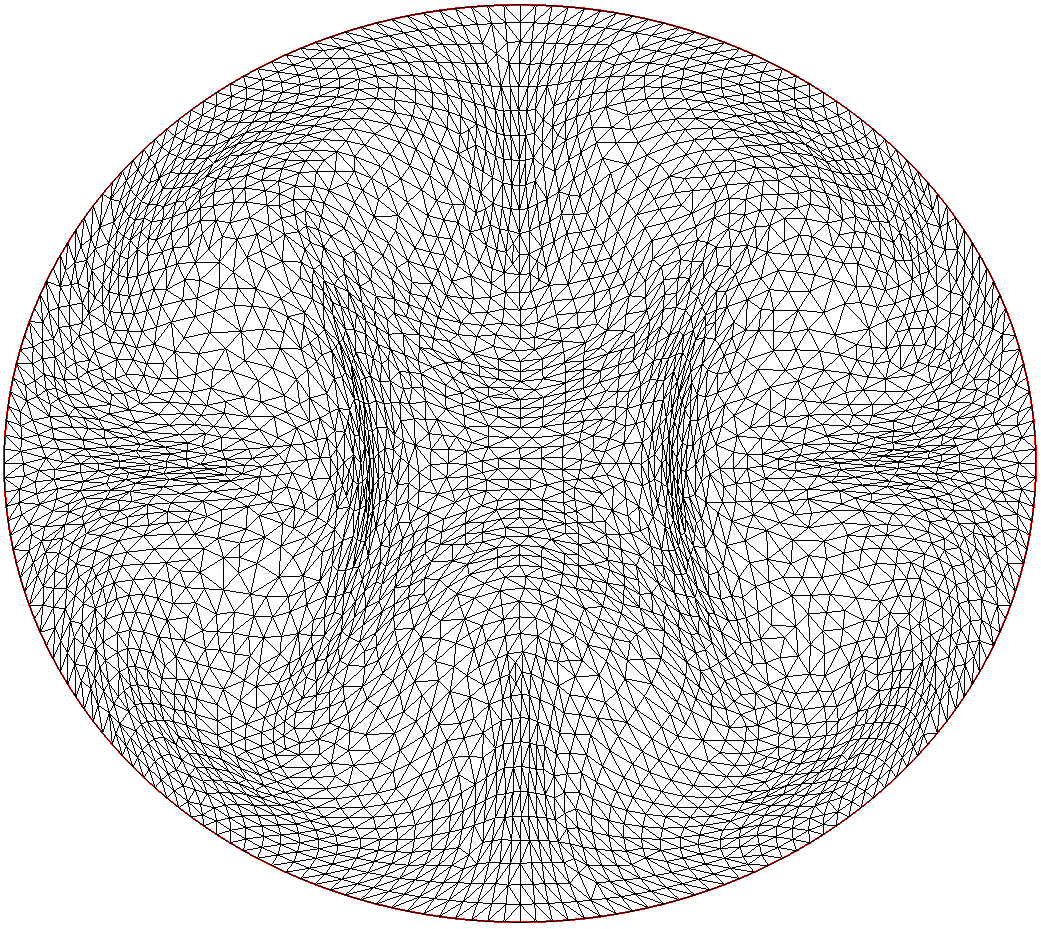}
		\includegraphics[width=0.46\linewidth]{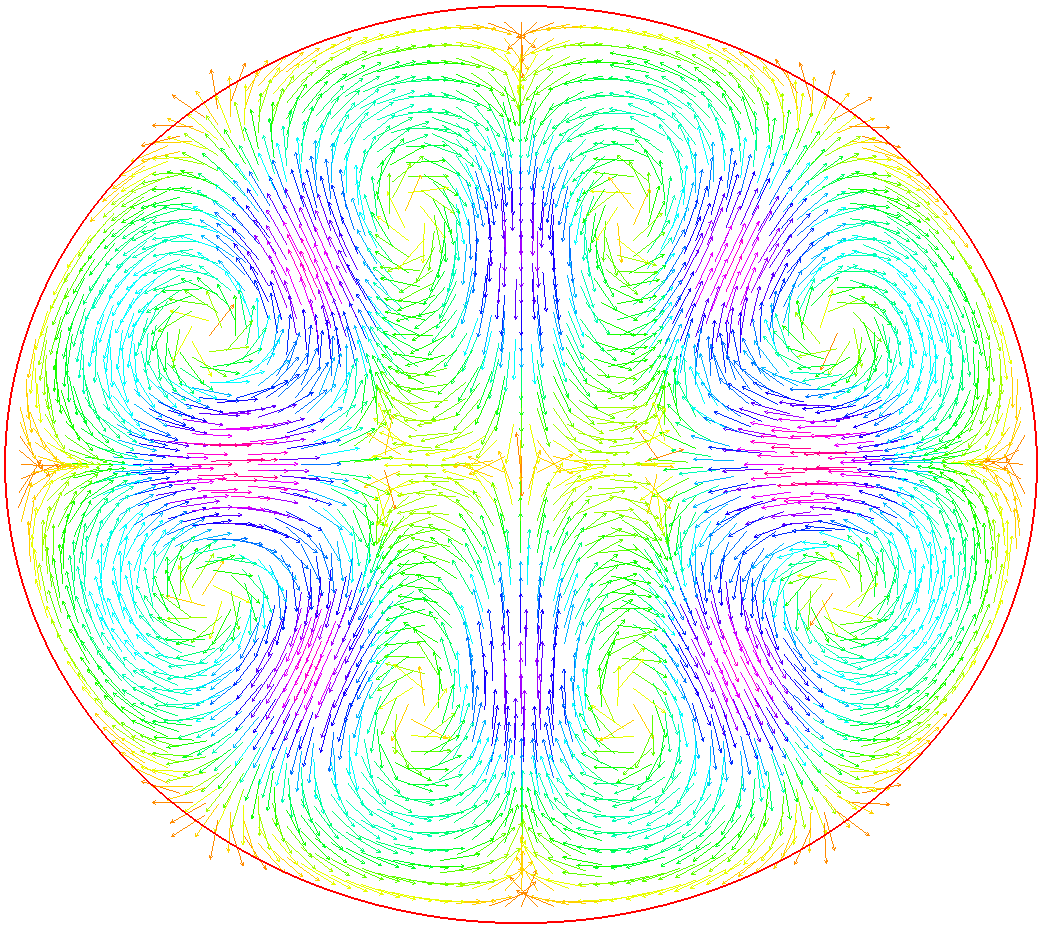}
		\caption*{$\eta = 52.2137$}
	\end{subfigure}
	\begin{subfigure}{0.49\linewidth}
		\includegraphics[width=0.45\linewidth]{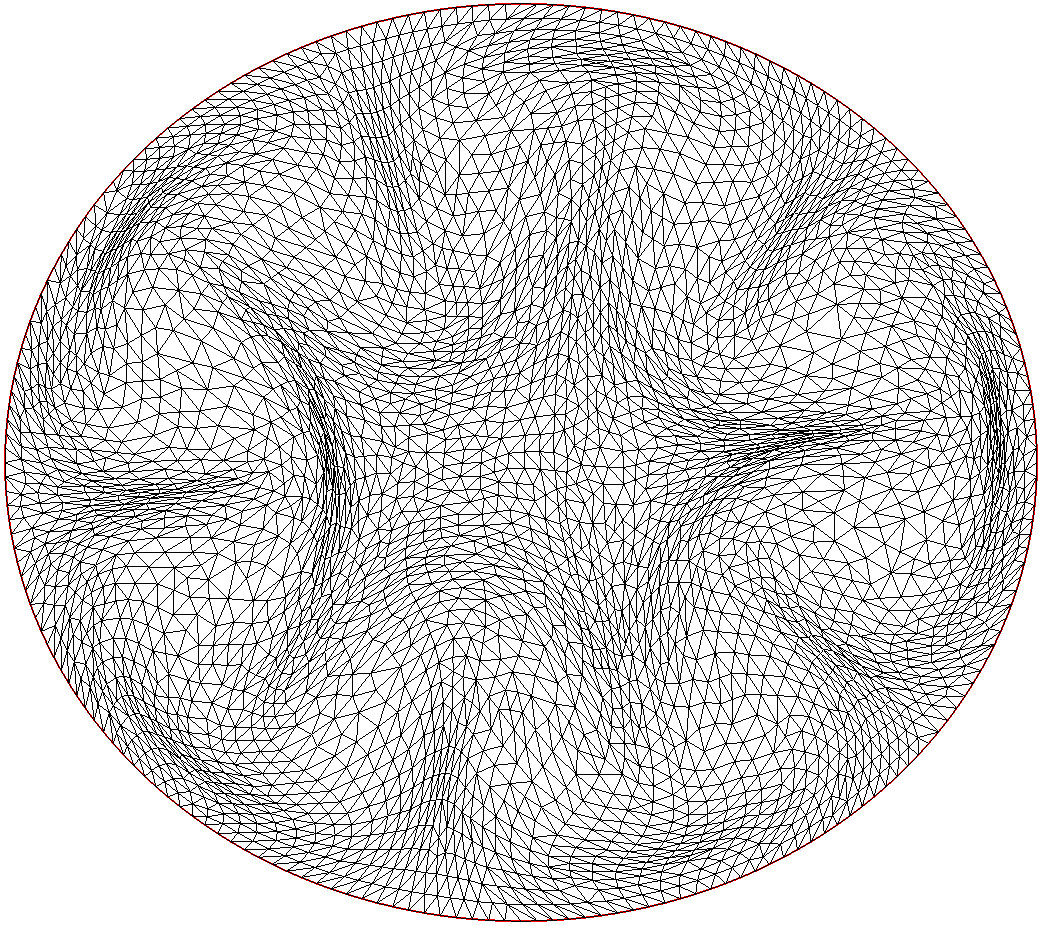}
		\includegraphics[width=0.45\linewidth]{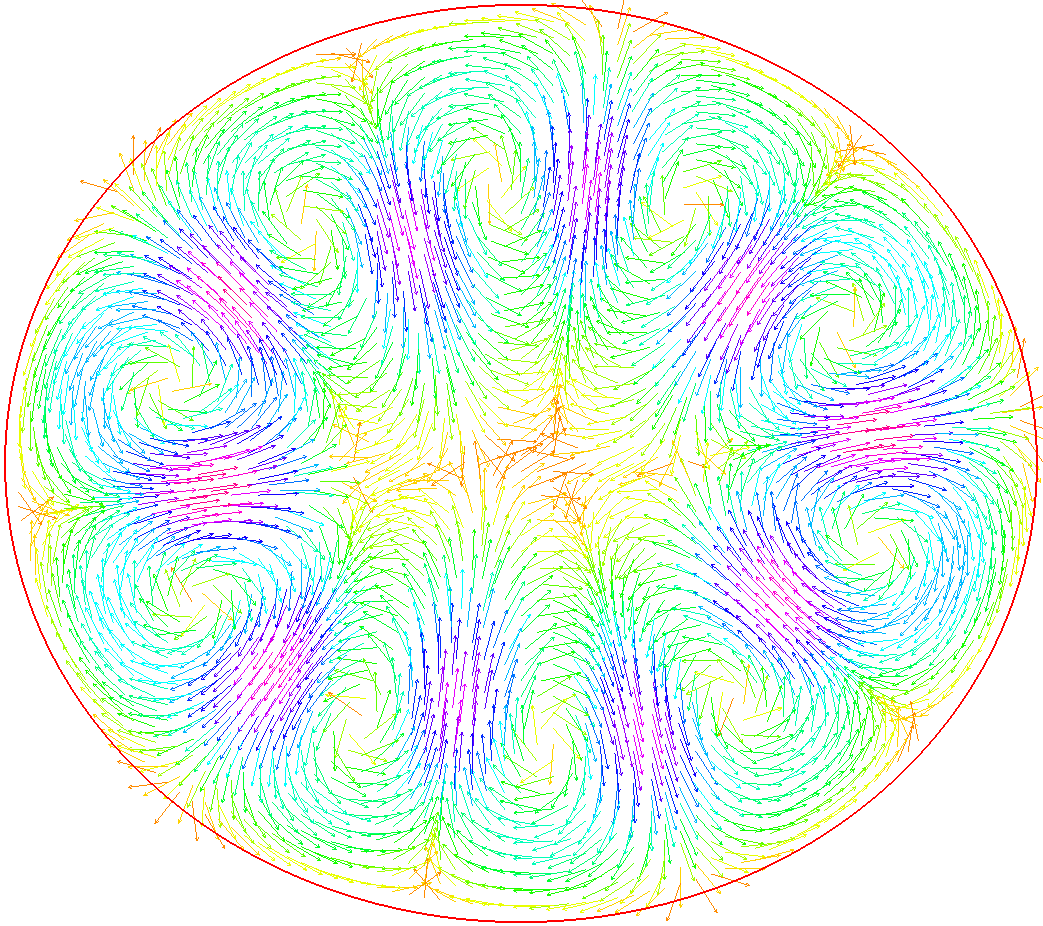}
		\caption*{$\eta = 64.7445$}
	\end{subfigure}
	\caption{Examples of eigenfunctions $\varphi$ of $\mathcal{A}_{0,0}^{\rm soft}$ with $\langle \varphi \rangle\neq 0$ and the associated eigenvalues $\eta$.  For each case, the first picture shows the deformation $x \mapsto x + \varphi(x)$ while the second shows the displacement vector field $x \mapsto \varphi(x)$.}
 \label{figureinclusions}
\end{figure}

\begin{figure}[b]
	\centering
	\includegraphics[trim={0.45cm 0.5cm 5.5cm 0.45cm},clip,width=1\linewidth]{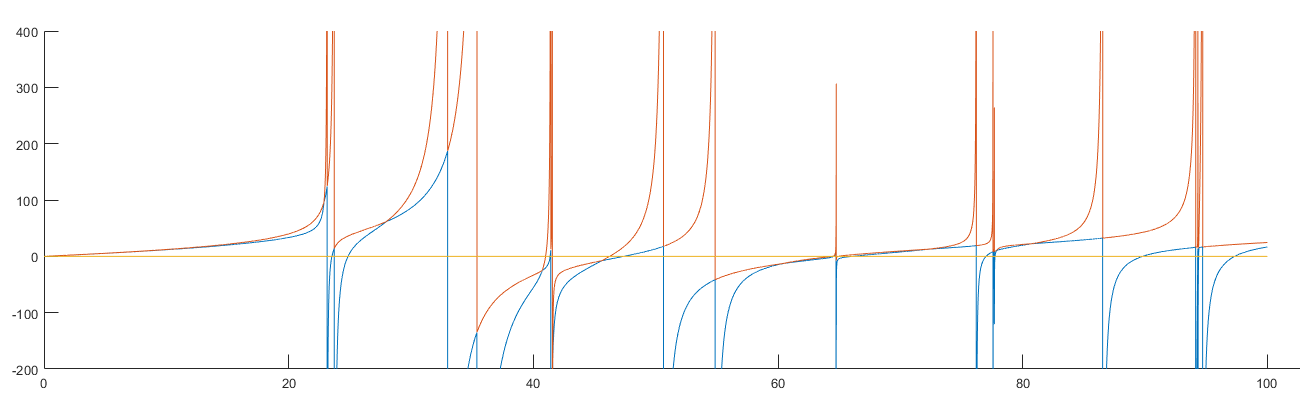}
	\caption{Plots of the eigenvalues $\beta_1(z)$, $\beta_2(z)$ of the ``truncation" $\mathcal{B}^n(z)$, see \eqref{Btruncation}, for $n=11.$ The $x$-axis represents the frequency $z$, the blue curve is the function $\beta_1(z)$ and the red curve is $\beta_2(z)$. The regions of $z$ for which both $\beta_1(z)$ and $\beta_2(z)$ are negative yield no solutions to \eqref{drelation}. }
    \label{figurebetafunction}
\end{figure}

\begin{figure}[t]
	\centering
	\begin{subfigure}{\linewidth}
		\centering
		\begin{subfigure}{0.49\linewidth}
			\centering
			\includegraphics[trim={6cm 2.5cm 5.5cm 1.5cm},clip,width=0.95\linewidth]{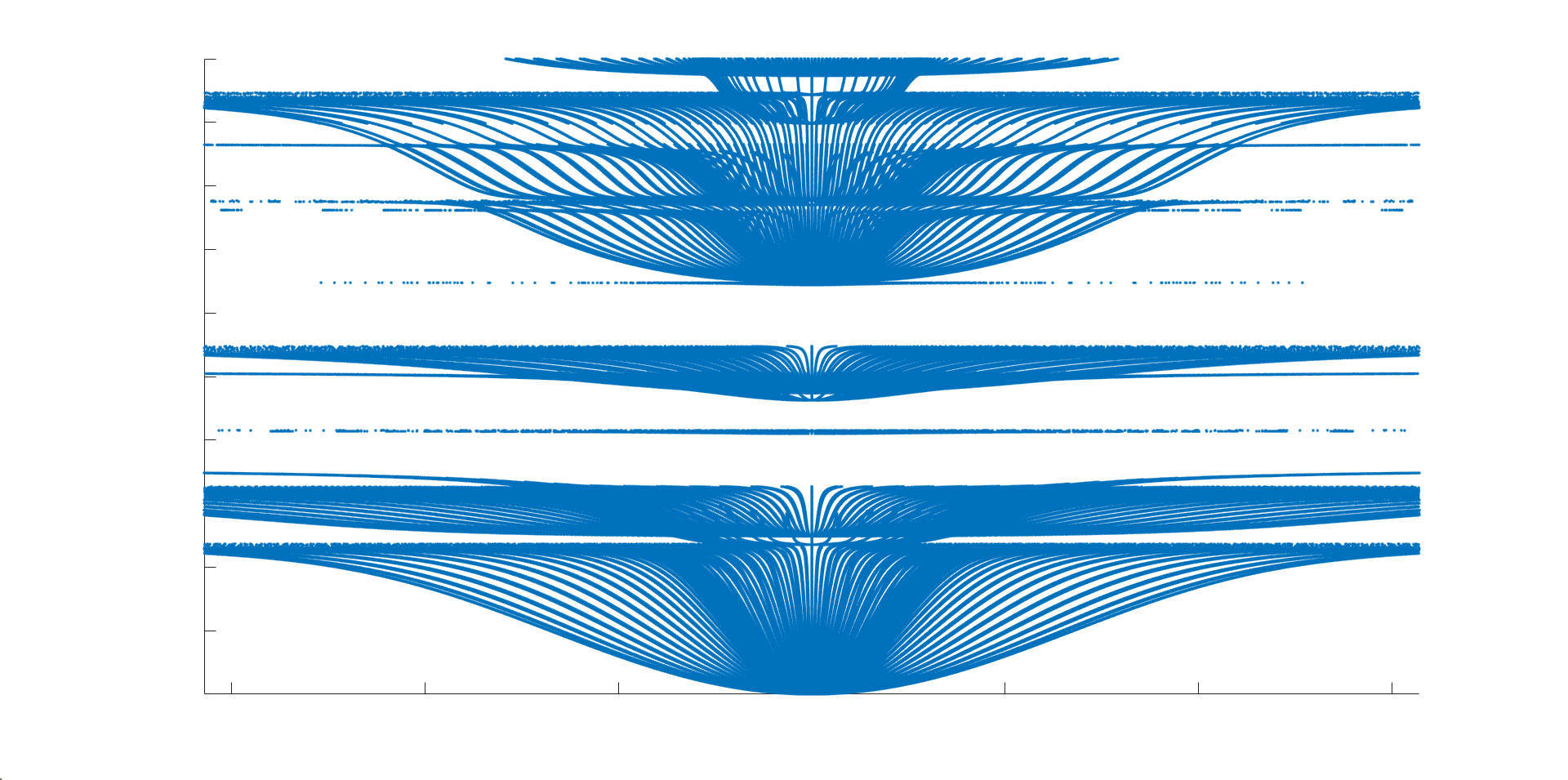}
		\end{subfigure}
		\begin{subfigure}{0.49\linewidth}
			\centering
			\includegraphics[trim={6cm 2.5cm 5.5cm 1.5cm},clip,width=0.95\linewidth]{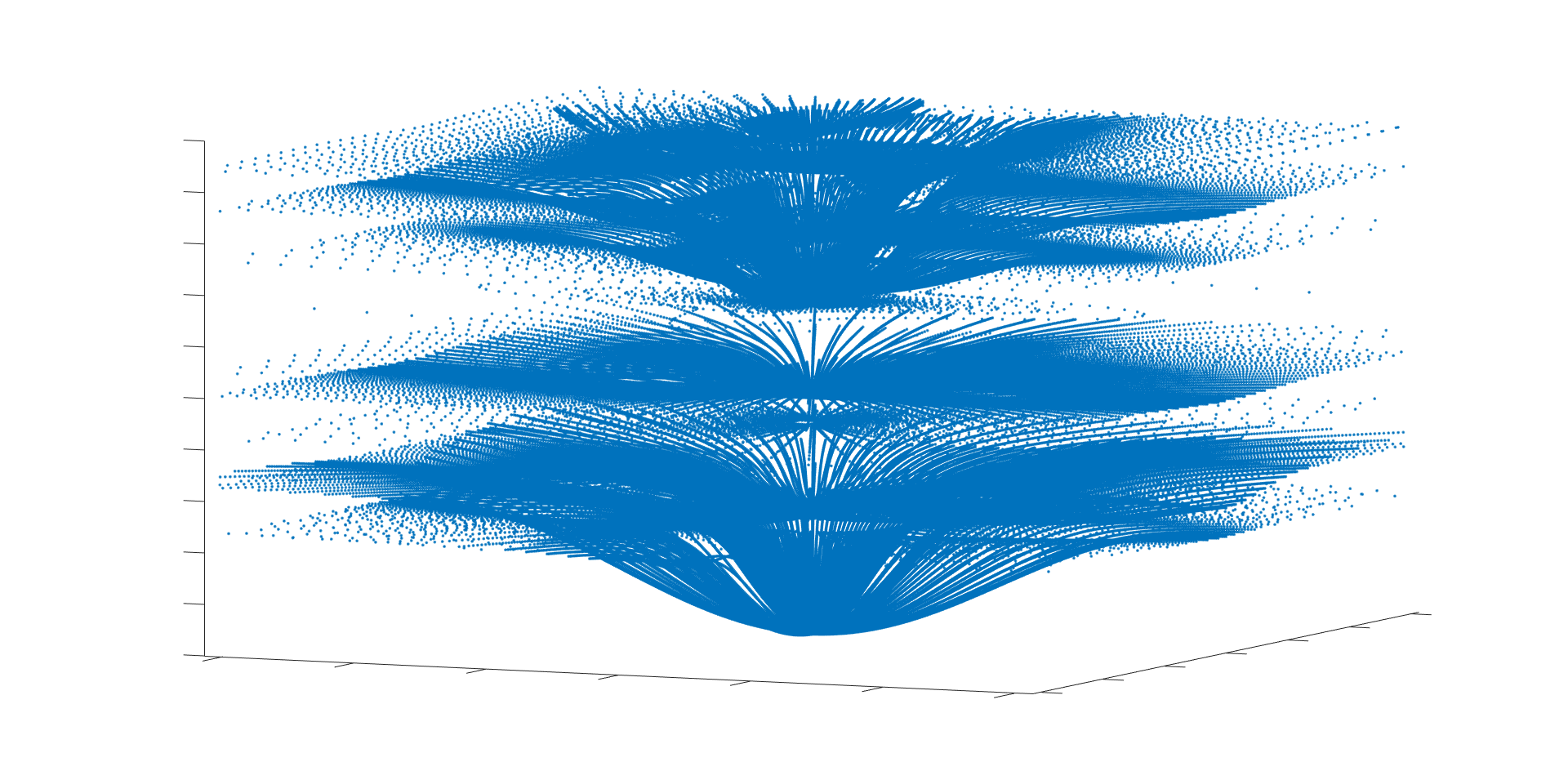}
		\end{subfigure}
		\caption*{$\varepsilon = 0.001$}
	\end{subfigure}
	\vfill
	\begin{subfigure}{\linewidth}
		\centering
		\begin{subfigure}{0.49\linewidth}\centering
			\includegraphics[trim={6cm 2.5cm 5.5cm 1.5cm},clip,width=0.95\linewidth]{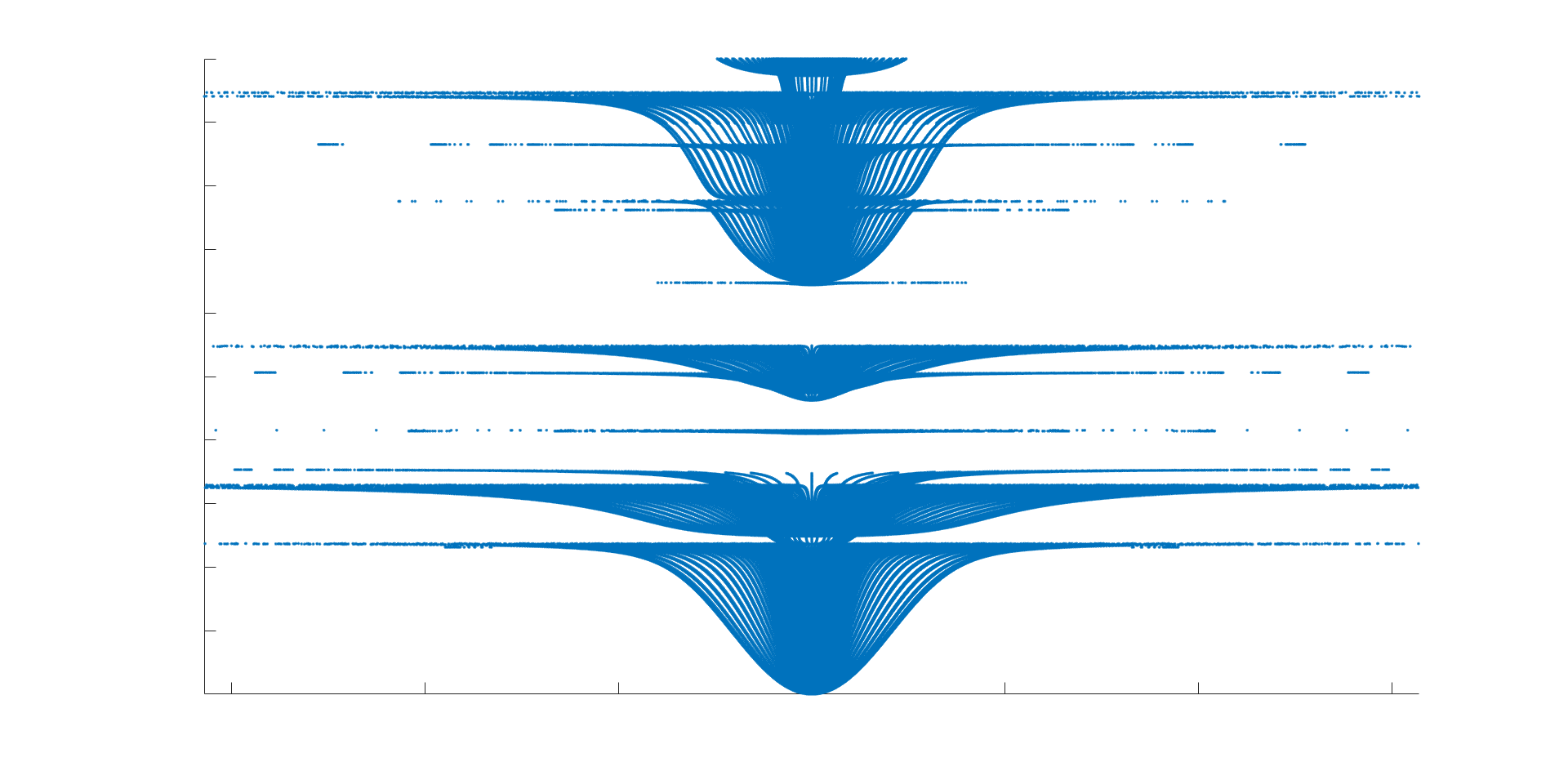}
		\end{subfigure}
		\begin{subfigure}{0.49\linewidth}\centering
			\includegraphics[trim={6cm 2.5cm 5.5cm 1.5cm},clip,width=0.95\linewidth]{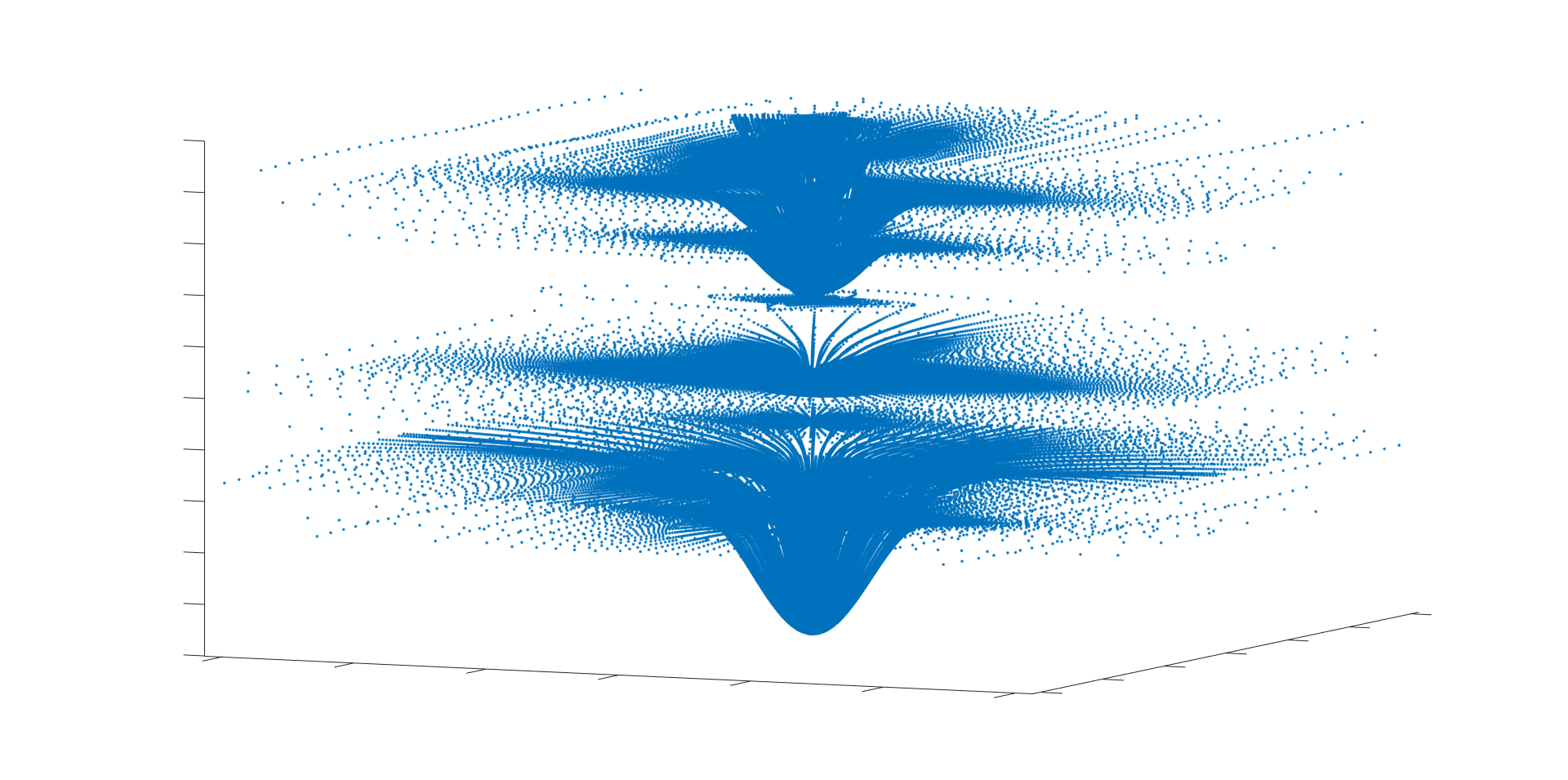}
		\end{subfigure}
		\caption*{$\varepsilon = 0.0001$}
	\end{subfigure}
	\vfill
	\begin{subfigure}{\linewidth}\centering
		\begin{subfigure}{0.49\linewidth}\centering
			\includegraphics[trim={6cm 2.5cm 5.5cm 1.5cm},clip,width=0.95\linewidth]{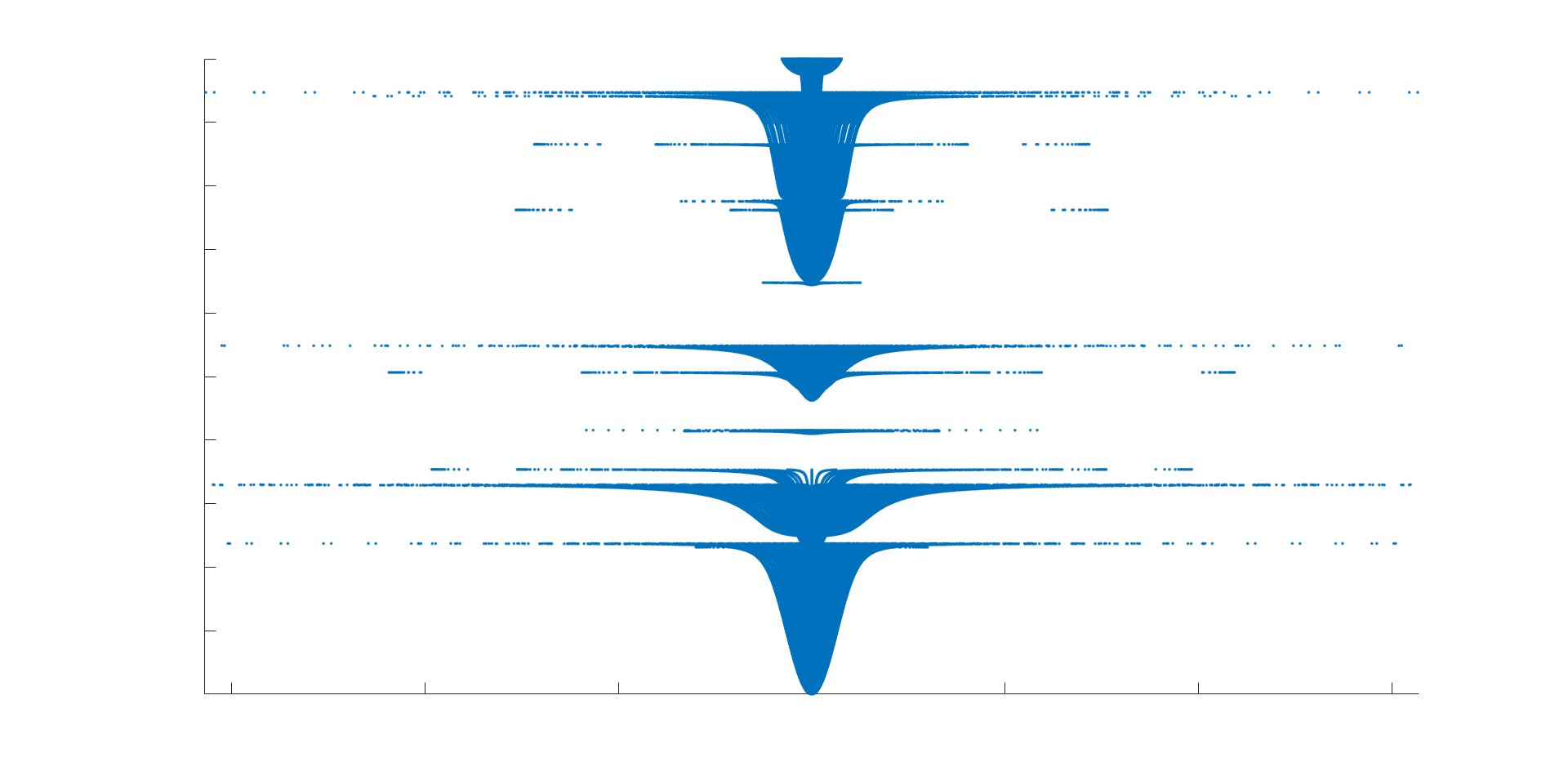}
		\end{subfigure}
		\begin{subfigure}{0.49\linewidth}\centering
			\includegraphics[trim={6cm 2.5cm 5.5cm 1.5cm},clip,width=0.95\linewidth]{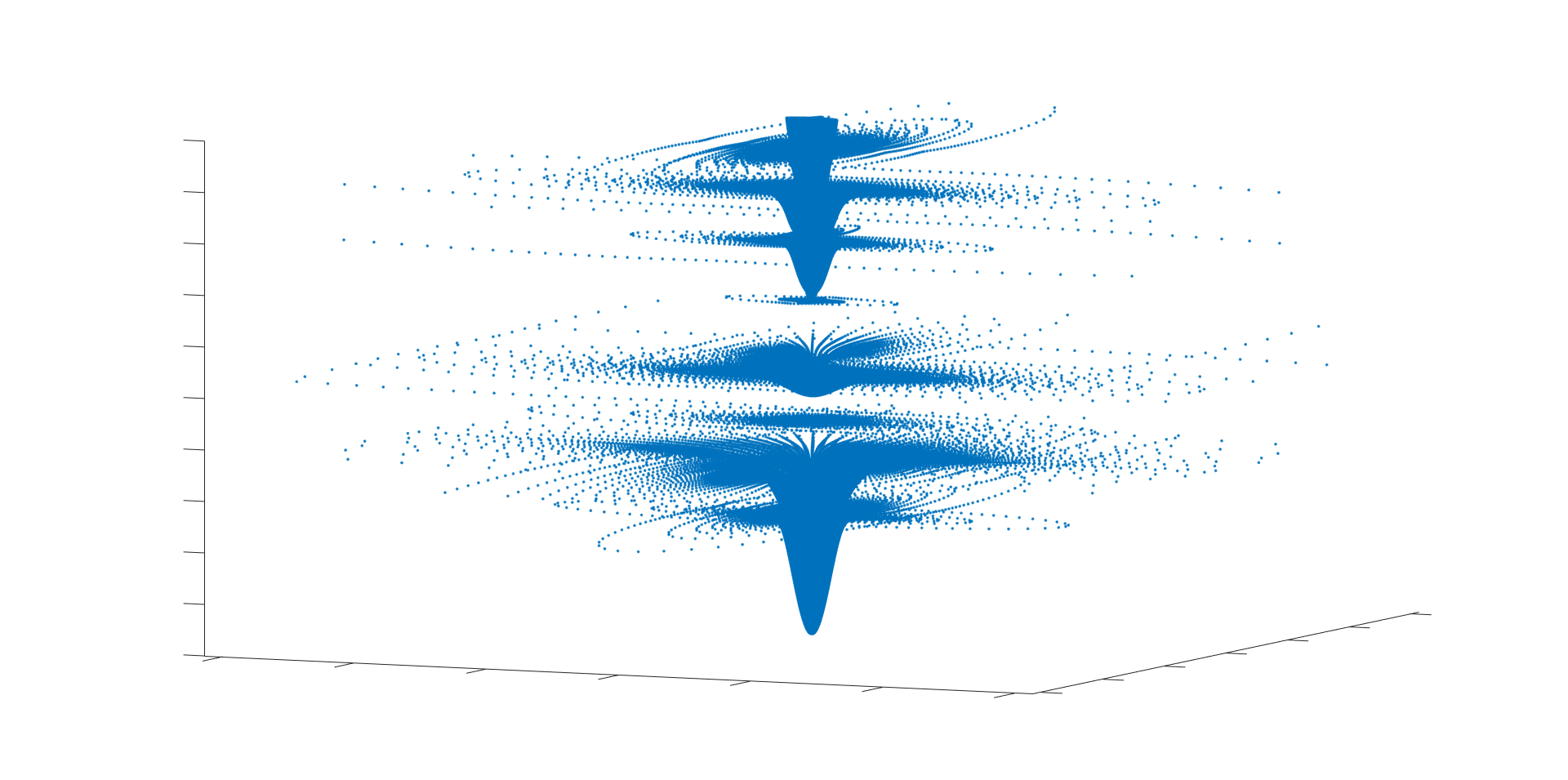}
		\end{subfigure}
		\caption*{$\varepsilon = 0.00001$}
	\end{subfigure}
	\caption{Solutions $(\chi,z)$ to the dispersion relation \eqref{drelation}: $\chi$ is horizontal, $z$ is vertical. Left panel: side view in the direction of $\chi_2;$ right panel: 3D view. 
		The number of dispersion surfaces at every $z$ is the number of non-negative eigenvalues of $\mathcal{B}(z)$. As $\varepsilon\to0,$ the gaps between the surfaces converge to the regions in which $\mathcal{B}(z)$ is negative-definite.}
	\label{figuredispersion}
\end{figure}


\renewcommand{\theequation}{A.\arabic{equation}}
\renewcommand{\thesubsection}{A.\arabic{subsection}}
\renewcommand{\thetheorem}{A.\arabic{theorem}}
\setcounter{theorem}{0}
\setcounter{subsection}{0}
\setcounter{equation}{0}



\nsection{Appendix}



\subsection{Operator theory}

\begin{lemma}
\label{lemmaforinvertibility}
Let $\mathcal{A}$ be a closed, densely defined linear operator on a complex Hilbert space $\mathcal H$, such that
    $\lvert \left\langle \mathcal{A}x,x\right\rangle\rvert \geq C \lVert x\rVert^2, \quad \forall x\in \mathcal{D}(\mathcal{A})$ and $\ker(\mathcal{A}^*) \subset \mathcal{D}(\mathcal{A}).$
Then the inverse $\mathcal{A}^{-1}$ exists and
    $\lVert \mathcal{A}^{-1}\rVert \leq C^{-1}.$
\end{lemma}
\begin{proof}
First, we show that $\overline{\mathcal{R}(\mathcal{A})}=\mathcal{H}=\mathcal{R}(\mathcal{A})$. To this end, take $x \in \ker(\mathcal{A}^*)$. Since $x \in \mathcal{D}(\mathcal{A}),$ one has
\begin{equation*}
    0 = \bigl\lvert \left\langle x,\mathcal{A}^* x\right\rangle\bigr\rvert =\bigl\lvert \left\langle \mathcal{A}x,x\right\rangle\bigr\rvert \geq C \lVert x\rVert^2,
\end{equation*}
and therefore $x = 0$. Therefore,
    $\overline{\mathcal{R}(\mathcal{A})} = \bigl(\ker(\mathcal{A}^*) \bigr)^\perp = \{0\}^\perp = \mathcal{H}.$
Clearly, $\mathcal{A}$ is an injection. Thus $\mathcal{A}^{-1}$ exists. For $y \in \mathcal{R}(\mathcal{A})$ we put $x = \mathcal{A}^{-1}y$. Thus, one has
\begin{equation*}
    \lVert x \rVert^2  = \lVert \mathcal{A}^{-1} y \rVert^2 \leq C^{-1}\bigl\lvert \left\langle y,\mathcal{A}^{-1}y\right\rangle\bigr\rvert \leq C^{-1}\lVert y \rVert \lVert \mathcal{A}^{-1}y \rVert,
\end{equation*}
and the claim follows. 
\end{proof}
\begin{corollary}
\label{boundfrombelowappendix}
Let $\mathcal{A}$ be a closed, densely defined linear operator on a complex Hilbert space $\mathcal{H}$ such that $\mathcal{D}(\mathcal{A}) = \mathcal{D}(\mathcal{A}^*)$. Denote by $\Re \mathcal{A}$ and $\Im \mathcal{A}$ the real and imaginary part of $\mathcal{A}.$  Assume that $\max\bigl\{\bigl\lvert  \langle \Im\mathcal{A}x, x\rangle\bigr\rvert, \bigl\lvert\langle \Re\mathcal{A}x, x\rangle\bigr\rvert\bigr\}\ge C\lVert x\rVert^2$ for all $x\in\mathcal{D}(\mathcal{A})$ for some $C>0.$
Then the inverse $\mathcal{A}^{-1}$ exists and
    $\lVert \mathcal{A}^{-1}\rVert \leq C^{-1}.$
\end{corollary}
\begin{proof}
	The claim follows immediately from Lemma \ref{lemmaforinvertibility}, since
\begin{equation*}
    \bigl\lvert\langle \mathcal{A}x, x\rangle\bigr\rvert = \sqrt{\bigl\lvert\langle \Re \mathcal{A}x, x\rangle \bigr\rvert^2+\bigl\lvert\langle\Im\mathcal{A}x, x\rangle\bigr\rvert^2} \geq \max\bigl\{\bigl\lvert  \langle \Im\mathcal{A}x, x\rangle\bigr\rvert, \bigl\lvert\langle \Re\mathcal{A}x, x\rangle\bigr\rvert\bigr\}\qquad \forall x \in \mathcal{D}(\mathcal{A}).\qedhere\popQED
\end{equation*}
\end{proof}

\subsection{Auxiliary estimates}
Here we state some of the results which we use in the main text. In this part we state various versions of Korn's and trace inequalities. First, we state a special case of the classical trace theorem. 
\begin{proposition}\label{josipapp2} 
There exists $C>0$ such that for every $\vect g \in H^{1/2}(\Gamma;\C^3)$ there is an extension $\vect G \in H_\#^1(Y_{\rm stiff(soft)};\C^3)$ satisfying
   $ \left\lVert \vect G \right\rVert_{H^1(Y_{\rm stiff(soft)};\C^3)} \leq C \left\lVert \vect g \right\rVert_{H^{1/2}(\Gamma;\C^3)}.$
\end{proposition}
Next, we recall a well-known Korn's inequality (see, e.g., \cite{Oleinik}). 
\begin{proposition}
\label{appendixkorn1}
Let $\Omega\subset \R^3$ be a bounded open set with Lipschitz boundary. There exists a constant $C>0$ such that for every $\vect u \in H^1(\Omega;\C^3)$, we have 
\begin{equation}
  \label{korninequality1}
    \lVert \vect u \rVert_{H^1(\Omega;\C^3)} \leq C \left( \Vert\simgrad \vect u\Vert_{L^2(\Omega;\C^{3\times 3})} +\lVert \vect u \rVert_{L^2(\Omega;\C^3)}\right),
\end{equation}
where the constant $C$ depends only on the domain $\Omega$.
\end{proposition}
We use the following version of Korn's inequality, proved by contradiction from Proposition \ref{appendixkorn1}.  
\begin{proposition}
\label{appendixkorn2}
Let $\Omega\subset \R^3$ be a bounded open connected set with Lipschitz boundary and $B \subset \partial \Omega$, $|B|>0$. 
Then there exists a constant $C>0$ such that for every $\vect u \in H^1(\Omega;\C^3)$ we have
\begin{equation}
\label{korninequality2}
    \lVert \vect u \rVert_{H^1(\Omega;\C^3)} \leq C\left( \lVert \simgrad \vect u \rVert_{L^2(\Omega;\C^{3 \times 3})} + \lVert  \vect u \rVert_{L^2(B;\C^3)} \right),
\end{equation}
where the constant $C$ depends only on $\Omega$ and $B$.  
\end{proposition}
\begin{proof}
By the inequality \eqref{korninequality1} we see that it suffices to show that
\begin{equation*}
    \lVert \vect u \rVert_{L^2(\Omega;\C^3)} \leq C\left( \lVert \simgrad \vect u \rVert_{L^2(\Omega;\C^{3 \times 3})} + \lVert  \vect u \rVert_{L^2(B;\C^3)} \right).
\end{equation*}
Suppose the contrary, namely that there is a sequence $(\vect u_k)_{k\in\N}\subset H^1(\Omega; {\mathbb C}^3)$, $\lVert \vect u_k \rVert_{L^2(\Omega;\C^3)}= 1$ such that
\begin{equation*}
    \left( \lVert \simgrad \vect u_k \rVert_{L^2(\Omega;\C^{3 \times 3})} + \lVert  \vect u_k \rVert_{L^2(B;\C^3)} \right)<1/k \quad \forall k\in \N.
\end{equation*}
By \eqref{korninequality1}, the sequence $(\vect u_k)_{k\in \N}$ is bounded in $H^1(\Omega;\C^3)$ and therefore converges (up to a subsequence) weakly in $H^1(\Omega;\C^3)$ and strongly in $L^2(\Omega;\C^3)$ to some $\vect u \in H^1(\Omega;\C^3)$, $\lVert \vect u \rVert_{L^2(\Omega;\C^3)}= 1.$ By lower semicontinuity of the norm, we have
\begin{equation}
\label{simgradzero}
    \lVert \simgrad \vect u \rVert_{L^2(\Omega;\C^{3 \times 3})} \leq \lim_{k \to \infty}\lVert \simgrad \vect u_k \rVert_{L^2(\Omega;\C^{3 \times 3})} = 0.
\end{equation}
Also, since $\vect u_k \weak \vect u$ in $H^1(\Omega;\C^3),$ by trace compactness we infer that  ${\vect u}\vert_B=0.$ By \eqref{simgradzero} one has $\vect u(x) = \vect{A}x + \vect c$ for some skew-symmetric $\vect{A} \in \C^{3 \times 3}$ and $\vect{c} \in \C^3$ (see, e.g., \cite{Oleinik}) Since $|B|>0,$ this implies that $\vect{A}=0$, $\vect{c}=0$. 
\end{proof}
Another useful version of Korn's inequality is also well known and is a direct consequence of \cite[Theorem 2.5]{Oleinik}. 
\begin{proposition}
Let $\Omega$ be a bounded open set with Lipschitz boundary. There exists $C>0$ dependent only on $\Omega$ such that for every $\vect u \in H^1(\Omega;\C^3)$ the estimate
    $\lVert \vect u -  \vect w\rVert_{H^1(\Omega;\C^3)} \leq C ||\simgrad \vect u||_{L^2(\Omega;\C^{3\times 3})}$
holds with $\vect w = Ax + \vect c$, where
\begin{equation*}
\begin{aligned}
    A&= \begin{bmatrix}
    0 & d & a \\
    -d & 0 & b \\
    -a & -b & 0 \\
\end{bmatrix}, \quad \vect c = \begin{bmatrix}
    c_1 \\
    c_2 \\
    c_3
\end{bmatrix}, \quad c_j = \int_{\Omega}(\vect u_j), \quad j=1,2,3, \\[0.6em]
    a&= \int_{\Omega}(\partial_3 \vect u_1 - \partial_1 \vect u_3), \quad  b = \int_{\Omega}(\partial_3 \vect u_2 - \partial_2 \vect u_3), \quad d = \int_{\Omega}(\partial_2 \vect u_1 - \partial_1 \vect u_2),
\end{aligned}
\end{equation*}
\end{proposition}

We proceed with the following simple assertion. 
\begin{lemma} \label{nak42} 
Let $\Omega \subset \R^3$ be a bounded open set. There exist constants $C_0,C_1>0$, independent of $\chi \in Y'$, such that
\begin{equation}
\label{equivalenceofnorms}
    C_0 \left\lVert \vect u \right\rVert_{H^1(\Omega;\C^3)} \leq  \bigl\lVert {\rm e}^{{\rm i}\chi y}\vect u \bigr\rVert_{H^1(\Omega;\C^3)} \leq  C_1 \left\lVert \vect u \right\rVert_{H^1(\Omega;\C^3)} \qquad \forall \chi \in Y', \quad \vect u \in H^1(\Omega;\C^3).
\end{equation}
\end{lemma}
\begin{proof}
We clearly have (due to $|{\rm e}^{i\chi y}|=1$)
\begin{equation*}
    \left\lVert {\rm e}^{{\rm i}\chi y}\vect u\right\rVert_{L^2(\Omega;\C^3)} = \left\lVert \vect u\right\rVert_{L^2(\Omega;\C^3)}, \qquad \left\lVert\nabla({\rm e}^{{\rm i}\chi y}\vect u )\right\rVert_{L^2(\Omega;\C^3)} = \bigl\lVert  \nabla \vect u +  \vect u \otimes {\rm i}\chi\bigr\rVert_{L^2(\Omega;\C^3)}.
\end{equation*}
Now, we calculate
\begin{equation*}
         \bigl\lVert {\rm e}^{{\rm i}\chi y}\vect u\bigr\rVert_{H^1(\Omega;\C^3)}^2 = \left\lVert \vect u\right\rVert_{L^2(\Omega;\C^3)}^2 + \bigl\lVert  \nabla \vect u +  \vect u \otimes {\rm i}\chi\bigr\rVert_{L^2(\Omega;\C^3)}^2 \leq \left(1 + |\chi|^2 \right) \left\lVert \vect u\right\rVert_{L^2(\Omega;\C^3)}^2 + \left\lVert  \nabla  \vect u\right\rVert_{L^2(\Omega;\C^3)}^2 
         \leq C\left\lVert \vect u \right\rVert_{H^1(\Omega;\C^3)}^2.
\end{equation*}
Conversely:
\begin{equation*}
\begin{aligned}
        \left\lVert \vect u \right\rVert_{H^1(\Omega;\C^3)}^2 &= \left\lVert \vect u\right\rVert_{L^2(\Omega;\C^3)}^2 + \left\lVert  \nabla  \vect u\right\rVert_{L^2(\Omega;\C^3)}^2  \leq \left\lVert \vect u\right\rVert_{L^2(\Omega;\C^3)}^2 + \bigl\lVert  \nabla \vect u+  \vect u \otimes {\rm i}\chi\bigr\rVert_{L^2(\Omega;\C^3)}^2 + |\chi|^2\left\lVert \vect u\right\rVert_{L^2(\Omega;\C^3)}^2 \\[0.4em]
        & = \left(1 + |\chi|^2 \right)\left\lVert \vect u\right\rVert_{L^2(\Omega;\C^3)}^2 + \bigl\lVert\nabla \vect u+  \vect u \otimes {\rm i}\chi\bigr\rVert_{L^2(\Omega;\C^3)}^2 \leq C\lVert {\rm e}^{{\rm i}\chi y}\vect u \rVert_{H^1(\Omega;\C^3)}.
\end{aligned}
\end{equation*}
\end{proof}
Using the above lemma as a starting point, we prove five propositions.   
\begin{proposition}\label{josipapp1} 
Let $\Omega$ be a bounded open set with Lipschitz boundary and $B \subset \partial \Omega$, $|B|>0$.  There exists a constant $C,$ depending only on the domain $\Omega,$ such that for every $\vect u \in H^1(\Omega;\C^3)$ and $|\chi|\in Y'$ one has
\begin{equation*}
    \lVert \vect u \rVert_{H^1(\Omega;\C^3)} \leq C\left( \left\lVert \left(\simgrad + {\rm i}X_{\chi}\right) \vect u \right\rVert_{L^2(\Omega;\C^{3 \times 3})} + \lVert  \vect u \rVert_{L^2(B;\C^3)} \right).
\end{equation*}
\end{proposition}
\begin{proof}
Plugging $\vect w = {\rm e}^{i \chi y} \vect u$, $\vect u\in H^1(\Omega;\C^3)$ into the inequality \eqref{korninequality2}, we obtain, by virtue of Lemma \ref{nak42}, 
\begin{equation*}
        \left\lVert \vect u\right\rVert_{H^1(\Omega;\C^3)} \leq \left\lVert \vect w\right\rVert_{H^1(\Omega;\C^3)}
         \leq  C\left( \lVert \simgrad \vect w \rVert_{L^2(\Omega;\C^{3 \times 3})} + \lVert  \vect w \rVert_{L^2(B;\C^3)} \right) 
        =  C\left( \left\lVert \left(\simgrad +{\rm i} X_\chi \right) \vect u \right\rVert_{L^2(\Omega;\C^{3 \times 3})} + \lVert  \vect u \rVert_{L^2(B\BBB;\C^3)} \right),
\end{equation*}
where we have also used \eqref{equivalenceofnorms}.
\end{proof}
\begin{proposition}
	\label{Korn_est}
There exists a constant $C>0$ such that for all $\vect u \in H_\#^1(Y;\C^3)$, $\chi \in Y'\setminus\{0\}$ one has
\begin{align}
\label{estimate1}
    \left\lVert \vect u \right\rVert_{H^1(Y;\C^3)}&\leq \frac{C}{|\chi|}\left\lVert \left(\simgrad +{\rm i}X_{\chi}\right) \vect u \right\rVert_{L^2(Y;\C^{3 \times 3})},\\[0.4em]
    \nonumber
    \left\lVert \nabla \vect u \right\rVert_{L^2(Y;\C^{3 \times 3})}&\leq C\left\lVert \left(\simgrad +{\rm i}X_{\chi}\right) \vect u \right\rVert_{L^2(Y;\C^{3 \times 3})},\\[0.4em]
\label{estimate12}
  \left\lVert \vect u - \int_Y \vect u\right\rVert_{H^1(Y;\C^3)}&\leq C\left\lVert \left(\simgrad +{\rm i}X_{\chi}\right) \vect u \right\rVert_{L^2(Y;\C^{3 \times 3})}.
\end{align}
\end{proposition}
\begin{proof}
For a function $\vect u \in H_\#^1(Y;\C^3)$ we have the Fourier series decomposition
\begin{equation*}
    \vect u = \sum_{k \in \Z^3} a_k {\rm e}^{2\pi{\rm i}k \cdot y}, \quad  \nabla \vect u = \sum_{k\in \Z^3} {\rm e}^{2\pi{\rm i}k \cdot y} a_k \otimes \left(2 \pi{\rm i}k \right),
\end{equation*}
from which, by Parseval's identity,
\begin{equation*}
    \left\lVert \vect u\right\rVert_{L^2(Y;\C^3)}^2 = \sum_{k\in \Z^3} |a_k|^2, \quad \left\lVert \nabla \vect u\right\rVert_{L^2(Y;\C^3)}^2 = \sum_{k\in \Z^3} |2 \pi |^2|a_k \otimes k|^2.
\end{equation*}
It follows that
\begin{equation*}
    \nabla \vect u + \vect u \otimes{\rm i}\chi= \sum_{k\in \Z^3} {\rm e}^{2\pi{\rm i}k \cdot y} a_k \otimes \left(2 \pi{\rm i}k +{\rm i}\chi \right), \quad \left\lVert \nabla \vect u + \vect u \otimes ({\rm i}\chi)\right\rVert_{L^2(Y;\C^3)}^2 = \sum_{k\in \Z^3} |a_k \otimes \left(2 \pi{\rm i}k + {\rm i}\chi \right)|^2,
\end{equation*}
and therefore
\begin{equation*} 
\left\lVert (\sym \nabla+{\rm i}X_{\chi})\vect{u}\right\rVert_{L^2(Y;\C^3)}^2 = \sum_{k\in \Z^3} |a_k \odot \left(2 \pi{\rm i}k + {\rm i}\chi \right)|^2.
\end{equation*} 
Combining this with the inequality $ |\vect{a} \odot \vect{b}|\geq|\vect{a}||\vect{b}|/\sqrt{2},$ we infer \eqref{estimate1}--\eqref{estimate12}.
\end{proof}
\begin{proposition}\label{extension} 
There is an extension operator mapping $\vect u\in H_\#^1(Y_{\rm stiff};\C^3)$ to an element of $H_\#^1(Y;\C^3)$ (for which we keep the same notation $\vect u$) such that
\begin{equation}
\label{extension1}
        \left\lVert \vect u\right\rVert_{H^1(Y;\C^3)} \leq C\left\lVert \vect u\right\rVert_{H^1(Y_{\rm stiff};\C^3)},\qquad\quad 
        \bigl\lVert \bigl(\simgrad +{\rm i}X_{\chi}\bigr) \vect u \bigr\rVert_{L^2(Y;\C^{3 \times 3})} \leq C\bigl\lVert \bigr(\simgrad + {\rm i}X_{\chi}\bigr) \vect u \bigr\rVert_{L^2(Y_{\rm stiff};\C^{3 \times 3})},
\end{equation}
for all $\chi \in Y$, where the constant $C$ depends only on $Y_{\rm stiff}$.
\end{proposition}
\begin{proof}
This is a straightforward consequence of an analogous result for the extension operator applied to quasiperiodic functions. For the related construction, see \cite{Oleinik}.
\end{proof}
\begin{proposition} \label{estimate0000} 
There exists a constant $C>0$ such that for every $\vect u \in H_\#^1(Y_{\rm stiff};\C^3)$, $\chi \in Y'\setminus\{0\}$ we have the following estimate:
\begin{equation}
\label{estimate2}
    \left\lVert \vect u \right\rVert_{H^1(Y_{\rm stiff};\C^3)} \leq C|\chi|^{-1}\bigl\lVert \bigl(\simgrad +{\rm i}X_{\chi}\bigr) \vect u \bigr\rVert_{L^2(Y_{\rm stiff};\C^{3 \times 3})}.
\end{equation}
\end{proposition}
\begin{proof}
By Proposition \ref{extension}, the function $\vect u$ 
is extended to $\vect u \in H_\#^1(Y;\C^3)$. Combining \eqref{estimate1} and \eqref{extension1} yields
\begin{equation*}
    \left\lVert \vect u \right\rVert_{H^1(Y_{\rm stiff};\C^3)} \leq \left\lVert \vect u \right\rVert_{H^1(Y;\C^3)} \leq C|\chi|^{-1}\bigl\lVert \bigl(\simgrad +{\rm i}X_{\chi}\bigr) \vect u \bigr\rVert_{L^2(Y;\C^{3 \times 3})} \leq C|\chi|^{-1}\bigl\lVert \bigl(\simgrad +{\rm i}X_{\chi}\bigr) \vect u \bigr\rVert_{L^2(Y_{\rm stiff};\C^{3 \times 3})}.
\end{equation*}
\end{proof}

\begin{proposition}
	\label{nakk112}
There exists a constant $C>0$ such that for all $\vect u \in H_\#^1(Y_{\rm stiff};\C^3)$, $\chi \in Y'$ one has
\begin{align}
\label{finalestimate1}
    \left\lVert \vect u - \frac{1}{|Y_{\rm stiff}|}\int_{Y_{\rm stiff}} \vect u \right\rVert_{H^1(Y_{\rm stiff};\C^3)} \leq C\bigl\lVert \bigl(\simgrad +{\rm i}X_{\chi}\bigr) \vect u \bigr\rVert_{L^2(Y_{\rm stiff};\C^{3 \times 3})},\\[0.3em]
\label{finalestimate2}
    \left\lVert \vect u - \frac{1}{|\Gamma|}\int_{\Gamma} \vect u \right\rVert_{H^{1/2} (\Gamma;\C^3)} \leq C\bigl\lVert \bigl(\simgrad +{\rm i}X_{\chi}\bigr) \vect u \bigr\rVert_{L^2(Y_{\rm stiff};\C^{3 \times 3})}.
\end{align}
\end{proposition}
\begin{proof} 
The bound \eqref{finalestimate1} is deduced from \eqref{estimate12} and Proposition \ref{extension} similar to how Proposition \ref{estimate0000} was established.  To prove \eqref{finalestimate2}, note first that by \eqref{finalestimate1} and the continuity of traces, one has  	
\begin{equation} \label{nak41} 
\begin{split}
    \left\lVert \vect u - \frac{1}{|\Gamma|}\int_{\Gamma} \vect u \right\rVert_{L^2(\Gamma;\C^3)} \leq
\left\lVert \vect u - \frac{1}{|Y_{\rm stiff}|}\int_{Y_{\rm stiff}} \vect u \right\rVert_{L^2(\Gamma;\C^3)}  & \leq C
 \left\lVert \vect u - \frac{1}{|Y_{\rm stiff}|}\int_{Y_{\rm stiff}} \vect u \right\rVert_{H^1(Y_{\rm stiff};\C^3)} \\[0.4em]
 &\leq C\left\lVert \left(\simgrad +{\rm i}X_{\chi}\right) \vect u \right\rVert_{L^2(Y_{\rm stiff};\C^{3 \times 3})}. 
\end{split}
\end{equation} 
Second, from Proposition \ref{josipapp1}  and the trace inequality, one has 
\begin{equation}
	\label{nak43} 
 \begin{aligned}
    &\left\Vert \vect u- \frac{1}{|\Gamma|}\int_{\Gamma} \vect{u}\right\Vert_{H^{1/2}(\Gamma ;\C^3)}\\[0.4em]
    &\hspace{2cm}\leq C\left( \left\lVert \left(\simgrad +{\rm i}X_{\chi}\right) \left(\vect u - \frac{1}{|\Gamma|}\int_{\Gamma} \vect{u}\right) \right\rVert_{L^2(Y_{\rm stiff};\C^{3 \times 3})} + \left\lVert  \vect u - \frac{1}{|\Gamma|}\int_{\Gamma}  \vect{u} \right\rVert_{L^2(\Gamma;\C^3)} \right) \\[0.4em]
    &\hspace{2cm}\leq C \left( \left\lVert \left(\simgrad +{\rm i}X_{\chi}\right) \vect u  \right\rVert_{L^2(Y_{\rm stiff};\C^{3 \times 3})} + |\chi|\left\lvert   \frac{1}{|\Gamma|}\int_{\Gamma}  \vect{u} \right\rvert_{L^2(\Gamma;\C^3)} + \left\lVert  \vect u - \frac{1}{|\Gamma|}\int_{\Gamma}  \vect{u} \right\rVert_{L^2(\Gamma;\C^3)} \right).
\end{aligned} 
\end{equation}
The bound \eqref{finalestimate2} now follows from \eqref{nak43} by using \eqref{nak41}, \eqref{estimate2}, and the continuity of traces. 
\end{proof}

\subsection{Well-posedness and regularity of elasticity boundary value problems}
Here we state some results regarding the properties of the weak solution to the boundary value problem
\begin{equation}
\label{appendixrobinstongform}
- \div\left(\A \simgrad \vect u \right)  = \vect f \ \ \mbox{ on $\Omega$},\qquad \ \ 
(\A\simgrad \vect u){\vect n}|_{\partial\Omega} + \vect u = \vect g \ \  \mbox{ on $\partial \Omega$}, 
\end{equation}
where $\Omega \subset \R^3$ is a bounded Lipschitz domain with the exterior normal ${\vect n},$ $\vect f \in L^2(\Omega;\C^3)$, $\vect g \in H^{-1/2}(\partial \Omega;\C^3),$ and the tensor of material properties $\A \in L^\infty(\Omega;\R^{3 \times 3\times 3 \times 3})$  satisfies the following assumptions (cf. Assumption \ref{coffassumption}).
\begin{assumption}
	\label{coffassumptionappendix}
	\begin{itemize}
		\item 
		Uniform 
		positive-definiteness and uniform boundedness on symmetric matrices: there exists $\nu>0$ such that
		$\nu|\vect\xi|^2 \leq \A(x)\vect\xi :\vect\xi \leq\nu^{-1}|\vect\xi|^2$ for all $\vect\xi \in \R^{3\times 3}, \vect\xi^\top = \vect\xi,$ $x\in\Omega.$
		\item 
		Material symmetries:
		$[\A]_{ijkl}=[\A]_{jikl}=[\A]_{klij},$ $i,j,k,l\in\{1,2,3\}.$
	\end{itemize}
\end{assumption}
The weak form of this problem is stated in the following definition.
\begin{definition}[Robin boundary problem for elasticity]
	\label{appendixrobinweakform}
	 For given functions $\vect f \in H^{-1}(\Omega;\C^3)$, $\vect g \in H^{-1/2}(\partial \Omega;\C^3),$ find $\vect u \in H^1(\Omega;\C^3)$ such that
	\begin{equation*}
	\int_\Omega\A \simgrad \vect u: \overline{\simgrad \vect v} + \int_{\partial \Omega} \vect u \cdot \overline{\vect v}  = \int_\Omega \vect f \cdot\overline{\vect v} + \int_{\partial \Omega} \vect g\cdot \overline{\vect v}\qquad \forall\vect v \in H^1(\Omega;\C^3).
	\end{equation*}
\end{definition}
Notice that the map $\vect v \mapsto \int_{\partial \Omega}\vect{g} \cdot  \overline{\vect{v}}$ is a bounded linear functional on $H^1(\Omega;\C^3)$. Also, the form
\begin{equation*}
(\vect u, \vect v) \mapsto \int_\Omega\A \simgrad \vect u: \overline{\simgrad \vect v} + \int_{\partial \Omega} \vect u \cdot \overline{\vect v},
\end{equation*}
is coercive on $H^1(\Omega;\C^3)$ by Proposition \ref{appendixkorn2}. Thus by the Lax-Milgram lemma, there exists a unique weak solution $\vect u \in H^1(\Omega;\C^3)$ of \eqref{appendixrobinstongform} in the sense of Definition \ref{appendixrobinweakform}, and
\begin{equation*}
\lVert \vect u \rVert_{H^1(\Omega;\C^3)} \leq C \bigl(\lVert \vect f\rVert_{H^{-1}(\Omega;\C^3)} + \lVert \vect g\rVert_{H^{-1/2}(\partial\Omega;\C^3)}\bigr).
\end{equation*}
 The following two lemmata can be found in, e.g., \cite{McLean}.
\begin{lemma}[Regularity of the solution of Robin problem]
	\label{appendixregularitiyofrobin}
	Let $\Omega \subset \R^3$ be a bounded domain with boundary $\partial\Omega$ of class $C^{1,1}$. Let $\vect f \in L^2(\Omega;\C^3)$, $\vect g \in H^{1/2}(\partial \Omega;\C^3)$ and, in addition to Assumption \ref{coffassumptionappendix}, $\A \in \mathcal{C}^{0,1}(\overline{\Omega};\R^{3 \times 3 \times 3 \times 3})$. Then the unique weak solution to the problem \eqref{appendixrobinstongform}
	belongs to $H^2(\Omega;\C^3),$ and
	\begin{equation*}
		\lVert \vect u\rVert_{H^2(\Omega;\C^3)} \leq C \bigl(\lVert \vect f\rVert_{L^2(\Omega;\C^3)} + \lVert \vect g\rVert_{H^{1/2}(\partial\Omega;\C^3)} \bigr).
	\end{equation*}
\end{lemma}

We are also interested in the solution to the Dirichlet boundary value problem
\begin{equation}
	\label{appendixdirichletstrong}
	\left\{ \begin{array}{ll}
		- \div\left(\A \simgrad \vect u \right)  = \vect f \quad &\mbox{on $\Omega$},\\[0.15em]
		\vect u = \vect g \quad &\mbox{ on $\partial \Omega$}, \end{array} \right.
\end{equation}
where $\vect f \in H^{-1}(\Omega;\C^3)$, $ \vect g \in H^{1/2}(\partial \Omega;\C^3)$. Applying the Lax-Milgram lemma, we infer that it has  a unique weak solution $\vect u \in H^1(\Omega;\C^3)$.
\begin{lemma}[Regularity of the solution of Dirichlet problem]
	\label{appendixregularityofdirichlet}
	Suppose that $\Omega \subset \R^3$ be a bounded domain with $C^{1,1}$ boundary $\partial\Omega,$ $\vect f \in L^2(\Omega;\C^3)$, $\vect g \in H^{3/2}(\partial \Omega;\C^3),$ and $\A \in \mathcal{C}^{0,1}(\overline{\Omega};\R^{3 \times 3 \times 3 \times 3})$. Then the solution to \eqref{appendixdirichletstrong} belongs to $H^2(\Omega;\C^3),$ and 
	\begin{equation*}
		\lVert \vect u\rVert_{H^2(\Omega;\C^3)} \leq C \bigl(\lVert \vect f\rVert_{L^2(\Omega;\C^3)}+ \lVert \vect g\rVert_{H^{3/2}(\partial\Omega;\C^3)} \bigr).
	\end{equation*}
	where the constant $C$ depends only on $\Omega$ and the $\mathcal{C}^{0,1}$-norm of $\mathbb{A}$.  
\end{lemma}

\subsubsection{Trace extension lemma}
A version of the following theorem can be found in \cite{Buffa}.
\begin{theorem}
	\label{theoremtraceextension}
	Let $\Omega \subset \R^3$ be a bounded domain with $C^{0,1}$ boundary $\partial \Omega.$ Let $\vect g_0 \in H^1(\partial \Omega;\C^3)$, $\vect g_1 \in L^2(\partial \Omega;\C^3)$. Then there exists $\vect u \in H^2(\Omega;\C^3)$ such that
			$\partial_{{\vect n}} \vect u  = \vect g_1$ 
			and
			$\vect u = \vect g_0$ on $\partial \Omega$ 
	if and only if 
	\begin{equation}
		\label{thmtraceextensioncondition}
		\nabla_{\partial \Omega} \vect g_0 + \vect g_1 \otimes {\vect n} \in H^{1/2}(\partial \Omega;\C^3),
	\end{equation}
where $\nabla_{\partial \Omega}$ is the tangential gradient.
\end{theorem}
This leads us to the trace extension lemma.
\begin{lemma}
	\label{traceextensionlemma}
	Let $\Omega \subset \R^3$ be a bounded domain with boundary $\partial\Omega$ of class $C^{1,1}$ and $\A \in \mathcal{C}^{0,1}(\overline{\Omega};\R^{ 3 \times 3 \times 3 \times 3})$.     Let $\vect g \in H^{1/2}(\partial \Omega;\C^3)$. Then there exists  $\vect u \in H^2(\Omega;\C^3)$ such that
	\begin{equation}
		\label{traceproblem}
			(\A \simgrad\vect u) {\vect n} = \vect g \quad \mbox{ on $\partial \Omega$, }\quad\qquad 
			\vect u = 0 \quad \mbox{ on $\partial \Omega$. } 
	\end{equation}
\end{lemma}
\begin{proof}
	The first step is to note that for  $\vect u \in H^2(\Omega;\C^3)$ such that $\vect{u}\vert_\Gamma=0$ one has 
	\begin{equation}
		\label{nakk55} 
		\nabla \vect u |_{\partial \Omega} = \partial_{{\vect n}} \vect u|_{\partial \Omega} \otimes {\vect n}.
	\end{equation}
	Indeed, for an arbitrary point on $\partial \Omega$ consider an arbitrary vector ${\vect v} = \alpha_1 {\vect \tau}_1 + \alpha_2 {\vect \tau}_2 + \alpha_3 {\vect n}$,  where ${\vect \tau}_1, {\vect \tau}_2, {\vect n}$ is an orthonormal basis of vectors ${\vect \tau}_1, {\vect \tau}_2$ tangential on $\partial \Omega$ and the vector ${\vect n}$ is normal to $\partial \Omega$. Then
	\begin{equation*}
		\nabla \vect u |_{\partial \Omega} \cdot {\vect v} = \alpha_1 \partial_{{\vect \tau}_1} \vect u|_{\partial \Omega} + \alpha_2 \partial_{{\vect \tau}_2} \vect u|_{\partial \Omega} + \alpha_3 \partial_{{\vect n}} \vect u|_{\partial \Omega} = \alpha_3 \partial_{{\vect n}} \vect u|_{\partial \Omega},
	\end{equation*}
	due to the fact that $\partial_{{\vect \tau}_1} \vect u |_{\partial \Omega} = \partial_{{\vect \tau}_2} \vect u = 0|_{\partial \Omega} $. But, since $\alpha_3 = {\vect v} \cdot {\vect n}$, one has
	\begin{equation*}
		\bigl(\nabla \vect u|_{\partial \Omega}\bigr){\vect v} = \left( {\vect v} \cdot {\vect n}\right) \partial_{{\vect n}} \vect u|_{\partial \Omega} = \bigl(\left( \partial_{{\vect n}} \vect u|_{\partial \Omega} \otimes {\vect n} \right)\bigr){\vect v},
	\end{equation*}
	and \eqref{nakk55} follows.
	With this in hand, the first equation in \eqref{traceproblem} becomes
		$(\A   \left( \partial_{{\vect n}} \vect u|_{\partial \Omega} \odot {\vect n}\right)){\vect n} = \vect g.$
	We proceed by introducing the new variable $\vect \omega := \partial_{{\vect n}} \vect u|_{\partial \Omega}$ and consider the problem of finding $\vect \omega$ such that
	\begin{equation}
		\label{omegaboundaryproblem}
		\bigl(\A  \left( \vect \omega \odot {\vect n}\right)\bigr){\vect n} = \vect g.
	\end{equation}

	The second step of the proof consists in showing that there exists a unique solution $\vect \omega \in H^{1/2}(\partial \Omega;\C^3)$ to  \eqref{omegaboundaryproblem}. Note that this  
	is an algebraic equation for every $x \in \partial \Omega$. 
	For a fixed point $x \in \partial \Omega$, we introduce an operator $L_x :\C^3 \to \C^3$ as follows:
		$L_x \vect \omega :=(\A\left( \vect \omega \odot {\vect n}(x)\right)){\vect n(x)}.$
	It is symmetric and positive definite:
		\begin{align}
		\langle L_x \vect \omega, \vect w \rangle_{\C^3}&= \bigl\langle\bigl(\A  \left( \vect \omega \odot {\vect n}(x)\right)\bigr){\vect n}(x), \vect w \bigr\rangle_{\C^3} = \A    \left( \vect \omega \odot {\vect n}(x)\right) :   \left( \vect w \odot {\vect n}(x)\right) = \langle L_x \vect w, \vect \omega \rangle_{\C^3},\nonumber\\[0.3em]
	\label{nakk60} 
		\langle L_x \vect \omega, \vect \omega \rangle_{\C^3}&= \A    \left( \vect \omega \odot {\vect n}(x)\right) :    \left( \vect \omega \odot {\vect n}(x)\right) \geq \nu| \vect \omega \odot {\vect n}(x)|^2\geq\frac{\nu}{2}  | \vect \omega|^2 |{\vect n}(x)|^2 = \frac{\nu}{2} |\vect \omega|^2,
		\end{align}
	where we have used Assumption \ref{coffassumptionappendix}.
	Thus, the problem
		$L_x \vect \omega(x)=\vect g(x)$
	has a unique solution for a.e. $x \in \partial \Omega$. It follows from \eqref{nakk60} that $\det \left(L_x \right) \geq \alpha >0$, uniformly in $x \in \partial \Omega$. By  $C^{0,1}$ regularity of both the material coefficients and the normal ${\vect n}$, we know that both $\det \left(L_x \right)$ and $\left(\det \left(L_x \right)\right)^{-1}$ are of class $C^{0,1}$. Then, by virtue of Cramer's rule, the function $\vect \omega(x) := \left( L_x \right)^{-1} \vect g(x)$ belongs to $H^{1/2}(\partial \Omega;\C^3)$.
	Thus, we have reduced the problem \eqref{traceproblem} to finding $\vect u \in H^2(\Omega;\C^3)$ such that 
			$\partial_{{\vect n}} \vect u  = \vect \omega$ on $\partial \Omega,$ 
			$\vect u = 0$ on $\partial \Omega$,  
	where $\vect{\omega}$ is the solution of \eqref{omegaboundaryproblem}. To complete the proof, it remains to check the validity of the condition \eqref{thmtraceextensioncondition} and apply  Theorem \ref{theoremtraceextension}, which is possible as
		$\vect \omega \otimes {\vect n} \in H^{1/2}(\partial \Omega;\C^3)$
	due to the $C^{0,1}(\partial\Omega;\R^3)$ regularity of ${\vect n}$.
\end{proof}

\subsection{Auxiliary proofs}
\begin{proof}[Proof of Theorem \ref{Piasymptheorem}]
	The definition of $\Pi_\chi^{\rm stiff}$ can be restated as follows:
		$\Pi_\chi^{\rm stiff} \vect  g := \vect G + \vect  u,$ 
	where $\vect  u \in H_\#^1(Y_{\rm stiff};\C^3),$ $\vect u|_\Gamma = 0$ is such that
	\begin{equation*}
			a_{0,\chi}^{\rm stiff}(\vect u,\vect v) = -a_{0,\chi}^{\rm stiff}(\vect G,\vect v)\qquad \forall \vect v \in H_\#^1(Y_{\rm stiff};\C^3),\quad\vect v\vert_{\Gamma}=0,
	\end{equation*}
	and $\vect G \in H_\#^1(Y_{\rm stiff};\C^3)$ is an extension of $\vect g\in H^{1/2}(\Gamma;\C^3)$ satisfying the bound
		$\left\lVert \vect G\right\rVert_{H^1} \leq C \left\lVert \vect g \right\rVert_{H^{1/2}(\Gamma;\C^3)},$ see Proposition \ref{josipapp2}.
	Recall that $a_{0,\chi}^{\rm stiff}: H_\#^1(Y_{\rm stiff};\C^3)\times H_\#^1(Y_{\rm stiff};\C^3) \to \R$ is the sesquilinear form associated with the operator $\mathcal{A}_{0, \chi}^{\rm stiff}$.
We aim to find an expansion of the solution $\vect u \in H_\#^1(Y_{\rm stiff};\C^3)$,  $\vect{u}\vert_\Gamma=0,$  to
	\begin{equation}
		\label{Piequationdefinition}
		\begin{aligned}
			&\int_{Y_{\rm stiff}} \A_{\rm stiff}\left(\simgrad +{\rm i}X_\chi \right) \vect u(y):  \left(\simgrad +{\rm i}X_\chi \right) \vect v(y) dy\\[0.3em] 
			&=-\int_{Y_{\rm stiff}} \A_{\rm stiff}\left(\simgrad +{\rm i}X_\chi \right) \vect G(y)  :  \left(\simgrad +{\rm i}X_\chi \right) \vect v(y) dy\qquad\forall\vect v\in H_\#^1(Y_{\rm stiff};\C^3),\ \vect v\vert_\Gamma = 0. 
		\end{aligned}
	\end{equation}
	in the form
	\begin{equation}
		\label{expansionforPi}
		\vect u = \vect u_0 + \sum_{n=1}^\infty \vect u_{n},
	\end{equation}
where ${\vect u}_n\vert_\Gamma=0$ for all $n,$ 
	the $H^1$ norm of $\vect u_n$ is of order $|\chi|^n$, while the error of approximation
	\begin{equation*}
		\vect u_{\rm error,n}:= \vect u - \vect u_0 - \sum_{k=1}^n  \vect u_k
	\end{equation*}
	is of order $|\chi|^{n+1}$. 
	The leading-order term $\vect u_0$ corresponds to the harmonic lift in the case $\chi = 0$.
	With this apriori intention, we plug the expansion \eqref{expansionforPi} into \eqref{Piequationdefinition} and equate the terms which would be of order $\mathcal{O}(1)$ to obtain
	\begin{equation*}
		\begin{aligned}
			\int_{Y_{\rm stiff}} \A_{\rm stiff}\simgrad \vect u_0(y) :  \simgrad \vect v(y)dy &= -\int_{Y_{\rm stiff}} \A_{\rm stiff}\simgrad \vect  G(y) :  \simgrad \vect  v(y)dy 
			\qquad\forall \vect v\in H_\#^1(Y_{\rm stiff};\C^3), \quad
			\vect v\vert_\Gamma=0.
		\end{aligned}
	\end{equation*}
	We recognise here the definition of the order-zero lift operator $\Pi_0^{\rm stiff}$, namely
		$\Pi_0^{\rm stiff} \vect g = \vect u_0 + \vect G =:\widetilde{\vect u}.$
	It is clear that
	\begin{equation*}
		\left\lVert \vect u_0 \right\rVert_{H^1(Y_{\rm stiff};\C^3)} \leq C \left\lVert \vect G \right\rVert_{H^1(Y_{\rm stiff};\C^3)} \leq C \left\lVert \vect g \right\rVert_{H^{1/2}(\Gamma;\C^3)}.
	\end{equation*}
	We proceed to define $\vect u_1$ as the solutions to
	\begin{equation}
		\label{definitionofpichi1}
		\begin{aligned}
			&\int_{Y_{\rm stiff}}\A_{\rm stiff}\simgrad \vect u_1(y):  \simgrad \vect v(y)dy\\[0.3em]
			=&-\int_{Y_{\rm stiff}} \A_{\rm stiff}{\rm i} X_\chi  \widetilde{\vect u}(y) :  \simgrad \vect v(y)dy  - \int_{Y_{\rm stiff}} \A_{\rm stiff}\simgrad  \widetilde{\vect u}(y):{\rm i}X_\chi \vect v(y)dy\qquad  
			\forall \vect v\in H_\#^1(Y_{\rm stiff};\C^3),\ \vect v\vert_{\Gamma}=0.
		\end{aligned}
	\end{equation}
	Clearly, one has
		$\left\lVert \vect u_1 \right\rVert_{H^1(Y_{\rm stiff};\C^3)} \leq C |\chi| \left\lVert\vect g \right\rVert_{H^{1/2}(\Gamma;\C^3)}.$
The next-order term of the expansion is defined by
	\begin{equation*}
		\begin{aligned}
			\int_{Y_{\rm stiff}}\A_{\rm stiff}\simgrad  \vect u_2(y):\simgrad \vect v(y)dy
		    =&-\int_{Y_{\rm stiff}} \A_{\rm stiff}{\rm i} X_\chi  \widetilde{\vect u}(y) :{\rm i}X_\chi \vect v(y)dy-\int_{Y_{\rm stiff}} \A_{\rm stiff}{\rm i} X_\chi  \vect u_1(y) :  \simgrad \vect v(y)dy\\[0.3em]
		    &-\int_{Y_{\rm stiff}} \A_{\rm stiff}\simgrad  \vect u_1(y) :{\rm i}X_\chi \vect v(y)dy\qquad\forall \vect v\in H_\#^1(Y_{\rm stiff};\C^3),\ \vect v\vert_{\Gamma}=0.
		\end{aligned}
	\end{equation*}
	We see that
	$\left\lVert \vect u_2 \right\rVert_{H^1(Y_{\rm stiff};\C^3)} \leq C |\chi|^2\left\lVert \vect  g \right\rVert_{H^{1/2}(\Gamma;\C^3)}.$
	Continuing this process by induction for $n \geq 3$,  we define $\vect u_n$ by the recurrence relation
	\begin{equation*}
		\begin{aligned}
			\int_{Y_{\rm stiff}}&\A_{\rm stiff}\simgrad  \vect u_n(y):\simgrad \vect v(y)dy
			=-\int_{Y_{\rm stiff}} \A_{\rm stiff}{\rm i}X_\chi  \vect u_{n-2}(y) :  {\rm i}X_\chi \vect v(y)dy\\[0.3em]
			&-\int_{Y_{\rm stiff}} \A_{\rm stiff}{\rm i}X_\chi  \vect u_{n-1}(y) :  \simgrad \vect v(y)dy 
			-\int_{Y_{\rm stiff}} \A_{\rm stiff}\simgrad  \vect u_{n-1}(y) :{\rm i}X_\chi \vect v(y)dy\qquad
			\forall \vect v\in H_\#^1(Y_{\rm stiff};\C^3),\ \vect v\vert_\Gamma=0,
		\end{aligned}
	\end{equation*}
	so that 
		$\left\lVert \vect u_n \right\rVert_{H^1(Y_{\rm stiff};\C^3)}  \leq C |\chi|^n \left\lVert \vect  g \right\rVert_{H^{1/2}(\Gamma;\C^3)}$ for all $n \in \N.$
	The error term
	\begin{equation*}
		\vect u_{\rm error, n}:= \vect u - \sum_{k = 0}^n \vect u_k, \quad n \in \N,
	\end{equation*}
	satisfies
	\begin{equation*}
		\begin{aligned}
			\int_{Y_{\rm stiff}}&\A_{\rm stiff}\left(\simgrad + {\rm i}X_\chi \right)\vect u_{\rm error,n}(y):  \left(\simgrad +{\rm i}X_\chi \right) \vect v(y) dy 
			\\[0.3em]
			&= -\int_{Y_{\rm stiff}} \A_{\rm stiff}{\rm i}X_\chi  \vect u_{n}(y):{\rm i}X_\chi \vect v(y)dy  -\int_{Y_{\rm stiff}} \A_{\rm stiff}{\rm i}X_\chi  \vect u_{n-1}(y):{\rm i}X_\chi \vect v(y)dy \\[0.3em]
			&-\int_{Y_{\rm stiff}} \A_{\rm stiff}{\rm i}X_\chi  \vect u_{n}(y) :  \simgrad \vect v(y)dy 
			- \int_{Y_{\rm stiff}} \A_{\rm stiff}\simgrad  \vect u_{n}(y):{\rm i}X_\chi \vect v(y)dy\qquad
			\forall \vect v\in H_\#^1(Y_{\rm stiff};\C^3),\ \vect  v\vert_\Gamma= 0.
		\end{aligned}
	\end{equation*}
	It follows that
	$\lVert \vect u_{\rm error,n}\rVert_{H^1(Y_{\rm stiff};\C^3)} \leq C |\chi|^{n+1} \left\lVert \vect g \right\rVert_{H^{1/2}(\Gamma;\C^3)}.$
\end{proof}

\begin{proof}[Proof of Lemma \ref{lemmanormequivalence_f}]

The proof is by contradiction. Assume that
for all $n \in \N$ there exist $\chi_n \in Y'$ and $\vect f_n \in \widehat{\mathcal{H}}_{\chi_n}^{\rm stiff(soft)}$ such that $\lVert \vect f_n\rVert_{H^1(Y_{\rm stiff(soft)};\C^3)}=1$ and $ \lim_{n \to \infty} \lVert \vect f_n\rVert_{L^2(Y_{\rm stiff(soft)};\C^3)}=0$. From this, it is clear that
\begin{equation}
\label{fweak}
    \vect f_n \xrightharpoonup{H^1(Y_{\rm stiff(soft)};\C^3)} 0.
\end{equation}

We proceed by defining 
$\vect g_n:=\vect f_n |_{\Gamma}\in \widehat{\mathcal{E}}_{\chi_n}$, i.e., $\vect f_n =  \widehat{\Pi}_{\chi_n}^{\rm stiff(soft)} \vect g_n.$ Since $\lVert \vect f_n\rVert_{H^1(Y_{\rm stiff(soft)};\C^3)}=1$, due to the continuity  of the trace operator, we conclude that the sequence $(\vect g_n)_{n \in \N}$ is bounded in $H^{1/2}(\Gamma;\C^3)$ and there exists $\vect g \in \mathcal{E}$ such that (on a subsequence)
   $\vect g_n \xrightharpoonup{H^{1/2}(\Gamma;\C^3)} \vect g.$
Furthermore, due to Proposition \ref{propositionpiestimate}, the sequence $\left(\lVert\vect g_n  \rVert_{H^{1/2}(\Gamma;\C^3)} \right)_{n \in \N}$ is bounded below by a $\chi$-independent $C>0:$
\begin{equation}
\label{glowerbound}
    \lVert\vect g_n  \rVert_{H^{1/2}(\Gamma;\C^3)} \geq C.
\end{equation}
However, due to the compactness of the trace operator acting from $H^1(Y_{\rm stiff(soft)};\C^3)$ to $L^2(\Gamma;\C^3)$, from \eqref{fweak} we infer that
    $\vect g_n \xrightarrow{L^2(\Gamma;\C^3)}0,$
hence $\vect g = 0.$ 
Next, we note that, on the one hand, 
\begin{equation*}
    \langle \Lambda_{\chi_n}^{\rm stiff} \vect g_n,\vect g_n\rangle_{L^2(\Gamma;\C^3)} + \langle
\vect g_n , \vect g_n \rangle_{L^2(\Gamma;\C^3)} \leq \left(C|\chi|^2 + 1 \right) \lVert \vect g_n\rVert_{L^2(\Gamma;\C^3)}^2,
\end{equation*}
where we use the fact that $\vect g_n \in \widehat{\mathcal{E}}_{\chi_n}$. On the other hand, one has
\begin{equation*}
    \begin{aligned}
         \langle \Lambda_{\chi_n}^{\rm stiff} \vect g_n,\vect g_n\rangle_{L^2(\Gamma;\C^3)} + \langle
\vect g_n , \vect g_n \rangle_{L^2(\Gamma;\C^3)} &= \lambda_{\chi_n}^{\rm stiff}(\vect g_n, \vect g_n) + \lVert \vect g_n\rVert_{L^2(\Gamma;\C^3)}^2 \\
& \geq C\left( \left\lVert \left(\simgrad + {\rm i}X_{\chi_n} \right)\Pi_{\chi_n}^{\rm stiff} \vect g_n\right\rVert_{L^2(Y_{\rm stiff};\C^3)}^2 + \lVert \vect g_n\rVert_{L^2(\Gamma;\C^3)}^2  \right)\\
& \geq C \left\lVert \Pi_{\chi_n}^{\rm stiff} \vect g_n\right\rVert_{H^1(Y_{\rm stiff};\C^3)}^2 \geq C \left\lVert  \vect g_n\right\rVert_{H^{1/2}(\Gamma;\C^3)}^2,
    \end{aligned}
\end{equation*}
where we use 
the coercivity of the sesquilinear form \eqref{lambda_chi_form_def},
Proposition \eqref{josipapp1}, and the continuity of the trace operator. It follows that 
\begin{equation*}
    \left\lVert  \vect g_n\right\rVert_{H^{1/2}(\Gamma;\C^3)} \leq C\lVert \vect g_n\rVert_{L^2(\Gamma;\C^3)} \to 0,
\end{equation*}
which contradicts \eqref{glowerbound}.
\end{proof}

\section*{Acknowledgements}
KC and AVK have been supported by EPSRC (Grants EP/L018802/2, EP/V013025/1.)
IV and J\v{Z} have been supported by the Croatian Science Foundation under Grant agreement No. IP-2018-01-8904 (Homdirestroptcm) and under 
Grant agreement No. IP-2022-10-5181 (HOMeOS).

\section*{Data availability statement}

This manuscript has no associated data.



\end{document}